\documentclass[10pt, leqno]{article}
\usepackage{amsmath,amsfonts,amsthm,amscd,amssymb}
\numberwithin{equation}{section}
\theoremstyle{plain}
\newtheorem{theo}{Theorem}[section]
\newtheorem{prop}[theo]{Proposition}
\newtheorem{coro}[theo]{Corollary} 
\newtheorem{lemm}[theo]{Lemma}
\theoremstyle{definition}
\newtheorem{defi}[theo]{Definition}
\newtheorem{rema}[theo]{Remark}
\newtheorem{theo-defi}[theo]{Theorem-Definition}
\newtheorem{prop-defi}[theo]{Proposition-Definition}
\newtheorem{rema-defi}[theo]{Remark-Definition}
\newtheorem{exem-defi} [theo]{Example-Definition}

\newtheorem{exem}[theo]{Example}
\newtheorem{conj}[theo]{Conjecture}

\def \al{\alpha}
\def \bet{\beta}
\def \bul{\bullet}
\def \col{\colon}
\def \Del{\Delta}

\def \eps{\epsilon}
\def \Gam{\Gamma}
\def \gam{\gamma}

\def \kap{\kappa}
\def \Lam{\Lambda}
\def \lam{\lambda}

\def \Lo{\Longrightarrow}
\def \lo{\longrightarrow}
\def \lom{\longmapsto}
\def \mab{\mathbb}

\def \Om{\Omega}
\def \om{\omega}
\def \ol{\overline}
\def \os{\overset}
\def \parno{\par\noindent}

\def \sus{\subset}

\def \ul{\underline}
\def \us{\underset}

\def \vp{\varpi}

\def \vpl{\varprojlim}

\def \wh{\widehat}
\def \wt{\widetilde}
\bigskip

\newcommand{\getsfrom}{\ensuremath{
\longleftarrow\kern-.50em\lower.0ex\hbox%
{$\shortmid\,$}}}

\begin{document}
\title{Theory of weights for log convergent cohomologies II: 
the case of a proper SNCL scheme in characteristic $p>0$
}
\author{Yukiyoshi Nakkajima \date{}
\thanks{   
\endgraf}}
\maketitle

\begin{flushright}
\end{flushright}

\par

$${\rm ABSTRACT}$$
\bigskip 
\par\noindent 
For a flat $p$-adic formal family $S$ of log points 
over a complete discrete valuation ring 
${\cal V}$ with perfect residue field of mixed characteristics $(0,p)$ and 
for a simple normal crossing log scheme $X$ over an exact closed 
log subscheme of $S$ defined 
by a nonzero element $\pi$ of the maximal ideal of ${\cal V}$, 
we construct fundamental filtered complexes  
$(A_{\rm conv}(X/S),P)$ and $(A^N_{\rm conv}(X/S),P)$ 
in the convergent topos $(\os{\circ}{X}/\os{\circ}{S})_{\rm conv}$ 
of $\os{\circ}{X}/\os{\circ}{S}$ 
(not in the log convergent topos $(X/S)_{\rm conv}$ of $X/S$) 
by using $(X/S)_{\rm conv}$. 
We prove that 
there exists a canonical isomorphism $(A^N_{\rm conv}(X/S),P)
\os{\sim}{\lo} (A_{\rm conv}(X/S),P)$. 
Because the complex $A_{\rm conv}(X/S)$ is shown to calculate 
the log convergent cohomology sheaf of 
$X/S$, the filtered complex $(A_{\rm conv}(X/S),P)$ produces 
the weight filtration on the log convergent cohomology sheaf 
if $\os{\circ}{X}/\os{\circ}{S}$ is proper. 
The filtered complex $(A_{\rm conv}(X/S),P)$ 
is a $p$-adic analogue of a generalization of Steenbrink-Fujisawa-Nakayama's $\infty$-adic filtered complex. 
The filtered complex $(A^N_{\rm conv}(X/S),P)$ 
is a $p$-adic analogue of a generalization of Rapoport-Zink-Nakayama's $l$-adic filtered complex 
and Steenbrink-Zucker-Fujisawa-Nakayama's $\infty$-adic filtered complex.  
We give a comparison theorem between 
the projection of $(A_{\rm conv}(X/S),P)$ 
in the Zariski topos $\os{\circ}{X}_{\rm zar}$ of $\os{\circ}{X}$ 
and the isozariskian weight-filtered complex 
$(A_{\rm zar}(X/S),P)\otimes^L_{\mab Z}{\mab Q}$ 
of $X/(S,p{\cal O}_S,[~])$ constructed in \cite{nb}, 
where $[~]$ is the canonical PD-structure on  $p{\cal O}_S$. 
\bigskip

$${\bf Contents}$$
\bigskip 
\parno
\S\ref{sec:intro}. Introduction 
\medskip 
\parno 
\S\ref{sec:logcd}. 
Log convergent topoi 
\medskip 
\parno
\S\ref{sec:lll}. Log convergent linearization functors
\medskip 
\parno
\S\ref{sec:vflvc}. Vanishing cycle sheaves in log convergent topoi 
\medskip 
\parno
\S\ref{eclf}. Log convergent linearization functors over the underlying formal scheme 
of a $p$-adic formal family of log points
\medskip 
\parno
\S\ref{sec:lcs}. 
Log convergent linearization functors of SNCL schemes
\medskip 
\parno
\S\ref{sec:rlct}. 
Simplicial log convergent topoi
\medskip 
\parno
\S\ref{sec:mplf}. 
Modified $P$-filtered convergent complexes and $p$-adic semi-purity
\medskip 
\parno
\S\ref{sec:wfcipp}. Convergent Steenbrink filtered complexes  
\medskip 
\parno
\S\ref{sec:wt}.  
Convergent Rapoport-Zink-Steenbrink-Zucker filtered complexes
\medskip 
\parno
\S\ref{sec:fpw}.
The functoriality of weight-filtered convergent and isozariskian complexes 
\medskip 
\parno
\S\ref{sec:bd}. Edge morphisms 
\medskip 
\parno
\S\ref{sec:mn}. 
Monodromy operators
\medskip 
\parno
\S\ref{sec:ct}. 
Comparison theorem 
\medskip
\parno
References

\section{Introduction}\label{sec:intro} 
This paper is a continuation of \cite{nb} and 
a simple normal crossing log version of \cite{nhw}. 
\par 
Inspired by the Weil conjecture and theory of Hodge-Deligne, A.~Grothendieck 
has proposed the motivic theory of weights which realizes  
the theories of weights on various cohomologies of algebraic varieties. 
Supported by his philosophy, in this paper 
we construct the $p$-adic weight spectral sequence of  
a proper SNCL(=simple normal crossing log) scheme over a $p$-adic family of log points
in a way which is philosophically similar (but technically non-similar)
to Rapoport-Zink-Nakyama's construction 
of the $l$-adic weight spectral sequence of 
a proper SNCL variety over a log point in \cite{rz} and \cite{nd}.  
In Rapoport-Zink's construction, the canonical filtration on a certain complex 
has an important role; in our construction,  
the canonical filtration on another complex has an important role. 
Our construction uses the theory of (log) convergent topoi in 
\cite{oc}, \cite{ofo}, \cite{s1}, \cite{s2}, \cite{s3} and \cite{nhw}.  
The use of the theory of the (log) convergent topoi is definitive because the construction 
using the theory of (log) crystalline topoi in \cite{bb}
and \cite{klog1} does not enable us to obtain the $p$-adic version of Rapoport-Zink's construction 
because the $p$-adic purity for the vanishing cycle sheaf obtained 
by the use of the morphism forgetting the log structure of a smooth scheme with 
an SNCD(=simple normal crossing divisor)
does {\it not} hold in the theory of the log crystalline topoi (\cite{nh2}). On the other hand, 
we have proved the $p$-adic purity for the vanishing cycle sheaf of the log scheme
in the log convergent topoi (\cite{nhw}). 
(We shall review these in this Introduction.) 
\par 
In order to explain our construction in more details, 
we recall results in \cite{nb} and \cite{nhw}. First we recall results in \cite{nb}, 
which are ideal generalizations of results in \cite{msemi} (see also \cite{ndw}). 
\par 
For a ringed topos $({\cal T},{\cal A})$, let 
${\rm D}^+{\rm F}({\cal A})$ and $D^+({\cal A})$ be the derived category of 
bounded below filtered complexes of ${\cal A}$-modules and the derived category of 
bounded below complexes of ${\cal A}$-modules, respectively. 
\par
For a log (formal) scheme $Y$, let $\os{\circ}{Y}$ be the 
underlying (formal) scheme of $Y$. 
Let $S$ be a $p$-adic formal family of log points defined in \cite{nb}; 
locally on $S$, $S$ is isomorphic to a log $p$-adic formal scheme 
$(\os{\circ}{S}, {\mab N}\oplus {\cal O}_S^*\lo {\cal O}_S)$, 
where the morphism ${\mab N}\oplus {\cal O}_S^*\lo {\cal O}_S$ 
is defined by the morphism
$(n,a)\lom 0^na$ $(n\in {\mab N}, a\in {\cal O}_S^*)$, where $0^n=0\in {\cal O}_S$ 
for $n\not =0$ and $0^0:=1\in {\cal O}_S$. 
Assume that the $p$-adic formal scheme 
$\os{\circ}{S}$ is flat over ${\rm Spf}({\mab Z}_p)$. 
Let $(S,{\cal I},\gam)$ be a $p$-adic formal PD-family of log points 
($S$ is a $p$-adic formal family of log points and 
${\cal I}$ is a quasi-coherent $p$-adic PD-ideal sheaf of ${\cal O}_S$ 
with PD-structure $\gam$). 
Let $S_0$ be an exact closed log subscheme of $S$ defined by ${\cal I}$.  
Let $X/S_0$ be a (not necessarily proper) 
SNCL scheme with structural morphism 
$f\col X\lo S_0\os{\sus}{\lo} S$. 
Let $\{\os{\circ}{X}_{\lam}\}_{\lam \in \Lam}$ be 
the set of smooth components of $X/S_0$ defined in \cite{nb}.  
(When $\os{\circ}{S}_0$ is the spectrum of a field of characteristic $p>0$, 
$\{\os{\circ}{X}_{\lam}\}_{\lam \in \Lam}$ can be taken as 
the set of the irreducible components of $\os{\circ}{X}$.) 
For a nonnegative integer $k$, let 
$$\os{\circ}{X}{}^{(k)}:=
\coprod_{\{\{\lam_0,\ldots, \lam_k\}~\vert \lam_i\in \Lam, \lam_i\not=\lam_j (i\not=j)\}} 
\os{\circ}{X}_{\lam_0}\cap \cdots \cap \os{\circ}{X}_{\lam_k}$$ 
be a scheme over $\os{\circ}{S}_0$ well-defined in \cite{nb}. 
Let $a^{(k)} \col \os{\circ}{X}{}^{(k)}\lo \os{\circ}{X}$ be the natural morphism. 
Let $\os{\circ}{f}{}^{(k)}\col \os{\circ}{X}{}^{(k)}\lo \os{\circ}{S}$ be the structural morphism. 
Let $\eps_{X/S} \col X\lo \os{\circ}{X}$ be a morphism forgetting the log structure of $X$ 
over a morphism $S\lo \os{\circ}{S}$ forgetting the log structure of $S$.  
Let $\eps^{\rm crys}_{X/S} \col (X/S)_{\rm crys}\lo (\os{\circ}{X}/\os{\circ}{S})_{\rm crys}$ 
be a morphism of (log) crystalline topoi induced by $\eps_{X/S}$. 
For fine log schemes $Z$ over $S_0$ and $W$ over $\os{\circ}{S}_0$, 
let $u^{\rm crys}_{Z/S} \col (Z/S)_{\rm crys}\lo Z_{\rm zar}$ and 
$u^{\rm crys}_{W/\os{\circ}{S}} \col 
(W/\os{\circ}{S})_{\rm crys}\lo W_{\rm zar}$ be the canonical projections, respectively. 
Let $E$ be a flat quasi-coherent crystal of ${\cal O}_{\os{\circ}{X}/\os{\circ}{S}}$-modules. 
In \cite{nb} we have constructed a filtered complex 
$(A_{\rm zar}(X/S,E),P)\in {\rm D}^+{\rm F}(f^{-1}({\cal O}_S))$ 
such that there exist the following canonical isomorphisms 
\begin{align*} 
\theta \col Ru^{\rm crys}_{X/S*}(\eps^{{\rm crys}*}_{X/S}(E))\os{\sim}{\lo} A_{\rm zar}(X/S,E)
\tag{1.0.1}\label{ali:ixuoa} 
\end{align*} 
and 
\begin{align*} 
{\rm gr}^P_kA_{\rm zar}(X/S,E)\os{\sim}{\lo} \bigoplus_{j\geq \max \{-k,0\}} 
&a^{(2j+k)}_* 
(Ru^{\rm crys}_{\os{\circ}{X}{}^{(2j+k)}/\os{\circ}{S}*}
(a^{(2j+k)*}_{{\rm crys}}(E)) \tag{1.0.2}\label{ali:ruovp}\\
&\otimes_{\mab Z}\vp_{\rm crys}^{(2j+k)}(\os{\circ}{X}/\os{\circ}{S}))[-2j-k]
\end{align*}
in $D^+(f^{-1}({\cal O}_S))$.  
Here $\vp_{\rm crys}^{(m)}(\os{\circ}{X}/\os{\circ}{S})$ $(m\in {\mab N})$ 
is the crystalline orientation sheaf associated to 
the set $\{\os{\circ}{X}_{\lam}\}_{\lam \in \Lam}$. That is, 
$\vp_{\rm crys}^{(m)}(\os{\circ}{X}/\os{\circ}{S})$ is 
the extension to $(\os{\circ}{X}{}^{(m)}/\os{\circ}{S})_{\rm crys}$ 
of the direct sum of  
$\os{m}{\bigwedge}{\mab Z}^E_{\os{\circ}{X}_{\lam_0}\cap \cdots \cap \os{\circ}{X}_{\lam_m}}$ 
in the Zariski topos $\os{\circ}{X}{}^{(m)}_{\rm zar}$ of $\os{\circ}{X}{}^{(m)}$ 
for the subsets $E=\{\os{\circ}{X}_{\lam_0}, \ldots, \os{\circ}{X}_{\lam_m}\}$'s of 
$\{\os{\circ}{X}_{\lam}\}_{\lam \in \Lam}$.  
We call the filtered complex $(A_{\rm zar}(X/S,E),P)$ the 
{\it zariskian Steenbrink complex} of $E$. 
When $E$ is trivial, we call 
$(A_{\rm zar}(X/S,E),P)$ the {\it zariskian Steenbrink complex} of $X/S$. 
If $E$ is an $F$-crystal on the crystalline site ${\rm Crys}(\os{\circ}{X}/\os{\circ}{S})$ 
of $\os{\circ}{X}/\os{\circ}{S}$, 
then the isomorphism (\ref{ali:ruovp}) becomes the following isomorphism 
\begin{align*} 
{\rm gr}^P_kA_{\rm zar}(X/S,E)\os{\sim}{\lo} \bigoplus_{j\geq \max \{-k,0\}} 
&a^{(2j+k)}_* 
(Ru^{\rm crys}_{\os{\circ}{X}{}^{(2j+k)}/\os{\circ}{S}*}
(a^{(2j+k)*}_{{\rm crys}}(E)) \tag{1.0.3}\label{ali:ruoavp}\\
&\otimes_{\mab Z}\vp_{\rm crys}^{(2j+k)}(\os{\circ}{X}/\os{\circ}{S}))(-j-k)[-2j-k]. 
\end{align*}
One may think that $(-j-k)$ means the usual Tate twist 
and that there arises no new problem. However this is not so. 
Let $F_{\os{\circ}{S}_0}\col \os{\circ}{S}_0\lo \os{\circ}{S}_0$ 
be the absolute Frobenius endomorphism 
of $\os{\circ}{S}_0$. 
First we have to consider the relative Frobenius morphism 
$F_{S_0/\os{\circ}{S}_0}\col  S_0\lo S_0\times_{\os{\circ}{S}_0,F_{\os{\circ}{S}_0}}\os{\circ}{S}_0=:S^{[p]}_0$ 
of $S_0$ over $\os{\circ}{S}_0$ as in \cite{ofo}. 
Then we have to consider a new fine log formal scheme 
$S_0^{[p]}(S)$, which has been considered 
in \cite{nb} as a special case. 
The log formal scheme $S_0^{[p]}(S)$ is, by definition, as follows: 
the underlying formal scheme of 
$S_0^{[p]}(S)$ is $\os{\circ}{S}$ and the log structure of 
$S_0^{[p]}(S)$ is the sub-log structure of $S$ 
such that the isomorphism 
$M_S/{\cal O}_S^*\os{\sim}{\lo} M_{S_0}/{\cal O}_{S_0}^*$ induces 
the following isomorphism 
\begin{align*} 
M_{S_0^{[p]}(S)}/{\cal O}_S^*\os{\sim}{\lo} 
{\rm Im}(F_{S_0/\os{\circ}{S}_0}^*\col 
F_{S_0/\os{\circ}{S}_0}^*(M_{S_0\times_{\os{\circ}{S}_0,F_{\os{\circ}{S}_0}}\os{\circ}{S}_0})
\lo M_{S_0})/{\cal O}_{S_0}^*. \tag{1.0.4}\label{ali:rus0avp}
\end{align*} 
Then we have an obvious morphism $(S,{\cal I},\gam)
\lo (S_0^{[p]}(S),{\cal I},\gam)$ of log PD-formal schemes. 
Now the Tate-twist $(-j-k)$ means the Tate twist 
with respect to 
the morphism $X\lo X\times_{\os{\circ}{S}_0,F_{\os{\circ}{S}_0}}\os{\circ}{S}_0$ 
induced by the absolute Frobenius endomorphism $F_X\col X\lo X$ of $X$
over $(S,{\cal I},\gam) \lo (S_0^{[p]}(S),{\cal I},\gam)$. 
Note also that the actions of the Frobenius endomorphisms  
on $p$-adic Steenbrink complexes and the $p$-adic weight spectral sequences 
of proper SNCL schemes over the log point of 
a perfect field of characteristic $p>0$ in some contexts 
are mistaken in published papers 
except \cite{ndw}, \cite{mat} and \cite{nb} (as far as I know)
because the actions of the Frobenii on the non-zero local section ``$d\log t$'' 
are not considered in the $p$-adic Steenbrink double complexes in almost all papers,  
where $t$ is a ``local log parameter'' of the log point. 
As a consequence of (\ref{ali:ixuoa}) and (\ref{ali:ruoavp}), 
we obtain the following spectral sequence: 
\begin{align*} 
E_1^{-k,q+k}=&\bigoplus_{j\geq \max \{-k,0\}} 
R^{q-2j-k}f^{\rm crys}_{\os{\circ}{X}{}^{(2j+k)}/\os{\circ}{S}*}(a^{(2j+k)*}_{{\rm crys}}(E)
\otimes_{\mab Z}\vp_{\rm crys}^{(2j+k)}(\os{\circ}{X}/\os{\circ}{S})))(-j-k)
\tag{1.0.5}\label{ali:ixsxa}\\
&\Lo R^qf^{\rm crys}_{X/S*}(E)
\end{align*}
if $E$ is an $F$-crystal on ${\rm Crys}(\os{\circ}{X}/\os{\circ}{S})$, 
where $f^{\rm crys}_{\os{\circ}{X}{}^{(2j+k)}/\os{\circ}{S}}
:=\os{\circ}{f}{}^{(2j+k)}\circ u^{\rm crys}_{\os{\circ}{X}{}^{(2j+k)}/\os{\circ}{S}}$ 
and $f^{\rm crys}_{X/S}:=f\circ u^{\rm crys}_{X/S}$. 
In \cite{nb} we have called this spectral sequence 
the $p$-adic weight spectral sequence of $X/S$ 
in the case where $\os{\circ}{f}$ is  proper and $E$ is trivial.

\par 
Next we recall main results in \cite{nhw}. This motivates us to 
gives a nice technical framework in this paper. 
\par 
Let ${\cal V}$ be a complete discrete valuation ring 
of mixed characteristics $(0,p)$ 
whose residue field $\kap$ is perfect. 
Let $\pi$ be a nonzero element of the maximal ideal of ${\cal V}$. 
Let $S$ be a $p$-adic formal ${\cal V}$-scheme in the sense of \cite{of}, 
that is, a noetherian formal scheme over ${\rm Spf}({\cal V})$ 
with the $p$-adic topology which is topologically of finite type 
over ${\rm Spf}({\cal V})$. 
Let $S_1$ be a closed subscheme of $S$ defined by 
the ideal sheaf $\pi{\cal O}_S$ of ${\cal O}_S$. 
Let $Y$ be a smooth scheme over $S_1$ and 
let $D$ be a relative SNCD(=simple normal crossing divisor) 
on $Y/S_1$. 
Stimulated by \cite{klog1} and \cite{fao}, 
we have constructed an 
fs(=fine and saturated) log structure 
$M(D)$ on $Y_{\rm zar}$ 
with morphism $M(D) \lo ({\cal O}_Y,*)$ in \cite{nh2}. 
If there exists the following cartesian diagram 
\begin{equation*}  
\begin{CD} 
D\cap V @>{\subset}>> V \\ 
@VVV @VVV \\ 
\ul{\rm Spec}_{S_1}
({\cal O}_{S_1}[x_1,\ldots,x_d]/(x_1\cdots x_a))
@>{\subset}>> 
\ul{\rm Spec}_{S_1}({\cal O}_{S_1}[x_1,\ldots,x_d]) 
\end{CD}
\end{equation*}
($a,d\in {\mab Z}_{\geq 1}$ $a\leq d$)
for an open subscheme $V$ of $Y$, 
where the right vertical morphism 
in the diagram above is \'{e}tale, 
then 
$$(M(D)\vert_V \lo ({\cal O}_V,*))
=(({\cal O}_V^*x_1^{\mab N}\cdots x_a^{\mab N},*) 
\os{\sus}{\lo} ({\cal O}_V,*)).$$ 
We denote the log scheme $(Y,M(D))$ simply by $(Y,D)$.  
\par 
Let $g\col Y \lo S_1 \os{\sus}{\lo} S$ be the structural morphism. 
Let $({Y/S})_{\rm conv}$ be the convergent topos of $Y/S$ defined in \cite{nhw},  
which is the relative version of the convergent topos in \cite{oc}. 
Let ${\cal K}_{Y/S}$ be the isostructure sheaf in $({Y/S})_{\rm conv}$.  
Let $({(Y,D)/S})_{\rm conv}(=
({(Y,D)/S})_{{\rm conv},{\rm zar}})$ be the log convergent topos 
of $(Y,D)/S$ which is the relative version of 
the log convergent topos in \cite{s2}.  
Let 
$$\eps^{\rm conv}_{(Y,D)/S} \col 
(({(Y,D)/S})_{\rm conv},{\cal K}_{(Y,D)/S}) 
\lo (({Y/S})_{\rm conv},{\cal K}_{Y/S})$$ 
be a morphism of ringed topoi forgetting the log structure $M(D)$.  
Set ${\cal K}_S:={\cal O}_S\otimes_{\mab Z}{\mab Q}$. 
Let 
$$u^{\rm conv}_{Y/S} \col 
(({Y/S})_{\rm conv},{\cal K}_{Y/S})
\lo ({Y}_{\rm zar},g^{-1}({\cal K}_S))$$ 
be the canonical projection. 
In \cite{nhw} we have constructed 
the following filtered complex
$$(C_{\rm conv}({\cal K}_{(Y,D)/S}),P)\in 
{\rm D}^+{\rm F}({\cal K}_{Y/S}).$$   
If there exists an immersion $(Y,D)\os{\sus}{\lo}{\cal Q}$ 
into a log smooth scheme over $S$, 
then 
$$(C_{\rm conv}({\cal K}_{(Y,D)/S}),P)
=(L^{\rm conv}_{\os{\circ}{\cal Q}{}^{\rm ex}/S}(\Om^{\bul}_{{\cal Q}^{\rm ex}/S}\otimes_{\mab Z}{\mab Q}),
\{L^{\rm conv}_{\os{\circ}{\cal Q}{}^{\rm ex}/S}(P_k\Om^{\bul}_{{\cal Q}^{\rm ex}/S}\otimes_{\mab Z}{\mab Q})\}_{k\in {\mab Z}})\in 
{\rm D}^+{\rm F}({\cal K}_{Y/S}),$$ 
where ${\cal Q}^{\rm ex}$ is the exactification of the immersion 
$(Y,D)\os{\sus}{\lo}{\cal Q}$ and 
$\{P_k\}_{k\in {\mab Z}}$ is the usual filtration on  $\Om^{\bul}_{{\cal Q}^{\rm ex}/S}$ 
obtained by counting the numbers of log poles of local sections of 
$\Om^{\bul}_{{\cal Q}^{\rm ex}/S}$
and $L^{\rm conv}_{\os{\circ}{\cal Q}{}^{\rm ex}/S}$ is the convergent linearization functor for 
coherent ${\cal O}_{{\cal Q}^{\rm ex}}\otimes_{\mab Z}{\mab Q}$-modules with respect to 
the immersion $Y\os{\sus}{\lo} \os{\circ}{\cal Q}{}^{\rm ex}$ over $S$. 
We have proved that the natural morphism 
$L^{\rm conv}_{\os{\circ}{\cal Q}{}^{\rm ex}/S}
(P_k\Om^{\bul}_{{\cal Q}^{\rm ex}/S}\otimes_{\mab Z}{\mab Q})
\lo L^{\rm conv}_{\os{\circ}{\cal Q}{}^{\rm ex}/S}(\Om^{\bul}_{{\cal Q}^{\rm ex}/S}\otimes_{\mab Z}{\mab Q})$
is indeed injective in \cite{nhw}, while we have proved that 
the integral log crystalline analogue of this morphism is {\it not} injective in \cite{nh2}.    
Hence we cannot consider the integral log crystalline analogue 
``$(C_{\rm crys}({\cal O}_{(Y,D)/S}),P)$'' of 
$(C_{\rm conv}({\cal K}_{(Y,D)/S}),P)$. 
\par 
In \cite{nhw} we have proved the following three theorems:

\begin{theo}[{\bf $p$-adic purity}]\label{theo:pap} 
Let $\{D_{\lam}\}_{\lam \in \Lam}$ be the set of the smooth components of $D/S_1$ defined in 
{\rm \cite{nh2}}. 
For a nonnegative integer $k$, set 
$$D^{(k)}:=
\coprod_{\{\{\lam_1,\ldots, \lam_k\}~\vert \lam_i\in \Lam, \lam_i\not=\lam_j (i\not=j)\}} 
D_{\lam_1}\cap \cdots \cap D_{\lam_k}$$ 
be a scheme over $\os{\circ}{S}_1$ well-defined in {\rm \cite{nh2}}. 
Let $b^{(k)}\col D^{(k)}\lo Y$ be the natural morphism. 
Then there exists a canonical isomorphism 
\begin{equation*} 
R^k\eps^{\rm conv}_{(Y,D)/S*}({\cal K}_{(Y,D)/S}) \os{\sim}{\lo} 
b^{(k)}_{{\rm conv}*}({\cal K}_{D^{(k)}/S}
\otimes_{\mab Z}
\vp^{(k)}_{\rm conv}(D/S))
\quad (k\in {\mab N}),
\tag{1.1.1}\label{eqn:vpcs}
\end{equation*}
where $\vp^{(k)}_{\rm conv}(D/S)$ is the convergent  
orientation sheaf associated to the set $\{D_{\lam}\}_{\lam \in \Lam}$. 
\end{theo} 

On the other hand, in \cite{nh2} we have proved that 
\begin{equation*} 
R^k\eps^{\rm crys}_{(Y,D)/S*}({\cal O}_{(Y,D)/S}) \not=
b^{(k)}_{{\rm conv}*}({\cal O}_{D^{(k)}/S}
\otimes_{\mab Z}
\vp^{(k)}_{\rm crys}(D/S))
\quad (k\in {\mab N})
\end{equation*}
in general in the case where $\pi{\cal O}_S$ has a PD-structure. 
Here $\vp^{(k)}_{\rm crys}(D/S)$ is the crystalline   
orientation sheaf associated to the set $\{D_{\lam}\}_{\lam \in \Lam}$.

\begin{theo}[{\bf Comparison theorem between $(C_{\rm conv},P)$ 
and $(C_{\rm conv},\tau)$}]\label{theo:cteck} 
Let $\tau$ be the canonical filtration on a complex. 
Then the following canonical morphism 
\begin{equation*}
(C_{\rm conv}({\cal K}_{(Y,D)/S}),\tau)
{\lo} 
(C_{\rm conv}({\cal K}_{(Y,D)/S}),P)
\tag{1.2.1}\label{eqn:excrs}
\end{equation*}
is an isomorphism.  
\end{theo} 

\par 
Set 
$$(C_{\rm iso{\textrm -}zar}({\cal K}_{(Y,D)/S}),P)
:=
Ru^{\rm conv}_{Y/S*}((C_{\rm conv}({\cal K}_{(Y,D)/S}),P))
\in {\rm D}^+{\rm F}(g^{-1}({\cal K}_S)).$$ 

\begin{theo}[{\bf Comparison theorem between $(C_{\rm iso{\textrm -}zar},P)$ 
and $(C_{\rm zar},P)\otimes^L_{\mab Z}{\mab Q}$}]\label{theo:izaz} 
Let $(C_{\rm zar}({\cal O}_{(Y,D)/S}),P)\in 
{\rm D}^+{\rm F}(g^{-1}({\cal O}_S))$ be the filtered complex 
constructed in {\rm \cite{nh2}}. 
Then there exists a canonical isomorphism 
\begin{equation*}
(C_{\rm iso{\textrm -}zar}({\cal K}_{(Y,D)/S}),P) 
\os{\sim}{\lo} 
(C_{\rm zar}({\cal O}_{(Y,D)/S}),P)\otimes^L_{\mab Z}{\mab Q}.
\tag{1.3.1}\label{eqn:crcs}
\end{equation*} 
Consequently 
$(C_{\rm conv}({\cal K}_{(Y,D)/S}),P)\in  {\rm D}^+{\rm F}({\cal K}_{Y/S})$ 
is an ``object above'' 
$(C_{\rm zar}({\cal O}_{(Y,D)/S}),P)\otimes^L_{\mab Z}{\mab Q}
\in {\rm D}^+{\rm F}(g^{-1}({\cal K}_S))$. 
$($The filtered complex $(C_{\rm conv}({\cal K}_{(Y,D)/S}),P)$ 
lives in a ``higher world'' ${\rm D}^+{\rm F}({\cal K}_{Y/S})$ 
than ${\rm D}^+{\rm F}(g^{-1}({\cal K}_S)).)$ 
\end{theo}
\par
Now we can state the main results in this paper as follows. 
\par 
Let ${\cal V}$ be as above. 
Let $S$ be a $p$-adic formal family of log points over ${\cal V}$. 
Let $X/S_1$ be an SNCL scheme 
and let $f\col X\lo S_1\os{\sus}{\lo} S$ be 
the structural morphism. 
Let ${\cal K}_{X/S}$ and ${\cal K}_{\os{\circ}{X}/\os{\circ}{S}}$ 
be the isostructure sheaves in $({X/S})_{\rm conv}$ and 
$({\os{\circ}{X}/\os{\circ}{S}})_{\rm conv}$, respectively. 
Let $E$ be an isocrystal in $(\os{\circ}{X}/\os{\circ}{S})_{\rm conv}$. 
(In this paper an isocrystal is always assumed to be coherent as in \cite{bprep}.)
In this paper we construct the following two filtered complexes 
$$(A_{\rm conv}(X/S,E),P)\in {\rm D}^+{\rm F}({\cal K}_{\os{\circ}{X}/\os{\circ}{S}}),$$ 
$$(A^N_{\rm conv}(X/S,E),P)\in {\rm D}^+{\rm F}({\cal K}_{\os{\circ}{X}/\os{\circ}{S}}).$$ 
(We prove that these filtered complexes are canonically isomorphic.  
The notation $N$ in the latter complex stands for the notation of 
the ``logarithm of the local monodromy''.)
We also define the following filtered complex by using the projection 
$u^{\rm conv}_{\os{\circ}{X}/\os{\circ}{S}}\col ((\os{\circ}{X}/\os{\circ}{S})_{\rm conv},
{\cal K}_{\os{\circ}{X}/\os{\circ}{S}}) \lo (\os{\circ}{X}_{\rm zar},f^{-1}({\cal K}_S))$: 
$$(A_{\rm iso{\textrm -}zar}(X/S,E),P):=
Ru^{\rm conv}_{\os{\circ}{X}/\os{\circ}{S}*}((A_{\rm conv}(X/S,E),P))\in 
{\rm D}^+{\rm F}(f^{-1}({\cal K}_S)).$$ 
We call the filtered complexes $(A_{\rm conv}(X/S,E),P)$ and 
$(A^N_{\rm conv}(X/S,E),P)$ the 
{\it convergent Steenbrink complex} of $E$ and 
the 
{\it convergent Rapoport-Zink-Steenbrink-Zucker complex} of $E$, respectively. 
We call $(A_{\rm iso{\textrm -}zar}(X/S,E),P)$ 
the 
{\it isozariskian Steenbrink complex} of $E$. 
When $E$ is trivial, we call 
$(A_{\rm conv}(X/S,E),P)$, $(A^N_{\rm conv}(X/S,E),P)$ 
 and $(A_{\rm iso{\textrm -}zar}(X/S,E),P)$ 
the {\it convergent Steenbrink complex} of $X/S$, 
the {\it convergent Rapoport-Zink-Steenbrink-Zucker complex} of $X/S$ 
and the {\it isozariskian Steenbrink complex} of $X/S$, respectively. 
\par 
The filtered complex 
$(A_{\rm iso{\textrm -}zar}(X/S,E),P)$ is a generalization of the filtered complex 
in \cite{gkwf} by the use of the different language in the case 
where $\os{\circ}{X}/\os{\circ}{S}_1$ is proper, 
though the key notion ``the derived category of bounded below filtered complexes'' does not appear 
in [loc.~cit.]. 
(Consequently the filtered complex nor the the weight spectral sequence in [loc.~cit.] 
is not well-defined.)  
Usually the $p$-adic filtered Steenbrink complexes are defined 
as filtered complexes in the Zariski topos 
$\os{\circ}{X}_{\rm zar}$
(e.g., \cite{msemi}, \cite{gkwf}, \cite{ndw}, \cite{nb}).  
In this paper we define fundamental filtered complexes 
in a ``higher'' category ${\rm D}^+{\rm F}({\cal K}_{\os{\circ}{X}/\os{\circ}{S}})$ 
than the category ${\rm D}^+{\rm F}(f^{-1}({\cal K}_{S}))$ as in \cite{nhw}. 
(I have learnt this from A.~Shiho.) 
This formulation enables us to obtain a kind of the $p$-adic ``semi-purity''
(see (\ref{theo:sp}) below for this). 
\par 
Let $\eps_{X/S} \col X\lo \os{\circ}{X}$ be a morphism 
forgetting the log structure of $X$ over the morphism 
$S\lo \os{\circ}{S}$ forgetting the log structure of $S$. 
Let $\eps_{X/\os{\circ}{S}}\col X\lo \os{\circ}{X}$ 
be a morphism forgetting the log structure of $X$
over $\os{\circ}{S}$ (not over $S\lo \os{\circ}{S}$). 
Let 
$$\eps^{\rm conv}_{X/S} \col ((X/S)_{\rm conv},{\cal K}_{X/S})\lo 
((\os{\circ}{X}/\os{\circ}{S})_{\rm conv},{\cal K}_{\os{\circ}{X}/\os{\circ}{S}})$$ 
and 
$$\eps^{\rm conv}_{X/\os{\circ}{S}} \col ((X/\os{\circ}{S})_{\rm conv},{\cal K}_{X/\os{\circ}{S}})\lo 
((\os{\circ}{X}/\os{\circ}{S})_{\rm conv},{\cal K}_{\os{\circ}{X}/\os{\circ}{S}})$$ 
be the morphisms of ringed topoi induced by $\eps_{X/S}$ and 
$\eps_{X/\os{\circ}{S}}$, respectively.  
To construct the filtered complex $(A_{\rm conv}(X/S,E),P)$, 
we construct a key fundamental filtered complex 
$$(\wt{R}\eps^{\rm conv}_{X/\os{\circ}{S}*}(\eps^{{\rm conv}*}_{X/\os{\circ}{S}}(E)),P)
\in {\rm D}^+{\rm F}({\cal K}_{\os{\circ}{X}/\os{\circ}{S}}).$$
(The underlying complex 
$\wt{R}\eps^{\rm conv}_{X/\os{\circ}{S}*}(\eps^{{\rm conv}*}_{X/\os{\circ}{S}}(E))$ 
is different from the usual complex 
$R\eps^{\rm conv}_{X/\os{\circ}{S}*}(\eps^{{\rm conv}*}_{X/\os{\circ}{S}}(E))$.)
The filtered complex 
$(\wt{R}\eps^{\rm conv}_{X/\os{\circ}{S}*}(\eps^{{\rm conv}*}_{X/\os{\circ}{S}}(E)),P)$ 
is a new complex in our theory 
which does not appear in the theory in \cite{nhw}. 
We call this filtered complex the {\it modified $P$-filtered convergent complex} 
of $\eps^{{\rm conv}*}_{X/\os{\circ}{S}}(E)$. 
To define the complex 
$\wt{R}\eps^{\rm conv}_{X/\os{\circ}{S}*}(\eps^{{\rm conv}*}_{X/\os{\circ}{S}}(E))$, 
we need a slightly careful argument for defining the differentials of 
$\wt{R}\eps^{\rm conv}_{X/\os{\circ}{S}*}(\eps^{{\rm conv}*}_{X/\os{\circ}{S}}(E))$ 
because the differential is not directly obtained in a general framework.  
We prove that 
the complex $\wt{R}\eps^{\rm conv}_{X/\os{\circ}{S}*}(\eps^{{\rm conv}*}_{X/\os{\circ}{S}}(E))$ 
fits into the following triangle: 
\begin{align*} 
R\eps^{\rm conv}_{X/S*}(\eps^{{\rm conv}*}_{X/S}(E))[-1]\os{\theta}{\lo}  
\wt{R}\eps^{\rm conv}_{X/\os{\circ}{S}*}(\eps^{{\rm conv}*}_{X/\os{\circ}{S}}(E))
\lo R\eps^{\rm conv}_{X/S*}(\eps^{{\rm conv}*}_{X/S}(E))
\tag{1.3.2}\label{ali:oiqy}
\os{+1}{\lo}
\end{align*}
in $D^+({\cal K}_{\os{\circ}{X}/\os{\circ}{S}})$. 
(See \S\ref{sec:mplf} for the precise definition of 
$(\wt{R}\eps^{\rm conv}_{X/\os{\circ}{S}*}(\eps^{{\rm conv}*}_{X/\os{\circ}{S}}(E)),P)$.)
\par 
In the text we calculate the graded complex of 
$(\wt{R}\eps^{\rm conv}_{X/\os{\circ}{S}*}(\eps^{{\rm conv}*}_{X/\os{\circ}{S}}(E)),P)$
as follows: 

\begin{prop}\label{theo:grtr}
Let $k$ be a nonnegative integer. 
Then
\begin{align*} 
{} & 
{\rm gr}_k^P\wt{R}\eps^{\rm conv}_{X/\os{\circ}{S}*}
(\eps^{{\rm conv}*}_{X/\os{\circ}{S}}(E))
\tag{1.4.1}\label{caseali:gretkr}\\
{} & = 
\begin{cases}
a^{(k-1)}_{{\rm conv}*}(a^{(k-1)*}_{{\rm conv}*}(E)
\otimes_{\mab Z}\vp^{(k-1)}_{\rm conv}(\os{\circ}{X}/\os{\circ}{S}))[-k]& (k>0),  \\
{\rm Ker}(a^{(0)}_{{\rm conv}*}(a^{(0)*}_{{\rm conv}*}(E)
\otimes_{\mab Z}
\vp^{(0)}_{\rm conv}(\os{\circ}{X}/\os{\circ}{S}))
\lo a^{(1)}_{{\rm conv}*}(a^{(1)*}_{{\rm conv}*}(E)\otimes_{\mab Z}
\vp^{(1)}_{\rm conv}(\os{\circ}{X}/\os{\circ}{S})))
 & (k=0) 
\end{cases} \\
\end{align*} 
in $D^+({\cal K}_{\os{\circ}{X}/\os{\circ}{S}})$, 
where $\vp^{(*)}_{\rm conv}(\os{\circ}{X}/\os{\circ}{S})$ is the convergent  
orientation sheaf associated to the set $\{\os{\circ}{X}_{\lam}\}_{\lam \in \Lam}$ of 
the smooth components of $X/S_1$.  
\par 
\end{prop} 

As a corollary of (\ref{theo:grtr}), 
we prove the following, which we call the $p$-adic {\it semi-purity} of 
$(\wt{R}\eps^{\rm conv}_{X/\os{\circ}{S}*}(\eps^{{\rm conv}*}_{X/\os{\circ}{S}}(E)),P)$: 

\begin{theo}[{\rm $p${\bf -adic semi-purity}}]\label{theo:sp}
The filtration $P$ is equal to the canonical filtration $\tau$ on 
$\wt{R}\eps^{\rm conv}_{X/\os{\circ}{S}*}(\eps^{{\rm conv}*}_{X/\os{\circ}{S}}(E))$. 
That is, 
\begin{align*} 
(\wt{R}\eps^{\rm conv}_{X/\os{\circ}{S}*}(\eps^{{\rm conv}*}_{X/\os{\circ}{S}}(E)),\tau)
=(\wt{R}\eps^{\rm conv}_{X/\os{\circ}{S}*}(\eps^{{\rm conv}*}_{X/\os{\circ}{S}}(E)),P)
\tag{1.5.1}\label{align:grxxtkr}
\end{align*} 
in ${\rm D}^+{\rm F}({\cal K}_{\os{\circ}{X}/\os{\circ}{S}})$. 
Consequently 
\begin{align*} 
&{\cal H}^k(\wt{R}\eps^{\rm conv}_{X/\os{\circ}{S}*}
(\eps^{{\rm conv}*}_{X/\os{\circ}{S}}(E)))\\
&=\begin{cases}
a^{(k-1)}_{{\rm conv}*}(a^{(k-1)*}_{{\rm conv}*}(E)
\otimes_{\mab Z}\vp^{(k-1)}_{\rm conv}(\os{\circ}{X}/\os{\circ}{S}))) & (k>0),  \\
{\rm Ker}(a^{(0)}_{{\rm conv}*}(a^{(0)*}_{{\rm conv}*}(E)
\otimes_{\mab Z}\vp^{(0)}_{\rm conv}(\os{\circ}{X}/\os{\circ}{S}))
\lo a^{(1)}_{{\rm conv}*}(a^{(1)*}_{{\rm conv}*}(E)
\otimes_{\mab Z}\vp^{(1)}_{\rm conv}(\os{\circ}{X}/\os{\circ}{S})))
 & (k=0)
\end{cases} \\
\end{align*} 
for a nonnegative integer $k$. 
\end{theo} 
It is not difficult to prove that 
$\wt{R}\eps^{\rm conv}_{X/\os{\circ}{S}*}(\eps^{{\rm conv}*}_{X/\os{\circ}{S}}(E))$ 
is contravariantly functorial for any morphism of SNCL schemes $X/S$'s. 
As a corollary of (\ref{theo:sp}), it is immediate to see that 
$(\wt{R}\eps^{\rm conv}_{X/\os{\circ}{S}*}(\eps^{{\rm conv}*}_{X/\os{\circ}{S}}(E)),P)$ 
is also contravariantly functorial.  
\par 
By (\ref{ali:oiqy}) we obtain the following composite morphism 
$$\theta \col 
\wt{R}\eps^{\rm conv}_{X/\os{\circ}{S}*}(\eps^{{\rm conv}*}_{X/\os{\circ}{S}}(E))
\lo 
R\eps^{\rm conv}_{X/S*}(\eps^{{\rm conv}*}_{X/S}(E))
\os{\theta}{\lo} 
\wt{R}\eps^{\rm conv}_{X/\os{\circ}{S}*}(\eps^{{\rm conv}*}_{X/\os{\circ}{S}}(E))[1]$$ 
of complexes.  
Set $H:=\wt{R}\eps^{\rm conv}_{X/\os{\circ}{S}*}(\eps^{{\rm conv}*}_{X/\os{\circ}{S}}(E))$. 
Then we define the filtered complex $(A_{\rm conv}(X/S,E),P)$ as follows. 
First we set  
\begin{align*} 
A_{\rm conv}(X/S,E):=s(H/\tau_0[1]\os{\theta}{\lo} H/\tau_1[1]\{1\}\os{\theta}{\lo} 
\cdots \os{\theta}{\lo} H/\tau_j[1]\{j\} \os{\theta}{\lo} \cdots),
\end{align*}  
where $s$ means the single complex of the double complex 
and $[1]$ means the $1$-shift on the left of the complex with the change of the signs  
of the differential morphisms as usual and $\{j\}$ means 
the $j$-shift on the left of the complex {\it without} the change of the signs  
of the differential morphisms unusually. 
Note that our definition of a double complex is {\it not} a traditional one; 
we demand the anti-commutative diagram for the definition of the double complex 
instead of the commutative diagram for the definition of the double complex.
(Strictly speaking, to define $A_{\rm conv}(X/S,E)$, we have to take a representative of 
$H$.) 
Then we define the filtration $P:=\{P_k\}_{k\in {\mab Z}}$ on $A_{\rm conv}(X/S,E)$ as follows: 
\begin{align*}
P_kA_{\rm conv}(X/S,E)
:=s(\cdots \os{\theta}{\lo} \tau_{\max \{2j+k+1,j\}}H/\tau_j[1]\{j\} \os{\theta}{\lo} \cdots) \quad
(k\in {\mab Z}, j\in {\mab N}).
\end{align*}  
Because we use the convergent topos $(\os{\circ}{X}/\os{\circ}{S})_{\rm conv}$ 
and because we prove the $p$-adic semi-purity (\ref{theo:sp}), 
we can define a $p$-adic analogue of Steenbrink-Fujisawa-Nakayama's $\infty$-adic filtered complex 
in a manner of the definition of Rapoport-Zink-Nakayama's $l$-adic filtered complex as above. 
Because $H=\wt{R}\eps^{\rm conv}_{X/\os{\circ}{S}*}(\eps^{{\rm conv}*}_{X/\os{\circ}{S}}(E))$ is 
contravariantly functorial as already stated, so is $(A_{\rm conv}(X/S,E),P)$, 
though one has to pay attention to the contravariant action on 
$d\log t$ as already remarked.  
In particular, on $(A_{\rm conv}(X/S,E),P)$,  
we can consider the pull-back of the Frobenius morphism 
$X\lo X\times_{\os{\circ}{S}_1,F_{\os{\circ}{S}_1}}\os{\circ}{S}_1$ 
over 
$S\lo S_1^{[p]}(S)$, where $S_1^{[p]}(S)$ is a $p$-adic formal log scheme 
defined similarly as in (\ref{ali:rus0avp}) when $E$ is an $F$-isocrystal 
on ${\rm Crys}(\os{\circ}{X}/\os{\circ}{S})$. 
\par 
We prove that the complex $A_{\rm conv}(X/S,E)$ 
computes the log convergent complex 
$R\eps^{\rm conv}_{X/S*}(\eps^{{\rm conv}*}_{X/S}(E))$ 
in the category ${\rm D}^+{\rm F}({\cal K}_{\os{\circ}{X}/\os{\circ}{S}})$ 
which is above the category ${\rm D}^+{\rm F}(f^{-1}({\cal K}_{S}))$
unlike \cite{gkwf} nor \cite{nb}.  
That is, we prove the following: 

\begin{prop}\label{theo:cexs}
There exists the following isomorphism  
\begin{align*} 
\theta \col R\eps^{\rm conv}_{X/S*}(\eps^{{\rm conv}*}_{X/S}(E))\os{\sim}{\lo} 
A_{\rm conv}(X/S,E)
\tag{1.6.1}\label{ali:ixaueova} 
\end{align*} 
in $D^+({\cal K}_{\os{\circ}{X}/\os{\circ}{S}})$. 
\end{prop} 
We calculate the graded complex of 
$(A_{\rm conv}(X/S,E),P)$ as follows: 

\begin{prop}\label{theo:psap} 
There exists a canonical isomorphism 
\begin{align*} 
{\rm gr}^P_kA_{\rm conv}(X/S,E) 
\os{\sim}{\lo} \bigoplus_{j\geq \max \{-k,0\}} 
&a^{(2j+k)}_{{\rm conv}*}(a^{(2j+k)*}_{\rm conv}(E) \tag{1.7.1}\label{ali:ruaovp}\\
&\otimes_{\mab Z}\vp_{\rm conv}^{(2j+k)}(\os{\circ}{X}/\os{\circ}{S}))[-2j-k]. 
\end{align*}
Here $\vp_{\rm conv}^{(\star)}(\os{\circ}{X}/\os{\circ}{S})$ 
is the convergent orientation sheaf associated to 
the set $\{\os{\circ}{X}_{\lam}\}_{\lam \in \Lam}$. 
If $E$ is an $F$-isocrystal on ${\rm Crys}(\os{\circ}{X}/\os{\circ}{S})$, then 
there exists a canonical isomorphism 
\begin{align*} 
{\rm gr}^P_kA_{\rm conv}(X/S,E) 
\os{\sim}{\lo} \bigoplus_{j\geq \max \{-k,0\}} 
&a^{(2j+k)}_{{\rm conv}*}(a^{(2j+k)*}_{\rm conv}(E) \tag{1.7.2}\label{ali:ruamp}\\
&\otimes_{\mab Z}\vp_{\rm conv}^{(2j+k)}(\os{\circ}{X}/\os{\circ}{S}))(-j-k)[-2j-k]. 
\end{align*}
\end{prop} 

Set $\os{\circ}{f}{}^{\rm conv}_{\os{\circ}{X}{}^{(k)}/\os{\circ}{S}}
:=\os{\circ}{f}{}^{(k)}\circ u^{\rm conv}_{\os{\circ}{X}{}^{(k)}/\os{\circ}{S}}$  
and $f^{\rm conv}_{X/S}:=f\circ u^{\rm conv}_{X/S}$.  
As a corollary of (\ref{ali:ixaueova}) and (\ref{ali:ruamp}), 
we obtain the following spectral sequence
for an $F$-isocrystal $E$ on ${\rm Crys}(\os{\circ}{X}/\os{\circ}{S})$: 
\begin{align*}
E_1^{-k,q+k} & = 
\bigoplus_{j\geq \max \{-k,0\}} R^{q-2j-k}f^{\rm conv}_{\os{\circ}{X}{}^{(2j+k)}/\os{\circ}{S}*}
(a^{(2j+k)*}_{\rm conv}(E)
\otimes_{\mab Z}
\vp^{(2j+k)}_{{\rm conv}}(\os{\circ}{X}/\os{\circ}{S}))(-j-k) \tag{1.7.3}\label{ali:wtawztc}\\
{} & \Lo  R^qf^{\rm conv}_{X/S*}(\eps^{{\rm conv}*}_{X/S}(E)). 
\end{align*} 
If $\os{\circ}{f}$ is proper and if $E={\cal K}_{\os{\circ}{X}/\os{\circ}{S}}$, 
then we call 
(\ref{ali:wtawztc}) 
the {\it weight spectral sequence} of 
$R^qf^{\rm conv}_{X/S*}({\cal K}_{X/S})$. 
\par
We also define another filtered complex 
$(A^N_{\rm conv},P)$ by using the ``derived Hirsch extension'' of $H$ by adding a one variable to  
the complex $H$ (the notion of the derived Hirsch extension has been defined in \cite{nhi})
and by considering the mapping fiber of the monodromy operator on this new complex 
as in the complex analytic case. 
(See \S\ref{sec:wt} for the precise definition of $(A^N_{\rm conv},P)$.)
We prove the following: 

\begin{theo}[{\bf Comparison theorem between $(A^N_{\rm conv},P)$ and  
$(A_{\rm conv},P)$}]\label{theo:ctaeck} 
There exists a canonical isomorphism
\begin{equation*}
(A^N_{\rm conv}(X/S,E),P)\os{\sim}{\lo} (A_{\rm conv}(X/S,E),P). 
\tag{1.8.1}\label{eqn:execrs}
\end{equation*}
\end{theo}

\par 
Finally we state a relationship  
between $(A_{\rm iso{\textrm -}zar},P)$  and $(A_{\rm zar},P)$ 
in \cite{nb}. 
Since $\os{\circ}{X}$ is not smooth over $\os{\circ}{S}$ in general, 
the relationship is 
not so simple for a general coefficient: 

\begin{theo}[{\bf A certain comparison theorem between $(A_{\rm iso{\textrm -}zar},P)$ 
and $(A_{\rm zar},P)\otimes^L_{\mab Z}{\mab Q}$}]\label{theo:izz} 
Let ${\rm Isoc}_{\rm conv}(X/S)$ and 
${\rm Isoc}_{\rm crys}(X/S)$ 
be the category of isocrystals on the log convergent site ${\rm Conv}(X/S)$ 
and that of isocrystals on the log crystalline site ${\rm Crys}(X/S)$, respectively. 
Denote by 
\begin{equation*}
\Xi \col {\rm Isoc}_{\rm conv}(X/S)
\lo {\rm Isoc}_{\rm crys}(X/S)
\tag{1.9.1}\label{eqn:icybm}
\end{equation*} 
the relative version of 
the functor defined in  {\rm \cite[Theorem 5.3.1]{s1}} which is the log version of 
{\rm \cite[Theorem 7.2]{oc}}
and denoted by $\Phi$ in {\rm \cite[Theorem 5.3.1]{s1}}. 
Let $L$ be a crystal on ${\rm Crys}(X/S)$ 
such that $K\otimes_{\cal V}L=\Xi(\eps^{*{\rm conv}}_{X/S}(E))$.  
$($It is a nontrivial fact that $L$ indeed exists $($cf.~{\rm \cite{oc}, \cite{s2}}$).)$ 
Then there exists a certain filtered complex $(A_{{\rm iso{\textrm -}zar}}(X/S,\os{\circ}{L}),P)$ 
and there exists a filtered isomorphism 
\begin{equation*}
(A_{\rm iso{\textrm -}zar}(X/S,E),P) 
\os{\sim}{\lo} 
(A_{\rm iso{\textrm -}zar}(X/S,\os{\circ}{L}),P)
\tag{1.9.2}\label{eqn:crs}
\end{equation*} 
fitting into the following commutative diagram 
\begin{equation*} 
\begin{CD}
A_{\rm iso{\textrm -}zar}(X/S,E)@>{\sim}>> A_{\rm iso{\textrm -}zar}(X/S,\os{\circ}{L})\\
@A{\theta \wedge}AA @AA{\theta \wedge}A \\
Ru^{\rm conv}_{X/S*}(\eps^{{\rm conv}*}_{X/S}(E))@>{\sim}>>
Ru^{\rm crys}_{X/S*}(L)\otimes^L_{\cal V}K. 
\end{CD}
\tag{1.9.3}\label{eqn:cad}
\end{equation*} 
$($Though we define $(A_{\rm iso{\textrm -}zar}(X/S,\os{\circ}{L}),P)$, 
we do not define $\os{\circ}{L}$ itself.
If $E={\cal K}_{\os{\circ}{X}/\os{\circ}{S}}$, 
then 
$(A_{\rm iso{\textrm -}zar}(X/S,\os{\circ}{L}),P)=
(A_{\rm zar}(X/S,{\cal O}_{\os{\circ}{X}/\os{\circ}{S}}),P)\otimes^L_{\cal V}K.)$
\end{theo}
By using this comparison theorem and 
results for $(A_{\rm iso{\textrm -}zar}(X/S,\os{\circ}{L}),P)$ 
which are obtained as in \cite{nb}, 
we can prove properties of $(A_{\rm iso{\textrm -}zar}(X/S,E),P)$. 
\par 
The contents of this paper are as follows. 
\par  
In \S\ref{sec:logcd} we recall the definition of a log convergent topos, 
the system of log universal enlargements and basic properties of them. 
\par 
In \S\ref{sec:lll}
we recall the definition of a log convergent linearization functor
and basic properties of them, especially the compatibility of 
the log convergent linearization functor with 
a closed immersion of log schemes. 
\par 
In \S\ref{sec:vflvc}  we prove the Poincar\'{e} lemma of the vanishing cycle sheaf 
in a more general situation than that in \cite{nhw}. 
\par 
In \S\ref{eclf}
we give basic results for log convergent linearization functors 
when the base log formal scheme is the underlying formal scheme 
of a $p$-adic family of log points. 
\par 
In \S\ref{sec:lcs} we give basic results for 
the log convergent linearization functor of an SNCL scheme.  
\par 
In \S\ref{sec:rlct} we give the notations of the simplicial log convergent topoi. 
\par 
In \S\ref{sec:mplf} we show the well-definedness of 
$(\wt{R}\eps^{\rm conv}_{X/\os{\circ}{S}*}(\eps^{{\rm conv}*}_{X/\os{\circ}{S}}(E)),P)$ 
by using (\ref{ali:oiqy}). We also prove (\ref{theo:grtr}) and (\ref{theo:sp}). 
\par 
In \S\ref{sec:wfcipp} we give the definition of $(A_{\rm conv}(X/S,E),P)$  
and we prove (\ref{theo:cexs}) and (\ref{theo:psap}). 
\par 
In \S\ref{sec:wt} we give the definition of $(A^N_{\rm conv}(X/S,E),P)$ and 
we prove (\ref{theo:ctaeck}). 
\par 
In \S\ref{sec:fpw} we prove the contravariant functoriality of 
$(A_{\rm conv}(X/S,E),P)$ and 
consequently that of $(A_{\rm iso{\textrm -}zar}(X/S,E),P)$.
\par 
In \S\ref{sec:bd} we give an explicit description of 
the edge morphism of the spectral sequence 
(\ref{ali:wtawztc}). 
\par 
In \S\ref{sec:mn} 
we give the definitions of monodromy operators and we give formulation 
of a variational log convergent monodromy-weight conjecture and a 
variational log convergent hard Lefschetz conjecture.  
\par 
In \S\ref{sec:ct} we prove (\ref{theo:izz}). 
Set 
$f^{\rm crys}_{\os{\circ}{X}{}^{(k)}/S}
:=f^{(k)}\circ u^{\rm crys}_{\os{\circ}{X}{}^{(k)}/S}$ $(k\in {\mab N})$ 
and 
$f^{\rm crys}_{X/S}:=f\circ u^{\rm crys}_{X/S}$. 
As a corollary of (\ref{theo:izz}),  
(\ref{ali:wtawztc}) for the trivial coefficient 
turns out to be canonically isomorphic to 
the following weight spectral sequence 
\begin{align*}
E_1^{-k,q+k} & = 
\bigoplus_{j\geq \max \{-k,0\}} R^{q-2j-k}
f^{\rm crys}_{\os{\circ}{X}{}^{(2j+k)}/\os{\circ}{S}*}
({\cal O}_{\os{\circ}{X}{}^{(2j+k)}/\os{\circ}{S}}\otimes_{\mab Z}
\vp^{(2j+k)}_{{\rm crys}}(\os{\circ}{X}/\os{\circ}{S}))\otimes_{\cal V}K(-j-k)
\tag{1.9.4}\label{ali:wtbwcc}\\
{} & \Lo  R^qf^{\rm crys}_{X/S*}({\cal O}_{X/S})\otimes_{\cal V}K, 
\end{align*}
which has been constructed in \cite{nb}. 
As a corollary of the comparison of (\ref{ali:wtawztc})  
with (\ref{ali:wtbwcc}) for the case of the trivial coefficient, 
if $X$ is proper over $S_1$, 
then we obtain the $E_2$-degeneration of (\ref{ali:wtawztc}) 
in this case 
by reducing it to that of (\ref{ali:wtbwcc}),  
which we have already proved in \cite{nb}.  
In the same section, as another corollary of (\ref{theo:izz}), 
we deduce the filtered base change theorem 
of $(Rf^{\rm conv}_{\os{\circ}{X}/\os{\circ}{S}*}(A_{\rm conv}(X/S)),P)$ from 
the filtered base change theorem of 
$(Rf^{\rm crys}_{\os{\circ}{X}/\os{\circ}{S}*}(A_{\rm crys}(X/S)),P)$. 
\bigskip
\parno
{\bf Notations.}
(1) The log structure of a log scheme in this paper is 
a sheaf of monoids in the Zariski topos. 
For a log (formal) scheme $S$, 
$\os{\circ}{S}$ denotes 
the underlying (formal) scheme of $S$ and 
${\cal K}_S$ denotes ${\cal O}_S\otimes_{\mab Z}{\mab Q}$. 
More generally, for a ringed topos $({\cal T},{\cal O}_{\cal T})$,  
${\cal K}_{\cal T}$ denotes 
${\cal O}_{\cal T}\otimes_{\mab Z}{\mab Q}$. 
For a morphism $f\col X\lo S$ of log (formal) schemes, 
$\os{\circ}{f}$ denotes the underlying morphism 
$\os{\circ}{X} \lo \os{\circ}{S}$ of (formal) schemes. 
\par
(2) SNCD=simple normal crossing divisor, SNCL=simple normal crossing log. 
\par 
(3) For a morphism $X \lo S$ of log (formal) schemes, 
we denote by 
$\Om^i_{X/S}$ ($=\om^i_{X/S}$ in 
\cite{klog1}) the sheaf of 
(formal) relative logarithmic differential forms on $X/S$
of degree $i$ $(i\in {\mab N})$. 
\par 
(4) For a commutative monoid $P$ with unit element 
and for a formal scheme $S$ 
with an ideal of definition of ${\cal I}\subset {\cal O}_S$, 
we denote by ${\cal O}_S\{P\}$ 
the sheaf $\vpl_n({\cal O}_S[P]/{\cal I}^n{\cal O}_S[P])$ 
in the Zariski topos ${S}_{\rm zar}$ of $S$. 
When ${\cal I}=0$, we denote 
$\ul{\rm Spf}_{S}({\cal O}_S\{P\})$ by 
$\ul{\rm Spec}_{S}({\cal O}_S[P])$.  
\par 
For a log (formal) scheme $S$ 
with an ideal of definition of ${\cal I}\subset {\cal O}_S$, 
we denote by $S~{\rm mod}~{\cal I}$
the log scheme whose underlying scheme is  
$\ul{\rm Spec}_{\os{\circ}{S}}({\cal O}_S/{\cal I})$ 
and whose log structure 
is the inverse image of the log structure of $S$ 
by the immersion 
$\ul{\rm Spec}_{\os{\circ}{S}}({\cal O}_S/{\cal I})\os{\sus}{\lo} S$. 
\par 
(5) For the log structure $M_S$ of a log (formal) scheme $S$, 
we denote $M_S/{\cal O}_S^*$ by $\ol{M}_S$. 
\par 
(6) We denote the completed tensor $\wh{\otimes}$ of 
two topological modules 
simply by $\otimes$. 

\section{Log convergent topoi}\label{sec:logcd}
In this section we recall basic notions and basic results 
about the log convergent topos in \cite{oc}, \cite{ofo}, \cite{s2}, \cite{s3} and \cite{nhw}. 
See \cite[\S2]{nhw} for the proofs of the resutls. 
\par 
Let ${\cal V}$ be a complete discrete valuation ring 
of mixed characteristics $(0,p)$ 
with perfect residue field $\kap$. 
Let $S$ be a fine log formal scheme 
whose underlying formal scheme is a 
$p$-adic formal ${\cal V}$-scheme 
in the sense of \cite{of}: 
$\os{\circ}{S}$ is a noetherian formal scheme 
over ${\rm Spf}({\cal V})$ 
with the $p$-adic topology 
which is topologically of finite type over ${\rm Spf}({\cal V})$.  
We always assume that $\os{\circ}{S}$ is flat over ${\rm Spf}({\cal V})$. 
We call  $S$ a fine log $p$-adic formal ${\cal V}$-scheme.  
Let $\pi$ be a non-zero element of the maximal ideal 
of ${\cal V}$.  Set $S_1=S~{\rm mod}~\pi{\cal O}_S$. 
Let $Y$ be a fine log scheme over $S_1$ 
whose underlying scheme $\os{\circ}{Y}$ 
is of finite type over $\os{\circ}{S}_1$. 
As in \cite[p.~212]{ofo}, \cite[Definition 2.1.8]{s2} and \cite[I (2.1)]{s3}, 
we consider a quadruple $(U,T,\iota,u)$, 
where $u \col U \lo Y$ is a strict morphism of 
fine log schemes over $S_1$ and 
$\os{\circ}{u}\col \os{\circ}{U}\lo \os{\circ}{Y}$ is a morphism of schemes 
of finite type,  
$T$ is a fine log noetherian formal scheme over $S$ 
whose underlying formal scheme 
is topologically of finite type over $\os{\circ}{S}$ 
and $\iota \col U \os{\sus}{\lo} T$ is 
a closed immersion over $S$. 
We say that $T$ is a fine log $S$-formal scheme. 
As in \cite{nhw} we have assumed the condition 
``the strictness of $u$'',  
which has not been assumed 
in \cite[Definition 2.1.1, 2.1.9]{s2} 
and \cite[I (2.1), (2.2)]{s3}. 
We call the $(U,T,\iota,u)$  
a {\it prewidening} of $Y/S$. 
If $\os{\circ}{T}$ is $p$-adic, then we say that 
$(U,T,\iota,u)$ is a $p$-{\it adic} {\it prewidening} of $Y/S$.  
Set $T_1:=T~{\rm mod}~\pi{\cal O}_T$ and 
let $T_0$ be the log maximal reduced subscheme of $T_1$.  
Set also $\wt{T}:=T~{\rm mod}~(\pi{\textrm -}{\rm torsion})$. 
By imitating \cite[(1.1)]{oc},  
\cite[Definition 2.1.1, 2.1.9]{s2} and \cite[I (2.2)]{s3}, 
we obtain the notion of a widening, an exact (pre)widening 
and an enlargement of $Y/S$ as follows. 
\par 
(1) We say that 
a ($p$-adic) prewidening 
$(U,T, \iota,u)$ of $Y/S$
is a ($p$-{\it adic}) {\it widening} of $Y/S$ if 
$\os{\circ}{\iota} \col \os{\circ}{U} \os{\sus}{\lo} \os{\circ}{T}$ 
is defined by an ideal sheaf of definition of $\os{\circ}{T}$. 
\par 
(2)  We say that a ($p$-adic) (pre)widening 
$(U,T, \iota,u)$ of $Y/S$
is {\it exact} if $\iota$ is exact. 
\par 
(3) We say that an exact widening 
$(U,T,\iota,u)$ of $Y/S$ is 
an {\it enlargement} of $Y/S$ if $\os{\circ}{T}$ is flat over 
${\rm Spf}({\cal V})$ and if $T_0 
\subset U$; 
the latter condition is equivalent to an inclusion 
\begin{equation*} 
({\rm Ker}({\cal O}_T \lo {\cal O}_U))^N \subset \pi{\cal O}_T  
\quad (\exists N\in {\mab N}).  
\tag{2.0.1}\label{eqn:kotu} 
\end{equation*} 
\par 
We define a {\it morphism of {\rm (}$p$-adic{\rm )} 
{\rm (}exact{\rm )} {\rm (}pre{\rm )}widenings and 
a morphism of  enlargements} in an obvious way. 
As usual, we denote $(U,T,\iota,u)$ simply by $T$.  
For a morphism 
$f\col (U,T,\iota,u) \lo (U',T',\iota',u')$ of 
($p$-adic) (exact) (pre)widenings 
or enlargements of $Y/S$, we denote the morphism 
$T\lo T'$ of fine log ($p$-adic) formal schemes over $S$ 
by $f\col T \lo T'$ as usual. 
\par  
The objects of the log convergent site ${\rm Conv}(Y/S)$ 
are, by definition, the enlargements of $Y/S$; the morphisms 
of ${\rm Conv}(Y/S)$ are the morphisms of enlargements of 
$Y/S$; for an object $(U,T,\iota,u)$ of 
${\rm Conv}(Y/S)$, a family 
$\{(U_{\lam},T_{\lam},\iota_{\lam},u_{\lam})
\lo (U,T,\iota,u)\}_{\lam}$ 
of morphisms in  ${\rm Conv}(Y/S)$ is defined to 
be a covering family of $(U,T,\iota,u)$ if 
$T=\bigcup_{\lam}T_{\lam}$ is a Zariski open covering and if 
the natural morphism $U_{\lam}\lo T_{\lam}\times_TU$ is an isomorphism 
(cf.~\cite[p.~135]{oc}, \cite[Definition (2.1.3)]{s2}, \cite[I (2.4)]{s3}).  
Let $(Y/S)_{\rm conv}$ 
be the topos associated to the site ${\rm Conv}(Y/S)$.  
An object $E$ in $(Y/S)_{\rm conv}$ 
gives a Zariski sheaf $E_T$ 
in the Zariski topos $T_{\rm zar}$ of $T$ 
for $T=(U,T,\iota,u)\in {\rm Ob}({\rm Conv}(Y/S))$.  
As in \cite{oc} and \cite{s2}, 
for a prewidening $T$ of $Y/S$, 
we denote simply by $T$ 
the following sheaf 
\begin{equation*} 
h_T \col  {\rm Conv}(Y/S)\owns T' \lom 
\{\text{the morphisms $T' \lo T$ of prewidenings 
of $Y/S$}\} \in ({\rm Sets})  
\end{equation*} 
on ${\rm Conv}(Y/S)$.  
Let ${\cal K}_{Y/S}$ be an isostructure sheaf in 
$(Y/S)_{\rm conv}$ defined by 
the following formula 
$\Gam((U,T,\iota,u),{\cal K}_{Y/S}):=
\Gam(T, {\cal O}_T)\otimes_{\mab Z}{\mab Q}$ 
for $(U,T, \iota,u)\in {\rm Ob}({\rm Conv}(Y/S))$. 
\par 
Let $f \col Y \lo S_1\os{\sus}{\lo} S$ 
be the structural morphism. 
Let $\os{\circ}{Y}_{\rm zar}$ be the Zariski topos of $\os{\circ}{Y}$ and 
let 
\begin{equation*}
u^{\rm conv}_{Y/S} \col (Y/S)_{\rm conv} \lo \os{\circ}{Y}_{\rm zar} 
\tag{2.0.2}\label{eqn:uwks}
\end{equation*}
be a morphism of topoi 
characterized by a formula 
$u^{{\rm conv}*}_{Y/S}(E)(T)=
\Gam(U,u^{-1}(E))$ for a sheaf  
$E$ in $\os{\circ}{Y}_{\rm zar}$ and an enlargement $T=(U,T,\iota,u)$ of $Y/S$.  
Set $f^{\rm conv}_{Y/S}:=f\circ u^{\rm conv}_{Y/S}$ and  
${\cal K}_S:={\cal O}_S\otimes_{\mab Z}{\mab Q}$. 
Let 
\begin{equation*}
u^{\rm conv}_{Y/S} \col 
((Y/S)_{\rm conv},{\cal K}_{Y/S})
\lo (Y_{\rm zar}, f^{-1}({\cal K}_S)) 
\tag{2.0.3}\label{eqn:uwksr}
\end{equation*} 
be a morphism of ringed topoi induced by the morphism (\ref{eqn:uwks}). 
\par
Let ${\cal V}'$ be another complete discrete valuation ring 
of mixed characteristics $(0,p)$ with perfect residue field. 
Let $\pi'$ be a non-zero element of the maximal ideal 
of ${\cal V}'$. Let ${\cal V} \lo {\cal V}'$ be a morphism of 
commutative rings with unit elements inducing a morphism 
${\cal V}/\pi{\cal V} \lo {\cal V}'/\pi' {\cal V}'$. 
Let $S'$ be a fine log $p$-adic formal ${\cal V}'$-scheme 
($\os{\circ}{S}{}'$ is assumed to be flat over ${\cal V}'$).  
Set $S'_1:=S'~{\rm mod}~\pi'{\cal O}_{S'}$. 
Let $Y'$ be a fine log scheme over $S'_1$ such that 
the underlying structural morphism 
$\os{\circ}{Y}{}'\lo \os{\circ}{S}{}'_1$ is of finite type. 
For a commutative diagram 
\begin{equation*} 
\begin{CD}
Y' @>>> S'_1 @>{\subset}>> S' @>>>{\rm Spf}({\cal V}')\\
@V{g}VV @VVV  @VV{u}V @VVV \\ 
Y    @>>> S_1 @>{\subset}>> S  @>>>{\rm Spf}({\cal V}) 
\end{CD}
\tag{2.0.4}\label{cd:yypssp}
\end{equation*}
of fine log (formal) schemes, we have a natural morphism 
\begin{equation*} 
g_{\rm conv} \col ((Y'/S')_{\rm conv},{\cal K}_{Y'/S'}) 
\lo ((Y/S)_{\rm conv},{\cal K}_{Y/S})  
\end{equation*} 
of ringed topoi. 
\par 
The following is easy to prove:

\begin{prop}[{\bf \cite[(2.1)]{nhw}}]\label{prop:toi} 
Assume that ${\cal V}'={\cal V}$ and that 
the morphism ${\cal V} \lo {\cal V}'$ is the identity of ${\cal V}$ 
$(\pi'$ may be different from $\pi)$. 
Assume also that $S'=S$ and $Y'=Y\times_{S_1}S'_1$. 
Denote $g$ in {\rm (\ref{cd:yypssp})} by $i$. 
Let $f' \col Y' \lo S$ 
be the structural morphism. 
Then the following hold$:$ 
\par 
$(1)$ The direct image $i_{{\rm conv}*}$ 
from the category of abelian sheaves in 
$(Y'/S')_{\rm conv}$ to 
the category of abelian sheaves in 
$(Y/S)_{\rm conv}$ 
is exact. 
\par 
$(2)$ For a bounded below filtered complex 
$(E^{\bul},P)$ of ${\cal K}_{Y'/S}$-modules, 
\begin{equation*} 
Rf^{\rm conv}_{Y/S*}(i_{{\rm conv}*}((E^{\bul},P)))
=Rf'{}^{\rm conv}_{Y'/S*}((E^{\bul},P)) 
\tag{2.1.1}\label{eqn:ysep}
\end{equation*} 
in ${\rm D}^+{\rm F}({\cal K}_S)$. 
\par 
$(3)$ 
\begin{equation*} 
Rf^{\rm conv}_{Y/S*}({\cal K}_{Y/S})=
Rf'{}^{\rm conv}_{Y'/S*}({\cal K}_{Y'/S}) 
\tag{2.1.2}\label{eqn:ksep}
\end{equation*} 
in $D^+({\cal K}_S)$. 
\end{prop}

\par 
For a fine log $p$-adic formal ${\cal V}$-scheme $T$, 
let  ${\rm Coh}({\cal O}_T)$ be 
the category of coherent ${\cal O}_T$-modules. 
Let  ${\rm Coh}({\cal K}_T)$ be the full subcategory 
of the category of ${\cal K}_T$-modules 
whose objects are isomorphic to $K\otimes_{\cal V}{\cal F}$ 
for some ${\cal F}\in {\rm Coh}({\cal O}_T)$
(\cite[(1.1)]{of}). 
In \cite{nhw} we have said that a ${\cal K}_{Y/S}$-module $E$ is 
a coherent isocrystal 
if $E_T\in {\rm Coh}({\cal K}_T)$ 
for any enlargement $T$ of $Y/S$  and if 
$\rho_f \col f^*(E_{T'}) \lo E_T$ is an isomorphism 
for any morphism $f\col T\lo T'$ of enlargements of  $Y/S$ 
as in \cite{of}, \cite{oc}, \cite{ofo}, \cite{s2}. 
As stated in the Introduction, we call a coherent isocrystal simply an isocrystal as in \cite{bprep} 
in this paper. 
\par 
Next we recall the system of universal enlargements  
(cf.~\cite[(2.1)]{oc}, \cite[Definition 2.1.22]{s2}, \cite[I (2.10), (2.11), (2.12)]{s3}).  
\par  
Let $(U,T,\iota,u)$ be 
an enlargement of $Y/S$. 
We say that $(U,T,\iota,u)$ is 
an enlargement of $Y/S$ with radius $\leq \vert \pi \vert$
if ${\rm Ker}({\cal O}_T \lo {\cal O}_U) \subset \pi{\cal O}_T$. 
The following is a log version of \cite[(2.1)]{oc} 
and an immediate generalization of 
\cite[Definition 2.1.22, Lemma 2.1.23]{s2} 
for the case $n=1$ in [loc.~cit.]: 

\begin{prop}[{\bf \cite[(2.2)]{nhw}}]\label{prop:exue} 
Let $T=(U,T,\iota,u)$ be 
a $(p$-adic$)$ $($pre$)$widening of  
$Y/S$. Then there exists an enlargement 
$T_1=(U_1,T_1,\iota_1,u_1)$ of $Y/S$ 
with radius $\leq \vert \pi \vert$ 
with a morphism $T_1 \lo T$ of 
$(p$-adic$)$ $($pre$)$widenings of $Y/S$ 
satisfying the following universality$:$ 
any morphism $T':=(U',T',\iota',u') \lo T$ 
of $(p$-adic$)$ $($pre$)$widenings of $Y/S$ 
from an enlargement of $Y/S$ 
with radius $\leq \vert \pi \vert$ 
factors through 
a unique morphism   
$T' \lo T_1$ of enlargements of $Y/S$. 
\end{prop}

\par 
By replacing  ${\rm Ker}({\cal O}_{T} \lo {\cal O}_{U})$ 
by $({\rm Ker}({\cal O}_{T} \lo {\cal O}_{U}))^n$ 
$(n\in {\mab Z}_{\geq 1})$ in (\ref{prop:exue}), 
we have the inductive system $\{T_n\}_{n=1}^{\infty}$ 
($T_n=(U_n, T_n,\iota_n,u_n)$) (cf.~\cite{of}, \cite{oc} and \cite{s2}). 
We call $\{T_n\}_{n=1}^{\infty}$ the {\it system of the universal enlargements} 
of $(U,T,\iota,u)$. We denote the log formal scheme $T_n$ 
by ${\mathfrak T}_{U,n}(T)$. 
By abuse of notation, we also denote 
the enlargement $(U_n,T_n,\iota_n,u_n)$ 
by ${\mathfrak T}_{U,n}(T)$.  
For a $($not necessarily closed$)$ immersion 
$\iota \col U \os{\sus}{\lo} T$, we can define 
$T_n:={\mathfrak T}_{U,n}(T)$ as usual.  
For a prewidening $T=(U,T,\iota,u)$ of $Y/S$, 
we can define the following functor 
$$h_T \col {\rm Conv}(Y/S)\lo ({\rm Sets})$$
by using the morphisms of prewidening of $Y/S$ as usual. 
We also obtain the localized category 
$(Y/S)_{\rm conv}\vert_T$ as usual. 
Let $\bet_n \col T_n \lo T$ $(n\in {\mab Z}_{\geq 1})$ 
be the natural morphism. 
Let $E$ be a sheaf in $(Y/S)_{\rm conv}$. 
Let  $E_T$ be a sheaf in $T_{\rm zar}$ defined by 
a formula $E_T(T')={\rm Hom}_{(Y/S)_{\rm conv}}(h_{T'},E)$ 
as in \cite[p.~140]{oc} for a log formal open subscheme $T'$ of $T$. 
We can prove that $E_T=\vpl_n\bet_{n*}(E_{T_n})$ as usual. 
Let $\os{\to}{T}$ be the associated topos 
to the site defined in \cite[Definition 2.1.28]{s2},  
which is the log version of the site in \cite[\S3]{oc}. 
An object  of this site is 
a log open formal subscheme $V_n$ of $T_n$ 
for a positive integer  $n$; let $V_m$ be 
a log open formal subscheme of $T_m$;  
for $n>m$, ${\rm Hom}(V_n,V_m):= \emptyset$ 
and, for $n\leq m$, 
${\rm Hom}(V_n,V_m)$ 
is the set of morphisms $V_n \lo V_m$'s of 
log formal subschemes over the natural morphism $T_n \lo T_m$.
A covering of $V_n$ is a Zariski open covering of $V_n$. 
Let $\os{\to}{T}$ be the topos associated to this site. 
Let $\phi_n \col T_n \lo T_{n+1}$ be the natural morphism 
of enlargements of $Y/S$. 
Then an object ${\cal F}$ in $\os{\to}{T}$
is a family $\{({\cal F}_n,\psi_n)\}_{n=1}^{\infty}$ of pairs, 
where ${\cal F}_n$ is a Zariski sheaf in $(T_n)_{\rm zar}$ 
and $\psi_n \col \phi^{-1}_n({\cal F}_{n+1}) \lo {\cal F}_n$ 
is a morphism of Zariski sheaves in $(T_n)_{\rm zar}$. 
Let ${\cal K}_{\os{\to}{T}}
=\{({\cal K}_{T_n},\phi_n^{-1})\}_{n=1}^{\infty}$
be the isostructure sheaf of $\os{\to}{T}$, 
where $\phi_n^{-1}$ is the natural pull-back morphism
$\phi_n^{-1}({\cal K}_{T_{n+1}}) \lo {\cal K}_{T_n}$.  
Following \cite[p.~141]{oc}, 
we say that a ${\cal K}_{\os{\to}{T}}$-module 
${\cal F}$ is {\it coherent}  
if ${\cal F}_n \in {\rm Coh}({\cal K}_{T_n})$ 
for all $n\in {\mab Z}_{>0}$ 
and 
we say that ${\cal F}$ is {\it admissible} (resp.~{\it isocrystalline})
if the natural morphism 
${\cal K}_{T_n}\otimes_{\phi^{-1}_{n+1}({\cal K}_{T_{n+1}})}
\phi^{-1}_{n+1}({\cal F}_{n+1}) \lo {\cal F}_n$ 
is surjective (resp.~isomorphic). 
For an object $E$ of 
$(Y/S)_{\rm conv}\vert_{T}$ and 
for a morphism $T' \lo T$ of quasi-prewidenings of $Y/S$, 
then we have an associated object 
${\cal E}'=\{{\cal E}'_n\}_{n=1}^{\infty}$ 
$({\cal E}'_n:=E_{T'_n})$ in $\os{\to}{T'}$. 
Following \cite[p.~147]{oc}, we say that 
a ${\cal K}_{Y/S}\vert_T$-module $E$ is 
{\it admissible} (resp.~{\it isocrystalline})
if ${\cal E}'$ is coherent and admissible (resp.~isocrystalline) for any morphism $T' \lo T$ 
of quasi-prewidenings.  
In this isocrystalline case, we say that $E$ is an isocrystal of 
${\cal K}_{Y/S}\vert_T$-modules. 
Let $\bet \col \os{\to}{T} \lo T_{\rm zar}$ 
be the natural morphism of topoi 
which is denoted by $\gam$ in \cite[p.~55]{s2} 
(=the log version of $\gam$ in \cite[p.~141]{of}).  
\par

The following is a relative log version of 
\cite[Lemma 2.4]{oc} (with an additional assumption) 
and \cite[Lemma 2.1.26]{s2} (see also \cite[I (2.16)]{s3}):

\begin{lemm}[{\bf \cite[(2.9)]{nhw}}]\label{lemm:bcue} 
Let $(U',T',\iota',u')\lo (U,T,\iota,u)$ be a morphism 
of quasi-prewidenings of $Y/S$ such that $U'=U\times_TT'$.  
Then the natural morphism 
${\mathfrak T}_{U',n}(T') \lo 
\wt{{\mathfrak T}_{U,n}(T)\times_{T}T'}$ 
$(n\in {\mab Z}_{\geq 1})$ is an isomorphism. 
If $\os{\circ}{T}{}' \lo \os{\circ}{T}$ is flat, then the natural morphism 
${\mathfrak T}_{U',n}(T') \lo 
{\mathfrak T}_{U,n}(T)\times_{T}T'$ 
$(n\in {\mab Z}_{\geq 1})$ is an isomorphism. 
\end{lemm}  
\par


\section{Log convergent linearization functors}\label{sec:lll}
In this section we only recall results in \cite[\S3]{nhw} without any proof. 
Let $S$ and $Y/S_1$ be as in \S\ref{sec:logcd}.  
Let $\iota_Y \col Y \os{\subset}{\lo} {\cal Y}$ be an immersion into 
a fine log $S$-formal scheme. 
As in \cite[\S3]{nhw} we do not necessarily assume that ${\cal Y}/S$ is formally log smooth 
except (\ref{theo:pl}) and (\ref{prop:cdfza}). 
\par 
Let ${\mathfrak T}_Y({\cal Y})$ be 
the quasi-prewidening 
$(Y,{\cal Y},\iota_Y, {\rm id}_{Y})$ of $Y/S$. 
Let 
$\{{\mathfrak T}_{Y,n}({\cal Y})\}_{n=1}^{\infty}$ 
(${\mathfrak T}_{Y,n}({\cal Y})
:=(Y_n,{\mathfrak T}_{Y,n}({\cal Y}),Y_n\os{\sus}{\lo}
{\mathfrak T}_{Y,n}({\cal Y}),Y_n\lo Y)$,  
$Y_n:=Y\times_{\cal Y}{\mathfrak T}_{Y,n}({\cal Y})$)
be the system of the universal enlargements of 
${\mathfrak T}_Y({\cal Y})$. 
We call $\{{\mathfrak T}_{Y,n}({\cal Y})\}_{n=1}^{\infty}$ 
the {\it system of the universal enlargements of} 
$\iota_Y$ for short.   
Let 
$((Y/S)_{\rm conv}
\vert_{{\mathfrak T}_Y({\cal Y})},
{\cal K}_{Y/S}\vert_{{\mathfrak T}_Y({\cal Y})})$ 
be the localized ringed topos 
of $((Y/S)_{\rm conv},{\cal K}_{Y/S})$ 
at ${\mathfrak T}_Y({\cal Y})$ 
and let  
$$j_{{\mathfrak T}_Y({\cal Y})} 
\col 
((Y/S)_{\rm conv}\vert_{{\mathfrak T}_Y({\cal Y})},
{\cal K}_{Y/S}\vert_{{\mathfrak T}_Y({\cal Y})}) 
\lo ((Y/S)_{\rm conv}, {\cal K}_{Y/S})$$ 
be the natural morphism of the ringed topoi and let 
\begin{align*}
\varphi^*_{\os{\to}{\mathfrak T}_Y({\cal Y})} 
\col & \{\text{the category of  isocrystals of }
{\cal K}_{\os{\to}{\mathfrak T}_Y({\cal Y})}\text{-modules}\} 
\tag{3.0.1}\label{eqn:phyt} \\
{} & \lo \{\text{the category of isocrystals of }
{\cal K}_{Y/S}\vert_{{\mathfrak T}_Y({\cal Y})}\text{-modules}\}
\end{align*}
be a natural functor which is 
a relative log version (with $\otimes_{\mab Z}{\mab Q}$) 
of the functor in \cite[p.~147]{oc}: 
$\varphi^*_{\os{\to}{\mathfrak T}_Y({\cal Y})}(\{{\cal E}_n\}_{n=1}^{\infty})(T)
=\Gam(T,\psi_n^*({\cal E}_n))$ for $n\in {\mab Z}_{\geq 1}$, where  
$\psi_n \col T\lo {\mathfrak T}_{Y,n}({\cal Y})$ is a morphism such that the composite morphism 
$T\lo {\mathfrak T}_{Y,n}({\cal Y})\lo {\mathfrak T}_Y({\cal Y})$ is 
the given morphism 
for an object $T$ of 
${\rm Conv}(Y/S)\vert_{{\mathfrak T}_Y({\cal Y})}$. 
We also have the following natural functor 
\begin{equation*}
\varphi_{\os{\to}{\mathfrak T}_Y({\cal Y})*} 
\col 
((Y/S)_{\rm conv}
\vert_{{\mathfrak T}_Y({\cal Y})},
{\cal K}_{Y/S}\vert_{{\mathfrak T}_Y({\cal Y})}) 
\lo 
({\os{\to}{\mathfrak T}_Y({\cal Y})}_{\rm zar}, 
{\cal K}_{\os{\to}{\mathfrak T}_Y({\cal Y})})
\tag{3.0.2}\label{eqn:ltts}
\end{equation*} 
by setting $\varphi_{\os{\to}{\mathfrak T}_Y({\cal Y})*}(E)(U) 
:=E((U\os{\subset}{\lo} {\mathfrak T}_{Y,n}({\cal Y}))\lo {\mathfrak T}_Y({\cal Y}))$. 
For an isocrystal of 
${\cal K}_{\os{\to}{\mathfrak T}_Y({\cal Y})}$-module ${\cal E}$, set 
\begin{equation*} 
L^{\rm UE}_{Y/S}({\cal E})
:=j_{{\mathfrak T}_Y({\cal Y})*}
\varphi^*_{\os{\to}{\mathfrak T}_Y({\cal Y})}({\cal E})
\tag{3.0.3}\label{eqn:luys}
\end{equation*} 
by abuse of notation (cf.~\cite[p.~152]{oc}, \cite[p.~95]{s2}) 
(``UE'' is the abbreviation of the universal enlargement). 
Strictly speaking, we should denote  
$L^{\rm UE}_{Y/S}({\cal E})$ by 
$L^{\rm UE}_{(Y\os{\sus}{\lo} {\cal Y})/S}({\cal E})$ since 
$L^{\rm UE}_{Y/S}({\cal E})$ depends on the immersion $Y\os{\sus}{\lo} {\cal Y}$ 
over $S$. 
\par 
For an object $T:=(U,T,\iota,u)$ of ${\rm Conv}(Y/S)$, 
the value $L^{\rm UE}_{Y/S}({\cal E})_T$ of 
$L^{\rm UE}_{Y/S}({\cal E})$ at $T$ is calculated as follows. 
First we have the natural immersion 
$(\iota,\iota_Y\circ u)\col U\lo T\times_S{\cal Y}$ over $S$. 
Let 
$\{{\mathfrak T}_n\}_{n=1}^{\infty}$ 
be the system of the universal enlargements 
of this immersion.  
Then we have the following diagram 
of the inductive systems of enlargements of $Y/S$:  
\begin{equation*}
\begin{CD}
T
@<{\{p'_n\}_{n=1}^{\infty}}<<
\{{\mathfrak T}_n\}_{n=1}^{\infty} 
@>{\{p_n\}_{n=1}^{\infty}}>> 
\{{\mathfrak T}_{Y,n}({\cal Y})\}_{n=1}^{\infty}.
\end{CD}
\tag{3.0.4}\label{cd:ppp}
\end{equation*} 
Let ${\cal E}=\{{\cal E}_n\}_{n=1}^{\infty}$ 
be an isocrystal of 
${\cal K}_{\os{\to}{\mathfrak T}_Y({\cal Y})}$-modules.  
Then, as in the proof of \cite[(5.4)]{oc} and 
\cite[Theorem 2.3.5]{s2},   
we have the following formula: 
\begin{equation*}
L^{\rm UE}_{Y/S}({\cal E})_{T}
=\vpl_np'_{n*}p_n^*({\cal E}_n).  
\tag{3.0.5}\label{eqn:lueyset}
\end{equation*}

\begin{defi}\label{defi:lf}  
We call the functor 
\begin{equation*}
L^{\rm UE}_{Y/S} \col 
\{{\rm isocrystals}~{\rm of}~
{\cal K}_{\os{\to}{\mathfrak T}_Y({\cal Y})}
{\textrm -}{\rm modules}\} 
\lo \{{\cal K}_{Y/S}\text{-}{\rm modules}\} 
\tag{3.1.1}\label{eqn:lvnc}
\end{equation*} 
the {\it log convergent linearization functor with respect to} 
$\iota_Y$ for isocrystals of
${\cal K}_{\os{\to}{\mathfrak T}_Y({\cal Y})}$-modules. 
\end{defi}

In the following we recall results in \cite[\S3]{nhw}.

\begin{lemm}[{\rm \bf \cite[(3.2)]{nhw}}]\label{lemm:rex} 
Let $T=(U,T,\iota,u)$ be 
an exact quasi-$($pre$)$widening of  
$Y/S$ and let $\{T_n\}_{n=1}^{\infty}$ be 
the system of the universal enlargements of $T$. 
Let $\bet_n \col T_n \lo T$ be the natural morphism. 
Then, for each $n$, 
the pull-back functor 
\begin{equation*} 
\beta^*_n \col 
\{{\rm coherent}~{\cal K}_T{\textrm -}{\rm modules}\} 
\lo 
\{{\rm coherent }~{\cal K}_{T_n}{\textrm -}{\rm modules}\} 
\end{equation*} 
is exact. 
\end{lemm}


\begin{prop}[{\rm \bf \cite[(3.5)]{nhw}}]\label{prop:afex} 
If the morphism $\os{\circ}{\cal Y} \lo \os{\circ}{S}$ is affine, 
then the functor $L^{\rm UE}_{Y/S}$ in {\rm (\ref{eqn:lvnc})} 
is right exact.  
\end{prop}

\begin{lemm}[{\rm \bf \cite[(3.6)]{nhw}}]\label{lemm:lcl} 
Assume that $\Om^1_{{\cal Y}/S}$ is locally free. 
Let ${\cal E}:=\{{\cal E}_n\}_{n=1}^{\infty}$ be an isocrystal of 
${\cal K}_{\os{\to}{\mathfrak T}_Y({\cal Y})}$-modules. 
Then there exists a family $\{\nabla^i\}_{i=0}^{\infty}$ 
of natural morphisms  
\begin{equation*} 
\nabla^i  \col 
L^{\rm UE}_{Y/S}({\cal E}\otimes_{{\cal O}_{\cal Y}}
\Om^i_{{\cal Y}/S}) 
\lo 
L^{\rm UE}_{Y/S}({\cal E}\otimes_{{\cal O}_{\cal Y}}
\Om^{i+1}_{{\cal Y}/S}) 
\end{equation*} 
which makes $\{(L^{\rm UE}_{Y/S}({\cal E}\otimes_{{\cal O}_{\cal Y}}
\Om^i_{{\cal Y}/S}),\nabla^i)\}_{i=0}^{\infty}$ 
a complex of ${\cal K}_{Y/S}$-modules. 
\end{lemm}

The following is a relative version of 
\cite[(5.4)]{oc} and \cite[Theorem 2.3.5 (2)]{s2}, 
though the formulation of it is rather different from them:

\begin{theo}[{\rm \bf Log converegent Poncar\'{e} lemma (\cite[(3.7)]{nhw})}]\label{theo:pl}
Assume that ${\cal Y}/S$ is log smooth. 
Let $E$ be an isocrystal of ${\cal K}_{Y/S}$-modules. 
Set ${\cal E}_n:=E_{{\mathfrak T}_{Y,n}({\cal Y})}$ and 
${\cal E}:=\{{\cal E}_n\}_{n=1}^{\infty}$. 
Then the natural morphism 
\begin{equation*}
E \lo
L^{\rm UE}_{Y/S}({\cal E}\otimes_{{\cal O}_{\cal Y}}
\Om^{\bul}_{{\cal Y}/S})
\tag{3.5.1}\label{eqn:epdlym}
\end{equation*}
is a quasi-isomorphism. 
\end{theo}

The following has been essentially proved 
in \cite[(6.6)]{oc} and \cite[Corollary 2.3.8]{s2}: 

\begin{prop}[{\rm \bf (\cite[(3.9)]{nhw})}]\label{prop:cdfza}
Let the notations be as in {\rm (\ref{theo:pl})}. 
Set $Y_n:=Y\times_{\cal Y}{\mathfrak T}_{Y,n}({\cal Y})$. 
Let 
$\bet_n \col Y_n \lo Y$ be 
the projection and identify 
${\mathfrak T}_{Y,n}({\cal Y})_{\rm zar}$ with 
$Y_{n,{\rm zar}}$. 
Then 
\begin{equation*}
Ru^{\rm conv}_{Y/S*}(E)=
\vpl_n\bet_{n*}({\cal E}_n{\otimes}_{{\cal O}_{{\cal Y}}} 
\Om^{\bul}_{{\cal Y}/S})
\tag{3.6.1}\label{eqn:uybo}
\end{equation*}
in $D^+(f^{-1}({\cal K}_S))$. 
\end{prop}



\par 
In the rest of this section we recall results in \cite[\S4]{nhw}, in which 
we have proved that log convergent linearization functors are 
compatible with closed immersions of log (formal) schemes. 
\par 
Let ${\cal V}$, $\pi$, $S$ and $S_1$ 
be as in \S\ref{sec:logcd}.

\begin{lemm}[{\bf {\cite[(4.1)]{nhw}}}]\label{lemm:increp}
Let $\iota_Y \col Y_1 \os{\sus}{\lo} Y_2$ be 
a closed immersion of 
fine log schemes over $S_1$ 
whose underlying schemes are of finite type 
over $\os{\circ}{S}_1$.   
Let $T_j:=(U_j,T_j,\iota_j,u_j)$ 
$(j=1,2)$ be a $($pre$)$widening of $Y_j/S$.
Let ${\iota}_U \col U_1 \os{\sus}{\lo} U_2$ 
and $\iota_T \col T_1 \os{\sus}{\lo} T_2$ 
be closed immersions  
of fine log $S$-formal schemes. 
Assume that ${\iota}_U$ and $\iota_T$
induce a morphism 
$(\iota_U,\iota_T) \col (U_1,T_1,\iota_1,u_1) 
\lo (U_2,T_2,\iota_2,u_2)$ 
of $($pre$)$widenings of $Y_1$ and $Y_2$. 
Let 
$$\iota^{\rm loc}_{\rm conv} \col 
(Y_1/S)_{\rm conv}\vert_{T_1} 
\lo 
(Y_2/S)_{\rm conv}\vert_{T_2}$$ 
be the induced morphism of topoi by $\iota_Y$ 
and $(\iota_U,\iota_T)$.
Let $(U,T,\iota,u)$ be a $($pre$)$widening of  $Y_2/S$ 
over $(U_2,T_2,\iota_2,u_2)$. 
Then $\iota^{{\rm loc}*}_{\rm conv}((U,T,\iota,u))=
(U\times_{U_2}U_1,T\times_{T_2}T_1,
\iota \times_{\iota_2}\iota_1,u\times_{u_2}u_1)$ 
as sheaves in $(Y_1/S)_{\rm conv}\vert_{T_1}$. 
\end{lemm}

\par 
Let $\iota_{Y,Z} \col Z \os{\subset}{\lo} Y$ 
be a closed immersion of fine log 
schemes over $S_1$.
Assume that there exists the following cartesian diagram
\begin{equation*}
\begin{CD}
Z  @>{\iota_{\cal Z}}>> {\cal Z} \\ 
@V{\iota_{Y,Z}}VV @VV{\iota_{{\cal Y},{\cal Z}}}V \\
Y @>{\iota_{\cal Y}}>> {\cal Y},
\end{CD}
\tag{3.7.1}\label{cd:zyectn}
\end{equation*}
where $\iota_{\cal Z}$ and $\iota_{\cal Y}$ 
are closed immersions into 
fine log $p$-adic $S$-formal schemes and 
$\iota_{{\cal Y},{\cal Z}}$ is a closed immersion of 
fine log $p$-adic $S$-formal schemes. 
Set ${\mathfrak T}_Z({\cal Z}):=
(Z,{\cal Z},\iota_{\cal Z},{\rm id}_Z)$ 
and  
${\mathfrak T}_Y({\cal Y})
:=(Y, {\cal Y},\iota_{\cal Y},{\rm id}_Y)$.  
Let 
$\{{\mathfrak T}_{Z,n}({\cal Z})\}_{n=1}^{\infty}$ 
and  
$\{{\mathfrak T}_{Y,n}({\cal Y})\}_{n=1}^{\infty}$ 
be the systems of the universal enlargements 
of $\iota_{\cal Z}$ and $\iota_{\cal Y}$ 
with natural morphisms 
$g_{{\cal Y},n} \col {\mathfrak T}_{Y,n}({\cal Y}) \lo {\cal Y}$ 
and 
$g_{{\cal Z},n} \col {\mathfrak T}_{Z,n}({\cal Z}) \lo {\cal Z}$,  
respectively. 
Because the diagram (\ref{cd:zyectn}) is cartesian, 
we have the following equality by (\ref{lemm:bcue}):
\begin{equation*} 
{\mathfrak T}_{Z,n}({\cal Z})
=\wt{{\mathfrak T}_{Y,n}({\cal Y})\times_{{\cal Y}}{\cal Z}}. 
\tag{3.7.2}\label{eqn:tzyb}
\end{equation*}
Set $Z_n:=Z\times_{\cal Z}{\mathfrak T}_{Z,n}({\cal Z})$ and 
$Y_n:=Y\times_{\cal Y}{\mathfrak T}_{Y,n}({\cal Y})$. 
Let $\iota_{{\cal Z},n}$ 
and $\iota_{{\cal Y},n}$ 
be the natural closed immersions 
$Z_n \os{\sus}{\lo} {\mathfrak T}_{Z,n}({\cal Z})$ 
and 
$Y_n \os{\sus}{\lo} {\mathfrak T}_{Y,n}({\cal Y})$, respectively.  
Let $\iota_{{\cal Y},{\cal Z},n} \col 
{\mathfrak T}_{Z,n}({\cal Z}) \os{\sus}{\lo} 
{\mathfrak T}_{Y,n}({\cal Y})$ be also the natural closed immersion. 
Let 
$$\os{\to}{\iota}{}_{{\cal Y},{\cal Z}*} \col 
\{{\cal K}_{\os{\to}{\mathfrak T}_Z({\cal Z})}\text{-modules}\} 
\lo \{{\cal K}_{\os{\to}{\mathfrak T}_Y({\cal Y})}\text{-modules}\}$$ 
be the natural direct image.

\begin{lemm}[{\bf \cite[(4.3)]{nh2}}]\label{lemm:phiit}  
Let 
$\iota^{\rm loc}_{\rm conv} \col 
(Z/S)_{\rm conv}
\vert_{{\mathfrak T}_{Z,n}({\cal Z})} 
\lo 
(Y/S)_{\rm conv}
\vert_{{\mathfrak T}_{Y,n}({\cal Y})}$ 
be the natural morphism of topoi. 
Then the following diagram 
\begin{equation*}
\begin{CD}
 \{\text{isocrystals of }
{\cal K}_{\os{\to}{\mathfrak T}_Z({\cal Z})}
\text{-modules}\}
@>{\varphi_{\os{\to}{\mathfrak T}_Z({\cal Z})}^*}>> \\
@V{\os{\to}{\iota}{}_{{\cal Y},{\cal Z}*}}VV  \\
\{\text{isocrystals of }
{\cal K}_{\os{\to}{\mathfrak T}_{Y}({\cal Y})}
\text{-modules}\} 
@>{\varphi_{\os{\to}{\mathfrak T}_Y({\cal Y})}^*}>> 
\end{CD}
\tag{3.8.1}
\end{equation*}
\begin{equation*}
\begin{CD}
\{\text{isocrystals of }
{\cal K}_{Z/S}\vert_{{\mathfrak T}_Z({\cal Z})} 
\text{-modules}\}  \\
@VV{\iota^{\rm loc}_{{\rm conv}*}}V\\
\{\text{isocrystals of }{\cal K}_{Y/S}
\vert_{{\mathfrak T}_Y({\cal Y})}\text{-modules}\}
\end{CD}
\end{equation*}
is commutative.
\end{lemm}

\begin{lemm}[{\bf \cite[(4.4)]{nh2}}]\label{lemm:jilt}
The following diagram of topoi
\begin{equation*}
\begin{CD}
(Z/S)_{\rm conv} \vert_{{\mathfrak T}_Z({\cal Z})} 
@>{j_{{\mathfrak T}_{Z}({\cal Z})}}>> (Z/S)_{\rm conv}\\ 
@V{\iota_{{\rm conv}}^{\rm loc}}VV 
@VV{\iota_{\rm conv}}V \\
(Y/S)_{\rm conv}\vert_{{\mathfrak T}_Y({\cal Y})} 
@>{j_{{\mathfrak T}_Y({\cal Y})}}>> 
(Y/S)_{\rm conv} \\  
\end{CD}
\tag{3.9.1}\label{cd:loctop}
\end{equation*}
is commutative. 
\end{lemm}

\begin{coro}[{\bf cf.~\cite[(4.5)]{nhw}}]\label{linc}
There exists a canonical 
isomorphism of functors 
\begin{equation*}
L^{\rm UE}_{Y/S} \circ \os{\to}{\iota}_{{\cal Y}, {\cal Z}*}
\lo 
\iota_{{\rm conv}*}\circ L^{\rm UE}_{Z/S} \tag{3.10.1}\label{coh:com}
\end{equation*}
for the coherent ${\cal K}_{\os{\to}{\mathfrak T}_Z({\cal Z})}$-modules.
\end{coro}

\section{Vanishing cycle sheaves in log convergent topoi}\label{sec:vflvc}
In this section we generalize results in \cite[\S5]{nhw}, 
which can be applicable for SNCL schemes
(see also \cite[(2.3)]{nh2} for the analogues of \cite[\S5]{nhw}).
Though one may think that the proofs in this section 
are imitations of those in \cite[\S5]{nhw}, 
we give the complete proofs because some proofs in this section are rather different 
from those in [loc.~cit.]. 
Because the formulation in this section is better than that of  [loc.~cit.],  
some of the proofs in this section become simpler than those of  [loc.~cit.]. 
\par 
Let $Y'/S'$, $Y/S$, $g\col Y'\lo Y$ and $S'\lo S$ be as in \S\ref{sec:logcd}. 
In this section we assume that $\os{\circ}{Y}{}'=\os{\circ}{Y}$. 
Let $M$ and $N$ be the log structures of $Y'$ and $Y$, respecitvely. 
We assume that $N\subset M$ and that the morphism $g\col Y'\lo Y$ 
is induced by the inclusion $N\subset M$ and ${\rm id}_{\os{\circ}{Y}}$.  
Denote the morphism $g\col Y'\lo Y$ over $S'_1\lo S_1$ by  
\begin{equation*} 
\eps_{(\os{\circ}{Y},M,N)/S'_1/S_1} \col Y' \lo Y 
\tag{4.0.1}\label{eqn:esfl}
\end{equation*} 
in this section. 
The morphism $\eps_{(\os{\circ}{Y},M,N)/S'_1/S_1}$ 
induces the following morphism of log convergent topoi: 
\begin{equation*}
\eps^{\rm conv}_{(\os{\circ}{Y},M,N)/S'/S} \col (Y'/S')_{\rm conv} \lo (Y/S)_{\rm conv}. 
\tag{4.0.2}\label{eqn:tsfl}
\end{equation*}
Let $f'\col Y'\lo S'$ and $f\col Y\lo S$ be the structural morphisms. 
We call the morphism $\eps_{(\os{\circ}{Y},M,N)/S'_1/S_1}$ 
a {\it morphism of log schemes} 
{\it forgetting the structure} $M\setminus N$ over $S'_1/S_1$ 
and we call the morphism 
$\eps^{\rm conv}_{(\os{\circ}{Y},M,N)/S'/S}$  
the {\it morphism of log convergent topoi forgetting the structure} 
$M\setminus N$ over $S'/S$. 
When the log structure of $S$ is trivial and $N$ is trivial, 
we call $\eps^{\rm conv}_{(\os{\circ}{Y},M,N)/S'/S}$  the 
{\it morphism forgetting the log structure} of $Y'$ over $S'/S$. 
For simplicity of notations, we denote 
$\eps_{(\os{\circ}{Y},M,N)/S'_1/S_1}$ and $\eps^{\rm conv}_{(\os{\circ}{Y},M,N)/S'/S}$ 
simply by $\eps$ and $\eps_{\rm conv}$ before (\ref{defi:eikys}). 
The morphism $\eps_{\rm conv}$ also induces the following morphism  
\begin{equation*}
\eps_{\rm conv} \col ((Y'/S')_{\rm conv}, {\cal K}_{Y'/S'}) \lo 
((Y/S)_{\rm conv}, {\cal K}_{Y/S}) 
\tag{4.0.3}\label{eqn:kmono}
\end{equation*} 
of ringed topoi. 
Let 
$$u^{\rm conv}_{Y'/S'}
\col 
((Y'/S')_{\rm conv},{\cal K}_{Y'/S'}) \lo 
(Y'_{\rm zar}, f'{}^{-1}({\cal K}_{S'}))=(Y_{\rm zar}, f'{}^{-1}({\cal K}_{S'}))$$ 
and 
$$u^{\rm conv}_{Y/S}
\col 
((Y/S)_{\rm conv},{\cal K}_{Y/S}) \lo (Y_{\rm zar}, f^{-1}({\cal K}_{S}))$$  
be the projections in (\ref{eqn:uwksr}). 
Let 
$$v\col (Y'_{\rm zar}, f'{}^{-1}({\cal K}_{S'}))\lo (Y_{\rm zar}, f^{-1}({\cal K}_{S}))$$  
be the natural morphism. 
Then we have the following equality:  
\begin{equation*}
u^{\rm conv}_{Y/S}\circ \eps_{\rm conv}= v\circ u^{\rm conv}_{Y'/S'}. 
\tag{4.0.4}\label{eqn:nbepu}
\end{equation*}
\par
Next we localize $\eps_{\rm conv}$'s in (\ref{eqn:tsfl}) and 
(\ref{eqn:kmono}).  
Let 
$$T'=(U',T',\iota',u')=
((\os{\circ}{U}{}',M_{U'}),(\os{\circ}{T}{}',M_{T'}),\iota',u')$$ 
and 
$$T=(U,T,\iota,u)=((\os{\circ}{U},N_U),(\os{\circ}{T},N_T),\iota,u)$$ 
be an exact prewidening of $Y'/S'$ and an exact prewidening of $Y/S$, respectively. 
Let $T' \lo T$ be a morphism of prewidenings over 
$\eps_{(\os{\circ}{Y},M,N)/S'_1/S_1}$. 
Then we have a morphism 
\begin{equation*}
\eps_{\rm conv} \vert_{T'T} \col 
(Y'/S')_{\rm conv}\vert_{T'} 
\lo 
(Y/S)_{\rm conv}\vert_{T} 
\tag{4.0.5}
\end{equation*}
of topoi and a morphism 
\begin{equation*}
\eps_{\rm conv} \vert_{T'T} \col 
((Y'/S')_{\rm conv}\vert_{T'},{\cal K}_{Y'/S'}\vert_{T'}) 
\lo 
((Y/S)_{\rm conv}\vert_{T}, {\cal K}_{Y/S}\vert_{T}) 
\tag{4.0.6}
\end{equation*}
of ringed topoi.
In \cite[(2.3.0.3)]{nh2} and \cite{nhw} we have assumed that, locally on $\os{\circ}{Y}$, 
there exists a finitely generated commutative monoid $P$ 
with unit element such that $P^{\rm gp}$ has no $p$-torsion 
and that there exists a chart $P \lo N$.  
However we do not assume this in this paper 
because we are given $N_T$ already as a setting. 
Henceforth we assume that the morphism $T'\lo T$ is integral and 
that the underlying morphism $\os{\circ}{T}{}'\lo \os{\circ}{T}$ is an immersion.  

The following is a generalization of \cite[(5.2)]{nhw}:  

\begin{lemm}[{\bf cf.~\cite[(5.2)]{nhw}}]\label{lemm:efstex}
Let the notations be as above. 
Then the functor 
$\eps_{\rm conv} \vert_{T'T*}$ is exact. 
\end{lemm}
\begin{proof}
Let $\phi_N \col T'':=(U'',T'',\iota'',u'')  \lo T=(U,T,\iota,u)$ 
($U''=(\os{\circ}{U}{}'',N_{U''})$, $T''=(\os{\circ}{T}{}'',N_{T''})$) 
be a morphism of widenings of $Y/S$ ($N_{U''}$ and $N_{T''}$ are, by definition, 
the log structures of $U''$ 
and $T''$, respectively). 
Consider the following quadruple 
$$(U''',T''',\iota''',u'''):=(U'\times_{U}U'',T'\times_TT'',\iota'\times_U\iota'',u'\times_Uu'').$$ 
Here the fiber products are considered in the category of fine log schemes. 
Because the morphism $T'\lo T$ is integral, $T'''$ and $U'''$ are integral. 
Let 
$$\phi_U \col (U''')^{\circ} \lo \os{\circ}{U}{}'$$ 
and 
$$\phi \col (T''')^{\circ} \lo \os{\circ}{T}{}'$$ 
be the underlying morphisms of 
(formal) schemes obtained by $\phi_N$. 
Set $M_{U''}:=\phi_U^*(M_{U'})$, $M_{T''}:=\phi^*(M_{T'})$,  $U'_M:=((U''')^{\circ},M_{U''})$ and 
$T'_M:=((T''')^{\circ},M_{T''})$. 
Let $\iota'_M \col {U}'_M \os{\sus}{\lo} T'_M$ 
and $u'_M \col U'_M \lo Y'$ 
be the natural morphisms. 
Let 
$$\phi_M \col (U'_M,T'_M,\iota'_M,u'_M)\lo (U',T',\iota',u')$$ 
be the natural morphism. 
Then $(U'_M,T'_M,\iota'_M,u'_M;\phi'_M)$ is an object of $(Y'/S')_{\rm conv} \vert_{T'}$.  
By the same proof as that of \cite[(2.3.3)]{nh2}, we see that 
\begin{equation*}
T'''=T'_M. 
\tag{4.1.1}\label{eqn:fibes}
\end{equation*}  
We also see that 
\begin{equation*}
U'''=U'_M. 
\tag{4.1.2}\label{eqn:fibu}
\end{equation*}
By the formula (\ref{eqn:fibes}) and 
(\ref{eqn:fibu}), $(\eps_{\rm conv} \vert_{T'T})^*(U'',T'',\iota'',u'';\phi_N)$ 
is represented by 
$(U'_M,T'_M,\iota'_M,u'_M;\phi_M)$.  
Therefore, 
for an object $E'$ in $(Y'/S')_{\rm conv} \vert_{T'}$, we obtain the following equalities 
\begin{align*}
{} & \Gamma((U'',T'',\iota'',u'';\phi_N), (\eps_{\rm conv} \vert_{T'T})_*(E')) 
\tag{4.1.3}\label{eqn:ephomep} \\ 
{} & =
{\rm Hom}_{(Y'/S')_{\rm conv}\vert_{T'}}
((\eps_{\rm conv} \vert_{T'T})^*(U'',T'',\iota'',u'';\phi_N), E') \\
{} & =E'(U'_M,T'_M,\iota',u'_M;\phi_M).
\end{align*}
Because $\os{\circ}{T}{}'\lo \os{\circ}{T}$ is an immersion, 
the projection $(T'_M)^{\circ}\lo \os{\circ}{T''}$ is an immersion. 
Hence, by the formula (\ref{eqn:ephomep}),  
we see that the functor $\eps_{\rm conv} \vert_{T'T*}$ is exact. 
\end{proof}

\begin{lemm}[{\bf cf.~\cite[(5.3)]{nhw}}]\label{lemm:ymjeps}
Let the notations be as above. Then the following diagram of 
topoi is commutative$:$
\begin{equation*}
\begin{CD}
(Y'/S')_{\rm conv} \vert_{T'} @>{j_{T'}}>> (Y'/S')_{\rm conv} \\ 
@V{\eps_{\rm conv} \vert_{T'T}}VV @VV{\eps_{\rm conv}}V  \\
(Y/S)_{\rm conv} \vert_{T} 
@>{j_{T}}>> (Y/S)_{\rm conv}.  \\ 
\end{CD}
\tag{4.2.1}\label{cd:ymnj}
\end{equation*}
The obvious analogue of {\rm (\ref{cd:ymnj})} for ringed topoi 
$((Y'/S')_{\rm conv},{\cal K}_{Y'/S'})$ and $((Y/S)_{\rm conv},{\cal K}_{Y/S})$ 
also holds.
\end{lemm}
\begin{proof} 
Let $G$ be an object of $(Y/S)_{\rm conv}$. 
By the proof of (\ref{lemm:efstex}),  
$(\eps_{\rm conv} \vert_{T'T})^* \allowbreak (T) \allowbreak = T'$.
Hence $(\eps_{\rm conv} \vert_{T'T})^*j_{T}^*(G)
=(\eps_{\rm conv} \vert_{T'T})^*(G\times T)=
\eps^*_{\rm conv}(G)\times T'=j_{T'}^*\eps^*_{\rm conv}(G)$. 
Hence the former statement follows. 
The latter statement immediately follows from this.
\end{proof}


\begin{lemm}[{\bf \cite[(5.4)]{nhw}}]\label{lemm:epsex}
Assume that $\os{\circ}{T}$ is affine. 
Let $E'$ be an admissible sheaf of 
${\cal K}_{Y'/S'}\vert_{T'}$-modules in $(Y'/S')_{\rm conv}\vert_{T'}$. 
Then the canonical morphism 
\begin{equation*} 
\eps_{{\rm conv}*}j_{T'*}(E') 
\lo R\eps_{{\rm conv}*}j_{T'*}(E') 
\tag{4.3.1}\label{eqn:ecjme} 
\end{equation*} 
is an isomorphism in the derived category 
$D^+({\cal K}_{Y/S})$. 
\end{lemm}
\begin{proof}
The proof is the same as that of 
\cite[(5.4.1)]{nhw} by using 
(\ref{lemm:efstex}) and 
(\ref{lemm:ymjeps}). 
\end{proof}

\begin{coro}[{\bf cf.~{\cite[(5.5)]{nhw}}}]\label{coro:epnr}
Let the notations be as above. 
Let ${\cal Y}':=(\os{\circ}{\cal Y},{\cal M})$ be a fine log $S'$-formal scheme. 
Let $\iota \col Y' \os{\subset}{\lo} {\cal Y}'$ be an immersion over $S'$.  
Let 
$\iota\col Y\os{\sus}{\lo} {\cal Y}$ be an immersion over $S$ 
fitting into the following commutative diagram 
of log formal schemes over $S'\lo S:$ 
\begin{equation*}
\begin{CD}
Y' @>{\iota'}>> {\cal Y}'  \\ 
@V{\eps^{\rm conv}_{(\os{\circ}{Y},M,N)/S'/S}}VV 
@VV{}V  \\
Y @>{\iota}>> {\cal Y}.   \\ 
\end{CD}
\tag{4.4.1}\label{cd:cchgmn}
\end{equation*}
Set ${\mathfrak T}:=(Y,{\cal Y},\iota,{\rm id}_Y)$ and 
${\mathfrak T}':=(Y',{\cal Y}',\iota',{\rm id}_{Y'})$. 
Let $\{{\mathfrak T}_n\}_{n=1}^{\infty}$ and $\{{\mathfrak T}'_n\}_{n=1}^{\infty}$ be 
the systems of the universal enlargements 
of $\iota$ and $\iota'$, respectively, 
with the following natural commutative diagram 
for $n\in {\mab Z}_{\geq 1}:$ 
\begin{equation*}
\begin{CD}
{\mathfrak T}'_n @>{g'_n}>> {\cal Y}'\\ 
@VVV @VV{}V  \\
{\mathfrak T}_n 
@>{g_n}>> {\cal Y}, \\ 
\end{CD}
\tag{4.4.2}\label{cd:zmnhg}
\end{equation*}
where $g'_n$ and $g_n$ are the natural morphisms. 
Assume that the underlying morphism 
$\os{\circ}{\mathfrak T}{}'_n\lo \os{\circ}{\mathfrak T}_n$
of the left vertical morphism ${\mathfrak T}'_n\lo {\mathfrak T}_n$ 
in {\rm (\ref{cd:zmnhg})} is an immersion. 
$($Then this morphism becomes a closed immersion.$)$ 
Let $L^{\rm UE}_{Y'/S'}$ 
be the log convergent linearization functor  for isocrystals of 
${\cal K}_{\os{\to}{\mathfrak T}{}'}$-modules  with respect to $\iota'$.  
Let ${\cal E}'$ be an isocrystal of 
${\cal K}_{\os{\to}{\mathfrak T}{}'}$-modules. 
Then the canonical morphism 
\begin{equation*}
\eps_{{\rm conv}*}L^{\rm UE}_{Y'/S'}({\cal E}') 
\lo R\eps_{{\rm conv}*}L^{\rm UE}_{Y'/S'}({\cal E}') 
\tag{4.4.3}\label{eqn:epsre}
\end{equation*}
is an isomorphism in $D^+({\cal K}_{Y/S})$. 
\end{coro}
\begin{proof} 
(\ref{coro:epnr}) immediately follows from (\ref{lemm:epsex}). 
\end{proof} 


\begin{lemm}[{\bf cf.~{\cite[(5.6)]{nhw}}}]\label{lemm:itmne}
Let the notations and the assumption be as in {\rm (\ref{coro:epnr})}. 
Let $L^{\rm UE}_{Y/S}$ be the log convergent linearization functor for isocrystals of 
${\cal K}_{\os{\to}{\mathfrak T}}$-modules. 
Then there exist natural isomorphisms 
\begin{equation*}
L^{\rm UE}_{Y/S} \os{\sim}{\lo} 
\eps_{\rm conv*}
L^{\rm UE}_{Y'/S'} 
\tag{4.5.1}\label{eqn:lnelmn}
\end{equation*} 
of functors for isocrystals of ${\cal K}_{\os{\to}{\mathfrak T}{}'}$-modules. 
Here we consider isocrystals of ${\cal K}_{\os{\to}{\mathfrak T}{}'}$-modules
as isocrystals of ${\cal K}_{\os{\to}{\mathfrak T}}$-modules by using the closed immersions 
${\mathfrak T}'_n\os{\sus}{\lo} {\mathfrak T}_n$'s.   
\end{lemm}
\begin{proof}
Let $\eps_{\rm conv} \vert_{{\mathfrak T}'_n{\mathfrak T}_n} 
\col  
(Y'/S')_{\rm conv}\vert_{{\mathfrak T}'_n} 
\lo 
(Y/S)_{\rm conv}\vert_{{\mathfrak T}_n}$ 
$(n\in {\mab Z}_{\geq 1})$ 
be the localized morphism of $\eps_{\rm conv}$ at 
${\mathfrak T}'_n$ and ${\mathfrak T}_n$.  
Using the formula (\ref{eqn:ephomep}), 
we can immediately check that 
$(\eps_{\rm conv} \vert_{{\mathfrak T}'_n{\mathfrak T}_n})_*
\varphi_{{\mathfrak T}'_n}^{*}=\varphi_{{\mathfrak T}_n}^*$ holds for 
isocrystals of ${\cal K}_{\os{\to}{\mathfrak T}{}'}$-modules 
${\cal E}'=\{{\cal E}'_n\}_{n=1}^{\infty}$ for sufficiently large $n$ 
for an object of $(T'',\phi)$ of $(Y/S)_{\rm conv}\vert_{{\mathfrak T}_n}$
as follows: 
\begin{align*} 
\Gam(T'',(\eps_{\rm conv} 
\vert_{{\mathfrak T}'_n{\mathfrak T}_n})_*\varphi_{{\mathfrak T}'_n}^{*}({\cal E}'_n))
&=\Gam(T'_M,\varphi_{{\mathfrak T}'_n}^{*}({\cal E}'_n))
=\Gam(T'_M,{\cal K}_{T'_M}\otimes_{{\cal K}_{{\mathfrak T}'_n}}{\cal E}'_n)\\
&=\Gam(T'_M,({\cal K}_{{\mathfrak T}'_n}
\otimes_{{\cal K}_{{\mathfrak T}_n}}{\cal K}_{T''})\otimes_{{\cal K}_{{\mathfrak T}'_n}}{\cal E}'_n)\\
&=\Gam(T'',{\cal K}_{T''}\otimes_{{\cal K}_{{\mathfrak T}_n}}{\cal E}'_n). 
\end{align*} 
Here we have assumed that the morphism $T''\lo {\mathfrak T}$ factors through 
$T''\lo {\mathfrak T}_n$ for some $n$ and we have used 
the assumption in (\ref{lemm:itmne}). 
Furthermore, by (\ref{lemm:ymjeps}) 
we obtain the following commutative diagram: 
\begin{equation*} 
\begin{CD} 
\{\text{isocrystals of }
{\cal K}_{\os{\to}{\mathfrak T}{}'}\text{-modules}\}   
@>{\varphi^*_{\os{\to}{\mathfrak T}{}'}}>> 
\{\text{isocrystals of }
{\cal K}_{Y'/S'}\vert_{{\mathfrak T}'}\text{-modules}\} \\
@V{\bigcap}VV @V{\{(\eps_{\rm conv} \vert_{{\mathfrak T}_n'{\mathfrak T}_n})_*\}_{n=1}^{\infty}}VV \\
\{\text{isocrystals of }
{\cal K}_{\os{\to}{\mathfrak T}}\text{-modules}\}  
@>{{\varphi^*_{\os{\to}{\mathfrak T}}}}>> 
\{\text{isocrystals of }{\cal K}_{Y/S}\vert_{\mathfrak T}\text{-modules}\}   
\end{CD}
\tag{4.5.2}\label{cd:longcom}
\end{equation*} 
\begin{equation*} 
\begin{CD}  
@>{j_{{\mathfrak T}{}'_*}}>> 
\{\text{crystals of }
{\cal K}_{Y'/S'}\text{-modules}\}   \\ 
@. @V{\eps_{\rm conv*}}VV  \\
@>{j_{{\mathfrak T}_*}}>> 
\{\text{crystals of }{\cal K}_{Y/S}\text{-modules}\}.
\end{CD}
\end{equation*} 
Hence we obtain the isomorphism (\ref{eqn:lnelmn}). 
\par 
\end{proof}


\begin{defi}\label{defi:eikys}
Let $E'$ be an isocrystal of ${\cal K}_{Y'/S'}$-modules. 
We call $R\eps^{\rm conv}_{(\os{\circ}{Y},M,N)/S'/S*}(E')$ 
the {\it vanishing cycle sheaf} of $E'$ {\it along} $M\setminus N$ over $S'/S$. 
We call $R\eps^{\rm conv}_{(\os{\circ}{Y},M,N)/S'/S*}({\cal K}_{Y'/S'})$ the 
{\it vanishing cycle sheaf} of $Y/S'/S$ {\it along} $M\setminus N$.
If $N$ is trivial, we omit the word 
``along $M\setminus N$''. 
\end{defi} 

\par 
Now assume that ${\cal Y}'$ is formally log smooth over $S'$. 
Then we have the following integrable connection for 
a coherent crystal 
${\cal E}':=\{{\cal E}'_n\}_{n=1}^{\infty}$ of 
${\cal K}_{\os{\to}{\mathfrak T}{}'}$-modules by (\ref{lemm:lcl}): 
\begin{align*} 
\nabla^i\col 
L^{\rm UE}_{Y'/S'}({\cal E}'\otimes_{{\cal O}_{{\cal Y}'}}\Om^i_{{\cal Y}'/S}) 
\lo 
L^{\rm UE}_{Y'/S'}({\cal E}'\otimes_{{\cal O}_{{\cal Y}'}}
\Om^{i+1}_{{\cal Y}'/S'}). 
\end{align*} 
Hence we can define the following integrable connection 
\begin{align*} 
\nabla^i \col 
L^{\rm UE}_{Y/S}({\cal E}'\otimes_{{\cal O}_{{\cal Y}'}}\Om^i_{{\cal Y}'/S'}) 
\lo 
L^{\rm UE}_{Y/S}({\cal E}'\otimes_{{\cal O}_{{\cal Y}'}}\Om^{i+1}_{{\cal Y}'/S'}) 
\tag{4.6.1}\label{cd:lnhg}
\end{align*} 
fitting into the following  commutative$:$ diagram
\begin{equation*}
\begin{CD}
L^{\rm UE}_{Y/S}({\cal E}'\otimes_{{\cal O}_{{\cal Y}'}}\Om^i_{{\cal Y}'/S'}) 
@>{\nabla^i}>> L^{\rm UE}_{Y/S}
({\cal E}'\otimes_{{\cal O}_{{\cal Y}'}}\Om^{i+1}_{{\cal Y}'/S'})  \\ 
@V{\simeq}VV @VV{\simeq}V  \\
\eps^{\rm conv}_{(\os{\circ}{Y},M,N)/S'/S*}
L^{\rm UE}_{Y'/S'}({\cal E}'\otimes_{{\cal O}_{{\cal Y}}}\Om^i_{{\cal Y}'/S'}) 
@>{\eps^{\rm conv}_{(\os{\circ}{Y},M,N)/S'/S*}(\nabla^i)}>> 
\eps^{\rm conv}_{(\os{\circ}{Y},M,N)/S'/S*}
L^{\rm UE}_{Y'/S'}({\cal E}'\otimes_{{\cal O}_{{\cal Y}'}}
\Om^{i+1}_{{\cal Y}'/S'}).  \\ 
\end{CD}
\tag{4.6.2}\label{cd:lnfhg}
\end{equation*}


\par 
The following theorem is a main result in this section:

\begin{theo}
[{\bf Poincar\'{e} lemma of a vanishing cycle sheaf (cf.~\cite[(5.8)]{nhw})}]
\label{theo:cpvcs} 
Let the notations and the assumption 
be as in {\rm (\ref{coro:epnr})}. 
Assume that ${\cal Y}'$ is formally log smooth over $S'$. 
Let $E'$ be an isocrystal of ${\cal K}_{Y'/S'}$-modules.   
Let $({\cal E}',\nabla')$ be the ${\cal K}_{\os{\to}{\mathfrak T}{}'}$-module 
with integrable connection  corresponding to $E:$ 
$\nabla \col {\cal E}' \lo {\cal E}'\otimes_{{\cal O}_{{\cal Y}'}}\Om^1_{{\cal Y}'/S'}$. 
Assume that we are given the commutative diagram {\rm (\ref{cd:cchgmn})}. 
Then there exists a canonical isomorphism 
\begin{equation*}
R\eps^{\rm conv}_{(\os{\circ}{Y},M,N)/S*}(E') \os{\sim}{\lo} 
L^{\rm UE}_{Y/S}({\cal E}'\otimes_{{\cal O}_{{\cal Y}'}}
\Om^{\bul}_{{\cal Y}'/S'})
\tag{4.7.1}\label{eqn:ceslvc}
\end{equation*} 
in $D^+({\cal K}_{Y/S})$. 
\end{theo} 
\begin{proof} 
By (\ref{eqn:epdlym}) we obtain the following isomorphism 
\begin{align*}
E' \os{\sim}{\lo} 
L^{\rm UE}_{Y'/S'}
({\cal E}'\otimes_{{\cal O}_{{\cal Y}'}}\Om^{\bul}_{{\cal Y}'/S'}).
\tag{4.7.2}\label{eqn:cplvc}   
\end{align*}  
Applying $R\eps^{\rm conv}_{(\os{\circ}{Y},M,N)/S'/S*}$ to 
both hand sides of (\ref{eqn:cplvc}) and 
using (\ref{coro:epnr}), (\ref{lemm:itmne}) and the definition of 
the differential (\ref{cd:lnhg}), 
we obtain the following: 
\begin{align*}
R\eps^{\rm conv}_{(\os{\circ}{Y},M,N)/S'/S*}(E') 
& \os{\sim}{\lo}  
R\eps^{\rm conv}_{(\os{\circ}{Y},M,N)/S'/S*}
L^{\rm UE}_{Y'/S'}({\cal E}'\otimes
_{{\cal O}_{{\cal Y}'}}\Om^{\bul}_{{\cal Y}'/S'})\\
{} & \os{\sim}{\longleftarrow}
\eps^{\rm conv}_{(\os{\circ}{Y},M,N)/S'/S*}
L^{\rm UE}_{Y'/S'}({\cal E}'\otimes_{{\cal O}_{{\cal Y}'}}\Om^{\bul}_{{\cal Y}'/S'})\\ 
{}&= 
L^{\rm UE}_{Y/S}({\cal E}'
\otimes_{{\cal O}_{{\cal Y}'}}\Om^{\bul}_{{\cal Y}'/S'}).
\end{align*}
\end{proof} 

\begin{rema}\label{rema:fgmm} 
In \cite[(5.8)]{nhw} we have assumed that 
the open log analogue of ${\cal Y}/S$ is log smooth, 
while we do not assume that 
${\cal Y}/S$ is log smooth in this paper. This is a different point from the theory in [loc.~cit.]. 
In ???
\end{rema}

\begin{prop}\label{prop:nia} 
Assume that ${\cal Y}'/S'$ and ${\cal Y}/S$ 
are formally log smooth over $S'$ and $S$, respectively. 
Let ${\cal E}:=\{{\cal E}_n\}_{n=1}^{\infty}$ be 
a coherent crystal  of 
${\cal K}_{\os{\to}{\mathfrak T}}$-modules. 
Set ${\cal E}':= {\cal E}\otimes_{{\cal K}_{\os{\to}{\mathfrak T}}}
{\cal K}_{\os{\to}{\mathfrak T}{}'}$. 
Let 
\begin{align*} 
\nabla \col 
L^{\rm UE}_{Y/S}({\cal E}) \lo 
L^{\rm UE}_{Y/S}({\cal E}\otimes_{{\cal O}_{\cal Y}}
\Om^{1}_{{\cal Y}/S})
\tag{4.9.1}\label{cd:lnaefhg}
\end{align*} 
be the integrable connection obtained in {\rm (\ref{lemm:lcl})}. 
Let 
\begin{align*} 
\nabla^i\col L^{\rm UE}_{Y/S}({\cal E}\otimes_{{\cal O}_{\cal Y}}\Om^i_{{\cal Y}/S}) 
\lo 
 L^{\rm UE}_{Y/S}
({\cal E}\otimes_{{\cal O}_{\cal Y}}\Om^{i+1}_{{\cal Y}/S})  
\quad (i\in {\mab N})
\end{align*} 
be the induced differential by {\rm (\ref{cd:lnaefhg})}. 
Then the following diagram is commutative$:$
\begin{equation*}
\begin{CD}
L^{\rm UE}_{Y/S}({\cal E}'\otimes_{{\cal O}_{{\cal Y}'}}\Om^i_{{\cal Y}'/S'}) 
@>{\nabla^i}>> L^{\rm UE}_{Y/S}
({\cal E}'\otimes_{{\cal O}_{{\cal Y}'}}\Om^{i+1}_{{\cal Y}'/S'})  \\ 
@AAA @AAA  \\
L^{\rm UE}_{Y/S}({\cal E}\otimes_{{\cal O}_{{\cal Y}}}\Om^i_{{\cal Y}/S}) 
@>{\nabla^i}>>
L^{\rm UE}_{Y/S}({\cal E}\otimes_{{\cal O}_{\cal Y}}
\Om^{i+1}_{{\cal Y}/S}). 
\end{CD}
\tag{4.9.2}\label{cd:lnafhg}
\end{equation*}
\end{prop}
\begin{proof} 
Though the proof is quite long, the proof is the same as that of \cite[(5.6)]{nh2}
by writing down the connection in (\ref{lemm:lcl}). 
\end{proof}

\section{Log convergent linearization functors over the underlying formal scheme 
of a $p$-adic formal family of log points}\label{eclf}
In this section we consider the (log) convergent linearization functor 
in the case where $S'$ and $S$ in the previous section are 
the underlying formal scheme of a 
$p$-adic formal ${\cal V}$-family of log points. 
\par 
Let $S$ be a formal family of log points, that is, 
$S$ is an fs(=fine and saturated) log formal scheme which is locally isomorphic to 
${\mab N}\oplus {\cal O}_S^*$ with structural morphism 
${\mab N}\owns n\lom 0^n\in {\cal O}_S$. 
In the following we assume that $\os{\circ}{S}$ is 
a flat $p$-adic formal ${\cal V}$-scheme in the sense of \cite{of}.  
We call this $S$ a $p$-adic ${\cal V}$-formal family of log points. 
Let us recall the formal scheme $\ol{S}$ defined in \cite{nb} as follows. 
\par 
Recall that, for the log structure $M_S$ of $S$, we denote 
$M_S/{\cal O}_S^*$ by $\ol{M}_S$. 
Let $S=\bigcup_{i\in I}S_i$ be an open covering of $S$ such that 
$\ol{M}_{S_i}\simeq {\mab N}$.  Take local sections 
$t_i$'s in $\Gam(S_i,M_S)$'s such that 
the images of $t_i$'s in $\Gam(S_i,\ol{M}_S)$'s are generators. 
Set $S_{ij}:= S_i\cap S_j$.  
Then there exist unique sections $u_{ji} \in 
\Gam(S_{ij},{\cal O}^*_S)$ such that 
\begin{equation*} 
t_{j}\vert_{S_{ij}}=u_{ji}t_{i}\vert_{S_{ij}}
\tag{5.0.1}\label{eqn:tkjd} 
\end{equation*} 
in $\Gam(S_i\cap S_j,M_S)$ for all $i$, $j$'s. 
Consider the ${\cal V}$-formal scheme 
$\wh{\mab A}{}^1_{\os{\circ}{S}_i}=
\ul{\rm Spf}_{\os{\circ}{S}_i}({\cal O}_{S_i}\{\tau_i\})$  
and the log ${\cal V}$-formal scheme 
$(\wh{\mab A}{}^1_{\os{\circ}{S}_i}, 
({\mab N}\owns n \lom \tau_i^n \in {\cal O}_{S_i}\{\tau_i\})^a)$.  
Denote this log $p$-adic ${\cal V}$-formal scheme by $\ol{S}_i$. 
Then, by patching $\ol{S}_i$ and $\ol{S}_j$ along 
$\ol{S}_{ij}:=\ol{S}_i\cap \ol{S}_j$ by the following equation 
\begin{equation*} 
\tau_j\vert_{\ol{S}_{ij}}=u_{ji}\tau_i\vert_{\ol{S}_{ij}},
\tag{5.0.2}\label{eqn:tkijd} 
\end{equation*} 
we have the log $p$-adic ${\cal V}$-formal scheme $\ol{S}=\bigcup_{i\in I}\ol{S}_i$. 
The isomorphism class of the log formal scheme $\ol{S}$ 
is independent of the choice of the system of 
generators $t_{i}$'s. 
By using this easy fact and by considering 
the refinement of two open coverings of $S$, 
we see that the isomorphism class
of the log formal scheme $\ol{S}$ is independent of 
the choice of the open covering $S=\bigcup_{i\in I}S_i$. 
Because 
$\os{\circ}{\ol{S}}_i
=\os{\circ}{S}_i\times_{{\rm Spf}({\cal V})}{\rm Spf}({\cal V}\{\tau_i\})$, 
the natural morphism $\ol{S} \lo \os{\circ}{S}$ is formally log smooth 
(cf.~\cite[(3.5)]{klog1}). 
In the following we denote $\tau_i$ by $t_i$ by convenience of notation. 
\par 
By killing ``$t_i$'''s in ${\cal O}_{\ol{S}}$, 
we have the following natural exact closed immersion 
\begin{equation*} 
S\os{\sus}{\lo} \ol{S}  
\tag{5.0.3}\label{eqn:stas}
\end{equation*} 
over $\os{\circ}{S}$. 
The local section $d\log t_i\in \Om^1_{\ol{S}_i/\os{\circ}{S}_i}$ defines a global section of 
$\Om^1_{\ol{S}/\os{\circ}{S}}$ by the equation (\ref{eqn:tkijd}) since 
$d\log u_{ij}=0$ in $\Om^1_{\ol{S}/\os{\circ}{S}}$. 
We denote this global section by $d\log t$. We also denote by $d\log t$ the image of 
$d\log t$ by the surjective morphism $\Om^1_{\ol{S}/\os{\circ}{S}}\lo 
\Om^1_{S/\os{\circ}{S}}$. 
\par 
Set $S_1:=\ul{\rm Spec}^{\log}_S({\cal O}_S/\pi{\cal O}_S)$.  
Let $Y/S_1$ be a log smooth scheme.  
Let $Y \os{\sus}{\lo} \ol{\cal Q}$ be an immersion into a log smooth 
log $p$-adic formal ${\cal V}$-scheme over $\ol{S}$. 
Let $\ol{T}$ be the prewidening over $\ol{S}$ obtained by the 
immersion $Y \os{\sus}{\lo} \ol{\cal Q}{}^{\rm ex}$. 
Set ${\cal Q}^{\rm ex}:=\ol{\cal Q}{}^{\rm ex}\times_{\ol{S}}S$. 
Let $T$ be the prewidening over $S$ obtained by the immersion 
$Y\os{\sus}{\lo} {\cal Q}^{\rm ex}$.  
Let $g\col {\cal Q}^{\rm ex}\lo S$ be the structural morphism. 
Let $\{(\ol{V}_n,\ol{T}_n)\}_{n=1}^{\infty}$ be the system of 
the universal enlargements of the immersion 
$Y\os{\sus}{\lo} \ol{\cal Q}$ over $\ol{S}$. 
Set $(V_n,T_n):=(\ol{V}_n,\ol{T}_n)\times_{\ol{S}}S$. 
Let $\bet_n \col V_n \lo Y$ 
be the natural morphism.  
Let $F$ be a flat isocrystal of ${\cal K}_{Y/\os{\circ}{S}}$-modules. 
Let $(\ol{\cal F}_n,\ol{\nabla}_n)$ be the flat coherent ${\cal K}_{\ol{T}_n}$-modules 
with integrable connection obtained by $F$: 
$\ol{\nabla}_n \col \ol{\cal F}_n\lo \ol{\cal F}_n
\otimes_{{\cal O}_{\ol{\cal Q}^{\rm ex}}}\Om^1_{\ol{\cal Q}^{\rm ex}/\os{\circ}{S}}$.  
Set ${\cal F}_n:=\ol{\cal F}_n{\otimes}_{{\cal O}_{\ol{S}}}{\cal O}_S$. 
Since $dt=td\log t$ for a local section $t$ such that 
${\cal O}_{\ol{S}}={\cal O}_S\{t\}$ locally on $\ol{S}$, 
$\ol{\nabla}_n$ 
induces the following integrable connection 
\begin{align*} 
{\nabla}_n \col {\cal F}_n\lo {\cal F}_n
\otimes_{{\cal O}_{{\cal Q}^{\rm ex}}}\Om^1_{{\cal Q}^{\rm ex}/\os{\circ}{S}}
\end{align*} 
and we obtain the projective system $\{{\nabla}_n\}_{n=0}^{\infty}$ and 
the following integrable connection 
\begin{align*} 
\vpl_n{\nabla}_n \col \vpl_n\bet_{n*}{\cal F}_n\lo 
(\vpl_n\bet_{n*}{\cal F}_n)\otimes_{{\cal O}_{{\cal Q}^{\rm ex}}}
\Om^1_{{\cal Q}^{\rm ex}/\os{\circ}{S}}. 
\end{align*} 
For simplicity of notation, set 
\begin{align*} 
{\cal F}\otimes_{{\cal O}_{{\cal Q}^{\rm ex}}}
\Om^{\bul}_{{\cal Q}^{\rm ex}/\os{\circ}{S}}:=
\{{\cal F}_n\otimes_{{\cal O}_{{\cal Q}^{\rm ex}}}
\Om^{\bul}_{{\cal Q}^{\rm ex}/\os{\circ}{S}}\}_{n=0}^{\infty}. 
\end{align*} 
Let $L^{\rm UE}_{Y/\os{\circ}{S}}$ and $L^{\rm UE}_{\os{\circ}{Y}/\os{\circ}{S}}$ 
be the (log) convergent linearization functors 
for coherent ${\cal K}_{\os{\to}{\ol{T}}}$-modules 
with respect to the immersions 
$Y \os{\sus}{\lo} \ol{\cal Q}{}^{\rm ex}$ over $\os{\circ}{S}$ 
and $\os{\circ}{Y} \os{\sus}{\lo} ({\ol{\cal Q}{}^{\rm ex}})^{\circ}$ 
over $\os{\circ}{S}$. 
Then the morphism 
\begin{align*} 
(\ol{T}_n)^{\circ}=({\mathfrak T}_{Y,n}(\ol{\cal Q}{}^{\rm ex}))^{\circ} \lo 
{\mathfrak T}_{\os{\circ}{Y},n}((\ol{\cal Q}{}^{\rm ex})^{\circ})
\end{align*} 
is an isomorphism and the assumption in (\ref{coro:epnr}) is satisfied. 
Let
$\eps_{Y/\os{\circ}{S}}\col ((Y/\os{\circ}{S})_{\rm conv},{\cal K}_{Y/\os{\circ}{S}})
\lo 
((\os{\circ}{Y}/\os{\circ}{S})_{\rm conv},{\cal K}_{\os{\circ}{Y}/\os{\circ}{S}})$ 
be the morphism forgetting the log structure of $Y$ over $\os{\circ}{S}$. 
Consider the complex 
$$\eps^{\rm conv}_{Y/\os{\circ}{S}*}
L^{\rm UE}_{Y/\os{\circ}{S}}
(\ol{\cal F}\otimes_{{\cal O}_{\ol{\cal Q}^{\rm ex}}}\Om^{\bul}_{\ol{\cal Q}{}^{\rm ex}/\os{\circ}{S}}).
$$
By (\ref{eqn:lnelmn}) and the log Poincar\'{e} lemma of a vanishing cycle sheaf 
((\ref{theo:cpvcs})), 
we have the following equality: 
\begin{align*} 
\eps^{\rm conv}_{Y/\os{\circ}{S}*}
L^{\rm UE}_{Y/\os{\circ}{S}}
(\ol{\cal F}\otimes_{{\cal O}_{\ol{\cal Q}{}^{\rm ex}}}
\Om^{\bul}_{\ol{\cal Q}{}^{\rm ex}/\os{\circ}{S}})
=L^{\rm UE}_{\os{\circ}{Y}/\os{\circ}{S}}
(\ol{\cal F}\otimes_{{\cal O}_{\ol{\cal Q}{}^{\rm ex}}}
\Om^{\bul}_{\ol{\cal Q}{}^{\rm ex}/\os{\circ}{S}})
(\simeq R\eps^{\rm conv}_{Y/\os{\circ}{S}*}
L^{\rm UE}_{Y/\os{\circ}{S}}
(\ol{\cal F}
\otimes_{{\cal O}_{\ol{\cal Q}{}^{\rm ex}}}\Om^{\bul}_{\ol{\cal Q}{}^{\rm ex}/\os{\circ}{S}})). 
\end{align*} 
However this complex is {\it not} our desired complex. 
Instead we consider the following sheaves for various $i$'s: 
\begin{align*} 
L^{\rm UE}_{\os{\circ}{Y}/\os{\circ}{S}}
({\cal F}\otimes_{{\cal O}_{\ol{\cal Q}^{\rm ex}}}\Om^i_{\ol{\cal Q}{}^{\rm ex}/\os{\circ}{S}})
&=L^{\rm UE}_{\os{\circ}{Y}/\os{\circ}{S}}
({\cal F}\otimes_{{\cal O}_{{\cal Q}^{\rm ex}}}\Om^i_{{\cal Q}{}^{\rm ex}/\os{\circ}{S}})
\tag{5.0.4}\label{eqn:plavc} \\
&=\eps^{\rm conv}_{Y/\os{\circ}{S}*}(L^{\rm UE}_{Y/\os{\circ}{S}}
({\cal F}\otimes_{{\cal O}_{{\cal Q}^{\rm ex}}}\Om^i_{{\cal Q}{}^{\rm ex}/\os{\circ}{S}}))\\
&(\simeq R\eps^{\rm conv}_{Y/\os{\circ}{S}*}(L^{\rm UE}_{Y/\os{\circ}{S}}
({\cal F}\otimes_{{\cal O}_{{\cal Q}^{\rm ex}}}\Om^i_{{\cal Q}{}^{\rm ex}/\os{\circ}{S}}))).
\end{align*} 
Here note that 
$\Om^i_{\ol{\cal Q}{}^{\rm ex}/\os{\circ}{S}}\not=
\Om^i_{(\ol{\cal Q}{}^{\rm ex})^{\circ}/\os{\circ}{S}}$ 
and that ${\cal Q}{}^{\rm ex}$ is not formally log smooth over 
$\os{\circ}{S}$ in general. 

\begin{lemm}\label{lemm:cpax}
The differential 
\begin{align*} 
\nabla^i \col L^{\rm UE}_{Y/\os{\circ}{S}}
(\ol{\cal F}\otimes_{{\cal O}_{\ol{\cal Q}{}^{\rm ex}}}
\Om^i_{\ol{\cal Q}{}^{\rm ex}/\os{\circ}{S}})
\lo 
L^{\rm UE}_{Y/\os{\circ}{S}}
(\ol{\cal F}\otimes_{{\cal O}_{\ol{\cal Q}{}^{\rm ex}}}
\Om^{i+1}_{\ol{\cal Q}{}^{\rm ex}/\os{\circ}{S}})
\tag{5.1.1}\label{eqn:epdomaep}
\end{align*}
induces the following differential$:$ 
\begin{align*} 
\nabla^i \col 
\eps^{\rm conv}_{Y/\os{\circ}{S}*}L^{\rm UE}_{Y/\os{\circ}{S}}
({\cal F}\otimes_{{\cal O}_{{\cal Q}{}^{\rm ex}}}\Om^i_{{\cal Q}{}^{\rm ex}/\os{\circ}{S}})
\lo 
\eps^{\rm conv}_{Y/\os{\circ}{S}*}L^{\rm UE}_{Y/\os{\circ}{S}}
({\cal F}\otimes_{{\cal O}_{{\cal Q}{}^{\rm ex}}}\Om^{i+1}_{{\cal Q}{}^{\rm ex}/\os{\circ}{S}}). 
\tag{5.1.2}\label{eqn:eldomaep}
\end{align*}
\end{lemm}
\begin{proof} 
Let $(\os{\circ}{U}{}',\os{\circ}{T}{}')$ be an object of $(\os{\circ}{Y}/\os{\circ}{S})_{\rm conv}$. 
Since $(\ol{\cal Q}{}^{\rm ex})^{\circ}$ is formally smooth over $\os{\circ}{S}$, 
there exists a morphism $\os{\circ}{T}{}'\lo (\ol{\cal Q}{}^{\rm ex})^{\circ}$ over $\os{\circ}{S}$
which lifts the composite morphism 
$\os{\circ}{U}{}'\lo \os{\circ}{Y}\os{\sus}{\lo} 
(\ol{\cal Q}{}^{\rm ex})^{\circ}$. 
This morphism induces a morphism 
$\os{\circ}{T}{}'\lo \os{\circ}{\ol{T}}$ of prewideings of $\os{\circ}{Y}/\os{\circ}{S}$.  
Let
\begin{align*} 
\ol{\varphi}{}^*: 
\{{\rm isocrystals~of~}{\cal K}_{\os{\to}{\ol{T}}}{\rm -modules}\}
\lo 
\{{\rm isocrystals~of~}{\cal K}_{Y/\os{\circ}{S}}\vert_{\ol{T}}{\rm -modules}\}
\end{align*} 
be the functor (\ref{eqn:phyt}). 
Consider the following diagram 
\begin{equation*}
\begin{CD}
\{{\mathfrak T}_{\os{\circ}{U}{}',n}(\os{\circ}{T}{}'
\times_{\os{\circ}{S}}\ol{T})\}_{n=1}^{\infty} 
@>{\{\ol{p}_n\}_{n=1}^{\infty}}>> 
\{\ol{T}_n\}_{n=1}^{\infty}\\ 
@V{\{p'_n\}_{n=1}^{\infty}}VV \\
\os{\circ}{T}{}'@. @. .
\end{CD}
\tag{5.1.3}\label{cd:polpp}
\end{equation*} 
By (\ref{lemm:efstex}) and (\ref{lemm:ymjeps}) we have the following: 
\begin{align*}
(\eps^{\rm conv}_{Y/\os{\circ}{S}*}L^{\rm UE}_{Y/\os{\circ}{S}}
({\cal F}\otimes_{{\cal O}_{{\cal Q}{}^{\rm ex}}}\Om^i_{{\cal Q}{}^{\rm ex}
/\os{\circ}{S}})))_{\os{\circ}{T}{}'} 
&=(\eps^{\rm conv}_{Y/\os{\circ}{S}*}j_{\ol{T}*}\ol{\varphi}{}^*
({\cal F}\otimes_{{\cal O}_{{\cal Q}{}^{\rm ex}}}
\Om^i_{{\cal Q}{}^{\rm ex}/\os{\circ}{S}})
)_{\os{\circ}{T}{}'}
\tag{5.1.4}\label{eqn:ephaomep} \\ 
{} &=(j_{\os{\circ}{\ol{T}}*}
(\eps^{\rm conv}_{Y/\os{\circ}{S}})_{\ol{T}\,\os{\circ}{\ol{T}}{}*}
\ol{\varphi}{}^*({\cal F}\otimes_{{\cal O}_{{\cal Q}{}^{\rm ex}}}
\Om^i_{{\cal Q}{}^{\rm ex}/\os{\circ}{S}}))_{\os{\circ}{T}{}'} \\
{}& =((\eps^{\rm conv}_{Y/\os{\circ}{S}*})_{\ol{T}\,\os{\circ}{\ol{T}}{}*}
\ol{\varphi}{}^*({\cal F}\otimes_{{\cal O}_{{\cal Q}{}^{\rm ex}}}
\Om^i_{{\cal Q}{}^{\rm ex}/\os{\circ}{S}}))_{\os{\circ}{T}{}'\times_{\os{\circ}{S}}
\os{\circ}{\ol{T}}}\\
{} & =  (
\ol{\varphi}{}^*({\cal F}\otimes_{{\cal O}_{{\cal Q}{}^{\rm ex}}}
\Om^i_{{\cal Q}{}^{\rm ex}/\os{\circ}{S}}))_{\os{\circ}{T}{}'\times_{\os{\circ}{S}}{\ol{T}}}\\
{}& =\vpl_np'_{n*}\ol{p}_n^*({\cal F}
\otimes_{{\cal O}_{{\cal Q}{}^{\rm ex}}}\Om^i_{{\cal Q}{}^{\rm ex}/\os{\circ}{S}}).   
\end{align*} 
On the other hand, we have the following similarly: 
\begin{align*}
(\eps^{\rm conv}_{Y/\os{\circ}{S}*}L^{\rm UE}_{Y/\os{\circ}{S}}
(\ol{\cal F}\otimes_{{\cal O}_{{\cal Q}{}^{\rm ex}}}\Om^i_{\ol{\cal Q}{}^{\rm ex}/\os{\circ}{S}}))
_{\os{\circ}{T}{}'}
&=(\eps^{\rm conv}_{Y/\os{\circ}{S}*}j_{\ol{T}*}\ol{\varphi}^*
(\ol{\cal F}\otimes_{{\cal O}_{{\cal Q}{}^{\rm ex}}}\Om^i_{{\cal Q}{}^{\rm ex}/\os{\circ}{S}}))
_{\os{\circ}{T}{}'}
\tag{5.1.5}\label{eqn:epomep} \\ 
{} &=\vpl_np'_{n*}\ol{p}_n^*(\ol{\cal F}
\otimes_{{\cal O}_{\ol{\cal Q}{}^{\rm ex}}}\Om^i_{\ol{\cal Q}{}^{\rm ex}/\os{\circ}{S}}).
\end{align*} 
Because 
$dt=td\log t=0$ in ${\cal F}\otimes_{{\cal O}_{\cal Q}}\Om^1_{{\cal Q}{}^{\rm ex}/\os{\circ}{S}}$, 
we have the set $\{\nabla^i\}_{i=0}^{\infty}$ of 
differentials in (\ref{eqn:eldomaep}).  
\end{proof}

\begin{rema}\label{rema:ds}
The differential $\nabla^i$ in (\ref{eqn:eldomaep}) is not a special case of 
the differential (\ref{cd:lnhg}) in the case ${\cal Y}'/S'={\cal Q}^{\rm ex}/\os{\circ}{S}$ 
since ${\cal Q}^{\rm ex}$ is not necessarily formally log smooth over $\os{\circ}{S}$. 
\end{rema}

\begin{coro}\label{coro:ddn}
The differential {\rm (\ref{eqn:eldomaep})} 
induces the following differential$:$
\begin{align*}  
\nabla^i \col L^{\rm conv}_{\os{\circ}{Y}/\os{\circ}{S}}
({\cal F}\otimes_{{\cal O}_{{\cal Q}^{\rm ex}}}
\Om^i_{{\cal Q}^{\rm ex}/\os{\circ}{S}})
\lo 
L^{\rm conv}_{\os{\circ}{Y}/\os{\circ}{S}}
({\cal F}\otimes_{{\cal O}_{{\cal Q}^{\rm ex}}}
\Om^{i+1}_{{\cal Q}^{\rm ex}/\os{\circ}{S}}). 
\end{align*} 
\end{coro}
\begin{proof} 
This follows from (\ref{lemm:itmne}) and (\ref{lemm:cpax}). 
\end{proof} 
By (\ref{coro:ddn}) we can consider the following complex 
\begin{align*}  
L^{\rm conv}_{\os{\circ}{Y}/\os{\circ}{S}}
({\cal F}\otimes_{{\cal O}_{{\cal Q}^{\rm ex}}}
\Om^{\bul}_{{\cal Q}^{\rm ex}/\os{\circ}{S}})\simeq 
\eps_{Y/\os{\circ}{S}*}
L^{\rm conv}_{Y/\os{\circ}{S}}
({\cal F}\otimes_{{\cal O}_{{\cal Q}^{\rm ex}}}
\Om^{\bul}_{{\cal Q}^{\rm ex}/\os{\circ}{S}}),
\end{align*} 
which is isomorphic to 
$R\eps_{Y/\os{\circ}{S}*}
L^{\rm conv}_{Y_/\os{\circ}{S}}
({\cal F}\otimes_{{\cal O}_{{\cal Q}^{\rm ex}}}
\Om^{\bul}_{{\cal Q}^{\rm ex}/\os{\circ}{S}})$ ((\ref{coro:epnr})). 
In \S\ref{sec:mplf} below we consider 
the cosimplicial version of this complex. 


\begin{lemm}\label{lemm:objt}
Consider $T$ as an object
of $(Y/\os{\circ}{S})_{\rm conv}$ 
by using the morphism $S\lo \os{\circ}{S}$. 
Let 
\begin{align*}
\varphi^*_{/\os{\circ}{S}}
\col & \{\text{the category of isocrystals of }
{\cal K}_{\os{\to}{T}}\text{-modules}\} \\
{} & \lo \{\text{the category of isocrystals of }
{\cal K}_{Y/\os{\circ}{S}}\vert_{T}\text{-modules}\}
\end{align*}
be a natural functor defined by 
$\varphi^*_{/\os{\circ}{S}}(\{{\cal E}_n\}_{n=1}^{\infty})(T')
=\Gam(T',\psi_n^*({\cal E}_n))$ for $n\in {\mab Z}_{\geq 1}$, where  
$\psi_n \col T'\lo T_n$ is a morphism such that the composite morphism 
$T'\lo T_n\lo T$ is 
the given morphism 
for an object $T'$ of ${\rm Conv}(Y/\os{\circ}{S})\vert_{T}$. 
Let 
\begin{align*}
\os{\circ}{\varphi}{}^*
\col & \{\text{the category of isocrystals of }
{\cal K}_{\os{\to}{T}}\text{-modules}\} \\
{} & \lo \{\text{the category of isocrystals of }
{\cal K}_{\os{\circ}{Y}/\os{\circ}{S}}\vert_{\os{\circ}{T}}\text{-modules}\}
\end{align*}
be a natural functor defined by 
$\os{\circ}{\varphi}{}^*(\{{\cal E}_n\}_{n=1}^{\infty})(T')
=\Gam(T',\psi_n^*({\cal E}_n))$ for $n\in {\mab Z}_{\geq 1}$, where  
$\psi_n \col T'\lo \os{\circ}{T}_n$ is a morphism 
such that the composite morphism 
$T'\lo \os{\circ}{T}_n\lo \os{\circ}{T}$ is 
the given morphism 
for an object $T'$ of ${\rm Conv}(\os{\circ}{Y}/\os{\circ}{S})\vert_{\os{\circ}{T}}$. 
$($Note that, because the immersion $Y\os{\sus}{\lo} {\cal Q}^{\rm ex}$ is exact, 
$\{\os{\circ}{T}_n:=(T_n)^{\circ}\}_{n=1}^{\infty}$ is the system of 
the universal enlargements of the immersion 
$\os{\circ}{Y}\os{\sus}{\lo} ({\cal Q}^{\rm ex})^{\circ}.)$ 
Then 
\begin{align*} 
L_{Y/\os{\circ}{S}}^{\rm UE}({\cal E}) 
= j_{T*}\varphi^*_{/\os{\circ}{S}}(\{{\cal E}_n\}_{n=1}^{\infty}) \in 
(Y/\os{\circ}{S})_{\rm conv}, 
\tag{5.4.1}\label{eqn:oxlndz}
\end{align*} 
\begin{align*} 
L_{\os{\circ}{Y}/\os{\circ}{S}}^{\rm UE}({\cal E}) 
= j_{\os{\circ}{T}*}\os{\circ}{\varphi}{}^*(\{{\cal E}_n\}_{n=1}^{\infty}) \in 
(\os{\circ}{Y}/\os{\circ}{S})_{\rm conv}. 
\tag{5.4.2}\label{eqn:oxlpndz}
\end{align*} 
\end{lemm}
\begin{proof} 
For an object $(U',T')$ of $(Y/\os{\circ}{S})_{\rm conv}$, 
the morphism $U'\lo Y\os{\sus}{\lo} \ol{T}$ 
factors through a morphism $U'\lo Y\os{\sus}{\lo} T$. 
Hence $U'=U'\times_{(T'\times_{\os{\circ}{S}}\ol{T})}(T'\times_{\os{\circ}{S}}T)$. 
Consider the following diagram: 
\begin{equation*}
\begin{CD}
\{{\mathfrak T}_{U',n}(T'\times_{\os{\circ}{S}}\ol{T})\times_{\ol{T}_n}T_n\}_{n=1}^{\infty} 
@>{\{p_n\}_{n=1}^{\infty}}>> \{T_n\}_{n=1}^{\infty}@>>>T\\
@V{\{p'_n\}_{n=1}^{\infty}}VV @VVV @VVV \\
\{{\mathfrak T}_{U',n}(T'\times_{\os{\circ}{S}}\ol{T})\}_{n=1}^{\infty} 
@>{\{\ol{p}_n\}_{n=1}^{\infty}}>> \{\ol{T}_n\}_{n=1}^{\infty}@>>> \ol{T}\\ 
@V{\{\ol{p}'_n\}_{n=1}^{\infty}}VV \\
T'.@. @. 
\end{CD}
\tag{5.4.3}\label{cd:pdpp}
\end{equation*} 
The right square in (\ref{cd:pdpp}) is cartesian. 
By (\ref{lemm:bcue}), 
\begin{align*} 
{\mathfrak T}_{U',n}(T'\times_{\os{\circ}{S}}T)
=\wt{{\mathfrak T}_{U',n}(\ol{T}{}'\times_{\os{\circ}{S}}\ol{T})}\times_{\ol{T}}T. 
\tag{5.4.4}\label{cd:pdbpp}
\end{align*} 
Hence we have the following: 
\begin{align*} 
L_{Y/\os{\circ}{S}}^{\rm UE}({\cal E})_{T'}&= \vpl_n\ol{p}{}'_{n*}\ol{p}_n^*({\cal E}_n)
=\vpl_np'_{n*}p_n^*({\cal E}_n). 
\tag{5.4.5}\label{cd:plbpp}
\end{align*} 
The last projective limit in (\ref{cd:plbpp}) is equal to 
$j_{T*}\varphi^*_{/\os{\circ}{S}}(\{{\cal E}_n\}_{n=1}^{\infty}) \in (Y/\os{\circ}{S})_{\rm conv}$. 
We obtain the formula (\ref{eqn:oxlpndz}) similarly. 
\end{proof} 

The lemma (\ref{lemm:objt}) 
implies that the linearizations $L_{Y/\os{\circ}{S}}$ and $L_{\os{\circ}{Y}/\os{\circ}{S}}$
are equal to the linearizations with respect to the immersions 
$Y\os{\sus}{\lo} {\cal Q}^{\rm ex}$ over $\os{\circ}{S}$ and 
$\os{\circ}{Y}\os{\sus}{\lo} \os{\circ}{\cal Q}{}^{\rm ex}$ over $\os{\circ}{S}$, respectively, for 
coherent ${\cal K}_{\os{\to}{T}}$-moduless.  
Note that ${\cal Q}{}^{\rm ex}$ and 
$({\cal Q}{}^{\rm ex})^{\circ}$
are not necessarily log smooth over $\os{\circ}{S}$ 
and not smooth over $\os{\circ}{S}$, respectively.

\begin{exem}\label{exem:sss}
In \S\ref{sec:mplf} below  we consider the following two cases for 
a ${\cal V}$-formal family $S$ of log points and an  SNCL scheme $Y$ over $S_1$:   
\par 
(1)   
$$\eps_{(\os{\circ}{Y},M,N)/S'/S}:=
\eps_{(\os{\circ}{X},M_X,{\cal O}_X^*)/S/\os{\circ}{S}}=:\eps_{X/S}.$$ 
\par 
(2) 
$$\eps_{(\os{\circ}{Y},M,N)/S'/S}:=\eps_{(\os{\circ}{X},M_X,{\cal O}_X^*)/\os{\circ}{S}/\os{\circ}{S}}=:\eps_{X/\os{\circ}{S}}.$$
(The latter case is quite less familiar than the former case.) 
\end{exem}

\section{Log convergent linearization functors of SNCL schemes}\label{sec:lcs}
In this section we show several properties of 
log convergent linearization functors of SNCL schemes. 
These results are SNCL versions of results in \cite[\S6]{nhw} 
in which several properties of the log convergent linearization functor of 
a smooth scheme with a relative SNCD have been shown.     
\par
For a log smooth log formal scheme ${\cal Q}$ over a log formal scheme $T$ 
and for a nonnegative integer $i$ and an integer $k$,  
set 
\begin{equation*} 
P_k
\Om^i_{{\cal Q}/T} =
\begin{cases} 
0 & (k<0), \\
{\rm Im}(\Om^k_{{\cal Q}/T}
{\otimes}_{{\cal O}_{\cal Q}}\Om^{i-k}_{\os{\circ}{\cal Q}/T} \lo 
\Om^i_{{\cal Q}/T}) & (0\leq k\leq i), \\
\Om^i_{{\cal Q}/T} & (k > i).
\end{cases}
\tag{6.0.1}\label{eqn:leolpw} 
\end{equation*}
Then we have a filtration 
$P:=\{P_k\}_{k\in {\mab Z}}$ 
on $\Om^i_{{\cal Q}/T}$. 
\bigskip 
In this section we assume that there exists an immersion 
$X \os{\sus}{\lo} \ol{\cal P}$
into a formally log smooth log ${\cal V}$-formal scheme over $\ol{S}$. 
Let $X\os{\sus}{\lo} \ol{\cal P}{}^{\rm ex}$ 
be the exactification of this immersion. 
\par 
Set ${\cal P}:=\ol{\cal P}\times_{\ol{S}}S$ and ${\cal P}^{\rm ex}:=\ol{\cal P}{}^{\rm ex}\times_{\ol{S}}S$. 
Then ${\cal P}^{\rm ex}$ is the exactification of 
the immersion $X\os{\sus}{\lo} {\cal P}$.  
\par 
For a finitely generated monoid $Q$ 
and a set $\{q_1,\ldots,q_n\}$ 
$(n\in {\mab Z}_{\geq 1})$ of generators of $Q$, 
we say that $\{q_1,\ldots,q_n\}$ is {\it minimal} if 
there does not exit a generator $\{q'_1,\ldots,q'_r\}$ of $Q$ such that $r<n$. 
Let $Y=(\os{\circ}{Y},M_Y)$ be 
an fs(=fine and saturated) log (formal) scheme. 
Let $y$ be a point of $\os{\circ}{Y}$. 
Let $m_{1,y},\ldots, m_{r,y}$ be local sections of $M_Y$ around $y$ 
whose images in $M_{Y,y}/{\cal O}_{Y,y}^*$ 
form a minimal set of generators of $M_{Y,y}/{\cal O}_{Y,y}^*$. 
Let $D(M_Y)_i$ $(1\leq i \leq r)$ 
be the local closed subscheme of 
$\os{\circ}{Y}$ defined by the ideal sheaf generated by 
the image of $m_{i,y}$ in ${\cal O}_Y$. 
For a nonnegative integer $k$, 
let $D^{(k+1)}(M_Y)$ be the disjoint union of 
the $(k+1)$-fold intersections of different $D(M_Y)_i$'s. 
Assume that 
\begin{equation*} 
M_{Y,y}/{\cal O}^*_{Y,y}\simeq {\mab N}^r
\tag{6.0.2}\label{eqn:epynr}
\end{equation*}
for any point $y$ of $\os{\circ}{Y}$ 
and for some $r\in {\mab N}$ depending on $y$. 
It is easy to show that the scheme $D^{(k+1)}(M_Y)$ is 
independent of the choice of $m_{1,y},\ldots, m_{r,y}$ 
and it is globalized (cf.~\cite[\S4]{nh3}). 
We denote this globalized scheme 
by the same symbol $D^{(k+1)}(M_Y)$. 
Set $D^{(0)}(M_Y):=\os{\circ}{Y}$. 
Let $c^{(k+1)} \col D^{(k+1)}(M_Y) \lo \os{\circ}{Y}$ 
be the natural morphism. 
\par 
As in \cite[(3.1.4)]{dh2} and \cite[(2.2.18)]{nh2}, 
we have an orientation sheaf $\vp_{\rm zar}^{(k+1)}(D(M_Y))$ 
$(k\in {\mab N})$ associated to the set of $D(M_Y)_i$'s. 
We have the following equality 
\begin{equation*} 
c^{(k+1)}_*\vp_{\rm zar}^{(k+1)}(D(M_Y)) 
=\bigwedge^{k+1}(M_Y^{\rm gp}/{\cal O}_Y^*) 
\tag{6.0.3}\label{eqn:bkezps}
\end{equation*} 
as sheaves of abelian groups on $\os{\circ}{Y}$.   

\par 
The following has been proved in \cite{nh3}: 

\begin{prop}[{\bf \cite[(4.3)]{nh3}}]\label{prop:mmoo} 
Let $g\col Y \lo Y'$ be a morphism of fs log schemes 
satisfying the condition {\rm (\ref{eqn:epynr})}. 
Assume that, for each point $y\in \os{\circ}{Y}$ 
and for each member $m$ of the minimal generators 
of $M_{Y,y}/{\cal O}^*_{Y,y}$, there exists 
a unique member $m'$ of the minimal generators of 
$M_{Y',\os{\circ}{g}(y)}/{\cal O}^*_{Y',\os{\circ}{g}(y)}$ 
such that $g^*(m')\in m^{{\mab Z}_{>0}}$. 
Then there exists a canonical morphism 
$g^{(k+1)}\col D^{(k+1)}(M_Y)\lo D^{(k+1)}(M_{Y'})$ fitting 
into the following commutative diagram of schemes$:$
\begin{equation*} 
\begin{CD} 
D^{(k+1)}(M_Y) @>{g^{(k+1)}}>> D^{(k+1)}(M_{Y'}) \\ 
@VVV @VVV \\ 
\os{\circ}{Y} @>{\os{\circ}{g}}>> \os{\circ}{Y}{}'. 
\end{CD} 
\tag{6.1.1}\label{cd:dmgdm}
\end{equation*} 
\end{prop} 

\par 
When $Y=X$, 
we denote $D^{(k+1)}(M_Y)$ by $\os{\circ}{X}{}^{(k)}$ and 
$\vp_{\rm zar}^{(k+1)}(D(M_Y))$ by 
$\vp^{(k)}_{\rm zar}(\os{\circ}{X}/S_1)$. 
The sheaf $\vp^{(k)}_{\rm zar}(\os{\circ}{X}/S_1)$ is 
extended to an abelian sheaf  
$\vp^{(k)}_{{\rm conv}}(\os{\circ}{X}/S)$ 
in the convergent topos $(\os{\circ}{X}{}^{(k)}/\os{\circ}{S})_{\rm conv}$. 
We call
$\vp^{(k)}_{\rm zar}(\os{\circ}{X}/S_1)$ and $\vp^{(k)}_{\rm conv}(\os{\circ}{X}/S)$ 
the {\it zariskian orientation sheaf}  of 
$\os{\circ}{X}{}^{(k)}/\os{\circ}{S}_1$ and the {\it convergent orientation sheaf} 
of $\os{\circ}{X}{}^{(k)}/\os{\circ}{S}$, respectively.


Because ${\cal P}^{\rm ex}$ is a formal SNCL scheme over $S$ (\cite[(1.14.1)]{nb}), 
we can define $(\os{\circ}{\cal P}{}^{\rm ex})^{(k)}$ as usual and we see that 
$$D^{(k+1)}(M_{\ol{\cal P}{}^{\rm ex}})=D^{(k+1)}(M_{{\cal P}^{\rm ex}})
=\os{\circ}{\cal P}{}^{{\rm ex},(k)}.$$ 
Let 
$a^{(k)} \col \os{\circ}{X}{}^{(k)}\lo \os{\circ}{X}$ and 
$b^{(k)}: \os{\circ}{\cal P}{}^{{\rm ex},(k)}
\lo \os{\circ}{\cal P}{}^{\rm ex}$ 
be the morphisms induced by the natural closed immersions. 
\par 
Let the middle objects in the following table 
be the prewidenings obtained by
the left immersions over $S$ and $\os{\circ}{S}$ and 
let the right objects be the systems 
of the universal enlargements of 
these prewidenings:

\begin{equation*}
\begin{tabular}{|l|l|l|l|} \hline 
$X \os{\subset}{\lo} {\cal P}^{\rm ex}$  
& $T$ & 
$\os{\to}{T}:=\{T_n\}_{n=1}^{\infty}$ & over $S$ \\
\hline   
$X \os{\subset}{\lo} \ol{\cal P}{}^{\rm ex}$  
& $\ol{T}$ & $\os{\to}{\ol{T}}:=\{\ol{T}_n\}_{n=1}^{\infty}$ & over $\os{\circ}{S}$
\\ \hline   
$\os{\circ}{X} \os{\subset}{\lo} \os{\circ}{\ol{\cal P}}{}^{\rm ex}$  
& $\os{\circ}{\ol{T}}$ & 
$\os{\to}{\os{\circ}{\ol{T}}}:=\{\os{\circ}{\ol{T}}_n\}_{n=1}^{\infty}$
& over $\os{\circ}{S}$\\ \hline    
$\os{\circ}{X}{}^{(k)} \os{\subset}{\lo} 
\os{\circ}{\cal P}{}^{{\rm ex},(k)}$ & $\os{\circ}{T}{}^{(k)}$ & 
$\os{\to}{\os{\circ}{T}{}^{(k)}}:=
\{\os{\circ}{T}{}^{(k)}_n\}_{n=1}^{\infty}$ & over $\os{\circ}{S}$
\\ \hline 
\end{tabular}
\end{equation*}
\bigskip
\parno
Let $g_n \col T_n \lo  {\cal P}^{\rm ex}$ 
and  
$\os{\circ}{g}{}^{(k)}_n \col \os{\circ}{T}{}^{(k)}_n \lo \os{\circ}{\cal P}{}^{{\rm ex},(k)}$ 
$(n\in {\mab Z}_{\geq 1})$ 
be the natural morphisms.  
Let  $c^{(k)}_n \col \os{\circ}{T}{}^{(k)}_n \lo \os{\circ}{\ol{T}}_n$
be the morphism  induced by 
$b^{(k)} \col \os{\circ}{\cal P}{}^{{\rm ex},(k)} \lo \os{\circ}{\ol{\cal P}}{}^{\rm ex}$. 

\begin{lemm}[{\rm {\bf cf.~\cite[(2.2.16) (1)]{nh2}}}]
\label{lemm:dpinc}
The natural morphism 
$\os{\circ}{T}{}^{(k)}_n \lo 
\os{\circ}{\ol{T}}_n\times_{\os{\circ}{\ol{\cal P}}{}^{\rm ex}} 
(\os{\circ}{\cal P}{}^{\rm ex})^{(k)}=
\os{\circ}{{T}}_n\times_{\os{\circ}{{\cal P}}{}^{\rm ex}} 
(\os{\circ}{\cal P}{}^{\rm ex})^{(k)}$ 
is an isomorphism. 
\end{lemm}
\begin{proof}
Since the defining ideal of the immersion $\os{\circ}{X} \os{\sus}{\lo} 
\os{\circ}{\ol{\cal P}}{}^{\rm ex}$ has no relation with 
the ideal of the image of the log structure of 
$\ol{\cal P}{}^{\rm ex}$, the proof is similar to that of \cite[(6.6)]{nhw}. 
\end{proof}

\par
Set 
$X_n:=X\times_{{\cal P}^{\rm ex}}T_n$ 
and 
$\os{\circ}{X}{}^{(k)}_n:=\os{\circ}{X}{}^{(k)}
\times_{\os{\circ}{\ol{\cal P}}{}^{\rm ex}}
\os{\circ}{T}{}_n=\os{\circ}{X}{}^{(k)}
\times_{\os{\circ}{\cal P}{}^{\rm ex}}
\os{\circ}{T}{}_n$ $(n\in {\mab N})$.  
As usual, we denote the left representable objects 
in the following table 
by the right ones for simplicity of notation:

\bigskip
\begin{tabular}{|l|l|} \hline 
$(X_n \os{\subset}{\lo}T_n)\in 
(X/S)_{\rm conv}$ & $T_n$ \\  \hline 
$(X_n \os{\subset}{\lo}\ol{T}_n)\in 
(X/\os{\circ}{S})_{\rm conv}$ & $\ol{T}_n$ \\  \hline 
$(\os{\circ}{X}_n \os{\subset}{\lo}\os{\circ}{\ol{T}}_n)\in 
(\os{\circ}{X}/\os{\circ}{S})_{\rm conv}$ & $\os{\circ}{\ol{T}}_n$ \\  \hline 
$(\os{\circ}{X}{}^{(k)}_n \os{\subset}{\lo} \os{\circ}{T}{}^{(k)}_n)\in 
(\os{\circ}{X}{}^{(k)}/\os{\circ}{S})_{\rm conv}$
& $\os{\circ}{T}{}^{(k)}_n$  \\ \hline 
\end{tabular}
\bigskip
\parno
\par 
Let

\bigskip
\begin{tabular}{|l|} \hline 
$j_T: (X/S)_{\rm conv}\vert_{T} \lo (X/S)_{\rm conv}$ 
\\ \hline
$j_{\ol{T}}: (X/\os{\circ}{S})_{\rm conv}\vert_{\ol{T}} \lo (X/\os{\circ}{S})_{\rm conv}$ 
\\ \hline
$j_{\os{\circ}{\ol{T}}}: (\os{\circ}{X}/\os{\circ}{S})_{\rm conv}\vert_{\os{\circ}{\ol{T}}} 
\lo (\os{\circ}{X}/\os{\circ}{S})_{\rm conv}$ 
\\ \hline
$j_{\os{\circ}{T}{}^{(k)}}: 
(\os{\circ}{X}{}^{(k)}/\os{\circ}{S})_{\rm conv}\vert_{\os{\circ}{T}{}^{(k)}} \lo 
(\os{\circ}{X}{}^{(k)}/\os{\circ}{S})_{\rm conv}$  
\\ \hline 
\end{tabular}
\bigskip  
\parno
be localization functors and let

\bigskip
\begin{tabular}{|l|} \hline 
$\varphi^*: 
\{{\rm isocrystals~of~}{\cal K}_{\os{\to}{T}}{\rm -modules}\}
\lo $
$\{{\rm isocrystals~of~}{\cal K}_{X/S}\vert_T{\rm -modules}\}$ \\ \hline
$\ol{\varphi}{}^*: 
\{{\rm isocrystals~of~}{\cal K}_{\os{\to}{\ol{T}}}{\rm -modules}\}
\lo $
$\{{\rm isocrystals~of~}{\cal K}_{X/\os{\circ}{S}}\vert_{\ol{T}}{\rm -modules}\}$ \\ \hline 
$\os{\circ}{\ol{\varphi}}{}^*: 
\{{\rm isocrystals~of~}{\cal K}_{\os{\to}{{\ol{T}}}}{\rm -modules}\}
\lo $
$\{{\rm isocrystals~of~}{\cal K}_{\os{\circ}{X}/\os{\circ}{S}}\vert_{\os{\circ}{\ol{T}}}{\rm -modules}\}$ \\ \hline 
$\os{\circ}{\varphi}{}^{(k)*}: 
\{{\rm isocrystals~of~}
{\cal K}_{\os{\to}{\os{\circ}{T}{}^{(k)}}}{\rm -modules}\}\lo $  
$\{{\rm isocrystals~of~}{\cal K}_{\os{\circ}{X}{}^{(k)}/\os{\circ}{S}}\vert_{\os{\circ}{T}{}^{(k)}}{\rm -modules}\}$  \\ \hline 
\end{tabular} 
\bigskip
\parno 
be the morphisms of pull-backs 
defined in (\ref{eqn:phyt}).   
Note that 
\begin{align*} 
\{{\rm isocrystals}~{\rm of}~{\cal K}_{\os{\to}{T}}{\rm -modules}\}  
\subset \{{\rm isocrystals}~{\rm of}~
{\cal K}_{\os{\to}{{\ol{T}}}}{\rm -modules}\}. 
\tag{6.2.1}\label{ali:mal}
\end{align*} 
Note also that isocrystals of 
${\cal K}_{X/S}\vert_T$-modules e.t.c. 
are assumed to be coherent.

\par  
Let ${\cal E}:=\{{\cal E}_n\}_{n=1}^{\infty}$ be 
a coherent ${\cal K}_{\os{\to}{T}}$-module.  
Let ${\cal F}:=\{{\cal F}_n\}_{n=1}^{\infty}$ be 
a coherent ${\cal K}_{\os{\to}{\ol{T}}}$-module.  
Let ${\cal G}:=\{{\cal G}_n\}_{n=1}^{\infty}$ 
be a coherent ${\cal K}_{\os{\to}{\os{\circ}{T}{}^{(k)}}}$-module. 
Set 
$$L_{X/S}^{\rm UE}({\cal E}) 
:= j_{T*} \varphi^*(\{{\cal E}_n\}_{n=1}^{\infty}) 
\in (X/S)_{\rm conv},$$ 
$$L_{X/\os{\circ}{S}}^{\rm UE}({\cal F}) 
:= j_{\ol{T}*} \ol{\varphi}^*(\{{\cal F}_n\}_{n=1}^{\infty}) \in (X/\os{\circ}{S})_{\rm conv},$$ 
$$L_{\os{\circ}{X}/\os{\circ}{S}}^{\rm UE}({\cal F}) 
:= j_{\os{\circ}{\ol{T}}*} \os{\circ}{\ol{\varphi}}{}^*(\{{\cal F}_n\}_{n=1}^{\infty}) 
\in (\os{\circ}{X}/\os{\circ}{S})_{\rm conv}$$ 
and 
$$L^{\rm UE}_{\os{\circ}{X}{}^{(k)}/\os{\circ}{S}}({\cal G})
:= j_{\os{\circ}{T}{}^{(k)}*} \os{\circ}{\varphi}{}^{(k)*}(\{{\cal G}_n\}_{n=1}^{\infty}) 
\in 
(\os{\circ}{X}{}^{(k)}/\os{\circ}{S})_{\rm conv}.$$


\par 
Let $E$ be a flat isocrystal of 
${\cal K}_{\os{\circ}{X}/\os{\circ}{S}}$-modules. 
Let $\ol{\cal E}:=\{\ol{\cal E}_n\}_{n=1}^{\infty}$ be 
a flat coherent ${\cal K}_{\os{\to}{\os{\circ}{\ol{T}}}}$-module 
corresponding to $E$.
Since ${\cal K}_{\os{\to}{\os{\circ}{\ol{T}}}}={\cal K}_{\os{\to}{{\ol{T}}}}$, 
$\ol{\cal E}:=\{\ol{\cal E}_n\}_{n=1}^{\infty}$ is also 
a flat coherent ${\cal K}_{\os{\to}{\os{\circ}{\ol{T}}}}$-module 
corresponding to $\eps_{X/\os{\circ}{S}}^*(E)$.
Let ${\cal E}:=\{{\cal E}_n\}_{n=1}^{\infty}$ be 
a flat coherent ${\cal K}_{\os{\to}{T}}$-module 
corresponding to $\eps_{X/S}^*(E)$.  
Then ${\cal E}_n=\ol{\cal E}_n{\otimes}_{{\cal O}_{\ol{S}}}{\cal O}_S$. 
Let ${\cal E}^{(k)}:=\{{\cal E}^{(k)}_n\}_{n=1}^{\infty}$ be a flat coherent 
${\cal K}_{\os{\to}{\os{\circ}{T}{}^{(k)}}}$-module corresponding to $a^{(k)*}_{\rm crys}(E)$.  
By (\ref{lemm:lcl}) we have a complex
$L^{\rm UE}_{X/S}
({\cal E}\otimes_{{\cal O}_{\cal P}}\Om^{\bul}_{{\cal P}^{\rm ex}/S})$ of 
${\cal K}_{X/S}$-modules.
By the log convergent Poincar\'{e} lemma ((\ref{theo:pl})), 
we have a natural quasi-isomorphism
\begin{equation*}
\eps_{X/S}^*(E) \os{\sim}{\lo} L^{\rm UE}_{X/S}
({\cal E}\otimes_{{\cal O}_{{\cal P}^{\rm ex}}}\Om^{\bul}_{{\cal P}^{\rm ex}/S}).
\tag{6.2.2}\label{eqn:oxdz}
\end{equation*}
Similarly we have the following quasi-isomorphism:
\begin{equation*}
a^{(k)*}_{\rm conv}(E)
\os{\sim}{\lo}
L^{\rm UE}_{\os{\circ}{X}{}^{(k)}/\os{\circ}{S}}
({\cal E}^{(k)}\otimes_{{\cal O}_{{\cal P}^{{\rm ex},(k)}}}
\Om^{\bul}_{(\os{\circ}{\cal P}{}^{\rm ex})^{(k)}/\os{\circ}{S}}). 
\tag{6.2.3}\label{eqn:odkzs}
\end{equation*}

\par 
Let us recall the morphism 
$\eps_{X/\os{\circ}{S}}\col ((X/\os{\circ}{S})_{\rm conv},{\cal K}_{X/\os{\circ}{S}})
\lo 
((\os{\circ}{X}/\os{\circ}{S})_{\rm conv},{\cal K}_{\os{\circ}{X}/\os{\circ}{S}})$ 
forgetting the log structure of $X$ over $\os{\circ}{S}$ ((\ref{exem:sss})). 
By (\ref{lemm:cpax})
the differential 
\begin{align*} 
\nabla^i \col L^{\rm UE}_{X/\os{\circ}{S}}
(\ol{\cal E}\otimes_{{\cal O}_{\ol{\cal P}{}^{\rm ex}}}\Om^i_{\ol{\cal P}{}^{\rm ex}/\os{\circ}{S}})
\lo 
L^{\rm UE}_{X/\os{\circ}{S}}
(\ol{\cal E}\otimes_{{\cal O}_{\ol{\cal P}{}^{\rm ex}}}\Om^{i+1}_{\ol{\cal P}{}^{\rm ex}/\os{\circ}{S}})
\tag{6.2.4}\label{eqn:epdomep}
\end{align*}
induces the following differential$:$ 
\begin{align*} 
\nabla^i \col 
\eps^{\rm conv}_{X/\os{\circ}{S}*}L^{\rm UE}_{X/\os{\circ}{S}}
({\cal E}\otimes_{{\cal O}_{{\cal P}{}^{\rm ex}}}\Om^i_{{\cal P}{}^{\rm ex}/\os{\circ}{S}})
\lo 
\eps^{\rm conv}_{X/\os{\circ}{S}*}L^{\rm UE}_{X/\os{\circ}{S}}
({\cal E}\otimes_{{\cal O}_{{\cal P}{}^{\rm ex}}}\Om^{i+1}_{{\cal P}{}^{\rm ex}/\os{\circ}{S}}). 
\tag{6.2.5}\label{eqn:eldomep}
\end{align*}
Hence we have the following complex by (\ref{coro:ddn}): 
\begin{align*}  
L^{\rm conv}_{\os{\circ}{X}/\os{\circ}{S}}
({\cal E}\otimes_{{\cal O}_{{\cal P}^{\rm ex}}}
\Om^{\bul}_{{\cal P}^{\rm ex}/\os{\circ}{S}}). 
\tag{6.2.6}\label{ali:linp}
\end{align*}  
By (\ref{lemm:objt}) the linearization in (\ref{ali:linp}) can be calculated 
not only with respect to the immersion 
$\os{\circ}{X}\os{\sus}{\lo} (\ol{\cal P}{}^{\rm ex})^{\circ}$ over $\os{\circ}{S}$ 
but also 
with respect to the immersion 
$\os{\circ}{X}\os{\sus}{\lo} ({\cal P}^{\rm ex})^{\circ}$ over $\os{\circ}{S}$. 
In the following we calculate the linearization in (\ref{ali:linp}) 
with respect to the latter immersion.  
\par 
The family  
$\{P_k\Om^i_{{\cal P}^{\rm ex}/\os{\circ}{S}}\}_{i\in {\mab N}}$
gives a subcomplex 
$P_k\Om^{\bul}_{{\cal P}^{\rm ex}/\os{\circ}{S}}$ 
of $\Om^{\bul}_{{\cal P}^{\rm ex}/\os{\circ}{S}}$. 
Since the following diagram 
\begin{equation*} 
\begin{CD}
X@>{\subset}>> \ol{\cal P}{}^{\rm ex}\\
@V{\eps_{X/\os{\circ}{S}}}VV @VVV \\
\os{\circ}{X}@>{\subset}>> (\ol{\cal P}{}^{\rm ex})^{\circ}
\end{CD}
\end{equation*} 
is commutative, the following diagram is commutative by (\ref{prop:nia}): 
\begin{equation*} 
\begin{CD}
\eps^{\rm conv}_{X/\os{\circ}{S}*}L^{\rm UE}_{X/\os{\circ}{S}}
(\ol{\cal E}\otimes_{{\cal O}_{\ol{\cal P}{}^{\rm ex}}}\Om^i_{\ol{\cal P}{}^{\rm ex}/\os{\circ}{S}})
@>{\nabla^i}>> 
\eps^{\rm conv}_{X/\os{\circ}{S}*}L^{\rm UE}_{X/\os{\circ}{S}}
(\ol{\cal E}
\otimes_{\ol{\cal O}_{\ol{\cal P}{}^{\rm ex}}}\Om^{i+1}_{\ol{\cal P}{}^{\rm ex}/\os{\circ}{S}}) \\
@AAA @AAA \\
L^{\rm UE}_{\os{\circ}{X}/\os{\circ}{S}}
(\ol{\cal E}\otimes_{{\cal O}_{\ol{\cal P}{}^{\rm ex}}}
\Om^i_{(\ol{\cal P}{}^{\rm ex})^{\circ}/\os{\circ}{S}})
@>{\nabla^i}>> 
L^{\rm UE}_{\os{\circ}{X}/\os{\circ}{S}}
(\ol{\cal E}\otimes_{{\cal O}_{\ol{\cal P}{}^{\rm ex}}}
\Om^{i+1}_{(\ol{\cal P}{}^{\rm ex})^{\circ}/\os{\circ}{S}}). 
\end{CD}
\end{equation*}  
Hence we have the following commutative diagram by (\ref{coro:ddn}):
\begin{equation*} 
\begin{CD}
\eps^{\rm conv}_{X/\os{\circ}{S}*}L^{\rm UE}_{X/\os{\circ}{S}}
({\cal E}\otimes_{{\cal O}_{{\cal P}{}^{\rm ex}}}\Om^i_{{\cal P}{}^{\rm ex}/\os{\circ}{S}})
@>{\nabla^i}>> 
\eps^{\rm conv}_{X/\os{\circ}{S}*}L^{\rm UE}_{X/\os{\circ}{S}}
({\cal E}\otimes_{{\cal O}_{{\cal P}{}^{\rm ex}}}\Om^{i+1}_{{\cal P}{}^{\rm ex}/\os{\circ}{S}}) \\
@| @| \\
L^{\rm UE}_{\os{\circ}{X}/\os{\circ}{S}}
({\cal E}
\otimes_{{\cal O}_{{\cal P}{}^{\rm ex}}}\Om^i_{{\cal P}{}^{\rm ex}/\os{\circ}{S}})
@>{\nabla^i}>> 
L^{\rm UE}_{\os{\circ}{X}/\os{\circ}{S}}
({\cal E}\otimes_{{\cal O}_{{\cal P}{}^{\rm ex}}}
\Om^{i+1}_{{\cal P}{}^{\rm ex}/\os{\circ}{S}}) \\
@AAA @AAA \\
L^{\rm UE}_{\os{\circ}{X}/\os{\circ}{S}}
({\cal E}\otimes_{{\cal O}_{{\cal P}{}^{\rm ex}}}\Om^i_{({\cal P}{}^{\rm ex})^{\circ}/\os{\circ}{S}})
@>{\nabla^i}>> 
L^{\rm UE}_{\os{\circ}{X}/\os{\circ}{S}}
({\cal E}\otimes_{{\cal O}_{{\cal P}{}^{\rm ex}}}\Om^{i+1}_{({\cal P}{}^{\rm ex})^{\circ}/\os{\circ}{S}}). 
\end{CD}
\end{equation*}  
Since the connection $\nabla^0 \col 
L^{\rm UE}_{\os{\circ}{X}/\os{\circ}{S}}(\ol{\cal E})
\lo 
L^{\rm UE}_{\os{\circ}{X}/\os{\circ}{S}}(\ol{\cal E}\otimes_{{\cal O}_{\ol{\cal P}{}^{\rm ex}}}
\Om^1_{(\ol{\cal P}{}^{\rm ex})^{\circ}/\os{\circ}{S}})$ 
has no log poles, so is 
$\nabla^0 \col L^{\rm UE}_{\os{\circ}{X}/\os{\circ}{S}}({\cal E})
\lo 
L^{\rm UE}_{\os{\circ}{X}/\os{\circ}{S}}({\cal E}\otimes_{{\cal O}_{{\cal P}{}^{\rm ex}}}
\Om^1_{({\cal P}{}^{\rm ex})^{\circ}/\os{\circ}{S}})$.  
Consequently we indeed have a complex 
$$P_kL^{\rm UE}_{\os{\circ}{X}/\os{\circ}{S}}({\cal E}
\otimes_{{\cal O}_{\ol{\cal P}^{\rm ex}}}\Om^{\bul}_{{\cal P}^{\rm ex}/\os{\circ}{S}})
:=
L^{\rm UE}_{\os{\circ}{X}/\os{\circ}{S}}(P_k({\cal E}
\otimes_{{\cal O}_{{\cal P}^{\rm ex}}}
\Om^{\bul}_{{\cal P}^{\rm ex}/\os{\circ}{S}}))\quad (k\in {\mab Z}).$$   
Here note that we cannot obtain the complex 
``$L^{\rm UE}_{X/S}
(P_k({\cal E}\otimes_{{\cal O}_{{\cal P}^{\rm ex}}}
\Om^{\bul}_{{\cal P}^{\rm ex}/\os{\circ}{S}}))$''.

\par 
As in \cite[(6.7)]{nhw}, the following theorem (2) is 
one of key theorems in this paper.  

\begin{theo}\label{theo:injf}
$(1)$ 
The natural morphism 
\begin{equation*} 
{\cal E}_n{\otimes}_{{\cal O}_{{\cal P}^{\rm ex}}}P_k
\Om^{\bul}_{{\cal P}^{\rm ex}/\os{\circ}{S}} 
\lo 
{\cal E}_n{\otimes}_{{\cal O}_{{\cal P}^{\rm ex}}}
\Om^{\bul}_{{\cal P}^{\rm ex}/\os{\circ}{S}}
\tag{6.3.1}\label{eqn:yxpd}
\end{equation*}
is injective. 
\par 
$(2)$ The natural morphism 
\begin{equation*} 
P_k
L^{\rm UE}_{\os{\circ}{X}/\os{\circ}{S}}
({\cal E}\otimes_{{\cal O}_{{\cal P}^{\rm ex}}}
\Om^{\bul}_{{\cal P}^{\rm ex}/\os{\circ}{S}}) 
\lo 
L^{\rm UE}_{\os{\circ}{X}/\os{\circ}{S}}
({\cal E}\otimes_{{\cal O}_{{\cal P}^{\rm ex}}}
\Om^{\bul}_{{\cal P}^{\rm ex}/\os{\circ}{S}}) 
\tag{6.3.2}\label{eqn:yxpdz}
\end{equation*}
is injective. 
\end{theo}
\begin{proof} 
(1): (1) immediately follows from (\ref{lemm:rex}). 
\par 
(2):  
This is a local problem.  
Let $(\os{\circ}{U}{}',\os{\circ}{T}{}',\os{\circ}{\iota},\os{\circ}{u})$ be an object of 
${\rm Conv}(\os{\circ}{X}/\os{\circ}{S})$. 
We may assume that $\os{\circ}{T}{}'$ is affine. 
Since $\os{\circ}{\ol{\cal P}}{}^{\rm ex}$ 
is formally smooth over $\os{\circ}{S}$, 
we have a morphism $\ol{e}\col \os{\circ}{T}{}'\lo 
\os{\circ}{\ol{\cal P}}{}^{\rm ex}$ over $\os{\circ}{S}$ 
such that the composite morphism 
$\os{\circ}{U}{}'\os{\sus}{\lo} \os{\circ}{T}{}'\lo \os{\circ}{\ol{\cal P}}{}^{\rm ex}$ 
is equal to the composite morphism 
$\os{\circ}{U}{}'\lo \os{\circ}{X}\os{\sus}{\lo} \os{\circ}{\ol{\cal P}}{}^{\rm ex}$ . 
Let $r\col \os{\circ}{T}{}'\times_{\os{\circ}{S}}\os{\circ}
{\ol{\cal P}}{}^{\rm ex}\lo \os{\circ}{\ol{\cal P}}{}^{\rm ex}$ 
be the second projection.  
Let $T'$ be a log formal scheme  
whose underlying formal scheme is $\os{\circ}{T}{}'$ and 
whose log structure is the inverse image of that of $S$. 
Set $\ol{\cal P}{}^{\rm ex}_{T'}:=\ol{\cal P}{}^{\rm ex}\times_{\os{\circ}{S}}\os{\circ}{T}{}'$ 
(by abuse of notation), 
${\cal P}^{\rm ex}_{T'}:={\cal P}^{\rm ex}\times_{S}T'$ and 
${\cal P}^{\rm ex}_{\os{\circ}{T}{}'}:={\cal P}^{\rm ex}\times_{\os{\circ}{S}}\os{\circ}{T}{}'$. 
Then ${\cal P}^{\rm ex}_{T'}$ is 
a formal SNCL scheme over $T'$ and 
$\os{\circ}{\cal P}{}^{\rm ex}_{T'}:=\os{\circ}{\cal P}{}^{\rm ex}\times_{\os{\circ}{S}}\os{\circ}{T}{}'$. 
We have the following diagram:  
\begin{equation*}
\begin{CD}
\os{\circ}{T}{}'\\
@A{\{p'_n\}_{n=1}^{\infty}}AA \\
\{{\mathfrak T}_{\os{\circ}{U}{}',n}(\os{\circ}{\cal P}{}^{\rm ex}_{T'})\}_{n=1}^{\infty} 
@>{\{p_n\}_{n=1}^{\infty}}>>
\{{\mathfrak T}_{\os{\circ}{X},n}(\os{\circ}{{\cal P}}{}^{\rm ex})\}_{n=1}^{\infty}
=\{\os{\circ}{T}_n\}_{n=1}^{\infty} \\ 
 @V{\{q'_n\}_{n=1}^{\infty}}VV 
 @VV{\{q_n\}_{n=1}^{\infty}}V\\
\os{\circ}{\cal P}{}^{\rm ex}_{T'}
@>{r}>> \os{\circ}{{\cal P}}{}^{\rm ex}.\\ 
\end{CD}
\tag{6.3.3}\label{cd:utsq} 
\end{equation*} 
By the remark after (\ref{ali:linp}) and (\ref{eqn:lueyset}), 
\begin{align*} 
(P_kL^{\rm UE}_{\os{\circ}{X}/\os{\circ}{S}}
({\cal E}\otimes_{{\cal O}_{{\cal P}^{\rm ex}}}
\Om^{\bul}_{{\cal P}^{\rm ex}/\os{\circ}{S}}))_{\os{\circ}{T}{}'}
&=(P_kL^{\rm UE}_{\os{\circ}{X}/\os{\circ}{S}}
({\cal E}\otimes_{{\cal O}_{{\cal P}^{\rm ex}}}
\Om^{\bul}_{{\cal P}^{\rm ex}/\os{\circ}{S}}))_{\os{\circ}{T}{}'}\\
&= \vpl_np'_{n*}p_n^*
({\cal E}_n\otimes_{{\cal O}_{T_n}}
q_n^*(P_k\Om^{\bul}_{{\cal P}{}^{\rm ex}/\os{\circ}{S}}))\\
&=\vpl_np'_{n*}\{p_n^*
({\cal E}_n)\otimes_{p_n^*({\cal O}_{T_n})}p_n^*
q_n^*(P_k\Om^{\bul}_{{\cal P}{}^{\rm ex}/\os{\circ}{S}})\}\\
&= \vpl_np'_{n*}\{p_n^*
({\cal E}_n)\otimes_{q_n'\!{}^*r^*({\cal O}_{{\cal P}{}^{\rm ex}})}q_n'\!{}^*
r^*(P_k\Om^{\bul}_{{\cal P}{}^{\rm ex}/\os{\circ}{S}})\}.
\end{align*}  
Because $\vpl_n$ and $p'_{n*}$ are left exact 
and because $q_n'\!{}^*$ is left exact by (\ref{lemm:rex}), 
we have only to prove that the morphism 
$r^*(P_k\Om^{\bul}_{{\cal P}{}^{\rm ex}/\os{\circ}{S}})  \lo 
r^*(\Om^{\bul}_{{\cal P}{}^{\rm ex}/\os{\circ}{S}})$ is injective. 
This injectivity is clear because 
\begin{align*} 
r^*(P_k\Om^{\bul}_{{\cal P}{}^{\rm ex}/\os{\circ}{S}})
=P_k\Om^i_{{\cal P}{}^{\rm ex}_{\os{\circ}{T}{}'}/{\os{\circ}{T}{}'}}.
\end{align*}  
\end{proof}

By (\ref{theo:injf}) (2), the family 
$\{P_kL^{\rm UE}_{\os{\circ}{X}/\os{\circ}{S}}
({\cal E}\otimes_{{\cal O}_{{\cal P}^{\rm ex}}}
\Om^{\bul}_{{\cal P}^{\rm ex}/\os{\circ}{S}})\}_{k \in {\mab Z}}$ 
of ${\cal K}_{\os{\circ}{X}/\os{\circ}{S}}$-modules 
defines a filtration on the complex
$L^{\rm UE}_{\os{\circ}{X}/\os{\circ}{S}}
({\cal E}\otimes_{{\cal O}_{{\cal P}^{\rm ex}}}
\Om^{\bul}_{{\cal P}^{\rm ex}/\os{\circ}{S}})$. 
Hence we obtain an object 
$$(L^{\rm UE}_{\os{\circ}{X}/\os{\circ}{S}}
({\cal E}\otimes_{{\cal O}_{{\cal P}^{\rm ex}}}\Om^{\bul}_{{\cal P}^{\rm ex}/\os{\circ}{S}}), 
\{P_k
L^{\rm UE}_{\os{\circ}{X}/\os{\circ}{S}}
({\cal E}\otimes_{{\cal O}_{{\cal P}^{\rm ex}}}
\Om^{\bul}_{{\cal P}^{\rm ex}/\os{\circ}{S}})\}_{k\in {\mab Z}})$$  
in ${\rm C}^+{\rm F}({\cal K}_{X/S})$. 
For simplicity of notation, we denote it by 
$$(L^{\rm UE}_{\os{\circ}{X}/\os{\circ}{S}}({\cal E}\otimes_{{\cal O}_{{\cal P}^{\rm ex}}}
\Om^{\bul}_{{\cal P}^{\rm ex}/\os{\circ}{S}}),P).$$


\begin{prop}\label{prop:grla}
Let $k$ be a positive integer. 
Assume that $\os{\circ}{\cal P}{}^{\rm ex}$ is affine over $\os{\circ}{S}$. 
Then there exist the following quasi-isomorphisms 
\begin{align*} 
{\rm gr}_k^{P}
& L^{\rm UE}_{\os{\circ}{X}/\os{\circ}{S}}(
{\cal E}\otimes_{{\cal O}_{{\cal P}^{\rm ex}}}
\Om^{\bul}_{{\cal P}^{\rm ex}/\os{\circ}{S}}) 
\os{\sim}{\longleftarrow}  
a^{(k-1)}_{{\rm conv}*}
(a^{(k-1)*}_{\rm conv}(E)
\otimes_{\mab Z}\vp^{(k-1)}_{{\rm conv}}(\os{\circ}{X}/\os{\circ}{S}))[-k]
\tag{6.4.1}\label{eqn:grpdl}\\ 
\end{align*}
in $D^+({\cal K}_{X/S})$. 
\end{prop} 
\begin{proof} 
Let 
$b^{(k)}\col \os{\circ}{\cal P}{}^{(k)}\lo \os{\circ}{\cal P}$ 
be the natural morphism. 
Then, as in the proof of \cite[(1.3.14)]{nb}, 
we have the following isomorphism 
\begin{equation*}  
{\rm Res} \col {\rm gr}_k^P{\Om}^{\bul}_{{\cal P}^{\rm ex}/\os{\circ}{S}}
\os{\sim}{\lo} 
b^{(k-1)}_*
(\Om^{\bul}_{\os{\circ}{\cal P}{}^{{\rm ex},(k-1)}/\os{\circ}{S}}
\otimes_{\mab Z}\vp^{(k-1)}_{\rm zar}(\os{\circ}{\cal P}{}^{\rm ex}/\os{\circ}{S}))[-k] 
\quad (k\geq 1)
\tag{6.4.2}\label{eqn:grps}
\end{equation*} 
by using the immersion $X\os{\sus}{\lo} \ol{\cal P}$ over $\ol{S}$. 
Hence, by (\ref{theo:injf}) (2) and (\ref{prop:afex}) and by the SNCL version of 
the proof of \cite[(6.9)]{nhw}, 
we have the isomorphism (\ref{eqn:grpdl}).  
\end{proof}

For the case $k=0$, we obtain the following: 

\begin{prop}\label{prop:0case}
There exists the following isomorphism$:$ 
\begin{align*}
&P_0L^{\rm UE}_{\os{\circ}{X}/\os{\circ}{S}}
({\cal E}\otimes_{{\cal O}_{{\cal P}^{\rm ex}}}\Om^{\bul}_{{\cal P}^{\rm ex}/\os{\circ}{S}}) 
\os{\sim}{\lo} \tag{6.5.1}\label{ali:ixea} \\
&(a^{(0)}_{{\rm conv}*}(a^{(0)*}_{{\rm conv}*}(E)\otimes_{\mab Z}\varpi^{(0)}_{\rm conv}
(\os{\circ}{X}/\os{\circ}{S}))
\lo a^{(1)}_{{\rm conv}*}(a^{(1)*}_{{\rm conv}*}(E)\otimes_{\mab Z}\varpi^{(1)}_{\rm conv}
(\os{\circ}{X}/\os{\circ}{S}))\lo \cdots )
\\
&={\rm Ker}(a^{(0)}_{{\rm conv}*}
(a^{(0)*}_{{\rm conv}*}(E)\otimes_{\mab Z}\varpi^{(0)}_{\rm conv}
(\os{\circ}{X}/\os{\circ}{S}))\lo 
a^{(1)}_{{\rm conv}*}(a^{(1)*}_{{\rm conv}*}(E)\otimes_{\mab Z}\varpi^{(1)}_{\rm conv}
(\os{\circ}{X}/\os{\circ}{S}))). 
\end{align*} 
\end{prop}
\begin{proof} 
The following proof is a slightly simpler log convergent version of \cite[(1.3.21)]{nb}.
\par 
Let the notations be as in the proof of (\ref{theo:injf}). 
Set 
$\os{\circ}{\cal P}{}^{{\rm ex},(k)}_{T'}
:=\os{\circ}{\cal P}{}^{{\rm ex},(k)}\times_{\os{\circ}{S}}\os{\circ}{T}{}'$ 
$(k\in {\mab N})$. 
Let  
$b^{(k)}_{\os{\circ}{T}{}'}\col \os{\circ}{\cal P}{}^{{\rm ex},(k)}_{T'}
\lo \os{\circ}{\cal P}{}^{{\rm ex}}_{T'}$ 
and 
$c^{(k)}_{n\os{\circ}{T}{}'}\col 
{\mathfrak T}_{\os{\circ}{U}{}'{}^{(k)},n}(\os{\circ}{\cal P}{}^{{\rm ex},(k)}_{T'})
\lo {\mathfrak T}_{\os{\circ}{U}{}',n}(\os{\circ}{T}{}'\times_{\os{\circ}{S}}\os{\circ}{\cal P}{}^{\rm ex})$
be the natural morphisms. 
Consider the following commutative diagram: 
\begin{equation*}
\begin{CD}
\{{\mathfrak T}_{\os{\circ}{U}{}'{}^{(k)},n}(\os{\circ}{\cal P}{}^{{\rm ex},(k)}_{T'})\}_{n=1}^{\infty} 
@>{\{q'{}^{(k)}_n\}_{n=1}^{\infty}}>> \os{\circ}{\cal P}{}^{{\rm ex},(k)}_{T'} \\
@V{\{c^{(k)}_{n\os{\circ}{T}{}'}\}_{n=1}^{\infty}}VV 
@VV{b^{(k)}_{\os{\circ}{T}{}'}}V\\
\{{\mathfrak T}_{\os{\circ}{U}{}',n}(\os{\circ}{{\cal P}}{}^{\rm ex}_{T'})\}_{n=1}^{\infty} 
@>{\{q'_n\}_{n=1}^{\infty}}>> \os{\circ}{{\cal P}}{}^{\rm ex}_{T'} \\ 
@V{\{p_n\}_{n=1}^{\infty}}VV \\
\{T_n\}_{n=1}^{\infty}. 
\end{CD}
\tag{6.5.2}\label{cd:utqsq} 
\end{equation*} 
and the diagram (\ref{cd:utsq}). 
By (\ref{lemm:dpinc}) and (\ref{lemm:bcue}) we see that the square in (\ref{cd:utqsq}) 
is a cartesian diagram. 
Then $c^{(0)}_{n\os{\circ}{T}{}'}$ induces the following morphism of complexes: 
\begin{align*} 
0 &\lo \vpl_nP_0(p_n^*({\cal E}_n)
\otimes_{{\cal O}_{{\cal P}^{\rm ex}_{T'}}}{\Om}^{\bul}_{{\cal P}^{\rm ex}_{T'}/T'}) 
\tag{6.5.3}\label{ali:epn}\\
& \lo \vpl_n\{c^{(0)}_{n\os{\circ}{T}{}'*}(c^{(0)*}_{n\os{\circ}{T}{}'}(p_n^*({\cal E}_n))
\otimes_{{\cal O}_{{\os{\circ}{\cal P}{}^{{\rm ex},(0)}_{T'}}}}
q'{}^{(0)*}_n(\Om^{\bul}_{\os{\circ}{\cal P}{}^{{\rm ex},(0)}_{T'}/\os{\circ}{T}{}'})\otimes_{\mab Z}
\vp^{(0)}_{\rm zar}(\os{\circ}{\cal P}{}^{\rm ex}_{T'}/\os{\circ}{T}{}')) \\
& \quad \quad \quad \quad\lo 
\vpl_n{c}{}^{(1)}_{n\os{\circ}{T}{}'*}(c^{(1)*}_{n\os{\circ}{T'}}(p_n^*({\cal E}_n))
\otimes_{{\cal O}_{\os{\circ}{\cal P}{}^{{\rm ex},(1)}_{T'}}}
q'{}^{(1)*}_n(\Om^{\bul}_{\os{\circ}{\cal P}{}^{{\rm ex},(1)}_{T'}/\os{\circ}{T}{}'}
\otimes_{\mab Z}\vp^{(1)}_{\rm zar}(\os{\circ}{\cal P}{}^{\rm ex}_{T'}/\os{\circ}{T}{}'))) \\
& \quad \quad \quad \quad\lo \cdots \lo \cdots \}. 
\end{align*}  
\par 
We claim that the whole complex (\ref{ali:epn}) is exact.  
Indeed, by \cite[(4.2.2) (a), (c)]{di} 
we have the following exact sequence 
\begin{equation*}
0 \lo 
\Om^{\bul}_{\os{\circ}{\ol{\cal P}}{}^{\rm ex}_{T'}/\os{\circ}{T}{}'}/
{\Om}^{\bul}_{{\ol{\cal P}{}^{\rm ex}_{T'}}/\os{\circ}{T}{}'}
(-\os{\circ}{\cal P}{}^{\rm ex}_{T'})
\lo 
b^{(0)}_{\os{\circ}{T}{}'*}(\Om^{\bul}_{\os{\circ}{\cal P}{}^{(0),{\rm ex}}_{T'}/\os{\circ}{T}{}'}
\otimes_{\mab Z}\vp^{(0)}_{\rm zar}(\os{\circ}{\cal P}{}^{\rm ex}_{T'}/\os{\circ}{T}{}')) 
\lo \cdots.  
\tag{6.5.4}\label{eqn:tztt}
\end{equation*} 
By (\ref{lemm:rex}) we have the following exact sequence 
\begin{equation*}
0 \lo q'^*_n(\Om^{\bul}_{\os{\circ}{\ol{\cal P}}{}^{\rm ex}_{T'}/{T'}}/
{\Om}^{\bul}_{\ol{\cal P}{}^{\rm ex}_{T'}/{T'}}(-\os{\circ}{\cal P}{}^{\rm ex}_{T'}))  
\lo 
q'{}^{*}_n(b^{(0)}_{T''*}(\Om^{\bul}_{\os{\circ}{\cal P}{}^{(0),{\rm ex}}_{T'}/T'})
\otimes_{\mab Z}\vp^{(0)}_{\rm zar}(\os{\circ}{\cal P}{}^{\rm ex}_{T'}/\os{\circ}{T}{}'))
\lo \cdots.  
\tag{6.5.5}\label{eqn:tztzt}
\end{equation*} 
Here note that  
$\Om^i_{\os{\circ}{\ol{\cal P}}{}^{\rm ex}_{T'}/T'}/
\Om^i_{\ol{\cal P}{}^{\rm ex}_{T'}/T'}(-\os{\circ}{\cal P}{}^{\rm ex}_{T'})$ 
$(i\in {\mab N})$ is an ${\cal O}_{{\cal P}^{\rm ex}_{T'}}$-module.   
By \cite[(1.3.21) (4), (1.3.12.1)]{nb}  we have 
the following isomorphism: 
\begin{align*}
\Om^{\bul}_{\os{\circ}{\ol{\cal P}}{}^{\rm ex}_{T'}/T'}/
\Om^{\bul}_{\ol{\cal P}{}^{\rm ex}_{T'}/T'}(-\os{\circ}{\cal P}{}^{\rm ex}_{T'})
\simeq P_0\Om^{\bul}_{\ol{\cal P}{}^{\rm ex}_{T'}/T'}/
\Om^{\bul}_{\ol{\cal P}{}^{\rm ex}_{T'}/T'}(-\os{\circ}{\cal P}{}^{\rm ex}_{T'})
=P_0\Om^{\bul}_{{\cal P}^{\rm ex}_{T'}/\os{\circ}{T}{}'}. 
\tag{6.5.6}\label{ali:opt} 
\end{align*} 
Because $E$ is a flat ${\cal K}_{\os{\circ}{X}/\os{\circ}{S}}$-module, 
$\ol{\cal E}_n$ is a flat ${\cal K}_{\ol{T}_n}$-module. 
Hence ${\cal E}_n$ is a flat ${\cal K}_{T_n}$-module and 
we have the following exact sequence of ${\cal K}_{T_n}$-modules: 
\begin{align*}
0 \lo &\{p_n({\cal E}_n)\otimes_{q'{}^{*}_n({\cal O}_{{\cal P}^{\rm ex}_{T'}})}
q'{}^{*}_n(P_0\Om^{\bul}_{{\cal P}^{\rm ex}_{T'}/{T'}})\}_{n=1}^{\infty}  
\tag{6.5.7}\label{ali:o0pt} \\
& \lo 
\{p_n^*({\cal E}_n)\otimes_{q'{}^{*}_n({\cal O}_{{\cal P}^{\rm ex}_{T'}})}
q'{}^{*}_n(b^{(0)}_{T'*}(\Om^{\bul}_{\os{\circ}{\cal P}{}^{(0)}_{T'}/{T'}
}\otimes_{\mab Z}\vp^{(0)}_{\rm zar}(\os{\circ}{\cal P}_{T'}/\os{\circ}{T}{}')))\}_{n=1}^{\infty}   
\lo \cdots.  
\end{align*} 
Because $b^{(k)}_{\os{\circ}{T}{}'}$ is an affine morphism, 
the affine base change theorem tells us that 
the target of the morphism (\ref{ali:o0pt}) is equal to 
\begin{align*}
\{p_n^*({\cal E}_n)\otimes_{q'{}^{*}_n({\cal O}_{{\cal P}^{\rm ex}_{T'}})}
c^{(0)}_{n\os{\circ}{T}{}'*}q'{}^{(0)*}_n
(\Om^{\bul}_{\os{\circ}{\cal P}{}^{(0)}_{T'}/{T'}
}\otimes_{\mab Z}\vp^{(0)}_{\rm zar}(\os{\circ}{\cal P}_{T'}/\os{\circ}{T}{}')))\}_{n=1}^{\infty}   
\lo \cdots.  
\end{align*} 
Because the exactness in question is a local problem, we may neglect the orientation sheaves 
$\vp^{(k)}_{\rm zar}(\os{\circ}{\cal P}_{T'}/\os{\circ}{T}{}')$'s. 
By considering this 
neglectance into account, the target is isomorphic to 
\begin{align*}
\{p_n^*({\cal E}_n)\otimes_{q'{}^{*}_n({\cal O}_{{\cal P}^{\rm ex}_{T'}})}
c^{(0)}_{n\os{\circ}{T}{}'*}q'{}^{(0)*}_n
(\Om^{\bul}_{\os{\circ}{\cal P}{}^{(0)}_{T'}/{T'}})\}_{n=1}^{\infty}   
\lo \cdots.  
\end{align*} 
For a coherent crystal $\{H_n\}_{n=1}^{\infty}$ of 
$\{{\cal K}_{{\mathfrak T}_{\os{\circ}{U}{}',n}(\os{\circ}{\cal P}{}^{\rm ex}_{T'}
)}\}_{n=1}^{\infty}$-modules, 
$R^1\vpl_nH_n=0$ by the log version of \cite[(3.8)]{oc} 
which follows from the proof of [loc.~cit.]. 
Hence the sequence (\ref{ali:epn}) is exact. 
\par 
Because the right hand side of (\ref{ali:epn}) with the abbreviation of  the orientation sheaves 
is equal to 
$$\{a^{(0)}_{{\rm conv}*}(L_{\os{\circ}{X}{}^{(0)}/\os{\circ}{S}}(c^{(0)*}({\cal E})
\otimes_{{\cal O}_{\os{\circ}{\cal P}{}^{(0),{\rm ex}}}}
\Om^{\bul}_{\os{\circ}{\cal P}{}^{(0),{\rm ex}}/\os{\circ}{S}}))
\}_T\lo \cdots \lo,$$
the right hand side is equal to 
$$a^{(0)}_{{\rm conv}*}(a^{(0)*}_{{\rm conv}}(E))
\lo a^{(1)}_{{\rm conv}*}
(a^{(1)*}_{{\rm conv}}(E))\lo \cdots$$ 
by the convergent Poincar\'{e} lemma (\cite[(5.6)]{oc}, \cite[(3.7)]{nhw}). 
\par 
The equality in (\ref{ali:ixea}) follows from the argument above. 
\end{proof}

\par 
For simplicity of notation,  set 
$$({\cal E}\otimes_{{\cal O}_{{\cal P}^{\rm ex}}}
\Om^{\bul}_{{\cal P}^{\rm ex}/\os{\circ}{S}},P):=
(\{{\cal E}_{n}\otimes_{{\cal O}_{{\cal P}^{\rm ex}}}
\Om^{\bul}_{{\cal P}^{\rm ex}/\os{\circ}{S}}\}_{n=1}^{\infty}, 
\{\{{\cal E}_{n}\otimes_{{\cal O}_{{\cal P}^{\rm ex}}}P_k
\Om^{\bul}_{{\cal P}^{\rm ex}/\os{\circ}{S}}\}_{k\in {\mab Z}}\}_{n=1}^{\infty}).$$
\begin{prop}\label{prop:ulcrz}
Let $\bet_n \col X_n \lo X$ be a natural morphism 
and 
identify $(T_n)_{{\rm zar}}$ with $(X_n)_{\rm zar}$. 
Then
\begin{equation*}
Ru^{\rm conv}_{X/S*}
(L^{\rm UE}_{X/S}
({\cal E}\otimes_{{\cal O}_{{\cal P}^{\rm ex}}}
\Om^{\bul}_{{\cal P}^{\rm ex}/\os{\circ}{S}}),P)
=\vpl_n \bet_{n*}(({\cal E}_{n}\otimes_{{\cal O}_{{\cal P}^{\rm ex}}}
\Om^{\bul}_{{\cal P}^{\rm ex}/\os{\circ}{S}},P))
\tag{6.6.1}\label{eqn:uxzl}
\end{equation*}
in ${\rm D}^+{\rm F}(f^{-1}({\cal K}_S))$.
\end{prop}
\begin{proof}
The proof is the same as that of \cite[(3.9)]{nhw}. 
\end{proof} 

\parno 
The following is an SNCL analogue of \cite[(5.13)]{nh3}, 
whose proof is the same as that of [loc.~cit.]: 

\begin{prop}\label{prop:rescos}
Let ${\cal V}'$, $\pi'$, $\kap'$ and $S'$ 
be similar objects to ${\cal V}$, $\pi$ $\kap$ and $S$
as in {\rm \S\ref{sec:logcd}}, respectively. 
Let $X' \os{\sus}{\lo} \ol{\cal P}{}'$ 
be a similar closed immersion over $\ol{S}{}'$ to $X\os{\sus}{\lo} \ol{\cal P}$ over $\ol{S}$. 
Let $a'{}^{(k)}\col \os{\circ}{\cal P}{}'^{(k)}\lo \os{\circ}{\cal P}{}'$ $(k\in {\mab N})$
be the natural morphism of log schemes over $S'_1$.  
Assume that there exists the following commutative diagram 
\begin{equation*} 
\begin{CD} 
X@>{\sus}>> \ol{\cal P}\\
@V{g}VV @VV{{\mathfrak g}}V \\
X' @>{\sus}>> \ol{\cal P}{}'
\end{CD}
\end{equation*} 
over $\ol{S}\lo \ol{S}{}'$. 
Let $E'$ be a similar isocrystal to $E$ in 
$((\os{\circ}{X}{}'/\os{\circ}{S}{}')_{\rm conv},{\cal K}_{\os{\circ}{X}{}'/\os{\circ}{S}{}'})$. 
Let $\Phi \col \os{\circ}{g}{}^*_{\rm conv}(E')\lo E$ be a morphism of isocrystals of 
${\cal K}_{\os{\circ}{X}{}/\os{\circ}{S}{}}$-modules. 
Let ${\cal M}^{\rm ex}$ and ${\cal M}'{}^{\rm ex}$ be 
the log structures of ${\cal P}^{\rm ex}$ and ${\cal P}'{}^{\rm ex}$, 
respectively. 
Assume that, for each point $x\in \os{\circ}{\cal P}{}^{\rm ex}$ 
and for each member $m$ of the minimal generators 
of ${\cal M}^{\rm ex}_{x}/{\cal O}^*_{{\cal P}^{\rm ex},x}$, 
there exists a unique member $m'$ of the minimal generators of 
${\cal M}'{}^{\rm ex}_{\os{\circ}{g}(x)}
/{\cal O}^*_{{\cal P}'{}^{\rm ex},\os{\circ}{g}(x)}$ 
such that $g^*(m')= m$ and such that the image of 
the other minimal generators of 
${\cal M}'{}^{\rm ex}_{\os{\circ}{g}(x)}
/{\cal O}^*_{{\cal P}'{}^{\rm ex},\os{\circ}{g}(x)}$ by $g^*$ 
are the trivial element of 
${\cal M}^{\rm ex}_{x}/{\cal O}^*_{{\cal P}^{\rm ex},x}$. 
Then the following diagram is commutative$:$
\begin{equation*} 
\begin{CD}  
{\rm Res}: 
{\mathfrak g}_{{\rm conv}*}
{\rm gr}_k^{P}L^{\rm UE}_{\os{\circ}{X}/\os{\circ}{S}}
({\cal E}\otimes_{{\cal O}_{{\cal P}^{\rm ex}}}
\Om^{\bul}_{{\cal P}^{\rm ex}/\os{\circ}{S}})
@>{\sim}>> \\
@AAA  \\
{\rm Res}: 
{\rm gr}_k^PL^{\rm UE}_{\os{\circ}{X}{}'/\os{\circ}{S}{}'}
({\cal E}'\otimes_{{\cal O}_{{\cal P}'{}^{\rm ex}}}
\Om^{\bul}_{{\cal P}'{}^{\rm ex}/\os{\circ}{S}{}'}) 
@>{\sim}>>
\end{CD}
\tag{6.7.1}\label{eqn:relcmrp} 
\end{equation*} 
\begin{equation*} 
\begin{CD}  
{\mathfrak g}_{{\rm conv}*}
a^{(k-1)}_{{\rm conv}*}
(L_{\os{\circ}{X}{}^{(k-1)}/\os{\circ}{S}}
({\cal E}\otimes_{{\cal O}_{{\cal P}^{\rm ex}}}
\Om^{\bul}_{\os{\circ}{\cal P}{}^{\rm ex}{}^{(k-1)}/\os{\circ}{S}})
\otimes_{\mab Z}\vp^{(k-1)}_{\rm conv}(\os{\circ}{X}/\os{\circ}{S}))[-k]\\
@AAA \\
a'{}^{(k-1)}_{{\rm conv}*}
(L_{\os{\circ}{X}{}'{}^{(k-1)}/\os{\circ}{S}{}'}
({\cal E}'\otimes_{{\cal O}_{{\cal P}{}'^{\rm ex}}}
\Om^{\bul}_{({\cal P}{}'{}^{\rm ex})^{\circ}{}^{(k-1)}/\os{\circ}{S}})
\otimes_{\mab Z}\vp^{(k-1)}_{\rm conv}(\os{\circ}{X}{}'/\os{\circ}{S}{}'))[-k].  
\end{CD}
\end{equation*} 
\end{prop}
\begin{proof} 
We leave the proof of this proposition to the reader because the proof is only 
the SNCL analogue of that of 
\cite[(6.13)]{nhw}. 
\end{proof}

\section{Simplicial log convergent topoi}\label{sec:rlct}
In this section we give the notations of simplicial log convergent topoi and 
simplicial Zariski topoi for later sections. 
\par 
Let the notations be as in the beginning of 
\S\ref{sec:logcd}.  
Let $Y=\us{i\in I}{\bigcup}Y_i$ 
be an open covering of $Y$, where 
$I$ is a (not necessarily finite) set. 
Set $Y_0:=\coprod_{i\in I}Y_i$ and 
$Y_n:= {\rm cosk}_0^Y(Y_0)_n:= 
\underset{n~{\rm pieces}}{\underbrace{
Y_0\times_Y \cdots \times_YY_0}}$.  
Then we have the simplicial log scheme $Y_{\bul}=
Y_{\bul \in {\mab N}}$ over $Y$. 
We also have the ringed topos 
$(({Y_{\bul}/S})_{\rm conv},{\cal K}_{Y_{\bul}/S})$. 
Let 
$\pi^{\rm conv}_{Y/S} \col  
({Y_{\bul}/S})_{\rm conv} \lo 
({Y/S})_{\rm conv}$ 
be the natural morphism of topoi.
The morphism
$\pi^{\rm conv}_{Y/S}$ induces the following morphism 
\begin{equation*} 
\pi^{\rm conv}_{Y/S}
\col
(({Y_{\bul}/S})_{\rm conv},{\cal K}_{Y_{\bul}/S})
\lo 
(({Y/S})_{\rm conv},{\cal K}_{Y/S}) 
\tag{7.0.1}\label{eqn:pidf}
\end{equation*}
of ringed topoi and a morphism of filtered derived categories:
\begin{equation*}
R\pi^{\rm conv}_{Y/S*} \col
{\rm D}^+{\rm F}({\cal K}_{Y_{\bul}/S}) \lo 
{\rm D}^+{\rm F}({\cal K}_{Y/S}).
\tag{7.0.2}\label{eqn:rtalc}
\end{equation*}
\par
By using the composite morphisms $f_i \col 
\os{\circ}{Y}_i \os{\subset}{\lo} \os{\circ}{Y} 
\os{\os{\circ}{f}}{\lo} \os{\circ}{S}$ $(i\in I)$, we 
also have an analogous simplicial ringed topos 
$({\os{\circ}{Y}}_{\bul{\rm zar}},
f^{-1}_{\bul}({\cal K}_S))$ and 
an analogous morphism 
\begin{equation*} 
\pi^{\rm zar}_{Y/S*} \col ({\os{\circ}{Y}}_{\bul {\rm zar}}, 
f^{-1}_{\bul}({\cal K}_S)) \lo 
({\os{\circ}{Y}}_{\rm zar}, f^{-1}({\cal K}_S)) 
\tag{7.0.3}\label{eqn:pizs}
\end{equation*} 
of ringed topoi.
We have the following morphism of filtered derived categories:
\begin{equation*}
R\pi^{\rm zar}_{Y/S*} \col
{\rm D}^+{\rm F}(f^{-1}_{\bul}({\cal K}_S)) \lo 
{\rm D}^+{\rm F}(f^{-1}({\cal K}_S)). 
\tag{7.0.4}\label{eqn:rtazr}
\end{equation*}
Let $u^{\rm conv}_{Y/S} \col 
((Y/S)_{\rm conv}, {\cal K}_{Y/S})
\lo ({\os{\circ}{Y}}_{{\rm zar}},f^{-1}({\cal K}_S))$ 
be the projection in (\ref{eqn:uwksr}). 
We also have the following natural projection
\begin{equation*}
u^{\rm conv}_{Y_{\bul}/S} 
\col 
(({Y_{\bul}/S})_{\rm conv}, {\cal K}_{Y_{\bul}/S})
\lo ({\os{\circ}{Y}}_{\bul{\rm zar}},f^{-1}_{\bul}({\cal K}_S)). 
\tag{7.0.5}\label{eqn:uysr}
\end{equation*}
We have the following commutative diagram: 
\begin{equation*}
\begin{CD}
{\rm D}^+{\rm F}({\cal K}_{Y_{\bul}/S}) 
@>{R\pi^{\rm conv}_{Y/S*}}>> 
{\rm D}^+{\rm F}({\cal K}_{Y/S})  \\ 
@V{Ru^{\rm conv}_{Y_{\bul}/S*}}VV 
@VV{Ru^{\rm conv}_{Y/S*}}V \\
{\rm D}^+{\rm F}(f^{-1}_{\bul}({\cal K}_S))  
@>{R\pi^{\rm zar}_{Y/S*}}>> 
{\rm D}^+{\rm F}(f^{-1}({\cal K}_S)). 
\end{CD}
\tag{7.0.6}\label{eqn:ubys}
\end{equation*}
\par
Assume that we are given the commutative diagram 
(\ref{cd:yypssp}) and open coverings 
$\{Y_i\}_{i \in I}$ of $Y$ and 
$\{Y'_i\}_{i \in I}$ of $Y'$, respectively, 
such that $g$ induces a morphism 
$g_i \col Y'_i \lo Y_i$ of log schemes. 
Then the family $\{g_i\}_{i\in I}$ 
induces the natural morphism 
$g_{\bul}\col Y'_{\bul} \lo Y_{\bul}$
of simplicial log schemes over the morphism $Y'\lo Y$. 
Hence we obtain the following commutative diagram:
\begin{equation*}
\begin{CD}
{\rm D}^+{\rm F}({\cal K}_{Y'_{\bul}/S'}) 
@>{Rg_{\bul {\rm conv}*}}>> 
{\rm D}^+{\rm F}({\cal K}_{Y_{\bul}/S}) \\ 
@V{R\pi^{\rm conv}_{Y'/S'*}}VV 
@VV{R\pi^{\rm zar}_{Y/S*}}V \\
{\rm D}^+{\rm F}({\cal K}_{Y'/S'}) 
@>{Rg_{{\rm conv}*}}>>  
{\rm D}^+{\rm F}({\cal K}_{Y/S}).
\end{CD}
\tag{7.0.7}\label{cd:rgrtc}
\end{equation*} 
We also have the following commutative diagram 
\begin{equation*}
\begin{CD}
{\rm D}^+{\rm F}(f'{}^{-1}_{\bul}({\cal K}_{S'})) 
@>{Rg_{\bul{\rm zar}*}}>> 
{\rm D}^+{\rm F}(f^{-1}_{\bul}({\cal K}_S)) \\ 
@V{R\pi^{\rm zar}_{Y'/S'*}}VV @VV{R\pi^{\rm zar}_{Y/S*}}V \\
{\rm D}^+{\rm F}(f'{}^{-1}({\cal K}_{S'})) @>{Rg_{{\rm zar}*}}>>  
{\rm D}^+{\rm F}(f^{-1}({\cal K}_{S})).
\end{CD}
\tag{7.0.8}\label{cd:rgztac}
\end{equation*}
\par 
Let $M$ be the log structure of $Y$. 
Let $N$ be a sub log structure of $M$. 
Let $M_i$ (resp.~$N_i$) be the pull-back of 
$M$ (resp.~$N$) to $Y_i$ $(i\in I)$.  
Let 
$\eps_{(\os{\circ}{Y},M,N)/S'/S} \col Y=(\os{\circ}{Y},M) \lo (\os{\circ}{Y},N)$ 
be the morphism forgetting the structure $M\setminus N$ 
over a morphism $S'\lo S$ in \S{sec:logcd}. 
For simplicity of notation, denote $(\os{\circ}{Y},M)$ and 
$(\os{\circ}{Y},N)$ by $Y'$ and $Y$, respectively. 
The morphism $\eps_{(\os{\circ}{Y},M,N)/S'/S}$ 
induces the following morphisms of ringed topoi: 
\begin{equation*}
\eps^{\rm conv}_{(\os{\circ}{Y}_{\bul},M_{\bul},N_{\bul})/S'/S} \col 
((Y'_{\bul}/S')_{\rm conv},
{\cal K}_{Y'_{\bul}/S'}) 
\lo 
((Y_{\bul}/S)_{\rm conv},{\cal K}_{Y_{\bul}/S}).  
\tag{7.0.9}\label{eqn:ymnb}
\end{equation*} 
We have the following commutative diagram
\begin{equation*}
\begin{CD} 
((Y'_{\bul}/S')_{\rm conv},{\cal K}_{Y'_{\bul}/S'}) 
@>{\eps^{\rm conv}_{(\os{\circ}{Y}_{\bul},M_{\bul},N_{\bul})/S'/S}}>> 
((Y_{\bul}/S)_{\rm conv},{\cal K}_{Y_{\bul}/S})\\ 
@V{\pi^{\rm conv}_{Y'/S'}}VV  
@VV{\pi^{\rm conv}_{Y/S}}V \\
((Y'/S')_{\rm conv},{\cal K}_{Y'/S'}) 
@>{\eps^{\rm conv}_{(\os{\circ}{Y},M,N)/S'/S}}>> ((Y/S)_{\rm conv},{\cal K}_{Y/S}).
\end{CD}
\tag{7.0.10}\label{cd:epep}
\end{equation*} 

\section{Modified $P$-filtered convergent complexes and $p$-adic semi-purity}\label{sec:mplf}
In this section we define the modified $P$-filtered convergent complex of an isocrystal 
on the underlying scheme of an SNCL scheme in characteristic $p>0$. 
This filtered complex is the first fundamenetal complex in this paper. 
Especially it plays a central role in the construction of 
the convergent Steenbrink filtered complex, which will be proved in later sections. 
\par 
Let $S$ be a ${\cal V}$-formal family of log points. 
Let $Y/S_1$ be a log smooth scheme. 
Let $g\col Y\lo S$ be the structural morphism. 
Let $Y=\bigcup_{i\in I} Y_i$ be an affine open covering of $Y$. 
Let $Y_{\bul}$ be the simplicial log scheme obtained by this affine open covering of $Y$. 
\par 
Let $\star$ be nothing or $\bul$. 
Let  
\begin{align*} 
\eps_{Y_{\star}/S/\os{\circ}{S}} \col 
((Y_{\star}/S)_{\rm conv},{\cal K}_{Y_{\star}/S}) \lo 
((Y_{\star}/\os{\circ}{S})_{\rm conv},{\cal K}_{Y_{\star}/\os{\circ}{S}}), 
\tag{8.0.1}\label{ali:eysss}
\end{align*} 
\begin{align*} 
\eps_{Y_{\star}/\os{\circ}{S}} \col 
((Y_{\star}/\os{\circ}{S})_{\rm conv},{\cal K}_{Y_{\star}/\os{\circ}{S}}) \lo 
((\os{\circ}{Y}_{\star}/\os{\circ}{S})_{\rm conv},
{\cal K}_{\os{\circ}{Y}_{\star}/\os{\circ}{S}}) 
\tag{8.0.2}\label{ali:eykskws}
\end{align*} 
and 
\begin{align*} 
\eps_{Y_{\star}/S} \col 
((Y_{\star}/S)_{\rm conv},{\cal K}_{Y_{\star}/S}) \lo 
((\os{\circ}{Y}_{\star}/\os{\circ}{S})_{\rm conv},
{\cal K}_{\os{\circ}{Y}_{\star}/\os{\circ}{S}})
\tag{8.0.3}\label{ali:eyksks}
\end{align*} 
be the morphisms forgetting the log structures. 
Then $\eps_{Y_{\star}/S} =\eps_{Y_{\star}/\os{\circ}{S}}\circ \eps_{Y_{\star}/S/\os{\circ}{S}}$. 
\par 
Let $Y_i \os{\sus}{\lo} \ol{\cal Q}_i$ be an immersion into a log smooth 
log affine $p$-adic formal ${\cal V}$-scheme over $\ol{S}$ 
such that the image of $\ol{\cal Q}_i$ in $\ol{S}$ is 
a log affine $p$-adic formal ${\cal V}$-scheme. 
If $Y_i$ is sufficiently small, this immersion indeed exists by \cite[(1.1.6)]{nb} 
and the criterion of the log smoothness in \cite{klog1}. 
Set $\ol{\cal Q}_0:=\coprod_i\ol{\cal Q}_i$ and 
$\ol{\cal Q}_{\bul}:={\rm cosk}_0^{\ol{S}}(\ol{\cal Q}_0)$. 
The underlying formal scheme of $\ol{\cal Q}_n$ is the disjoint union of 
affine $p$-adic formal ${\cal V}$-schemes.  
We have a simplicial immersion 
$Y_{\bul} \os{\sus}{\lo} \ol{\cal Q}_{\bul}$ 
into a log smooth log $p$-adic formal ${\cal V}$-scheme over $\ol{S}$. 
Set ${\cal Q}_{\bul}:=\ol{\cal Q}_{\bul}\times_{\ol{S}}S$. 
Let $g_{\bul} \col {\cal Q}^{\rm ex}_{\bul}\lo S$ be the structural morphism. 
Let $\{(\ol{V}_{\bul n},\ol{T}_{\bul n})\}_{n=1}^{\infty}$ be the system of 
the universal enlargements of the simplicial immersion 
$Y_{\bul} \os{\sus}{\lo} \ol{\cal Q}_{\bul}$ over $\ol{S}$. 
Set $(V_{\bul n},T_{\bul n}):=(\ol{V}_{\bul n},\ol{T}_{\bul n})\times_{\ol{S}}S$. 
Let $\bet_{\bul n} \col V_{\bul n} \lo Y_{\bul}$ 
be the natural morphism.  
Let $F$ be a flat isocrystal of ${\cal K}_{Y/\os{\circ}{S}}$-modules. 
Let $(\ol{\cal F}{}^{\bul}_n,\ol{\nabla}_n)$ be the flat coherent ${\cal K}_{\ol{T}_{\bul n}}$-modules 
with integrable connection obtained by $F$: 
$\ol{\nabla}_n \col \ol{\cal F}{}^{\bul}_n\lo \ol{\cal F}{}^{\bul}_n
\otimes_{{\cal O}_{\ol{\cal Q}^{\rm ex}_{\bul}}}\Om^1_{\ol{\cal Q}^{\rm ex}_{\bul}/\os{\circ}{S}}$.  
Set ${\cal F}_n:=\ol{\cal F}{}^{\bul}_n\otimes_{{\cal O}_{\ol{S}}}{\cal O}_S$. 
Since $dt=td\log t$ for a local section $t$ such that 
${\cal O}_{\ol{S}}={\cal O}_S\{t\}$ locally on $\ol{S}$, 
$\ol{\nabla}_n$ 
induces the following integrable connection 
\begin{align*} 
{\nabla}_n \col {\cal F}^{\bul}_n\lo {\cal F}^{\bul}_n
\otimes_{{\cal O}_{{\cal Q}^{\rm ex}_{\bul}}}\Om^1_{{\cal Q}^{\rm ex}_{\bul}/\os{\circ}{S}}
\end{align*} 
and we obtain the projective system $\{{\nabla}_n\}_{n=0}^{\infty}$ and 
the following integrable connection 
\begin{align*} 
\vpl_n{\nabla}_n \col \vpl_n\bet_{\bul n*}{\cal F}_n\lo 
(\vpl_n\bet_{\bul n*}{\cal F}_n)\otimes_{{\cal O}_{{\cal Q}^{\rm ex}_{\bul}}}
\Om^1_{{\cal Q}^{\rm ex}_{\bul}/\os{\circ}{S}}. 
\end{align*} 
For simplicity of notation, set 
\begin{align*} 
{\cal F}^{\bul}\otimes_{{\cal O}_{{\cal Q}^{\rm ex}_{\bul}}}
\Om^{\bul}_{{\cal Q}^{\rm ex}_{\bul}/\os{\circ}{S}}:=
\{{\cal F}^{\bul}_n\otimes_{{\cal O}_{{\cal Q}^{\rm ex}_{\bul}}}
\Om^{\bul}_{{\cal Q}^{\rm ex}_{\bul}/\os{\circ}{S}}\}_{n=0}^{\infty}. 
\end{align*} 
Let $L^{\rm UE}_{Y_{\bul}/\os{\circ}{S}}$ and $L^{\rm conv}_{\os{\circ}{Y}_{\bul}/\os{\circ}{S}}$ 
be the (log) convergent linearization functors with respect to the immersions 
$Y_{\bul} \os{\sus}{\lo} \ol{\cal Q}_{\bul}$ over $\os{\circ}{S}$ 
and $\os{\circ}{Y}_{\bul} \os{\sus}{\lo} \os{\circ}{\ol{\cal Q}}_{\bul}$ over $\os{\circ}{S}$. 
Consider the following sheaf 
\begin{align*}  
L^{\rm UE}_{\os{\circ}{Y}_{\bul}/\os{\circ}{S}}
({\cal F}^{\bul}\otimes_{{\cal O}_{{\cal Q}^{\rm ex}_{\bul}}}
\Om^i_{{\cal Q}^{\rm ex}_{\bul}/\os{\circ}{S}})\simeq 
\eps_{Y_{\bul}/\os{\circ}{S}*}
L^{\rm UE}_{Y_{\bul}/\os{\circ}{S}}
({\cal F}^{\bul}\otimes_{{\cal O}_{{\cal Q}^{\rm ex}_{\bul}}}
\Om^i_{{\cal Q}^{\rm ex}_{\bul}/\os{\circ}{S}})
\tag{8.0.4}\label{ali:lys}
\end{align*} 
for $i\in {\mab N}$, 
which is isomorphic to 
$R\eps_{Y_{\bul}/\os{\circ}{S}*}
L^{\rm UE}_{Y_{\bul}/\os{\circ}{S}}
({\cal F}^{\bul}\otimes_{{\cal O}_{{\cal Q}^{\rm ex}_{\bul}}}
\Om^i_{{\cal Q}^{\rm ex}_{\bul}/\os{\circ}{S}})$ ((\ref{coro:epnr})). 
By (\ref{coro:ddn}) 
we obtain the complex 
$$L^{\rm UE}_{\os{\circ}{Y}_{\bul}/\os{\circ}{S}}
({\cal F}^{\bul}\otimes_{{\cal O}_{{\cal Q}^{\rm ex}_{\bul}}}
\Om^{\bul}_{{\cal Q}^{\rm ex}_{\bul}/\os{\circ}{S}}).$$ 
Let $L_{Y_{\bul}/S}$ be the log convergent linearization functor for the immersion 
$Y_{\bul} \os{\sus}{\lo} {\cal Q}_{\bul}$ over $S$.   
Set 
${\cal F}^{\bul}\otimes_{{\cal O}_{{\cal Q}^{\rm ex}_{\bul}}}
\Om^{\bul}_{{\cal Q}^{\rm ex}_{\bul}/S}
=\{{\cal F}^{\bul}_n\otimes_{{\cal O}_{{\cal Q}^{\rm ex}_{\bul}}}
\Om^{\bul}_{{\cal Q}^{\rm ex}_{\bul}/S}\}_{n=1}^{\infty}$. 

\begin{theo}\label{theo:pyil}
Let $\pi^{\rm conv}_{\os{\circ}{Y}/\os{\circ}{S}}$ 
be the morphism of ringed topoi for $\os{\circ}{Y}/\os{\circ}{S}$ in 
{\rm (\ref{eqn:pidf})}.   
Then the complex 
$R\pi^{\rm conv}_{\os{\circ}{Y}/\os{\circ}{S}*}(L^{\rm UE}_{\os{\circ}{Y}_{\bul}/\os{\circ}{S}}
({\cal F}^{\bul}\otimes_{{\cal O}_{{\cal Q}^{\rm ex}_{\bul}}}
\Om^{\bul}_{{\cal Q}^{\rm ex}_{\bul}/\os{\circ}{S}}))\in 
{\rm D}^+{\rm F}({\cal K}_{\os{\circ}{Y}/\os{\circ}{S}})$ 
is independent of the choice of an affine open covering of $Y$ and 
an immersion $Y_i \os{\sus}{\lo} \ol{\cal Q}_i$ $(i\in I)$ into a log smooth 
log  affine $p$-adic formal ${\cal V}$-scheme over $\ol{S}$. . 
\end{theo} 
\begin{proof} 
Since the natural morphism 
${\cal Q}_{\bul}\lo \ol{\cal Q}_{\bul}$ is an immersion over $S\lo \os{\circ}{S}$, 
the induced morphism 
${\mathfrak T}_{X,n}({\cal Q})\lo {\mathfrak T}_{X,n}(\ol{\cal Q})$ is an immersion. 
Hence we obtain the following equality by the log Poincar\'{e} lemma (\ref{eqn:epdlym}) 
and (\ref{theo:cpvcs}): 
\begin{align*} 
R\eps^{\rm conv}_{Y_{\bul}/S*}(\eps^{{\rm conv}*}_{Y_{\bul}/S/\os{\circ}{S}}
(\pi^{{\rm conv}*}_{Y/\os{\circ}{S}*}(F)))
&=
R\eps^{\rm conv}_{Y_{\bul}/S*}(L^{\rm UE}_{Y_{\bul}/S}
({\cal F}^{\bul}\otimes_{{\cal O}_{{\cal Q}^{\rm ex}_{\bul}}}
\Om^{\bul}_{{\cal Q}^{\rm ex}_{\bul}/S})) 
\tag{8.1.1}\label{ali:ysbs}\\
&=L^{\rm UE}_{\os{\circ}{Y}_{\bul}/\os{\circ}{S}}
({\cal F}^{\bul}
\otimes_{{\cal O}_{{\cal Q}^{\rm ex}_{\bul}}}\Om^{\bul}_{{\cal Q}^{\rm ex}_{\bul}/S}). 
\end{align*} 
Hence the following formula holds by the cohomological descent: 
\begin{align*} 
R\pi^{\rm conv}_{\os{\circ}{Y}/\os{\circ}{S}*}(L^{\rm UE}_{\os{\circ}{Y}_{\bul}/\os{\circ}{S}}
({\cal F}^{\bul}\otimes_{{\cal O}_{{\cal Q}^{\rm ex}_{\bul}}}\Om^{\bul}_{{\cal Q}^{\rm ex}_{\bul}/S})) 
&=R\pi^{\rm conv}_{\os{\circ}{Y}/\os{\circ}{S}*}
R\eps^{\rm conv}_{Y_{\bul}/S*}(\eps^{{\rm conv}*}_{Y_{\bul}/S/\os{\circ}{S}}
(\pi^{{\rm conv}*}_{Y/\os{\circ}{S}*}(F)))\tag{8.1.2}\label{ali:yybs}\\
&=R\eps^{\rm conv}_{Y/S*}R\pi^{\rm conv}_{Y/S*}
(\pi^{{\rm conv}*}_{Y/S}(\eps^{{\rm conv}*}_{Y/S/\os{\circ}{S}}(F)))   \\
&=R\eps^{\rm conv}_{Y/S*}(\eps^{{\rm conv}*}_{Y/S/\os{\circ}{S}}(F)). 
\end{align*} 
It is immediate to see that 
the following sequence 
\begin{align*} 
0\lo {\cal F}^{\bul}_n\otimes_{{\cal O}_{{\cal Q}^{\rm ex}_{\bul}}}
\Om^{\bul}_{{\cal Q}^{\rm ex}_{\bul}/S}[-1]
\os{\theta_{\bul} \wedge}{\lo} {\cal F}^{\bul}_n\otimes_{{\cal O}_{{\cal Q}^{\rm ex}_{\bul}}}
\Om^{\bul}_{{\cal Q}^{\rm ex}_{\bul}/\os{\circ}{S}}
\lo 
{\cal F}^{\bul}_n\otimes_{{\cal O}_{{\cal Q}^{\rm ex}_{\bul}}}
\Om^{\bul}_{{\cal Q}^{\rm ex}_{\bul}/S}\lo 0
\tag{8.1.3}\label{ali:foq}
\end{align*} 
is exact. 
Here $\theta_{\bul}:=g^*_{\bul}(d\log t) \in \Om^1_{{\cal Q}^{\rm ex}_{\bul}/\os{\circ}{S}}$, 
where $t$ is a local section of $M_S$ which gives the local generator 
of $\ol{M}_S$. (Note again that $g^*_{\bul}(d\log t)$ is independent of the choice of $t$.) 
In fact, the following sequence 
\begin{align*} 
0\lo {\cal F}^{\bul}_n\otimes_{{\cal O}_{{\cal Q}^{\rm ex}_{\bul}}}
\Om^{i-1}_{{\cal Q}^{\rm ex}_{\bul}/S}
\os{\theta_{\bul} \wedge}{\lo} {\cal F}^{\bul}_n\otimes_{{\cal O}_{{\cal Q}^{\rm ex}_{\bul}}}
\Om^i_{{\cal Q}^{\rm ex}_{\bul}/\os{\circ}{S}}
\lo 
{\cal F}^{\bul}_n\otimes_{{\cal O}_{{\cal Q}^{\rm ex}_{\bul}}}
\Om^i_{{\cal Q}^{\rm ex}_{\bul}/S}\lo 0
\end{align*} 
is locally split for each $i\in {\mab N}$. 
Hence we obtain the following exact sequence 
\begin{align*} 
0\lo L^{\rm UE}_{\os{\circ}{Y}_{\bul}/\os{\circ}{S}}
({\cal F}^{\bul}\otimes_{{\cal O}_{{\cal Q}^{\rm ex}_{\bul}}}
\Om^{\bul}_{{\cal Q}^{\rm ex}_{\bul}/S})[-1]  
\lo 
L^{\rm UE}_{\os{\circ}{Y}_{\bul}/\os{\circ}{S}}
({\cal F}^{\bul}\otimes_{{\cal O}_{{\cal Q}^{\rm ex}_{\bul}}}
\Om^{\bul}_{{\cal Q}^{\rm ex}_{\bul}/\os{\circ}{S}}) 
\lo 
L^{\rm UE}_{\os{\circ}{Y}_{\bul}/\os{\circ}{S}}
({\cal F}^{\bul}\otimes_{{\cal O}_{{\cal Q}^{\rm ex}_{\bul}}}
\Om^{\bul}_{{\cal Q}^{\rm ex}_{\bul}/S})\lo 0.  
\end{align*} 
Applying $R\pi^{\rm conv}_{\os{\circ}{Y}/\os{\circ}{S}*}$ to 
this exact sequence and using (\ref{ali:yybs}), 
we have the following triangle: 
\begin{align*} 
R\eps^{\rm conv}_{Y/S*}(\eps^{{\rm conv}*}_{Y/S/\os{\circ}{S}}(F))[-1]\os{\theta \wedge}{\lo}  
R\pi^{\rm conv}_{\os{\circ}{Y}/\os{\circ}{S}*}(L^{\rm UE}_{\os{\circ}{Y}_{\bul}/\os{\circ}{S}}
({\cal F}^{\bul}\otimes_{{\cal O}_{{\cal Q}^{\rm ex}_{\bul}}}
\Om^{\bul}_{{\cal Q}^{\rm ex}_{\bul}/\os{\circ}{S}})) 
\lo R\eps^{\rm conv}_{Y/S*}(\eps^{{\rm conv}*}_{Y/S/\os{\circ}{S}}(F))
\tag{8.1.4}\label{ali:oqy}
\os{+1}{\lo} . 
\end{align*} 
\par 
Let $Y'_{\bul}\os{\sus}{\lo} \ol{\cal Q}{}'_{\bul}$ be another analogous simplicial immersion. 
Then, by considering a refinement of two affine open coverings of $Y$ 
and the fiber product, 
we may assume that there exists the following commutative diagram:
\begin{equation*} 
\begin{CD}
Y_{\bul} @>{\sus}>> \ol{\cal Q}_{\bul}\\
@VVV @VVV \\
Y'_{\bul} @>{\sus}>> \ol{\cal Q}{}'_{\bul}. 
\end{CD}
\end{equation*} 
Now it is a routine work to show that 
$R\pi^{\rm conv}_{\os{\circ}{Y}/\os{\circ}{S}*}(L^{\rm conv}_{\os{\circ}{Y}_{\bul}/\os{\circ}{S}}
({\cal F}^{\bul}\otimes_{{\cal O}_{{\cal Q}^{\rm ex}_{\bul}}}
\Om^{\bul}_{{\cal Q}^{\rm ex}_{\bul}/\os{\circ}{S}}))$ 
is independent of the choice of an affine open covering of $Y$ 
and the immersion $Y_i \os{\sus}{\lo} \ol{\cal Q}_i$ over 
$\ol{S}$. 
\end{proof} 

\begin{defi}\label{defi:crpi} 
We call 
$R\pi^{\rm conv}_{\os{\circ}{Y}/\os{\circ}{S}*}
(L^{\rm conv}_{\os{\circ}{Y}_{\bul}/\os{\circ}{S}}
({\cal F}\otimes_{{\cal O}_{{\cal Q}^{\rm ex}_{\bul}}}
\Om^{\bul}_{{\cal Q}^{\rm ex}_{\bul}/\os{\circ}{S}}))$ 
the {\it modified convergent complex} of $F$ and we denote 
it by 
$\wt{R}\eps^{\rm conv}_{Y/\os{\circ}{S}*}(F)$. 
When $F={\cal K}_{Y/\os{\circ}{S}}$, we call $\wt{R}\eps^{\rm conv}_{Y/\os{\circ}{S}*}(F)$ 
the {\it modfifed convergent complex} of $Y/\os{\circ}{S}$ and 
denote it by $\wt{R}\eps^{\rm conv}_{Y/\os{\circ}{S}*}({\cal K}_{Y/\os{\circ}{S}})$. 
We call $Ru^{\rm conv}_{\os{\circ}{Y}/\os{\circ}{S}*}(\wt{R}\eps_{Y/\os{\circ}{S}*}(F))$ 
the {\it modified isozariskian complex} of $F$ and 
we denote it by $\wt{R}u^{\rm conv}_{Y/\os{\circ}{S}*}(F)$.  
When $F={\cal K}_{Y/\os{\circ}{S}}$, we call $\wt{R}u^{\rm conv}_{Y/\os{\circ}{S}*}(F)$
the {\it modfifed convergent-isozariskian complex} of $Y/\os{\circ}{S}$ and 
denote it by $\wt{R}u^{\rm conv}_{Y/\os{\circ}{S}*}({\cal K}_{Y/\os{\circ}{S}})$. 
\end{defi}

By (\ref{ali:oqy}) we obtain the following triangles: 
\begin{align*} 
R\eps^{\rm conv}_{Y/S*}(\eps^{{\rm conv}*}_{Y/S/\os{\circ}{S}}(F))[-1]\os{\theta \wedge}{\lo}  
\wt{R}\eps^{\rm conv}_{Y/\os{\circ}{S}*}(F) 
\lo R\eps^{\rm conv}_{Y/S*}(\eps^{{\rm conv}*}_{Y/S/\os{\circ}{S}}(F))
\tag{8.2.1}\label{ali:oqay}
\os{+1}{\lo}, 
\end{align*} 
\begin{align*} 
Ru^{\rm conv}_{Y/S*}(\eps^{{\rm conv}*}_{Y/S/\os{\circ}{S}}(F))[-1]\os{\theta \wedge}{\lo}
\wt{R}u^{\rm conv}_{Y/\os{\circ}{S}*}(F) 
\lo Ru^{\rm conv}_{Y/S*}(\eps^{{\rm conv}*}_{Y/S/\os{\circ}{S}}(F))
\tag{8.2.2}\label{ali:ouqby}
\os{+1}{\lo} . 
\end{align*} 

\begin{prop}\label{prop:cwass}
Let the notations be as in {\rm (\ref{cd:yypssp})}. 
Assume that $Y'$ and $Y$ are log smooth over $S'_1$ and $S_1$, respectively. 
For a morphism $g^* \col g^*_{\rm conv}(F)\lo F'$ 
of ${\cal K}_{Y'/\os{\circ}{S}{}'}$-modules, 
there exists a morphism 
\begin{align*} 
g^*\col  \wt{R}\eps^{\rm conv}_{Y/\os{\circ}{S}*}(F) \lo 
R\os{\circ}{g}_{{\rm conv}*}\wt{R}\eps^{\rm conv}_{Y'/\os{\circ}{S}{}'*}(F) 
\tag{8.3.1}\label{ali:team}
\end{align*} 
such that 
${\rm id}_Y^*={\rm id}_{\wt{R}\eps^{\rm conv}_{Y/\os{\circ}{S}*}(F)}$ and 
$(g\circ h)^*=R\os{\circ}{g}_{{\rm conv}*}(h^*)\circ g^*$ for an analogous morphism
$h\col Y''\lo Y'$ over $S''\lo S'$ to $g$ 
and an analogus morphism $h^* \col h^*_{\rm conv}(F')\lo F''$ of ${\cal K}_{Y''/\os{\circ}{S}{}''}$-modules 
to $g^*\col g^*_{\rm conv}(F)\lo F'$ of ${\cal K}_{Y'/\os{\circ}{S}{}'}$-modules. 
\end{prop}
\begin{proof} 
(The proof is not so standard.) 
Let $Y'_{\bul}\lo Y_{\bul}$ be a morphism of simplicial log schemes obtained by 
affine open coverings of $Y'$ and $Y$ in the beginning of this section.  
Let $Y_{\bul}\os{\sus}{\lo} \ol{\cal Q}_{\bul}$ and 
$Y'_{\bul}\os{\sus}{\lo} \ol{\cal Q}{}'_{\bul}$ be immersions into 
formally log smooth log formal schemes
over $\ol{S}$ and $\ol{S}{}'$, respectively, which are constructed in the beginning of this section. 
Set $\ol{\cal R}_{\bul} :=\ol{\cal Q}{}'_{\bul}\times_{S'}(\ol{\cal Q}_{\bul}\times_{\ol{S}}\ol{S}{}')$. 
By using the immersion $Y'_{\bul}\os{\sus}{\lo} \ol{\cal Q}{}'_{\bul}$, 
the composite morphism 
$Y'_{\bul}\lo Y_{\bul} \os{\sus}{\lo} \ol{\cal Q}_{\bul}$ and 
the projection $\ol{\cal Q}_{\bul}\times_{\ol{S}}\ol{S}{}'\lo \ol{\cal Q}_{\bul}$,  
we have the following commutative diagram 
\begin{equation*} 
\begin{CD} 
Y'_{\bul}@>{\subset}>> \ol{\cal R}_{\bul} @>>>\ol{S}{}'\\
@VVV @VVV @VVV\\
Y_{\bul} @>{\subset}>> \ol{\cal Q}_{\bul} @>>>\ol{S}. 
\end{CD}
\end{equation*} 
Let 
\begin{align*} 
{\cal F}{}'^{\bul}\otimes_{{\cal O}_{{\cal R}^{\rm ex}_{\bul}}}
\Om^{\bul}_{{\cal R}^{\rm ex}_{\bul}/\os{\circ}{S}{}'}:=
\{{\cal F}{}'^{\bul}_n\otimes_{{\cal O}_{{\cal R}^{\rm ex}_{\bul}}}
\Om^{\bul}_{{\cal R}^{\rm ex}_{\bul}/\os{\circ}{S}{}'}\}_{n=0}^{\infty}. 
\end{align*} 
be an analogous object to 
${\cal F}^{\bul}\otimes_{{\cal O}_{{\cal Q}^{\rm ex}_{\bul}}}
\Om^{\bul}_{{\cal Q}^{\rm ex}_{\bul}/\os{\circ}{S}}$. 
Then the commutative diagram above and 
the morphism $g^* \col g^*_{\rm conv}(F)\lo (F')$ 
induces the following morphism 
\begin{align*} 
\os{\circ}{g}{}^*_{\rm conv}
L^{\rm UE}_{\os{\circ}{Y}_{\bul}/\os{\circ}{S}}
({\cal F}^{\bul}\otimes_{{\cal O}_{{\cal Q}^{\rm ex}_{\bul}}}
\Om^{\bul}_{{\cal Q}^{\rm ex}_{\bul}/\os{\circ}{S}})
\lo 
L^{\rm UE}_{\os{\circ}{Y}{}'_{\bul}/\os{\circ}{S}{}'}
({\cal F}{}'^{\bul}\otimes_{{\cal O}_{{\cal R}^{\rm ex}_{\bul}}}
\Om^{\bul}_{{\cal R}^{\rm ex}_{\bul}/\os{\circ}{S}{}'}). 
\end{align*} 
This morphism induces the morphism (\ref{ali:team}). 
\par 
It is a routine work to show that 
the morphism (\ref{ali:team}) is well-defined 
and that the relations in (\ref{prop:cwass}) hold. 
\end{proof} 

\par 
Let $X$ be an SNCL scheme over $S_1$. Let $f\col X\lo S_1\os{\sus}{\lo} S$ be 
the structural morphism. Let $X_{\bul}$ be the simplicial log scheme 
obtained by an affine covering of $X$. 
Let $X_{\bul}\os{\sus}{\lo} \ol{\cal P}_{\bul}$ be an immersion into 
a log smooth simplicial log $p$-adic formal ${\cal V}$-scheme over $\ol{S}$. 
Since $\ol{\cal P}{}^{\rm ex}$ is a simplcial strict semistable formal scheme over $\os{\circ}{S}$, 
the underlying formal scheme $\os{\circ}{\ol{\cal P}}{}^{\rm ex}$ is formally smooth 
over $\os{\circ}{S}$. 
Set ${\cal P}_{\bul}:=\ol{\cal P}_{\bul}\times_{\ol{S}}S$.  
Let $E$ be a flat isocrystal of ${\cal K}_{\os{\circ}{X}/\os{\circ}{S}}$-modules. 
Let $\{(\ol{U}_{\bul n},\ol{T}_{\bul n})\}_{n=1}^{\infty}$ be the system of 
the universal enlargements of the immersion 
$X_{\bul} \os{\sus}{\lo} \ol{\cal P}_{\bul}$ over $\ol{S}$. 
Set $(U_{\bul n},T_{\bul n}):=(\ol{U}_{\bul n},\ol{T}_{\bul n})\times_{\ol{S}}S$. 
Let $(\ol{\cal E}{}^{\bul}_{n},\nabla)$ be the corresponding integrable connection of 
the ${\cal K}_{\ol{T}_{\bul n}}$-module to $\pi^{{\rm conv}*}_{X/\os{\circ}{S}}\eps^*_{X/\os{\circ}{S}}(E)$: 
$\ol{\nabla}_n \col \ol{\cal E}{}^{\bul}_n\lo \ol{\cal E}{}^{\bul}_n
\otimes_{\cal P}\Om^1_{\ol{\cal P}{}^{\rm ex}_{\bul}/\os{\circ}{S}}$.  
Let $\bet_{\bul n} \col U_{\bul n} \lo X_{\bul}$ 
be the natural morphism.  
Set ${\cal E}^{\bul}_n:=\ol{\cal E}{}^{\bul}_n\otimes_{{\cal O}_{\ol{S}}}{\cal O}_S$. 
Let 
\begin{align*} 
{\nabla}_n \col {\cal E}^{\bul}_n\lo 
{\cal E}^{\bul}_n\otimes_{{\cal P}^{\rm ex}}\Om^1_{{\cal P}{}^{\rm ex}_{\bul}/\os{\circ}{S}}
\end{align*} 
be the induced integrable connection by $\ol{\nabla}_n$ obtained by using the relation
``$dt=t\log t$'' in $\Om^1_{\ol{\cal P}{}^{\rm ex}_{\bul}/\os{\circ}{S}}$ as before. 
We obtain the projective system $\{{\nabla}_n\}_{n=0}^{\infty}$ and 
the following integrable connection 
\begin{align*} 
\vpl_n{\nabla}_n \col \vpl_n\bet_{\bul n*}{\cal E}^{\bul}_n\lo 
(\vpl_n\bet_{\bul n*}{\cal E}^{\bul}_n)\otimes_{{\cal O}_{{\cal P}_{\bul}}}
\Om^1_{{{\cal P}^{\rm ex}_{\bul}}/\os{\circ}{S}}. 
\end{align*} 
Here we have identified 
$({T_{\bul n}})_{\rm zar}$ with $({U_{\bul n}})_{\rm zar}$. 
For simplicity of notation, set 
\begin{align*} 
({\cal E}^{\bul}\otimes_{{\cal O}_{{\cal P}^{\rm ex}_{\bul}}}
\Om^{\bul}_{{\cal P}^{\rm ex}_{\bul}/\os{\circ}{S}},P):=
(\{{\cal E}^{\bul}_n\otimes_{{\cal O}_{{\cal P}^{\rm ex}_{\bul}}}
\Om^{\bul}_{{\cal P}^{\rm ex}_{\bul}/\os{\circ}{S}}\}_{n=0}^{\infty},
\{\{{\cal E}^{\bul}_n\otimes_{{\cal O}_{{\cal P}^{\rm ex}_{\bul}}}
P_k\Om^{\bul}_{{\cal P}^{\rm ex}_{\bul}/\os{\circ}{S}}\}_{n=0}^{\infty}\}_{k\in {\mab Z}}). 
\end{align*} 
This is indeed a filtered complex by (\ref{theo:injf}) (1). 
We also have the following filtered complex by (\ref{theo:injf}) (2): 
\begin{align*} 
(L^{\rm UE}_{\os{\circ}{X}_{\bul}/\os{\circ}{S}}
({\cal E}^{\bul}\otimes_{{\cal O}_{{\cal P}^{\rm ex}_{\bul}}}
\Om^{\bul}_{{\cal P}^{\rm ex}_{\bul}/\os{\circ}{S}}),P)
:= &
(L^{\rm UE}_{\os{\circ}{X}_{\bul}/\os{\circ}{S}}
(\{{\cal E}^{\bul}_n\otimes_{{\cal O}_{{\cal P}^{\rm ex}_{\bul}}}
\Om^{\bul}_{{\cal P}^{\rm ex}_{\bul}/\os{\circ}{S}}\}_{n=0}^{\infty}),\\
&\{L^{\rm conv}_{\os{\circ}{X}_{\bul}/\os{\circ}{S}}
(\{{\cal E}^{\bul}_n\otimes_{{\cal O}_{{\cal P}^{\rm ex}_{\bul}}}
P_k\Om^{\bul}_{{\cal P}^{\rm ex}_{\bul}/\os{\circ}{S}}\}_{n=0}^{\infty})\}_{k\in {\mab Z}}).   
\end{align*} 

In order to prove (\ref{theo:pil}) below, we need the following:

\begin{lemm}[{\rm {\bf cf.~\cite[(4.14)]{nh3}}}]
\label{lemm:knit}
Let $k$ be a nonnegative integer. Then the following hold$:$ 
\par 
$(1)$ For the morphism 
$g \col  \os{\circ}{\cal P}{}^{\rm ex}_n\lo 
\os{\circ}{\cal P}{}^{\rm ex}_{n'}$ 
corresponding to a morphism $[n']\lo [n]$ in $\Del$, 
the assumption in {\rm (\ref{prop:mmoo})} is satisfied  for $Y={\cal P}{}^{\rm ex}_n$ 
and $Y'={\cal P}{}^{\rm ex}_{n'}$. 
Consequently there exists a morphism 
$g^{(k)}\col \os{\circ}{\cal P}{}^{{\rm ex},(k)}_{n}\lo \os{\circ}{\cal P}{}^{{\rm ex},(k)}_{n'}$ 
$(k\in {\mab N})$ fitting into 
the commutative diagram {\rm (\ref{cd:dmgdm})} 
for $Y={\cal P}{}^{\rm ex}_n$ and $Y'={\cal P}{}^{\rm ex}_{n'}$.
\par 
$(2)$ 
The family 
$\{\os{\circ}{\cal P}{}^{{\rm ex},(k)}_{n}\}_{n\in {\mab N}}$ 
gives a simplicial scheme $\os{\circ}{\cal P}{}^{{\rm ex},(k)}_{\bul}$ 
with natural morphism 
$b^{(k)}_{\bul} \col \os{\circ}{\cal P}{}^{{\rm ex},(k)}_{\bul}
\lo \os{\circ}{\cal P}{}^{\rm ex}_{\bul}$ 
of simplicial schemes. 
\end{lemm} 
\begin{proof}
The proof is the same as that of \cite[(4.14)]{nh3}. 
\end{proof}

\begin{theo}\label{theo:pil}
Let $\pi^{\rm conv}_{\os{\circ}{X}/\os{\circ}{S}}$ 
be the morphism of ringed topoi for $\os{\circ}{X}/\os{\circ}{S}$ in 
{\rm (\ref{eqn:pidf})}.   
Then the filtered complex 
$R\pi^{\rm conv}_{\os{\circ}{X}/\os{\circ}{S}*}
((L^{\rm UE}_{\os{\circ}{X}_{\bul}/\os{\circ}{S}}
({\cal E}^{\bul}\otimes_{{\cal O}_{{\cal P}^{\rm ex}_{\bul}}}
\Om^{\bul}_{{\cal P}^{\rm ex}_{\bul}/\os{\circ}{S}}),P))\in 
{\rm D}^+{\rm F}({\cal K}_{\os{\circ}{X}/\os{\circ}{S}})$ 
is independent of the choice of the open covering of $X$ and the simplicial immersion
$X_{\bul}\os{\sus}{\lo} \ol{\cal P}_{\bul}$ over $\ol{S}$. 
\end{theo} 
\begin{proof} 
By (\ref{theo:pyil}),  
$R\pi^{\rm conv}_{\os{\circ}{X}/\os{\circ}{S}*}
(L^{\rm UE}_{\os{\circ}{X}_{\bul}/\os{\circ}{S}}
({\cal E}^{\bul}\otimes_{{\cal O}_{{\cal P}^{\rm ex}_{\bul}}}
\Om^{\bul}_{{\cal P}^{\rm ex}_{\bul}/\os{\circ}{S}}))=: 
\wt{R}\eps^{\rm conv}_{X/S/\os{\circ}{S}*}(\eps^{{\rm conv}*}_{X/\os{\circ}{S}}(E))$ 
is independent of the choice of the open covering of $X$.  
\par 
Let $a_{\bul}^{(k-1)}\col \os{\circ}{X}{}^{(k-1)}_{\bul}\lo \os{\circ}{X}_{\bul}$ 
be the natural morphism.  
By (\ref{prop:0case}) and the cohomological descent we obtain the following: 
\begin{align*} 
&{\rm gr}_0^P
R\pi^{\rm conv}_{\os{\circ}{X}/\os{\circ}{S}*}
(L^{\rm UE}_{\os{\circ}{X}_{\bul}/\os{\circ}{S}}
({\cal E}^{\bul}\otimes_{{\cal O}_{{\cal P}^{\rm ex}_{\bul}}}
\Om^{\bul}_{{\cal P}^{\rm ex}_{\bul}/\os{\circ}{S}}))
=P_0R\pi^{\rm conv}_{\os{\circ}{X}/\os{\circ}{S}*}
(L^{\rm UE}_{\os{\circ}{X}_{\bul}/\os{\circ}{S}}
({\cal E}^{\bul}\otimes_{{\cal O}_{{\cal P}^{\rm ex}_{\bul}}}
\Om^{\bul}_{{\cal P}^{\rm ex}_{\bul}/\os{\circ}{S}}))
\tag{8.5.1}\label{ali:cp0r}\\
& =
R\pi^{\rm conv}_{\os{\circ}{X}/\os{\circ}{S}*}
(a^{(0)}_{\bul{\rm conv}*}(a^{(0)*}_{\bul {\rm conv}}(E^{\bul})
\otimes_{{\mab Z}}\vp_{\rm conv}^{(0)}(\os{\circ}{X}/\os{\circ}{S}))
\lo a^{(1)}_{\bul{\rm conv}*}(a^{(1)*}_{\bul}{\rm conv}(E^{\bul}))\otimes_{{\mab Z}}
\vp_{\rm conv}^{(1)}(\os{\circ}{X}_{\bul}/\os{\circ}{S}))
\lo \cdots )\\
&= (a^{(0)}_{{\rm conv}*}(a^{(0)*}_{\rm conv}(E)
\otimes_{{\mab Z}}\vp_{\rm conv}^{(0)}(\os{\circ}{X}/\os{\circ}{S}))
\lo a^{(1)}_{{\rm conv}*}(a^{(1)*}_{\rm conv}(E)\otimes_{{\mab Z}}
\vp_{\rm conv}^{(1)}(\os{\circ}{X}/\os{\circ}{S}))
\lo \cdots )\\
&={\rm Ker}(a^{(0)}_{{\rm conv}*}(a^{(0)*}_{\rm conv}(E)
\otimes_{\mab Z}\varpi^{(0)}_{\rm conv}
(\os{\circ}{X}/\os{\circ}{S})
\lo 
a^{(1)}_{{\rm conv}*}(a^{(1)*}_{\rm conv}(E)\otimes_{\mab Z}\varpi^{(1)}_{\rm conv}
(\os{\circ}{X}/\os{\circ}{S}))). 
\end{align*} 
\par 
Let $k$ be a positive integer. 
Because the filtered direct image and the gr-functor commutes 
(\cite[(1.3.4)]{nh2}), we have the following equalities by (\ref{lemm:knit}),  
(\ref{eqn:grpdl}) and the convergent Poincar\'{e} lemma: 
\begin{align*} 
{\rm gr}_k^PR\pi^{\rm conv}_{\os{\circ}{X}/\os{\circ}{S}*}
(L^{\rm UE}_{\os{\circ}{X}_{\bul}/\os{\circ}{S}}
({\cal E}^{\bul}\otimes_{{\cal O}_{{\cal P}^{\rm ex}_{\bul}}}
\Om^{\bul}_{{\cal P}^{\rm ex}_{\bul}/\os{\circ}{S}}))
&=R\pi^{\rm conv}_{\os{\circ}{X}/\os{\circ}{S}*}({\rm gr}_k^P
L^{\rm UE}_{\os{\circ}{X}_{\bul}/\os{\circ}{S}}
({\cal E}^{\bul}\otimes_{{\cal O}_{{\cal P}^{\rm ex}_{\bul}}}
\Om^{\bul}_{{\cal P}^{\rm ex}_{\bul}/\os{\circ}{S}}))
\tag{8.5.2}\label{ali:cpr}\\
&=R\pi^{\rm conv}_{\os{\circ}{X}/\os{\circ}{S}*}
(a^{(k-1)}_{\bul {\rm conv}*}(a^{(k-1)*}_{\bul {\rm conv}}(E)
\otimes_{{\mab Z}}\vp_{\rm conv}^{(k-1)}(\os{\circ}{X}_{\bul}/\os{\circ}{S})))[-k]\\
&= a^{(k-1)}_{{\rm conv}*}(a^{(k-1)*}_{\rm conv}(E)
\otimes_{{\mab Z}}\vp_{\rm conv}^{(k-1)}(\os{\circ}{X}/\os{\circ}{S}))[-k]. 
\end{align*} 
\par 
Because the rest of the proof is a routine work, we leave the rest of the proof to the reader. 
\end{proof}

\begin{defi}\label{defi:crpxi} 
We call 
$$R\pi^{\rm conv}_{\os{\circ}{X}/\os{\circ}{S}*}
((L^{\rm UE}_{\os{\circ}{X}_{\bul}/\os{\circ}{S}}
({\cal E}^{\bul}\otimes_{{\cal O}_{{\cal P}^{\rm ex}_{\bul}}}
\Om^{\bul}_{{\cal P}^{\rm ex}_{\bul}/\os{\circ}{S}}),P))$$  
the {\it modified $P$-filtered convergent complex} of $E$ and we denote 
it by 
$$(\wt{R}\eps^{\rm conv}_{X/\os{\circ}{S}*}(\eps^{{\rm conv}*}_{X/\os{\circ}{S}}(E)),P).$$ 
When $E={\cal K}_{\os{\circ}{X}/\os{\circ}{S}}$, we call 
$(\wt{R}\eps^{\rm conv}_{X/\os{\circ}{S}*}(\eps^{{\rm conv}*}_{X/\os{\circ}{S}}(E)),P)$ 
the {\it modfifed $P$-filtered convergent complex} of $X/\os{\circ}{S}$ and 
denote it by $(\wt{R}\eps_{X/\os{\circ}{S}*}({\cal K}_{X/\os{\circ}{S}}),P)$. 
We call $Ru^{\rm conv}_{\os{\circ}{X}/\os{\circ}{S}*}(
(\wt{R}\eps^{\rm conv}_{X/\os{\circ}{S}*}(\eps^{{\rm conv}*}_{X/\os{\circ}{S}}(E)),P))$ 
the {\it modified isozariskian $P$-filtered complex} of $E$ and 
we denote it by $(\wt{R}u^{\rm conv}_{X/\os{\circ}{S}*}(\eps^{{\rm conv}*}_{X/\os{\circ}{S}}(E)),P)$.  
When $E={\cal K}_{\os{\circ}{X}/\os{\circ}{S}}$, 
we call 
$Ru^{\rm conv}_{\os{\circ}{X}/\os{\circ}{S}*}(
(\wt{R}\eps^{\rm conv}_{X/\os{\circ}{S}*}(\eps^{{\rm conv}*}_{X/\os{\circ}{S}}(E)),P))$
the {\it modfifed $P$-filtered log isozariskian complex} of $Y/\os{\circ}{S}$ and 
denote it by $(\wt{R}u^{\rm conv}_{\os{\circ}{X}/\os{\circ}{S}*}({\cal K}_{Y/\os{\circ}{S}}),P)$. 
\end{defi}

Set 
\begin{align*} 
\wt{R}^k\eps^{\rm conv}_{X/\os{\circ}{S}*}(\eps^{{\rm conv}*}_{X/\os{\circ}{S}}(E)):=
{\cal H}^k(\wt{R}\eps^{\rm conv}_{X/\os{\circ}{S}*}(\eps^{{\rm conv}*}_{X/\os{\circ}{S}}(E)))
\quad (k\in {\mab Z}). 
\end{align*} 

The following is a main result in this section: 

\begin{theo}[{\bf $p$-adic semi-purity}]\label{theo:calvpc}
Let $k$ be a nonnegative integer. Then
\begin{align*} 
{} & 
\wt{R}^k\eps^{\rm conv}_{X/\os{\circ}{S}*}
(\eps^{{\rm conv}*}_{X/\os{\circ}{S}}(E))
\tag{8.7.1}\label{caseali:grmtkr}\\
{} & = 
\begin{cases}
a^{(k-1)}_{{\rm conv}*}(a^{(k-1)*}_{\rm conv}(E)
\otimes_{\mab Z}\varpi^{(k-1)}_{\rm conv}
(\os{\circ}{X}/\os{\circ}{S})) & (k>0),  \\
{\rm Ker}(a^{(0)}_{{\rm conv}*}(a^{(0)*}_{\rm conv}(E)
\otimes_{\mab Z}\varpi^{(0)}_{\rm conv}
(\os{\circ}{X}/\os{\circ}{S}))
\lo a^{(1)}_{{\rm conv}*}(a^{(1)*}_{\rm conv}(E)
\otimes_{\mab Z}\varpi^{(1)}_{\rm conv}
(\os{\circ}{X}/\os{\circ}{S})))
 & (k=0).  
\end{cases} \\
\end{align*} 
\end{theo}
\begin{proof}
The ``increasing filtration''
$\{P_k\wt{R}\eps^{\rm conv}_{X/\os{\circ}{S}*}(\eps^{{\rm conv}*}_{X/\os{\circ}{S}}(E))\}_{k\in {\mab Z}}$  
on $\wt{R}\eps^{\rm conv}_{X/\os{\circ}{S}*}
(\eps^{{\rm conv}*}_{X/\os{\circ}{S}}(E))$ gives 
the following spectral sequence
\begin{equation*}
E_1^{-k,q+k} 
= {\cal H}^q({\rm gr}_k^{P}\wt{R}\eps^{\rm conv}_{X/\os{\circ}{S}*}
(\eps^{{\rm conv}*}_{X/\os{\circ}{S}}(E))) 
\Lo \wt{R}^q\eps^{\rm conv}_{X/\os{\circ}{S}*}(E). 
\tag{8.7.2}\label{eqn:ttc}
\end{equation*}
By (\ref{ali:cpr}) and (\ref{ali:cp0r}) we have the following equality: 
\begin{align*} 
{} & 
{\cal H}^q({\rm gr}_k^{P}\wt{R}
\eps^{\rm conv}_{X/\os{\circ}{S}*}(\eps^{{\rm conv}*}_{X/\os{\circ}{S}}(E)))
\tag{8.7.3}\label{caseali:grhtkr}\\
{} & = 
\begin{cases}
{\cal H}^{q-k}
(a^{(k-1)}_{{\rm conv}*}(a^{(k-1)*}_{\rm conv}(E)\otimes_{\mab Z}\varpi^{(k-1)}_{\rm conv}
(\os{\circ}{X}/\os{\circ}{S})) & (k>0),  \\
{\cal H}^q({\rm Ker}(a^{(0)}_{{\rm conv}*}(a^{(0)*}_{\rm conv}(E)
\otimes_{\mab Z}\varpi^{(0)}_{\rm conv}
(\os{\circ}{X}/\os{\circ}{S}))
\lo a^{(1)}_{{\rm conv}*}(a^{(1)*}_{\rm conv}(E)\otimes_{\mab Z}\varpi^{(1)}_{\rm conv}
(\os{\circ}{X}/\os{\circ}{S})))))
 & (k=0).  
\end{cases} \\
\end{align*} 
Hence $E_1^{-k,q+k}=0$ for $q\not= k$. 
Now we obtain (\ref{caseali:grmtkr}) by the following equalities: 
\begin{align*} 
&\wt{R}^k\eps^{\rm conv}_{X/\os{\circ}{S}*}(\eps^{{\rm conv}*}_{X/\os{\circ}{S}}(E))
= {\cal H}^k(\wt{R}\eps^{\rm conv}_{X/\os{\circ}{S}*}(\eps^{{\rm conv}*}_{X/\os{\circ}{S}}(E))) 
=E_1^{-k,2k} \\
{} & = 
\begin{cases}
a^{(k-1)}_{{\rm conv}*}(a^{(k-1)*}_{\rm conv}(E)
\otimes_{\mab Z}\varpi^{(k-1)}_{\rm conv}
(\os{\circ}{X}/\os{\circ}{S})) & (k>0),  \\
{\rm Ker}(a^{(0)}_{{\rm conv}*}(a^{(0)*}_{\rm conv}(E)
\otimes_{\mab Z}\varpi^{(0)}_{\rm conv}
(\os{\circ}{X}/\os{\circ}{S}))
\lo 
a^{(1)}_{{\rm conv}*}(a^{(1)*}_{\rm conv}(E)
\otimes_{\mab Z}\varpi^{(1)}_{\rm conv}
(\os{\circ}{X}/\os{\circ}{S})))
 & (k=0).  
\end{cases} \\. 
\end{align*}
\end{proof}

\begin{theo}[{\bf Comparison theorem}]\label{theo:wtvsca}  
Let $\tau$ be the canonical filtration on a complex. 
There exists a canonical isomorphism 
\begin{equation*}
(\wt{R}\eps^{\rm conv}_{X/\os{\circ}{S}*}(\eps^{{\rm conv}*}_{X/\os{\circ}{S}}(E)),\tau)
\os{\sim}{\lo} (\wt{R}\eps^{\rm conv}_{X/\os{\circ}{S}*}(\eps^{{\rm conv}*}_{X/\os{\circ}{S}}(E)),P).  
\tag{8.8.1}\label{eqn:ltxpxd}
\end{equation*} 
\end{theo} 
\begin{proof}   
There exists the following natural morphism 
of filtered complexes of ${\cal K}_{\os{\circ}{X}_{\bul}/\os{\circ}{S}}$-modules:
\begin{equation*}
(L^{\rm UE}_{\os{\circ}{X}_{\bul}/\os{\circ}{S}}
({\cal E}\otimes_{{\cal O}_{{\cal P}^{\rm ex}_{\bul}}}
\Om^{\bul}_{{\cal P}^{\rm ex}_{\bul}/\os{\circ}{S}}),\tau) 
\lo (L^{\rm UE}_{\os{\circ}{X}_{\bul}/\os{\circ}{S}}
({\cal E}\otimes_{{\cal O}_{{\cal P}^{\rm ex}_{\bul}}}
\Om^{\bul}_{{\cal P}^{\rm ex}_{\bul}/\os{\circ}{S}}),P).
\tag{8.8.2}\label{eqn:ltpxd}
\end{equation*} 
By \cite[(2.7.2)]{nh2} there exists a canonical morphism
\begin{equation*}
(R{\pi}^{\rm conv}_{\os{\circ}{X}/\os{\circ}{S}*}
L^{\rm UE}_{\os{\circ}{X}_{\bul}/\os{\circ}{S}}
({\cal E}\otimes_{{\cal O}_{{\cal P}^{\rm ex}_{\bul}}}
\Om^{\bul}_{{\cal P}^{\rm ex}_{\bul}/\os{\circ}{S}}),\tau) 
\lo 
R{\pi}^{\rm conv}_{\os{\circ}{X}/\os{\circ}{S}*}
((L^{\rm UE}_{\os{\circ}{X}_{\bul}/\os{\circ}{S}}
({\cal E}\otimes_{{\cal O}_{{\cal P}^{\rm ex}_{\bul}}}
\Om^{\bul}_{{\cal P}^{\rm ex}_{\bul}/\os{\circ}{S}}),\tau)).
\tag{8.8.3}\label{eqn:rtloxdz}
\end{equation*} 
Hence we obtain the following morphism 
\begin{equation*}
(R{\pi}^{\rm conv}_{\os{\circ}{X}/\os{\circ}{S}*}
L^{\rm UE}_{\os{\circ}{X}_{\bul}/\os{\circ}{S}}
({\cal E}\otimes_{{\cal O}_{{\cal P}^{\rm ex}_{\bul}}}
\Om^{\bul}_{{\cal P}^{\rm ex}_{\bul}/\os{\circ}{S}}),\tau)
\lo (R{\pi}^{\rm conv}_{\os{\circ}{X}/\os{\circ}{S}*}
L^{\rm UE}_{\os{\circ}{X}_{\bul}/\os{\circ}{S}}
({\cal E}\otimes_{{\cal O}_{{\cal P}^{\rm ex}_{\bul}}}
\Om^{\bul}_{{\cal P}^{\rm ex}_{\bul}/\os{\circ}{S}}),P). 
\tag{8.8.4}
\end{equation*}
which is nothing but a morphism 
\begin{equation*}
(\wt{R}\eps^{\rm conv}_{X/\os{\circ}{S}*}(\eps^{{\rm conv}*}_{X/\os{\circ}{S}}(E)),\tau)
\lo (\wt{R}\eps^{\rm conv}_{X/\os{\circ}{S}*}(\eps^{{\rm conv}*}_{X/\os{\circ}{S}}(E)),P).  
\tag{8.8.5}\label{eqn:repxto}
\end{equation*}
To prove that this morphism  is a 
filtered isomorphism, 
it suffices to prove that the induced morphism
\begin{equation*}
{\rm gr}_k^{\tau}
\wt{R}\eps^{\rm conv}_{X/\os{\circ}{S}*}(\eps^{{\rm conv}*}_{X/\os{\circ}{S}}(E))
\lo {\rm gr}_k^{P}
\wt{R}\eps^{\rm conv}_{X/\os{\circ}{S}*}(\eps^{{\rm conv}*}_{X/\os{\circ}{S}}(E)) \quad 
(k\in {\mab Z})
\tag{8.8.6}\label{eqn:xosxse}
\end{equation*}
is an isomorphism. 
By (\ref{caseali:grmtkr}) we have the following equalities: 
\begin{align*} 
{} & {\cal H}^q
({\rm gr}_k^{\tau}\wt{R}\eps^{\rm conv}_{X/\os{\circ}{S}*}
(\eps^{{\rm conv}*}_{X/\os{\circ}{S}}(E))) 
\tag{8.8.7}\label{caseali:grtxkr}\\ 
{} & = 
\begin{cases}  
a^{(k-1)}_{{\rm conv}*}
(a^{(k-1)*}_{{\rm conv}*}(E)
\otimes_{\mab Z}\varpi^{(k-1)}_{\rm conv}
(\os{\circ}{X}/\os{\circ}{S})) & (q=k>0),  \\
{\rm Ker}(a^{(0)}_{{\rm conv}*}(a^{(0)*}_{\rm conv}(E)
\otimes_{\mab Z}\varpi^{(0)}_{\rm conv}
(\os{\circ}{X}/\os{\circ}{S}))
\lo a^{(1)}_{{\rm conv}*}(a^{(1)*}_{\rm conv}(E)
\otimes_{\mab Z}\varpi^{(1)}_{\rm conv}
(\os{\circ}{X}/\os{\circ}{S})))
 & (q=k=0), \\
 0 & ({\rm  otherwise}). 
\end{cases} \\
\end{align*} 
By (\ref{caseali:grmtkr}),  
${\cal H}^q({\rm gr}_k^{P}
\wt{R}\eps^{\rm conv}_{X/\os{\circ}{S}*}(\eps^{{\rm conv}*}_{X/\os{\circ}{S}}(E)))$ 
is also equal to the last formulas in (\ref{caseali:grtxkr}). 
Hence the morphism (\ref{eqn:ltxpxd}) is an isomorphism. 
\par
The rest we have to show is that 
the morphism (\ref{eqn:ltxpxd}) is independent of 
the choice of an affine open covering of $X$ and 
the simplicial immersion $X_{\bul}\os{\sus}{\lo} \ol{\cal P}_{\bul}$. 
Because it is a routine work, 
we leave the rest of the proof to the reader. 
\end{proof}

\begin{coro}[{\rm {\bf Contravariant functoriality of the modified $P$-filttered complex}}]
\label{coro:cwxstxs}
Let the notations be as in {\rm (\ref{prop:cwass})}. 
Let $X'$ be an SNCL scheme over $S'_1$. 
Let  
\begin{equation*}
\begin{CD} 
X' @>{g}>> X\\
@VVV @VVV \\
S'_1  @>{u_1}>> S_1
\end{CD}
\tag{8.9.1}\label{cd:mpsmlgs}
\end{equation*}
be a commutative diagram of log schemes.  
Let $E'$ be an isocrystal on 
$(\os{\circ}{X}{}'/\os{\circ}{S}{}')_{\rm conv}$ 
and let 
\begin{align*} 
\os{\circ}{g}{}^* \col \os{\circ}{g}{}^*_{\rm crys}(E')\lo E
\tag{8.9.2}\label{ali:teawbfm}
\end{align*} 
be a morphism of isocrystals on $(\os{\circ}{X}{}'/\os{\circ}{S}{}')_{\rm conv}$. 
Then there exists a morphism 
\begin{align*} 
g^*\col  (\wt{R}\eps^{\rm conv}_{X'/\os{\circ}{S}{}'*}
(\eps^{{\rm conv}*}_{X'/\os{\circ}{S}{}'}(E')),P) \lo 
R\os{\circ}{g}_{{\rm conv}*}
((\wt{R}\eps^{\rm conv}_{X/\os{\circ}{S}*}(\eps^{{\rm conv}*}_{X/\os{\circ}{S}}(E)),P)) 
\tag{8.9.3}\label{ali:teafm}
\end{align*} 
such that 
${\rm id}_X^*={\rm id}_{\wt{R}\eps_{X/\os{\circ}{S}*}(\eps^{{\rm conv}*}_{X/\os{\circ}{S}}(E))}$ and 
$(g\circ h)^*=R\os{\circ}{g}_{{\rm conv}*}(h^*)\circ g^*$ for analogous morphisms 
$h\col X''\lo X'$ and $h^* \col \os{\circ}{h}{}^*_{\rm conv}(E')\lo E''$ of 
${\cal K}_{\os{\circ}{X}{}'/\os{\circ}{S}{}'}$-modules. 
\end{coro}
\begin{proof} 
This follows from (\ref{prop:cwass}) and (\ref{theo:wtvsca}). 
\end{proof}


\par 
Let $\pi'$ be another non-zero element of the maximal ideal of ${\cal V}$. 
Assume that $\pi{\cal V} \subset \pi' {\cal V}$. 
Set  $S'_1:=\ul{\rm Spec}^{\log}_S({\cal O}_S/\pi')$. 
Then we have a natural closed immersion 
$S'_1 \os{\subset}{\lo} S_1$.  Set $X':=X\times_{S_1}S'_1$. 
Let $i \col X' \os{\subset}{\lo} X$ be the natural closed immersion. 
\par 
Because $i_{{\rm conv}*}$ is exact ((\ref{prop:toi}) (1)), 
we have the following functor:  
\begin{equation*} 
i_{{\rm conv}*} \col {\rm D}^+{\rm F}({\cal K}_{X'/S}) 
\lo 
{\rm D}^+{\rm F}({\cal K}_{X/S}). 
\tag{8.9.4}\label{eqn:dfkxs}
\end{equation*} 

\begin{prop}\label{prop:iekp}  
\begin{equation*} 
i_{{\rm conv}*}((\wt{R}\eps^{\rm conv}_{X'/\os{\circ}{S}*}
({\cal K}_{X'/\os{\circ}{S}}),\tau))
= 
(\wt{R}\eps^{\rm conv}_{X/\os{\circ}{S}*}({\cal K}_{X/\os{\circ}{S}}),\tau).  
\tag{8.10.1}\label{eqn:ikcp}
\end{equation*} 
\end{prop} 
\begin{proof}   
By (\ref{prop:toi}) (3)
we have the following equalities: 
\begin{align*}  
i_{{\rm conv}*}
({\cal K}_{{\os{\circ}{X}{}'^{(k)}}/\os{\circ}{S}})
& ={\cal K}_{{\os{\circ}{X}{}^{(k)}}/\os{\circ}{S}} \quad (k\in {\mab N}). 
\end{align*}  
Hence (\ref{prop:iekp}) follows from (\ref{ali:cp0r}) and (\ref{ali:cpr}).  
\end{proof}

\section{Convergent Steenbrink filtered complexes}\label{sec:wfcipp}
Let ${\cal V}$, $\pi$, $S$ and $S_1$ be as in the Introduction. 
Let the notations be as in the previous section. 
In this section we construct a filtered complex 
$(A_{\rm conv}(X/S,E),P)$ in ${\rm D}^+{\rm F}({\cal K}_{\os{\circ}{X}/\os{\circ}{S}})$ and 
we calculate ${\rm gr}_k^PA_{\rm conv}(X/S,E)$ ((\ref{theo:psap})). 
\par 
\par 
Let 
\begin{align*} 
\theta \col \wt{R}\eps^{\rm conv}_{X/\os{\circ}{S}*}(\eps^{{\rm conv}*}_{X/\os{\circ}{S}}(E))
\lo \wt{R}\eps^{\rm conv}_{X/\os{\circ}{S}*}(\eps^{{\rm conv}*}_{X/\os{\circ}{S}}(E))[1]
\tag{9.0.1}\label{ali:tre}
\end{align*} 
be a morphism of complexes induced 
by $f^*(d\log t) \wedge$, where $t$ is a local section of $M_S$ 
whose image in $\ol{M}_S$ is a local generator. 
Note that $\theta$ is indeed a morphism of complexes since $d\log t$ is a one-form. 
Set $H:=\wt{R}\eps^{\rm conv}_{X/\os{\circ}{S}*}(\eps^{{\rm conv}*}_{X/\os{\circ}{S}}(E))$. 
Strictly speaking, we have to take a flasque resolution $I^{\bul \bul}$ of 
$L^{\rm conv}_{\os{\circ}{X}_{\bul}/\os{\circ}{S}}
({\cal E}^{\bul}\otimes_{{\cal O}_{{\cal P}^{\rm ex}_{\bul}}}
\Om^{\bul}_{{\cal P}^{\rm ex}_{\bul}/\os{\circ}{S}})$ 
and we have to set $H:=\pi^{\rm conv}_{\os{\circ}{X}/\os{\circ}{S}*}(I^{\bul \bul})$. 
Instead of considering (\ref{ali:tre}), we consider the following morphism 
\begin{align*} 
\theta \col H \lo H\{1\}. 
\tag{9.0.2}\label{ali:trae}
\end{align*} 
Here $\{1\}$ means the $1$-shift on the left without the change of the complex. 
Note that $\theta$ in (\ref{ali:trae}) is not a morphism of complexes: 
``$d\theta=-\theta d$'' since $d\log t$ is a one-form. 
Consider the double complex 
\begin{equation*} 
\begin{CD} 
\cdots\\
@A{\theta}AA \\
H/\tau_j[1]\{j\}\\
@A{\theta}AA \\
\cdots \\
@A{\theta}AA \\
 H/\tau_1[1]\{1\}\\
@A{\theta}AA \\
H/\tau_0[1].
\end{CD}
\tag{9.0.3}\label{cd:v}
\end{equation*} 
Here we mean the differential morphism of the double complex 
by the morphism ``$d+\partial$'', where the relation between the horizontal differential morphism 
``$d$'' and the vertical differential morphism ``$\partial$'' 
is ``$d\partial +\partial d=0$''.  
Set 
\begin{align*} 
(A_{\rm conv}(X/S,E),P):=(s(H/\tau_0[1]\os{\theta}{\lo} H/\tau_1[1]\{1\}\os{\theta}{\lo} 
\cdots \os{\theta}{\lo} H/\tau_j[1]\{j\} \os{\theta}{\lo} \cdots),\{P_k\}_{k\in {\mab Z}}),
\end{align*}  
where 
\begin{align*}
P_kA_{\rm conv}(X/S,E):=s(\cdots \os{\theta}{\lo} \tau_{\max \{2j+k+1,j\}}H
/\tau_j[1]\{j\} \os{\theta}{\lo} \cdots) \quad
(k\in {\mab Z}, j\in {\mab N}).
\end{align*}  
Here we denote the double complex (\ref{cd:v}) as 
$$H/\tau_0[1]\os{\theta}{\lo} H/\tau_1[1]\{1\}\os{\theta}{\lo} 
\cdots \os{\theta}{\lo} H/\tau_j[1]\{j\} \os{\theta}{\lo} \cdots$$ 
for saving the space. 
Usually, 
$P_kA_{\rm conv}(X/S,E)$ is defined to be 
\begin{align*}
(\cdots \os{\theta}{\lo} \tau_{2j+k+1}H/\tau_j[j+1] \os{\theta}{\lo} \cdots). 
\end{align*}  
However this is mistaken since the inequality $2j+k+1\geq j$ does not necessarily hold.

The following is a key lemma for 
(\ref{theo:ifc}) below:  

\begin{lemm}\label{lemm:grkpd}
\begin{align*}
{\rm gr}^P_kA_{\rm conv}(X/S,E)
\os{\sim}{\lo} \bigoplus_{j\geq \max \{-k,0\}} 
&a^{(2j+k)}_{{\rm conv}*} (a^{(2j+k)*}_{{\rm conv}}(E)
\tag{9.1.1}\label{ali:ruoxvp}\\
&\otimes_{\mab Z}\vp_{\rm conv}^{(2j+k)}(\os{\circ}{X}/\os{\circ}{S}))[-2j-k]. 
\end{align*}
\end{lemm}
\begin{proof}  
We have the following equalities by (\ref{caseali:grmtkr}): 
\begin{align*} 
{\rm gr}^P_kA_{\rm conv}(X/S,E)&=
(\cdots \os{\theta}{\lo} {\rm gr}^P_{2j+k+1}H[1]\{j\} \os{\theta}{\lo} \cdots) \quad
(j\geq \max \{-k,0\})\\
&=(\cdots \os{0}{\lo} {\rm gr}^P_{2j+k+1}H[1]\{j\} \os{0}{\lo} \cdots)\quad
(j\geq \max \{-k,0\})\\
&=\bigoplus_{j\geq \max \{-k,0\}}{\rm gr}^P_{2j+k+1}H[1]\{j\}\{-j\}\\
&= \bigoplus_{j\geq \max \{-k,0\}}a^{(2j+k)}_{{\rm conv}*}(a^{(2j+k)*}_{{\rm conv}*}(E)
\otimes_{\mab Z}\vp_{\rm conv}^{(2j+k)}(\os{\circ}{X}/\os{\circ}{S}))[-2j-k]. 
\end{align*} 
\end{proof}

\begin{theo}\label{theo:ifc} 
There exists the following canonical isomorphism 
\begin{align*} 
\theta \col 
R\eps^{\rm conv}_{X/S*}(\eps^{{\rm conv}*}_{X/S}(E))\os{\sim}{\lo} 
A_{\rm conv}(X/S,E). 
\tag{9.2.1}\label{ali:pi}
\end{align*} 
\end{theo}
\begin{proof}  
We have the following equalities by (\ref{eqn:ltxpxd}): 
\begin{align*} 
A_{\rm conv}(X/S,E)&=s(H/P_0[1]\os{\theta}{\lo} H/P_1[1]\{1\}\os{\theta}{\lo} 
\cdots \os{\theta}{\lo} H/P_j[1]\{j\} \os{\theta}{\lo} \cdots)\\
&=R{\pi}^{\rm conv}_{\os{\circ}{X}/\os{\circ}{S}*}
(\cdots \os{\theta_{\bul}}{\lo} L^{\rm UE}_{\os{\circ}{X}_{\bul}/\os{\circ}{S}}
({\cal E}^{\bul}\otimes_{{\cal O}_{{\cal P}^{\rm ex}_{\bul}}}
\Om^{\bul}_{{\cal P}^{\rm ex}_{\bul}/\os{\circ}{S}})/P_j[1]\{j\}\os{\theta_{\bul}}{\lo}\cdots)\\
&=
R{\pi}^{\rm conv}_{\os{\circ}{X}/\os{\circ}{S}*}
(\cdots \os{\theta_{\bul}}{\lo} L^{\rm UE}_{\os{\circ}{X}_{\bul}/\os{\circ}{S}}
({\cal E}^{\bul}\otimes_{{\cal O}_{{\cal P}^{\rm ex}_{\bul}}}
\Om^{\bul}_{{\cal P}^{\rm ex}_{\bul}/\os{\circ}{S}})/P_j[1]\{j\}\os{\theta_{\bul}}{\lo}\cdots ). 
\end{align*} 
It suffices to prove that the morphism 
\begin{align*} 
\theta_{\bul} \col L^{\rm UE}_{\os{\circ}{X}_{\bul}/\os{\circ}{S}}
({\cal E}^{\bul}\otimes_{{\cal O}_{{\cal P}^{\rm ex}_{\bul}}}
\Om^{\bul}_{{\cal P}^{\rm ex}_{\bul}/S})\lo &
(L^{\rm UE}_{\os{\circ}{X}_{\bul}/\os{\circ}{S}}({\cal E}^{\bul}\otimes_{{\cal O}_{{\cal P}^{\rm ex}_{\bul}}}
\Om^{\bul}_{{\cal P}^{\rm ex}_{\bul}/\os{\circ}{S}})/P_0L^{\rm UE}_{\os{\circ}{X}_{\bul}/\os{\circ}{S}}[1]
\os{\theta_{\bul}}{\lo} \tag{9.2.2}\label{ali:tt}\\
&
L^{\rm UE}_{\os{\circ}{X}_{\bul}/\os{\circ}{S}}({\cal E}^{\bul}\otimes_{{\cal O}_{{\cal P}^{\rm ex}_{\bul}}}
\Om^{\bul}_{{\cal P}^{\rm ex}_{\bul}/\os{\circ}{S}})/P_1
L^{\rm UE}_{\os{\circ}{X}_{\bul}/\os{\circ}{S}}[1]\os{\theta_{\bul}}{\lo}  \cdots )
\end{align*} 
is an isomorphism. 
Here one must note that 
the following diagram 
\begin{equation*} 
\begin{CD} 
L^{\rm UE}_{\os{\circ}{X}_{\bul}/\os{\circ}{S}}
({\cal E}^{\bul}\otimes_{{\cal O}_{{\cal P}^{\rm ex}_{\bul}}}
\Om^i_{{\cal P}^{\rm ex}_{\bul}/S}) @>{\theta_{\bul}}>>
(L^{\rm UE}_{\os{\circ}{X}_{\bul}/\os{\circ}{S}}
({\cal E}^{\bul}\otimes_{{\cal O}_{{\cal P}^{\rm ex}_{\bul}}}
\Om^{i+1}_{{\cal P}^{\rm ex}_{\bul}/\os{\circ}{S}})
/P_0L^{\rm UE}_{\os{\circ}{X}_{\bul}/\os{\circ}{S}}\\
@V{\nabla^i}VV @VV{-\nabla^{i+1}}V \\
L^{\rm UE}_{\os{\circ}{X}_{\bul}/\os{\circ}{S}}
({\cal E}^{\bul}\otimes_{{\cal O}_{{\cal P}^{\rm ex}_{\bul}}}
\Om^{i+1}_{{\cal P}^{\rm ex}_{\bul}/S}) @>{\theta_{\bul}}>>
(L^{\rm UE}_{\os{\circ}{X}_{\bul}/\os{\circ}{S}}
({\cal E}^{\bul}\otimes_{{\cal O}_{{\cal P}^{\rm ex}_{\bul}}}
\Om^{i+2}_{{\cal P}^{\rm ex}_{\bul}/\os{\circ}{S}})
/P_0L^{\rm UE}_{\os{\circ}{X}_{\bul}/\os{\circ}{S}}
\end{CD} 
\end{equation*} 
is commutative in order to confirm that the morphism $\theta_{\bul}$ in (\ref{ali:tt}) 
is a morphism of complexes. 
Because $\os{\circ}{\ol{\cal P}}_0$ is affine, 
the morphism $\os{\circ}{\ol{\cal P}}_{\bul}\lo \os{\circ}{\ol{S}}$ is affine. 
Hence the morphism $\os{\circ}{\cal P}_{\bul}\lo \os{\circ}{S}$ is also affine.   
By (\ref{prop:afex}) and (\ref{theo:injf}) (2),   
$L^{\rm UE}_{\os{\circ}{X}_{\bul}/\os{\circ}{S}}({\cal E}^{\bul}\otimes_{{\cal O}_{{\cal P}^{\rm ex}_{\bul}}}
\Om^{\bul}_{{\cal P}^{\rm ex}_{\bul}/\os{\circ}{S}})/P_jL^{\rm UE}_{\os{\circ}{X}_{\bul}/\os{\circ}{S}}$ 
is isomorphic to 
$L^{\rm UE}_{\os{\circ}{X}_{\bul}/\os{\circ}{S}}({\cal E}^{\bul}\otimes_{{\cal O}_{{\cal P}^{\rm ex}_{\bul}}}
\Om^{\bul}_{{\cal P}^{\rm ex}_{\bul}/\os{\circ}{S}}/P_j)$. 
\par 
Now we may assume that there exists an exact immersion 
$X\os{\sus}{\lo} \ol{\cal P}{}^{\rm ex}$ into an affine log scheme 
over $S_1\os{\sus}{\lo} \ol{S}$.  
Let the notations be as in the proofs of 
(\ref{theo:injf}) and (\ref{prop:0case}). 
Then we have only to prove that
\begin{align*} 
\theta \col p_n^*({\cal E}_n)\otimes_{q'{}^*_{\!n}({\cal O}_{{\cal P}^{\rm ex}_T})}
q'{}^*_{\!n}(\Om^{\bul}_{{\cal P}^{\rm ex}_T/T'})\lo &
[(p_n^*({\cal E}_n)\otimes_{q'{}^*_{\!n}({\cal O}_{{\cal P}^{\rm ex}_T})}
q'{}^*_{\!n}(\Om^{\bul}_{{\cal P}^{\rm ex}_T/T})/P_0[1]\os{\theta}{\lo} \\
&
(p_n^*({\cal E}_n)\otimes_{q'{}^*_{\!n}({\cal O}_{{\cal P}^{\rm ex}_T})}
q'{}^*_{\!n}(\Om^{\bul}_{{\cal P}^{\rm ex}_T/T})/P_1[1]\os{\theta}{\lo}  \cdots ].
\end{align*} 
is exact. 
Since ${\cal E}^{\bul}_n$ is a flat ${\cal K}_{T_n}$-module and $q'_n$ is exact ((\ref{lemm:rex})), 
it suffices to prove that the morphism 
\begin{align*} 
\theta \col \Om^{\bul}_{{\cal P}^{\rm ex}_{T'}/T'}\lo 
[\Om^{\bul}_{{\cal P}^{\rm ex}_T/T}/P_0[1]\os{\theta}{\lo} 
\Om^{\bul}_{{\cal P}^{\rm ex}_T/T}/P_1[1]\os{\theta}{\lo}  \cdots ]
\end{align*} 
is a quasi-isomorphism. 
This is well-known. 
(See \cite[1.3.5]{msemi} 
and \cite[(1.3.21), (1.4.3)]{nb} for the proof of this fact.)
\end{proof}

\begin{coro}\label{coro:ifc} 
There exists the following spectral sequences$:$ 
\begin{align*}
E_1^{-k,q+k}&=
\bigoplus_{j\geq \max \{-k,0\}} 
R^{q-2j-k}\os{\circ}{f}{}^{(2j+k)}_{{\rm conv}*}(a^{(2j+k)*}_{\rm conv}(E)
\otimes_{\mab Z}\vp_{\rm conv}^{(2j+k)}(\os{\circ}{X}/\os{\circ}{S}))
\tag{9.3.1}\label{ali:cccw}\\
&\Lo R^q\os{\circ}{f}_{{\rm conv}*}(R\eps^{\rm conv}_{X/S*}(\eps^{{\rm conv}*}_{X/S}(E))), 
\end{align*} 
\begin{align*}
E_1^{-k,q+k}&=
\bigoplus_{j\geq \max \{-k,0\}} 
R^{q-2j-k}\os{\circ}{f}_{\os{\circ}{X}{}^{(2j+k)}/\os{\circ}{S}*}(a^{(2j+k)*}_{\rm conv}(E)
\otimes_{\mab Z}\vp_{\rm conv}^{(2j+k)}(\os{\circ}{X}/\os{\circ}{S})))
\tag{9.3.2}\label{ali:ccwz}\\
&\Lo R^qf_{X/S*}(\eps^{{\rm conv}*}_{X/S}(E)). 
\end{align*} 
\end{coro}

\begin{defi}\label{defi:ccx}  
Set $(A_{\rm iso{\textrm -}zar}(X/S,E),P):=Ru_{\os{\circ}{X}/\os{\circ}{S}*}((A_{\rm conv}(X/S,E),P))$.  
We call $(A_{\rm conv}(X/S,E),P)$ and $(A_{\rm iso{\textrm -}zar}(X/S,E),P)$
the {\it weight-filtered convergent complex} of $E$ and 
the {\it weight-filtered isozariskian complex} of $E$, respectively. 
When $E= {\cal K}_{\os{\circ}{X}/\os{\circ}{S}}$ and $\os{\circ}{X}/\os{\circ}{S}$ is proper, 
we call $(A_{\rm conv}(X/S,E),P)$ and $(A_{\rm iso{\textrm -}zar}(X/S,E),P)$ 
the {\it weight-filtered convergent complex} of $X/S$ and 
the {\it weight-filtered convergent isozariskian complex} of $X/S$, respectively. 
When $E= {\cal K}_{\os{\circ}{X}/\os{\circ}{S}}$, 
we call the spectral sequences (\ref{ali:cccw}) and 
(\ref{ali:ccwz})
the {\it convergent weight spectral sequence} of $X/S$ and 
the {\it  weight spectral sequence} of $X/S$, respectively. 
\end{defi}

By (\ref{ali:pi}) we can denote the filtered complexes  
$(A_{\rm conv}(X/S,E),P)$ and $(A_{\rm iso{\textrm -}zar}(X/S,E),P)$ 
by 
\begin{align*} 
(R\eps^{\rm conv}_{X/S*}(\eps^{{\rm conv}*}_{X/S}(E)),P)
\tag{9.4.1}\label{ali:ceswz}
\end{align*} 
and 
\begin{align*} 
(Ru^{\rm conv}_{X/S*}(\eps^{{\rm conv}*}_{X/S}(E)),P),
\tag{9.4.2}\label{ali:ccwez}
\end{align*} 
respectively.  
We can also denote the filtered complexes  
$R\os{\circ}{f}_{{\rm conv}*}((A_{\rm conv}(X/S,E),P))$ 
and $Rf_{X/S*}((A_{\rm iso{\textrm -}zar}(X/S,E),P))$ 
by 
\begin{align*}  
(Rf_{{\rm conv}*}(\eps^{{\rm conv}*}_{X/S}(E)),P)
\tag{9.4.3}\label{ali:cefswz}
\end{align*} 
and 
\begin{align*} 
(Rf_{X/S*}(\eps^{{\rm conv}*}_{X/S}(E)),P),
\tag{9.4.4}\label{ali:ccwfez}
\end{align*} 
respectively.


\begin{rema}\label{rema:l}
Let $l$ be a prime to $p$. Let $X_{\rm et}$ and $\os{\circ}{X}_{\rm et}$ 
be the Kummer log \'{e}tale topos of $X$ and the \'{e}tale topos of $\os{\circ}{X}$, 
respectively. Let $\eps_{X/S_1}\col X\lo \os{\circ}{X}$ be the morphism 
forgetting the log structure of $X$. 
Assume that $\os{\circ}{S}_1$ is the spectrum of a field $\kap$ of 
characteristic $p\geq 0$. 
In this case 
$R\eps^{\rm conv}_{X/S*}({\cal K}_{X/S})$ is a $p$-adic analogue of 
$R\eps_{X/S_1,{\rm et}}({\mab Z}_l)
\otimes^L_{{\mab Z}_l}{\mab Q}_l$ in \cite[p.~173]{nd}.  
\end{rema} 

\section{Convergent Rapoport-Zink-Steenbrink-Zucker filtered complexes}\label{sec:wt} 
The results in this section are complements of those in the previous section. 
In this section we construct another weight-filtered convergent complex
in a manner which is analogous to \cite{rz} and \cite{stz}, 
which we call a convergent Rapoport-Zink-Steenbrink-Zucker filtered complex. 
See also \cite{nd} and \cite{fn}. 
\par 
First let us recall the Hirsch extension (cf.~\cite{nhi}). 
\par 
Let the notations be as in the previous section. 
Let $\wt{t}$ be a local section of $M_S$ 
whose image $t$ in $\ol{M}_S:=M_S/{\cal O}_S^*$ is the local generator. 
Let $U_S$ be a free ${\cal O}_S$-module of rank $1$ with a basis $u$: 
$U_S={\cal O}_Su$. 
Though we do not ask what $u$ is, we consider $u$ as 
``$\log t$'' (not ``$\log \wt{t}$'') in our mind. 
For a local section $\wt{t}{}'=a\wt{t}$ $(a\in {\cal O}_S^*)$ of $M_S$, 
we do not change $U_S$.  
Let $\varphi \col U_S\lo \Om^1_{{\cal P}^{\rm ex}_{\bul}/\os{\circ}{S}}$ be the following composite morphism 
\begin{align*}
f^{-1}(U_S)\lo f^{-1}(\Om^1_{S/\os{\circ}{S}})
\lo
{\rm Ker}(\Om^1_{{\cal P}^{\rm ex}_{\bul}/\os{\circ}{S}}\lo \Om^2_{{\cal P}^{\rm ex}_{\bul}/\os{\circ}{S}})
\end{align*}
of $f^{-1}({\cal O}_S)$-modules, where 
the first morphism 
is defined by $u\lom d\log t:=d\log \wt{t}$. 
Let $\Gam_{{\cal O}_S}(U_S)$ be a locally free PD-algebra over ${\cal O}_S$ generated by $U_S$. 
Consider the Hirsch extension 
$\Gam_{{\cal O}_S}(U_S)\otimes_{{\cal O}_S}{\cal E}^{\bul}\otimes_{{\cal O}_{{\cal P}^{\rm ex}_{\bul}}}
\Om^{\bul}_{{\cal P}^{\rm ex}_{\bul}/\os{\circ}{S}}$
of the complex 
${\cal E}^{\bul}\otimes_{{\cal O}_{{\cal P}^{\rm ex}_{\bul}}}
\Om^{\bul}_{{\cal P}^{\rm ex}_{\bul}/\os{\circ}{S}}$ by $\varphi$; 
the differential is defined by the following: 
$$d(u^{[n]}\om)=u^{[n-1]}d\log t\wedge \nabla(\om)+u^{[n]}\nabla(\om) 
\quad (\om \in {\cal E}^{\bul}\otimes_{{\cal O}_{{\cal P}^{\rm ex}_{\bul}}}
\Om^{\bul}_{{\cal P}^{\rm ex}_{\bul}/\os{\circ}{S}},n\in {\mab N}).$$ 
We denote this complex by 
${\cal E}^{\bul}\otimes_{{\cal O}_{{\cal P}^{\rm ex}_{\bul}}}
\Om^{\bul}_{{\cal P}^{\rm ex}_{\bul}/\os{\circ}{S}}[U_S]$.
 (We do not need to mind what $U_S$ is in this paper. If you want to know the definition of $U_S$, 
 see \cite{nhi} for the definition of $U_S$.)
Set 
\begin{align*} 
N_{\bul}:=d/du \col {\cal E}^{\bul}\otimes_{{\cal O}_{{\cal P}^{\rm ex}_{\bul}}}
\Om^{\bul}_{{\cal P}^{\rm ex}_{\bul}/\os{\circ}{S}}[U_S]\lo 
{\cal E}^{\bul}\otimes_{{\cal O}_{{\cal P}^{\rm ex}_{\bul}}}
\Om^{\bul}_{{\cal P}^{\rm ex}_{\bul}/\os{\circ}{S}}[U_S] \quad (d/du(u^{[n]}):=u^{[n-1]}).
\end{align*} 
Since $N_{\bul}$ is surjective, we have the following exact sequence 
\begin{align*} 
0\lo 
{\cal E}^{\bul}\otimes_{{\cal O}_{{\cal P}^{\rm ex}_{\bul}}}
\Om^i_{{\cal P}^{\rm ex}_{\bul}/\os{\circ}{S}}\lo 
{\cal E}^{\bul}\otimes_{{\cal O}_{{\cal P}^{\rm ex}_{\bul}}}
\Om^i_{{\cal P}^{\rm ex}_{\bul}/\os{\circ}{S}}[U_S]\os{N_{\bul}}{\lo} 
{\cal E}^{\bul}\otimes_{{\cal O}_{{\cal P}^{\rm ex}_{\bul}}}
\Om^i_{{\cal P}^{\rm ex}_{\bul}/\os{\circ}{S}}[U_S]\lo 0 \quad (i\in {\mab N}).
\end{align*} 
Because this sequence is locally split, 
the following sequence is exact:  
\begin{align*} 
0&\lo 
L^{\rm UE}_{\os{\circ}{X}_{\bul}/\os{\circ}{S}}({\cal E}^{\bul}\otimes_{{\cal O}_{{\cal P}^{\rm ex}_{\bul}}}
\Om^{\bul}_{{\cal P}^{\rm ex}_{\bul}/\os{\circ}{S}})
\lo 
L^{\rm UE}_{\os{\circ}{X}_{\bul}/\os{\circ}{S}}({\cal E}^{\bul}\otimes_{{\cal O}_{{\cal P}^{\rm ex}_{\bul}}}
\Om^{\bul}_{{\cal P}^{\rm ex}_{\bul}/\os{\circ}{S}}[U_S])\\
&\os{N_{\bul}}{\lo} 
L^{\rm UE}_{\os{\circ}{X}_{\bul}/\os{\circ}{S}}({\cal E}^{\bul}\otimes_{{\cal O}_{{\cal P}^{\rm ex}_{\bul}}}
\Om^{\bul}_{{\cal P}^{\rm ex}_{\bul}/\os{\circ}{S}}[U_S])\lo 0.
\end{align*}
Hence we have the following canonical isomorphism 
\begin{align*} 
\iota \col L^{\rm UE}_{\os{\circ}{X}_{\bul}
/\os{\circ}{S}}({\cal E}^{\bul}\otimes_{{\cal O}_{{\cal P}^{\rm ex}_{\bul}}}
\Om^{\bul}_{{\cal P}^{\rm ex}_{\bul}/\os{\circ}{S}})
&\os{\sim}{\lo} 
{\rm MF}(N_{\bul})\tag{10.0.1}\label{ali:lxsie}\\
&=L^{\rm UE}_{\os{\circ}{X}_{\bul}/\os{\circ}{S}}({\cal E}^{\bul}\otimes_{{\cal O}_{{\cal P}^{\rm ex}_{\bul}}}
\Om^{\bul}_{{\cal P}^{\rm ex}_{\bul}/\os{\circ}{S}}[U_S])
\oplus 
L^{\rm UE}_{\os{\circ}{X}_{\bul}/\os{\circ}{S}}({\cal E}^{\bul}\otimes_{{\cal O}_{{\cal P}^{\rm ex}_{\bul}}}
\Om^{\bul}_{{\cal P}^{\rm ex}_{\bul}/\os{\circ}{S}}[U_S])[-1]. 
\end{align*}
Consider the following morphism 
\begin{align*} 
\psi \col {\rm MF}(N_{\bul})
\lo L^{\rm UE}_{\os{\circ}{X}_{\bul}/\os{\circ}{S}}
({\cal E}^{\bul}\otimes_{{\cal O}_{{\cal P}^{\rm ex}_{\bul}}}
\Om^{\bul}_{{\cal P}^{\rm ex}_{\bul}/\os{\circ}{S}})
\tag{10.0.2}\label{ali:lxse}
\end{align*}
defined by the following 
$$(\sum_{n=0}u^n\om_n,\sum_{n=0}u^n\eta_n)
\lom \om_0+d\log t\wedge \eta_0.$$
As in \cite[(3.1)]{np} we can check that this morphism is a morphism of complexes.

\begin{prop}\label{prop:psii}
The morphism $\psi$ is a quasi-isomorphism. 
\end{prop} 
\begin{proof} 
Since $\iota$ is a quasi-isomorphism and 
$\psi \circ \iota$ is the identity, $\psi$ is a quasi-isomorphism. 
\end{proof} 

\begin{coro} 
The morphism $\psi$ induces the following quasi-isomorphism
\begin{align*} 
\psi \col ({\rm MF}(N_{\bul}),\tau)
\lo (L^{\rm UE}_{\os{\circ}{X}/\os{\circ}{S}}({\cal E}^{\bul}\otimes_{{\cal O}_{{\cal P}_{\bul}}}
\Om^{\bul}_{{\cal P}_{\bul}/\os{\circ}{S}}),\tau), 
\tag{10.2.1}\label{ali:lxqse}
\end{align*}
which we denote by $\psi$ again. 
\end{coro}

\par 
Now consider the morphism 
\begin{align*} 
\theta \col {\rm MF}(N_{\bul})\lo {\rm MF}(N_{\bul})[1]
\end{align*} 
defined by 
``$(x,y)\lom (0,x)$''. 
Set 
\begin{align*} 
{\rm MF}(N):=R\pi^{\rm coinv}_{\os{\circ}{X}/\os{\circ}{S}*}({\rm MF}(N_{\bul}))
\end{align*} 
and 
\begin{align*} 
A^N_{\rm conv}(X/S,E):=s({\rm MF}(N)/\tau_0[1]\os{\theta}{\lo}
{\rm MF}(N)/\tau_1[1]\{1\}\os{\theta}{\lo}\cdots)
\end{align*} 
and 
\begin{align*} 
(A^N_{\rm conv}(X/S,E),P):=(A^N_{\rm conv}(X/S,E),
\{(\cdots \os{\theta}{\lo}
\tau_{\max \{2j+k+1,j\}}{\rm MF}(N)/\tau_j[1]\{j\}\os{\theta}{\lo}
\cdots\}_{k\in {\mab Z}}). 
\end{align*}

\begin{prop}\label{prop:rza}
The morphism $\psi$ induces the following canonical filtered isomorphism 
\begin{align*}
(A^N_{\rm conv}(X/S,E),P)\os{\sim}{\lo}
(A_{\rm conv}(X/S,E),P) , 
\tag{10.3.1}\label{ali:aisom}
\end{align*}
which we denote by $\psi$ again. 
\end{prop} 
\begin{proof} 
By the definition of $\psi$, It is obvious that 
the following diagram 
\begin{equation*} 
\begin{CD} 
{\rm MF}(N_{\bul}) @>{\psi}>> 
L^{\rm UE}_{\os{\circ}{X}/\os{\circ}{S}}
({\cal E}^{\bul}\otimes_{{\cal O}_{{\cal P}_{\bul}}}
\Om^{\bul}_{{\cal P}_{\bul}/\os{\circ}{S}})\\
@A{\theta}AA @AA{\theta}A \\
{\rm MF}(N_{\bul}) @>{\psi}>> 
L^{\rm UE}_{\os{\circ}{X}/\os{\circ}{S}}
({\cal E}^{\bul}\otimes_{{\cal O}_{{\cal P}_{\bul}}}
\Om^{\bul}_{{\cal P}_{\bul}/\os{\circ}{S}})
\end{CD}
\end{equation*} 
is commutative.  
Because $\psi$ is an isomorphism, the morphism (\ref{ali:aisom}) is an isomorphism. 
The rest we have to prove is that the morphism is independent of the choices of 
the affine open covering of $X_{\bul}$ and 
the embedding system $X_{\bul}\os{\sus}{\lo} \ol{\cal P}_{\bul}$ over $\ol{S}$. 
Because this is a routine work, we leave the detail of the proof to the reader. 
\end{proof}

\begin{defi}\label{defi:rzm}
We call $(A^N_{\rm conv}(X/S,E),P)$ the 
{\it convergent Rapoport-Zink-Steenbrink-Zucker filtered complex} 
of $E$. When $E={\cal K}_{\os{\circ}{X}/\os{\circ}{S}}$, 
we call $(A^N_{\rm conv}(X/S,E),P)$ the 
{\it convergent Rapoport-Zink-Steenbrink-Zucker filtered complex} of $X/S$. 
\end{defi} 

\begin{rema}
(1) Let the notations be as in (\ref{rema:l}). 
Set $X_{l^{-n}}:=X\times_{{\rm Spec}^{\log}({\mab Z}[{\mab N}])}
{\rm Spec}^{\log}({\mab Z}[l^{-n}{\mab N}])$, where 
the log structure of ${\rm Spec}^{\log}({\mab Z}[l^{-n}{\mab N}])$ 
is associated to the inclusion 
$l^{-n}{\mab N}\os{\subset}{\lo} {\mab Z}[l^{-n}{\mab N}]$. 
As in \cite[p.~723]{nd}, set $X_{{\rm t}l}:=\vpl_nX_{l^{-n}}$. 
Let $\pi \col X_{{\rm t}l}:\lo X_{\rm et}$ be the natural morphism. 
Then 
$R\pi^{\rm coinv}_{\os{\circ}{X}/\os{\circ}{S}*}
(L^{\rm UE}_{\os{\circ}{X}_{\bul}/\os{\circ}{S}}
({\cal E}^{\bul}\otimes_{{\cal O}_{{\cal P}^{\rm ex}_{\bul}}}
\Om^{\bul}_{{\cal P}^{\rm ex}_{\bul}/\os{\circ}{S}}[U_S]))$ is 
a $p$-adic analogue of 
$R(\eps_{X/S_1,{\rm et}}\pi)_*({\mab Z}_l)\otimes_{{\mab Z}_l}{\mab Q}_l$. 
Consequently 
$(A^N_{\rm conv}(X/S,E),P)$ is a $p$-adic analogue of the 
$(s(A^{\bul \bul}),W)\in {\rm D}^+{\rm F}(({\mab Q}_l)_{\os{\circ}{X}_{\rm et}})$, 
though the filtered derived category 
${\rm D}^+{\rm F}(({\mab Q}_l)_{\os{\circ}{X}_{\rm et}})$ 
does not appear in [loc.~cit.] nor \cite{rz}. 
Here $A^{\bul \bul}$ is the double complex in [loc.~cit] with the shift of the complex 
as in \S\ref{sec:wfcipp}. 
(It is easy to prove that $(s(A^{\bul \bul}),W)$ is indeed independent of the choice of 
the injective resolution $({\mab Z}/l^n)_{\os{\circ}{X}_{\rm et}}\lo I^{\bul}$ 
$(n\in {\mab N})$ by using the $l$-adic analogue of 
(\ref{lemm:grkpd}) and (\ref{theo:ifc}).)  
\end{rema}


\section{The functoriality of weight-filtered convergent and isozariskian complexes}\label{sec:fpw}
Let the notations be as in (\ref{coro:cwxstxs}). 
In this section we prove the contravariant functoriality of  
$(A_{\rm conv}(X/S,E),P)$,  $(A^N_{\rm conv}(X/S,E),P)$ 
and $(A_{\rm iso{\textrm -}zar}(X/S,E),P)$ by using (\ref{theo:wtvsca}). 
This theme has not been discussed in a lot of references except 
\cite{ndw}, \cite{nb} and \cite{nhi} (cf.~\cite{tsd}, \cite{mat}); 
even the pull-back of the abrelative
and the absolute Frobenius on $p$-adic Steenbrink complexes 
has not been discussed except the references above; 
the action (\ref{cd:fdegcbc}) below dividing the ``degree mapping'' 
has not been considered in a lot of references.  
Consequently the yoga of arithmetic weights cannot be used except them.

\begin{theo}[{\bf Contravariant Functoriality}]\label{theo:fcpwczo}
Let the notations be as in {\rm (\ref{coro:cwxstxs})}. 
Then there exists  a canonical morphism 
\begin{equation*}
g^*_{{\rm conv}}\col (A_{\rm conv}(X/S,E),P) \lo 
R\os{\circ}{g}_{{\rm conv}*}((A_{\rm conv}(X'/S',E'),P)) 
\tag{11.1.1}\label{eqn:fdfcc}
\end{equation*} 
fitting into the following commutative diagram 
\begin{equation*}
\begin{CD} 
A_{\rm conv}(X/S,E)@>{g^*_{{\rm conv}}}>>
R\os{\circ}{g}_{{\rm conv}*}(A_{\rm conv}(X'/S',E')) \\
@A{\theta}A{\simeq}A @A{\simeq}A{R\os{\circ}{g}_{{\rm conv}*}(\theta')}A\\
R\eps^{\rm conv}_{X/S*}(\eps^{{\rm conv}*}_{X/S}(E))@>{g^*_{{\rm conv}}}>>
R\os{\circ}{g}_{{\rm conv}*}R\eps^{\rm conv}_{X/S*}(\eps^{{\rm conv}*}_{X/S}(E')). 
\end{CD} 
\tag{11.1.2}\label{cd:fdxfcc}
\end{equation*} 
The morphisms $g^*_{{\rm conv}}$'s satisfy the transitive law and 
$(({\rm id}_{X/S})_{\rm conv})^*={\rm id}_{(A_{\rm conv}(X/S,E),P)}$. 
\end{theo}
\begin{proof} 
Set $H:=\wt{R}\eps^{\rm conv}_{X/\os{\circ}{S}*}(\eps^{{\rm conv}*}_{X/\os{\circ}{S}}(E))$ 
and $H':=\wt{R}\eps^{\rm conv}_{X'/\os{\circ}{S}{}'*}(\eps^{{\rm conv}*}_{X'/\os{\circ}{S}{}'}(E'))$.  
Let $\theta' \col H'\lo H'[1]$ be an analogous morphism to $\theta$. 
By (\ref{coro:cwxstxs}) and \cite[(2.7.2)]{nh2} there exists the following composite morphism 
\begin{align*} 
g^*\col  (H,\tau) \lo (R\os{\circ}{g}_{{\rm conv}*}(H'),\tau)\lo 
R\os{\circ}{g}_{{\rm conv}*}((H',\tau)).  
\end{align*} 
Recall the morphism $u\col S'\lo S$ in (\ref{cd:yypssp}). 
Because $u^* \circ \theta=(\deg u) \theta'$ by the definition of ${\rm deg}(u)$ 
(\cite[(1.1.42)]{nb}), 
we have the following commutative diagram: 
\begin{equation*} 
\begin{CD} 
H[1]\{j\}/\tau_{j+1} @>{\deg (u)^{-1}g^*}>> R\os{\circ}{g}_{{\rm conv}*}(H'[1]\{j\}/\tau_{j+1})\\
@A{\theta}AA @AA{R\os{\circ}{g}_{{\rm conv}*}(\theta')}A \\
H[1]\{j-1\}/\tau_j @>{g^*}>> R\os{\circ}{g}_{{\rm conv}*}(H'[1]\{j-1\}/\tau_j)
\end{CD} 
\end{equation*} 
for $j\in {\mab N}$. 
Hence the morphisms 
\begin{align*} 
g_j:=(\deg u)^{-(j+1)}g^* \col H[1]\{j-1\}/\tau_j \lo 
R\os{\circ}{g}_{{\rm conv}*}(H'[1]\{j-1\}/\tau_j)
\tag{11.1.3}\label{cd:fdegcbc}
\end{align*} 
for $j$'s induce 
the following morphism 
\begin{align*} 
g^*_{{\rm conv}}\col A_{\rm conv}(X/S,E) \lo 
R\os{\circ}{g}_{{\rm conv}*}(A_{\rm conv}(X'/S',E')). 
\end{align*} 
In fact, this morphism induces the morphism (\ref{eqn:fdfcc}). 
The commutativity of the diagram (\ref{cd:fdxfcc}) 
follows from that of the following commutative diagram 
\begin{equation*}
\begin{CD} 
H/\tau_0[1]@>{g^*_{{\rm conv}}}>>
R\os{\circ}{g}_{{\rm conv}*}(H/\tau_0[1]) \\
@A{\theta'}A{\simeq}A @A{\simeq}A{R\os{\circ}{g}_{{\rm conv}*}(\theta)}A\\
R\eps^{\rm conv}_{X/S*}(\eps^{{\rm conv}*}_{X/S}(E))@>{g^*_{{\rm conv}}}>>
R\os{\circ}{g}_{{\rm conv}*}R\eps^{\rm conv}_{X'/S'*}(\eps^{{\rm conv}*}_{X'/S'}(E)), 
\end{CD} 
\tag{11.1.4}\label{cd:fdxfcbc}
\end{equation*} 
which can be shown by the local calculation as usual. 
\end{proof}

\begin{coro}\label{coro:spfcz} 
Let $E^{\rm conv}_{\rm ss}(X/S,E)$ and $E_{\rm ss}(X/S,E)$
be the  spectral sequences {\rm (\ref{ali:cccw})} 
and {\rm (\ref{ali:ccwz})} for $X/S$ and $E$, respectively.
Let $E^{\rm conv}_{\rm ss}(X'/S',E')$ and $E_{\rm ss}(X'/S',E')$
be the  spectral sequences {\rm (\ref{ali:cccw})} 
and {\rm (\ref{ali:ccwz})} for $X'/S'$ and $E'$. 
Then the morphism $g^*_{\rm conv}$
induces a morphism 
\begin{equation*}
g^*_{\rm conv} \col 
u^{-1}E^{\rm conv}_{\rm ss}(X'/S',E') \lo E^{\rm conv}_{\rm ss}(X/S,E) 
\tag{11.2.1}
\end{equation*} 
and 
\begin{equation*}
g^*_{\rm conv} \col 
u^{-1}E_{\rm ss}(X'/S',E') \lo E_{\rm ss}(X/S,E) 
\tag{11.2.2}
\end{equation*} 
of spectral sequences. 
\end{coro}
\begin{proof}
The proof is straightforward.
\end{proof}

\par 
Let 
$a'{}^{(k)} \col \os{\circ}{X}{}'{}^{(k)}\lo \os{\circ}{X}{}'$ 
be the natural morphism.  
By  (\ref{theo:fcpwczo}) the morphism 
$g^*_{\rm conv}$ 
induces the following morphism 
\begin{align*}
{\rm gr}^{P}_k(g^*_{\rm conv}) 
\col  a^{(k)}_{{\rm conv}*}
(a^{(k)*}_{\rm conv}(E)
\otimes_{\mab Z}\vp^{(k)}_{\rm conv}(\os{\circ}{X}{}/\os{\circ}{S}))[-k] 
\tag{11.2.3}\label{ali:gdkcl}\\ 
\lo 
R\os{\circ}{g}_{{\rm conv}*}a'{}^{(k)}_{{\rm conv}*}
(a'{}^{(k)*}_{\rm conv}(E)
\otimes_{\mab Z}
\vp^{(k)}_{\rm conv}(\os{\circ}{X}/\os{\circ}{S}))[-k]. 
\end{align*} 
Assume that  $g$ induces a morphism 
$g_{\os{\circ}{X}{}^{(k)}} \col \os{\circ}{X}{}^{(k)} \lo 
\os{\circ}{X}{}'{}^{(k)}$ for any $k\in {\mab N}$. 
In this case we make the morphism 
${\rm gr}^{P}_k(g^*_{\rm conv})$ in 
(\ref{ali:gdkcl}) explicit under the following assumptions.
Assume that the following two conditions hold:
\bigskip
\parno  
(11.2.4): there exists the same cardinality of 
smooth components of $\os{\circ}{X}$ and $\os{\circ}{X}{}'$ over 
$S_1$ and $S'_1$, respectively:
$\os{\circ}{X}=\bigcup_{\lam \in \Lam}\os{\circ}{X}_{\lam}$, 
$\os{\circ}{X}{}'=\bigcup_{\lam \in \Lam}\os{\circ}{X}{}'_{\lam}$, 
where $\os{\circ}{X}_{\lam}$ and $\os{\circ}{X}{}'_{\lam}$ are 
smooth components of 
$\os{\circ}{X}/S_1$ and $\os{\circ}{X}{}'/S'_1$, respectively. 
\medskip
\parno 
(11.2.5): 
there exist positive integers $e_{\lam}$ and local generators 
$x_{\lam}$ and $x'_{\lam}$ of $M_X$ and $M_{X'}$ 
defining the equations of $\os{\circ}{X}_{\lam}$ and $\os{\circ}{X}{}'_{\lam}$  
$(\lam \in \Lam)$ such that $g^*(x_{\lam})=(x'_{\lam})^{e_{\lam}}$ in $\ol{M}_X=M_X/{\cal O}_X^*$.  
\bigskip 
\par
Set $\ul{\lam}:=\{\lam_1, \ldots, \lam_k\}$
$(\lam_j \in \Lam,~(\lam_i \not= \lam_j~(i\not= j)))$ 
and 
$\os{\circ}{X}_{\ul{\lam}}:=\os{\circ}{X}_{\lam_1}\cap \cdots \cap \os{\circ}{X}_{\lam_k}$, 
$\os{\circ}{X}{}'_{\ul{\lam}}:=\os{\circ}{X}{}'_{\lam_1}\cap \cdots \cap \os{\circ}{X}{}'_{\lam_k}$.  
Let $a_{\ul{\lam}} \col \os{\circ}{X}_{\ul{\lam}}\lo \os{\circ}{X}$ and 
$a'_{\ul{\lam}} \col \os{\circ}{X}{}'_{\ul{\lam}}\lo \os{\circ}{X}{}'$
be the natural exact closed immersions.
Consider the following 
direct factor of the morphism (\ref{ali:gdkcl}): 
\begin{align*}
g^*_{\ul{\lam}{\rm conv}} \col 
a_{\ul{\lam}{\rm conv}*}
(a^*_{\ul{\lam}{\rm conv}}(E)\otimes_{\mab Z}
\vp_{\ul{\lam}{\rm conv}}(\os{\circ}{X}/{\os{\circ}{S}})
[-k]  \tag{11.2.6}\label{ali:grgm}\\
\lo 
Rg_{\ul{\lam}{\rm conv}*}a'{}_{\ul{\lam}{\rm conv}*}
(a'{}^*_{\ul{\lam}{\rm conv}}(E')
\otimes_{\mab Z}
\vp_{\ul{\lam}{\rm conv}}(\os{\circ}{X}{}'/\os{\circ}{S}{}'))
[-k],  
\end{align*}
where $\vp_{\ul{\lam}{\rm conv}}(\os{\circ}{X}{}'/{\os{\circ}{S}{}'})$ and 
$\vp_{\ul{\lam}{\rm conv}}(\os{\circ}{X}/\os{\circ}{S})$ are 
the convergent orientation sheaves  
associated to the sets $\{\os{\circ}{X}{}'_{\lam_j}\}_{j=1}^k$ 
and $\{\os{\circ}{X}_{\lam_j}\}_{j=1}^k$, respectively,  
defined as in \cite[p.~81]{nh2}.

\begin{prop}\label{prop:grloc}
Let the notations 
and the assumptions be as above.
Let 
$$\os{\circ}{g}_{\ul{\lam}}
\col \os{\circ}{X}{}'_{\ul{\lam}} \lo \os{\circ}{X}_{\ul{\lam}}$$ 
be the induced morphism by $g$ and 
let 
\begin{align*} 
\os{\circ}{g}{}^*_{\ul{\lam}} \col 
\os{\circ}{g}{}^*_{\ul{\lam},{\rm crys}}(a^*_{\ul{\lam}}(E))
\lo a'{}^*_{\ul{\lam}}(E')
\tag{11.3.1}\label{ali:teawfm}
\end{align*} 
be the induced morphism of isocrystals on 
$(\os{\circ}{X}{}'_{\lam}/\os{\circ}{S}{}')_{\rm conv}$ by 
{\rm (\ref{ali:teawfm})}. 
Then the morphism 
$g^*_{\ul{\lam}{\rm conv}}$ in 
{\rm (\ref{ali:grgm})}  is equal to $(\prod_{j=1}^ke_{\lam_j})
(a'_{\ul{\lam}{\rm conv}*}\os{\circ}{g}{}^*_{\ul{\lam},{\rm conv}})$ for 
$k \geq 0$. Here we define 
$\prod_{j=1}^ke_{\lam_j}$ as $1$ for $k=0$.
\end{prop}
\begin{proof}
One has only to imitate the proof of 
\cite[(2.9.3)]{nh2} and make the convergent linearization of it. 
\end{proof}

In \cite[(1.5.7)]{nb} we have proved the following: 

\begin{prop}\label{prop:xxle} 
Let the assumptions and the notations be as above. 
Set $\Lam(x):=
\{\lam \in \Lam~\vert~x \in \os{\circ}{X}_{\lam}\}$.  
Then $\deg(u)_x=e_{\lam}$ for $\lam \in \Lam(x)$.  
In particular, 
$e({\lam})$'s are independent of
the choice of an element of $\Lam(x)$.  
\end{prop} 

\begin{defi}\label{defi:wcel}
(1) We call 
$\{e_{\lam}\}_{\lam \in \Lam}\in 
({\mab Z}_{>0})^{\Lam}$ 
the {\it multi-degree} of $g$ 
with respect to the decompositions of 
$\{\os{\circ}{X}_{\lam}\}_{\lam}$ 
and $\{\os{\circ}{X}{}'_{\lam}\}_{\lam}$ of $\os{\circ}{X}$ and $\os{\circ}{X}{}'$, respectively.   
We  denote it by 
${\rm deg}_{\os{\circ}{X},\os{\circ}{X}'}(g) \in ({\mab Z}_{>0})^{\Lam}$. 
If $e_\lam$'s for all 
$\lam$'s 
are equal, 
we also denote 
$e_{\lam}\in {\mab Z}_{>0}$ 
by 
${\rm deg}_{\os{\circ}{X},\os{\circ}{X}{}'}(g) \in 
{\mab Z}_{>0}$.
\par
(2) Assume that 
$e_\lam$'s for all $\lam$'s are equal. 
Let $v \col {\cal E} \lo {\cal F}$ 
be a morphism of  
${\cal K}_S$-modules. 
Let $k$ be a nonnegative integer. 
The $D$-{\it twist} by $k$ 
$$v(-k) \col {\cal E}(-k;u) 
\lo {\cal F}(-k;u)$$  
of $v$ with respect to $u$
is, by definition, the morphism  
${\rm deg}(u)^kv 
\col {\cal E} \lo {\cal F}$.
\end{defi}

\begin{coro}\label{coro:dts}
Set $f^{\rm conv}_{X/S}:=f\circ u^{\rm conv}_{X/S}$. 
Assume that $e_\lam$'s for 
all $\lam$'s are equal. 
Let $E_{\rm ss}(X/S,E)$ be 
the following spectral sequence  
\begin{equation*} 
E_1^{-k,q+k}=\bigoplus_{j\geq \max \{-k,0\}} 
R^{q-2j-k}f^{\rm conv}_{\os{\circ}{X}{}^{(2j+k)}/\os{\circ}{S}*}(a^{(2j+k)*}_{\rm conv}(E)\otimes_{\mab Z}
\vp^{(2j+k)}_{{\rm conv}}(\os{\circ}{X}/\os{\circ}{S}))(-j-k,u) 
\tag{11.6.1}\label{eqn:dtxw}
\end{equation*}
$$\Lo 
R^qf^{\rm conv}_{X/S*}(\eps^{{\rm conv}*}_{X/S}(E))$$ 
and let $E_{\rm ss}(X'/S',E')$ be the obvious 
analogue of the above for $X'/S'$.  
Then there exists a morphism  
\begin{equation*} 
g^*_{\rm conv}
\col u^{-1}E_{\rm ss}(X'/S',E') 
\lo E_{\rm ss}(X/S,E)
\tag{11.6.2}\label{eqn:dtw}
\end{equation*} 
of spectral sequences. 
\end{coro}
\begin{proof}
(\ref{coro:dts}) immediately follows 
from (\ref{prop:grloc}).
\end{proof}

\par 
Let $F_{S_1} \col S_1 \lo S_1$ be 
the absolute Frobenius endomorphism of $S_1$.
Set 
$X^{[p]}:=X\times_{\os{\circ}{S}_1,F_{\os{\circ}{S}_1}}\os{\circ}{S}_1$.
The abrelative Frobenius morphism 
$F \col X\lo  X^{[p]} $ over $F_{S_1/\os{\circ}{S}_1}\col S_1\lo S_1^{[p]}:=
S_1\times_{\os{\circ}{S}_1,F_{\os{\circ}{S}_1}}\os{\circ}{S}_1$ induces 
the usual relative Frobenius morphism 
$\os{\circ}{F}{}^{(k)} \col \os{\circ}{X}{}^{(k)} \lo (X^{[p]})^{\circ,(k)}$. 
It is easy to see that 
$(X^{[p]})^{\circ,(k)}=\os{\circ}{X}{}^{(k)}
\times_{\os{\circ}{S}_1,F_{\os{\circ}{S}_1}}\os{\circ}{S}_1
=:\os{\circ}{X}{}^{(k)'}$. 
We also set $\os{\circ}{X}{}':=\os{\circ}{X}\times_{\os{\circ}{S}_1,F_{\os{\circ}{S}_1}}\os{\circ}{S}_1$. 
Let $a^{(k)} \col \os{\circ}{X}{}^{(k)} \lo \os{\circ}{X}$ 
and 
$a^{(k)'} \col \os{\circ}{X}{}^{(k)'} \lo \os{\circ}{X}{}'$ 
be the natural morphisms. 
We define the 
relative Frobenius action
$$\Phi_{\os{\circ}{X}{}^{(k)}/\os{\circ}{S}} 
\col  a^{(k)'}_{{\rm conv}*}\vp^{(k)}_{\rm conv}(\os{\circ}{X}{}'/\os{\circ}{S}) 
\lo F_{{\rm conv}*}a^{(k)}_{{\rm conv}*}
\vp^{(k)}_{{\rm conv}}(\os{\circ}{X}/\os{\circ}{S})$$
as the identity 
under the natural identification
$$\vp^{(k)}_{\rm conv}(\os{\circ}{X}{}'/\os{\circ}{S}{}')
\os{\sim}{\lo} F^{(k)}_{{\rm conv}*}\vp^{(k)}_{{\rm conv}}(\os{\circ}{X}/\os{\circ}{S}).$$
Let $S_1^{[p]}(S)$ be a fine log formal scheme 
whose underlying formal scheme is $\os{\circ}{S}$ and the log structure of 
$S_1^{[p]}(S)$ is the sub-log structure of $S$ 
such that the isomorphism 
$M_S/{\cal O}_S^*\os{\sim}{\lo} M_{S_1}/{\cal O}_{S_1}^*$ induces 
the following isomorphism 
\begin{align*} 
M_{S_1^{[p]}(S)}/{\cal O}_S^*\os{\sim}{\lo} 
{\rm Im}(F_{S_1/\os{\circ}{S}_1}^*\col 
F_{S_1/\os{\circ}{S}_1}^*(M_{S_1^{[p]}})
\lo M_{S_1})/{\cal O}_{S_1}^*. 
\end{align*} 
Then we have an obvious morphism $S\lo S_1^{[p]}(S)$ of log formal schemes. 
When $g$ in (\ref{cd:mpsmlgs}) 
is equal to the relative Frobenius $F \col X \lo X^{[p]}$ over $S\lo S_1^{[p]}(S)$,
we denote (\ref{ali:ccwz})+(the compatibility with Frobenius) 
by the following:  
\begin{align*}
E_1^{-k,q+k}(X/S)= &\bigoplus_{j\geq \max \{-k,0\}} 
R^{q-2j-k}f^{\rm conv}_{\os{\circ}{X}{}^{(2j+k)}/\os{\circ}{S}*}
(a^{(2j+k)*}_{\rm conv}(E)\otimes_{\mab Z}
\vp^{(2j+k)}_{{\rm conv}}(\os{\circ}{X}/\os{\circ}{S})))(-j-k)  
\tag{11.6.3}\label{ali:wtfapwt}
\\
\Lo 
{} & R^qf^{\rm conv}_{X/S*}(\eps^{{\rm conv}*}_{X/S}(E)).
\end{align*}

\begin{defi}
In the case where $E$ is the trivial coefficient and $\os{\circ}{X}$ is proper over $\os{\circ}{S}_1$, 
we call the sequence (\ref{ali:wtfapwt}) 
the {\it weight spectral sequence of} 
$X/S$ (more precisely over $(S_1,S)$). 
\end{defi}

\section{Edge morphisms}\label{sec:bd}
In this section we define the cycle class 
in a convergent cohomology sheaf for a smooth divisor on a smooth scheme
(cf.~\cite[\S2]{fao}, \cite[(2.8)]{nh2}).
Then, as in \cite[(10.1)]{nh2}, we give a 
description of the edge morphism between 
the $E_1$-terms of the spectral sequence 
(\ref{ali:wtfapwt}). Because the proof is the same as that of 
\cite[(10.1)]{ndw}, 
we omit the proofs in this section. 
\par
Let $S$ and $S_1$ be as in \S\ref{sec:lcs}. 
Let $g \col Z \lo \os{\circ}{S}_1$ be a smooth scheme over $\os{\circ}{S}_1$. 
By abuse of notation, 
we also denote by $g$ the composite 
morphism $Z \lo \os{\circ}{S}_1 \os{\subset}{\lo} \os{\circ}{S}$. 
Let $D$ be a smooth divisor on $Z$.  
Let $F$ be an isocrystal in $(Z/\os{\circ}{S})_{\rm conv}$. 
Let 
$a \col D\os{\subset}{\lo} Z$ be 
the natural closed immersion over $\os{\circ}{S}_1$.  
Let 
$a_{\rm zar} \col D_{\rm zar}\lo Z_{\rm zar}$ 
be the induced morphism of Zariski topoi.
Let $a_{\rm conv} \col
(({D/\os{\circ}{S}})_{\rm conv},{\cal K}_{D/\os{\circ}{S}}) 
\lo 
(({Z/\os{\circ}{S}})_{\rm conv},
{\cal K}_{Z/\os{\circ}{S}})$ 
be also the induced morphism of convergent ringed topoi.
\par
Let $Z_{\bul}$ be the smooth simplicial scheme over $\os{\circ}{S}_1$ obtained by 
an affine covering of $Z$. 
Let $(Z_{\bul},D_{\bul}) \os{\sus}{\lo} {\cal Q}_{\bul}$ 
be a simplicial immersion into a log smooth simplicial scheme over $\os{\circ}{S}$. 
Let ${\cal M}^{\rm ex}_{\bul}$ be the log structure of ${\cal Q}^{\rm ex}_{\bul}$. 
Let $b_{\bul} \col D^{(1)}({\cal M}^{\rm ex}_{\bul})
{\lo} {\cal Q}^{\rm ex}_{\bul}$ 
be the natural simplicial morphism. 
Let $({\cal F}^{\bul},\nabla):=\{({\cal F}^{\bul}_n,\nabla)\}_{n=1}^{\infty}$ 
be an integrable connection obtained by $F$. 
By using the Poincar\'{e} residue isomorphism 
with respect to $D^{(1)}({\cal M}^{\rm ex}_{\bul})$, 
we have the following exact sequence:
\begin{align*}
0 &\lo 
{\cal F}^{\bul}\otimes_{{\cal O}_{{\cal Q}^{\rm ex}_{\bul}}}
\Om_{\os{\circ}{\cal Q}{}^{\rm ex}_{\bul}/\os{\circ}{S}}^{\bul} 
\lo {\cal F}^{\bul}\otimes_{{\cal O}_{{\cal Q}^{\rm ex}_{\bul}}}
\Om_{{\cal Q}^{\rm ex}_{\bul}/\os{\circ}{S}}^{\bul} 
\tag{12.0.1}\label{eqn:omresd} \\
&\os{\rm Res}{\lo} 
b_{{\bul}*}
(b^*_{{\bul}}({\cal F}^{\bul})\otimes_{{\cal O}_{D^{(1)}({\cal M}^{\rm ex}_{\bul})}}
\Om_{D^{(1)}({\cal M}^{\rm ex}_{\bul})/\os{\circ}{S}}^{\bul}
\otimes_{\mab Z}
\vp^{(1)}_{\rm zar}(D^{(1)}({\cal M}^{\rm ex}_{\bul})/\os{\circ}{S}))
[-1] \lo 0.
\end{align*} 
Let $L^{\rm UE}_{Z_{\bul}/\os{\circ}{S}}$ 
(resp.~$L^{\rm UE}_{D^{(1)}_{\bul}/\os{\circ}{S}}$)
be the convergent linearization functor 
with respect to the simplicial immersion  
$Z_{\bul} \os{\subset}{\lo} \os{\circ}{\cal Q}{}^{\rm ex}_{\bul}$
(resp.~$D_{\bul} 
\os{\subset}{\lo} D^{(1)}({\cal M}^{\rm ex}_{\bul}$)). 
By (\ref{linc}), 
$L^{\rm UE}_{Z_{\bul}/\os{\circ}{S}}b_{\bul{\rm zar}*}= 
a_{{\rm conv}\bul*}
L^{\rm UE}_{D^{(1)}_{\bul}/\os{\circ}{S}}$. 
Hence we have the following exact 
sequence by (\ref{prop:grla}) and (\ref{theo:injf}) (2): 
\begin{equation*}
0 \lo 
L^{\rm UE}_{Z_{\bul}/\os{\circ}{S}}
({\cal F}^{\bul}\otimes_{{\cal O}_{{\cal Q}^{\rm ex}_{\bul}}}
\Om_{\os{\circ}{\cal Q}{}^{\rm ex}_{\bul}/\os{\circ}{S}}^{\bul}) \lo 
L^{\rm UE}_{Z_{\bul}/\os{\circ}{S}}
({\cal F}^{\bul}\otimes_{{\cal O}_{{\cal Q}^{\rm ex}_{\bul}}}
\Om_{{\cal Q}^{\rm ex}_{\bul}/\os{\circ}{S}}^{\bul})
\tag{12.0.2}\label{eqn:xzl}
\end{equation*} 
$$\lo 
a_{{\rm conv}\bul*}
(L^{\rm UE}_{D_{\bul}/\os{\circ}{S}}
(b^*_{{\bul}}({\cal F}^{\bul})\otimes_{{\cal O}_{D^{(1)}({\cal M}^{\rm ex}_{\bul})}}
\Om_{D^{(1)}({\cal M}^{\rm ex}_{\bul})/\os{\circ}{S}}^{\bul})
\otimes_{\mab Z}\vp^{(1)}_{{\rm conv}}(D_{\bul}/\os{\circ}{S}))
[-1]) 
\lo 0.$$
Recall the morphisms  
$\pi^{\rm conv}_{Z/\os{\circ}{S}}$ and $\pi^{\rm conv}_{D/\os{\circ}{S}}$ 
of ringed topoi in \S\ref{sec:rlct}.  
By applying $R{\pi}^{\rm conv}_{Z/\os{\circ}{S}*}$ to (\ref{eqn:xzl}), 
we obtain the following triangle:  
\begin{equation*}
\lo R{\pi}^{\rm conv}_{Z/\os{\circ}{S}*}
L^{\rm UE}_{Z_{\bul}/\os{\circ}{S}}
({\cal F}^{\bul}\otimes_{{\cal O}_{{\cal Q}^{\rm ex}_{\bul}}}
\Om_{\os{\circ}{\cal Q}{}^{\rm ex}_{\bul}/\os{\circ}{S}}^{\bul}) \lo 
R{\pi}^{\rm conv}_{Z/\os{\circ}{S}*}
L^{\rm UE}_{Z_{\bul}/\os{\circ}{S}}
({\cal F}^{\bul}\otimes_{{\cal O}_{{\cal Q}^{\rm ex}_{\bul}}}
\Om_{{\cal Q}^{\rm ex}_{\bul}/\os{\circ}{S}}^{\bul}) \lo 
\tag{12.0.3} 
\end{equation*}
$$
a_{{\rm conv}*}
R{\pi}_{D/\os{\circ}{S}{\rm conv}*}
(L^{\rm UE}_{D^{(1)}_{\bul}/\os{\circ}{S}}
(b^*_{{\bul}}({\cal F}^{\bul})\otimes_{{\cal O}_{D^{(1)}({\cal M}^{\rm ex}_{\bul})}}
\Om^{\bul}_{D^{(1)}({\cal M}^{\rm ex}_{\bul})/\os{\circ}{S}})  
\otimes_{\mab Z}
\vp^{(1)}_{{\rm conv}}(D_{\bul}/\os{\circ}{S}))[-1] 
\os{+1}{\lo} 
\cdots. $$
By (\ref{theo:pl}), (\ref{theo:cpvcs}) and 
by the cohomological descent,  
we have the following triangle:
\begin{equation*}
\lo F \lo 
R\eps^{\rm conv}_{(Z,D)/\os{\circ}{S}*}(\eps^{{\rm conv}*}_{(Z,D)/\os{\circ}{S}}(F)) \lo 
a_{{\rm conv}*}
(a^*_{\rm conv}(F)\otimes_{\mab Z}
\vp^{(1)}_{{\rm conv}}(\os{\circ}{D}/\os{\circ}{S}))
[-1] 
\os{+1}{\lo} \cdots.
\tag{12.0.4}\label{eqn:oealog}
\end{equation*}
Hence we have the boundary morphism 
\begin{equation*}
d \col a_{{\rm conv}*}
(a^*_{\rm conv}(F)\otimes_{\mab Z}
\vp^{(1)}_{{\rm conv}}(\os{\circ}{D}/\os{\circ}{S}))[-1] \lo F[1]
\tag{12.0.5}
\end{equation*}
in $D^+({\cal K}_{Z/S})$ and the following morphism
\begin{equation*}
d \col a_{{\rm conv}*}
(a^*_{\rm conv}(F)\otimes_{\mab Z}
\vp^{(1)}_{{\rm conv}}(\os{\circ}{D}/\os{\circ}{S})) \lo F[2]
\tag{12.0.6} 
\end{equation*}
in $D^+({\cal K}_{Z/\os{\circ}{S}})$. 
Set 
\begin{equation*}
G^{\rm conv}_{D/Z,F}:=-d.  
\tag{12.0.7}\label{eqn:gdxz}
\end{equation*}
and call $G^{\rm conv}_{D/Z}$ the 
{\it log convergent Gysin morphism} of $D$.
Then we have the following cohomology class 
\begin{align*}
c^{\rm conv}_{Z/\os{\circ}{S}}(D,F):=G^{\rm conv}_{D/Z,F} \in 
{\cal E}{\it xt}^0_{{\cal K}_{Z/\os{\circ}{S}}} 
(& a_{{\rm conv}*}
(a^*_{\rm conv}(F)\otimes_{\mab Z}
\vp^{(1)}_{{\rm conv}}(\os{\circ}{D}/\os{\circ}{S})),F[2]).  
\tag{12.0.8}\label{eqn:lcycls} 
\end{align*}
It is a routine work to prove that $G^{\rm conv}_{D/Z,F}$ 
(and $c^{\rm conv}_{Z/\os{\circ}{S},F}(D,F)$) depends only  
on $Z/\os{\circ}{S}$. 
Since $\vp^{(1)}_{{\rm conv}}(D/\os{\circ}{S})$ 
is canonically isomorphic to ${\mab Z}$ and since 
there exists a natural morphism 
$F \lo 
a_{{\rm conv}*}a^*_{\rm conv}(F)$,
we have a cohomology class 
\begin{align*}
c^{\rm conv}_{Z/\os{\circ}{S}}(D,F) & \in  
{\cal E}{\it xt}^0_{{\cal K}_{Z/\os{\circ}{S}}}
({\cal K}_{Z/\os{\circ}{S}},
{\cal K}_{Z/\os{\circ}{S}}[2]) =:
\ul{\cal H}_{{\log}{\textrm -}{\rm conv}}^2(Z/\os{\circ}{S}).
\end{align*}

\par
Next we give the expression of the edge morphism $d_1^{\bul \bul}$ of
(\ref{ali:wtfapwt}) by the \v{C}ech-Gysin morphism.
\par
First, fix a decomposition 
$\{\os{\circ}{X}_{\lam}\}_{\lam \in \Lam}$ of $\os{\circ}{X}$ 
by smooth components of $\os{\circ}{X}$ over $S_1$ and fix a total order on $\Lam$. 
For $\ul{\lam}:=\{\lam_0,  \ldots, \lam_{k-1}\}$ $(\lam_0<\cdots < \lam_{k-1})$, 
$\ul{\lam}_j:=\{\lam_0,\ldots,  \wh{\lam}_j,  
\ldots, \lam_{k-1}\}$, set 
$\os{\circ}{X}_{\ul{\lam}}:=
\os{\circ}{X}_{\lam_0}\cap  \cdots \cap\os{\circ}{X}_{\lam_{k-1}}$ 
and $\os{\circ}{X}_{\ul{\lam}_j}:=
\os{\circ}{X}_{\lam_0}\cap  \cdots 
\wh{\os{\circ}{X}_{{\lam}_j}}\cap \cdots \cap \os{\circ}{X}_{\lam_{k-1}}$ for $k\geq 2$. 
Here $~\wh{}~$ means the elimination. 
Then $\os{\circ}{X}_{\ul{\lam}}$
is a smooth divisor on $\os{\circ}{X}_{\ul{\lam}_j}$ over $S_1$.
Let 
$\iota^{\ul{\lam}_j}_{\ul{\lam}} 
\col \os{\circ}{X}_{\ul{\lam}} 
\os{\subset}{\lo} 
\os{\circ}{X}_{\ul{\lam}_j}$ 
be the closed immersion.
As in \cite[p.~81]{nh2}, 
we have the following orientation sheaves  
$$\vp_{\ul{\lam}{\rm conv}}(\os{\circ}{X}/\os{\circ}{S}):=
\vp_{\lam_0 \cdots \lam_{k-1}{\rm conv}}(\os{\circ}{X}/\os{\circ}{S})$$ 
and
$$\vp_{\ul{\lam}_j{\rm conv}}(\os{\circ}{X}/\os{\circ}{S}):=
\vp_{\lam_0 \cdots \wh{\lam}_j 
\cdots \lam_{k-1}{\rm conv}}(\os{\circ}{X}/\os{\circ}{S})$$
associated to the sets $\{\os{\circ}{X}_{\lam_i}\}_{i=0}^k$ 
and $\{\os{\circ}{X}_{\lam_i}\}_{i=0}^k\setminus 
\{\os{\circ}{X}_{\lam_j}\}$, respectively.   
We denote by $(\lam_0\cdots \lam_{k-1})$ 
the local section giving a basis of 
$\vp_{\lam_0\cdots \lam_{k-1}{\rm conv}}(\os{\circ}{X}/\os{\circ}{S})$; 
we also denote by $(\lam_0\cdots \lam_{k-1})$ 
the local section giving a basis of 
$\vp_{\lam_0\cdots \lam_{k-1}{\rm zar}}(\os{\circ}{X}/S)$. 
Let $a_{\ul{\lam}}\col \os{\circ}{X}_{\lam}\lo \os{\circ}{X}$ be the natural closed immersion. 
By (\ref{eqn:gdxz}) we have a morphism
\begin{align*}
G^{{\ul{\lam}_j}{\rm conv}}_{\ul{\lam}}:=
G^{\rm conv}_{\os{\circ}{X}_{\ul{\lam}}/\os{\circ}{X}_{\ul{\lam}_j},a_{\ul{\lam}_j}^*(E)} 
\col & \iota^{\ul{\lam}_j}_{\ul{\lam}{\rm conv}*}
(a_{\ul{\lam}{\rm conv}}^*(E)
\otimes_{\mab Z}\vp_{\lam_j{\rm conv}}(\os{\circ}{X}/\os{\circ}{S}))\lo a_{\ul{\lam}_j}^*(E)[2].
\tag{12.0.9}\label{eqn:glmoo} 
\end{align*}
We fix an isomorphism 
\begin{equation*}
\vp_{\lam_j{\rm conv}}(\os{\circ}{X}/\os{\circ}{S}) 
\otimes_{\mab Z}\vp_{\ul{\lam}_j{\rm conv}}(\os{\circ}{X}/\os{\circ}{S})  
\os{\sim}{\lo}\vp_{\ul{\lam}{\rm conv}}(\os{\circ}{X}/\os{\circ}{S}) 
\tag{12.0.10}\label{eqn:coxts}
\end{equation*}
by the following morphism
$$(\lam_j)\otimes (\lam_0\cdots  
\wh{\lam}_j \cdots \lam_{k-1})
\lom (-1)^j(\lam_0\cdots \lam_{k-1}).$$
We identify 
$\vp_{\lam_j{\rm conv}}(\os{\circ}{X}/\os{\circ}{S})\otimes_{\mab Z} 
\vp_{\ul{\lam}_j{\rm conv}}(\os{\circ}{X}/\os{\circ}{S})$ with 
$\vp_{\ul{\lam}{\rm conv}}(\os{\circ}{X}/\os{\circ}{S})$ 
by this isomorphism.
We also have the following composite morphism
\begin{equation*}
(-1)^jG^{{\ul{\lam}_j}{\rm conv}}_{\ul{\lam}}
\col \iota^{\ul{\lam}_j}_{\ul{\lam}{\rm conv}*}
(a_{\ul{\lam}{\rm conv}}^*(E)
\otimes_{\mab Z}\vp_{\ul{\lam}{\rm conv}}(\os{\circ}{X}/\os{\circ}{S})) 
\os{\sim}{\lo} \tag{12.0.11}\label{eqn:m1gom}
\end{equation*} 
$$\iota^{\ul{\lam}_j}_{\ul{\lam}{\rm conv}*}
(a_{\ul{\lam}{\rm conv}}^*(E)
\otimes_{\mab Z}
\vp_{\lam_j{\rm conv}}(\os{\circ}{X}/\os{\circ}{S})\otimes_{\mab Z} 
\vp_{\ul{\lam}_j{\rm conv}}(\os{\circ}{X}/\os{\circ}{S})) 
\os{G^{{\ul{\lam}_j}{\rm conv}}_{\ul{\lam}}\otimes 1}{\lo} $$
$$a_{\ul{\lam}_j{\rm conv}}^*(E)
\otimes_{\mab Z}\vp_{\ul{\lam}_j{\rm conv}}(\os{\circ}{X}/\os{\circ}{S})[2]$$ 
defined by 
\begin{equation*}{\rm ``} x\otimes (\lam_0 \cdots  \lam_{k-1}) \lom 
(-1)^jG^{{\ul{\lam}_j}{\rm conv}}_{\ul{\lam}}(x)\otimes 
(\lam_0\cdots \wh{\lam}_j \cdots
\lam_{k-1}){\rm "}.
\tag{12.0.12}\label{eqn:ooolll}
\end{equation*}
The morphism (\ref{eqn:ooolll}) induces 
the following morphism 
of log convergent cohomologies:
\begin{equation*}
(-1)^j G^{{\ul{\lam}_j}{\rm conv}}_{\ul{\lam}} \col 
R^{q-k}
f^{\rm conv}_{\os{\circ}{X}_{\ul{\lam}}/\os{\circ}{S}*}
(a_{\ul{\lam}{\rm conv}}^*(E)
\otimes_{\mab Z}
\vp_{\ul{\lam}{\rm conv}}(\os{\circ}{X}/\os{\circ}{S}))\lo 
\tag{12.0.13}\label{eqn:cohgxdz}
\end{equation*}
$$R^{q-k+2}
f^{\rm conv}_{\os{\circ}{X}_{\ul{\lam}_j}/\os{\circ}{S}*}
(a_{\ul{\lam}_j{\rm conv}}^*(E)
\otimes_{\mab Z}\vp_{\ul{\lam}_j{\rm conv}}(\os{\circ}{X}/\os{\circ}{S})).$$ 

Set
\begin{align*}
& G:=\bigoplus_{j\geq \max \{-k,0\}}
\sum_{\{\ul{\lam}:=\{\lam_{0},\ldots, \lam_{2j+k}\}~ 
\vert ~ \lam_{\gam} < \lam_{\al}\; (\gam < \al)\}}
\sum_{\bet=0}^{2j+k}(-1)^{\bet}
G_{\ul{\lam}}^{\ul{\lam}_{{\bet}}} \col 
\tag{12.0.14}\label{eqn:togsn}\\
&\bigoplus_{j\geq \max \{-k,0\}} 
R^{q-2j-k}
f_{\os{\circ}{X}{}^{(2j+k)}/\os{\circ}{S}*}
(a^{(2j+k)*}_{\rm conv}(E)\otimes_{\mab Z} 
\vp^{(2j+k)}_{\rm crys}(\os{\circ}{X}/\os{\circ}{S}))\\ 
&(-j-k,u)\lo \\
&\bigoplus_{j\geq \max \{-k+1,0\}} 
R^{q-2j-k+2}
f_{\os{\circ}{X}{}^{(2j+k-1)}/\os{\circ}{S}*}
(a^{(2j+k-1)*}_{\rm conv}(E)
\otimes_{\mab Z} 
\vp^{(2j+k-1)}_{\rm conv}(\os{\circ}{X}/\os{\circ}{S}))\\
&(-j-k+1,u).
\end{align*}
\par 
Let $\iota_{\ul{\lam}}^{\ul{\lam}_{\bet}} \col \os{\circ}{X}_{\ul{\lam}} 
\os{\sus}{\lo} \os{\circ}{X}_{\ul{\lam}_{\bet}}$ be 
the natural immersion.  
The morphism $\iota_{\ul{\lam}}^{\ul{\lam}_{\bet}}$ 
induces the morphism  
\begin{equation}
(-1)^{\bet}
\iota_{\ul{\lam}{\rm crys}}^{\ul{\lam}_{\bet}*}
\col 
\iota_{\ul{\lam}{\rm crys}}^{\ul{\lam}_{\bet}*}
(a^*_{\ul{\lam}_j{\rm conv}}(E)
\otimes_{\mab Z}\vp_{\ul{\lam}_{\bet}{\rm crys}}
(\os{\circ}{X}/\os{\circ}{T})) 
\lo 
a^*_{\ul{\lam}{\rm conv}}(E)
\otimes_{\mab Z}\vp_{\ul{\lam}{\rm conv}}
(\os{\circ}{X}/\os{\circ}{S}).  
\tag{12.0.15}\label{eqn:defcbd}
\end{equation}
Set 
\begin{align*}
& \rho:=\bigoplus_{j\geq \max \{-k,0\}}
\sum_{\{\ul{\lam}:=\{\lam_{0},\ldots,<\lam_{2j+k}\}~ 
\vert ~\lam_{\gam}< \lam_{\alpha}\; 
(\gam < \alpha)\}}
\sum_{\bet=0}^{2j+k}(-1)^{\bet}
\iota_{\ul{\lam}{\rm conv}}^{\ul{\lam}_{\bet*}} 
\col 
\tag{12.0.16}\label{eqn:rhogsn}\\
&\bigoplus_{j\geq \max \{-k,0\}} 
R^{q-2j-k}
f_{\os{\circ}{X}{}^{(2j+k)}/\os{\circ}{S}*}
(E^m_{\os{\circ}{X}{}^{(2j+k)}
/\os{\circ}{T}}
\otimes_{\mab Z}\vp^{(2j+k)}_{\rm crys}
(\os{\circ}{X}/\os{\circ}{S}))  \\ 
&(-j-k,u)\lo \\ 
& \bigoplus_{j\geq \max \{-k,0\}} 
R^{q-2j-k}
f_{\os{\circ}{X}{}^{(2j+k+1)}/\os{\circ}{S}}
(E^m_{\os{\circ}{X}{}^{(2j+k+1)}
/\os{\circ}{T}}
\otimes_{\mab Z}\vp^{(2j+k+1)}_{\rm crys}
(\os{\circ}{X}/\os{\circ}{S}))\\
& (-j-k,u).
\end{align*}

\begin{prop}\label{prop:deccbd} 
The edge morphism between the $E_1$-terms of 
the spectral sequence {\rm (\ref{ali:wtfapwt})} 
is given by the following diagram$:$ 
\begin{equation*} 
\tag{12.1.1}\label{cd:gsmsd}
\end{equation*} 
\begin{equation*} 
\begin{split} 
{} & \bigoplus_{j\geq \max \{-k,0\}} 
R^{q-2j-k}
f_{\os{\circ}{X}{}^{(2j+k)}
/\os{\circ}{T}*}
(a^{(2j+k)*}_{\rm conv}(E) \\
{} & \phantom{R^{q-2m-k}
f_{(\os{\circ}{X}^{(k)}, 
Z\vert_{\os{\circ}{X}^{(k+m)}})/S*}
({\cal O}\quad \quad} 
\otimes_{\mab Z}\vp^{(2j+k)}_{\rm conv}
(\os{\circ}{X}_{T_0}/\os{\circ}{T}))(-j-k,u). 
\end{split}  
\end{equation*}  
$$\text{\scriptsize
{${
G+\rho}$}}
~\downarrow \quad \quad \quad \quad \quad \quad \quad 
\quad \quad \quad \quad \quad$$
\begin{equation*} 
\begin{split} 
{} & 
\bigoplus_{j\geq \max \{-k+1,0\}} 
R^{q-2j-k+2}
f_{\os{\circ}{X}{}^{(2j+k-1)}/\os{\circ}{S}*}
(a^{(2j+k-1)*}_{\rm conv}(E) \\
{} & \phantom{R^{q-k}
f_{(\os{\circ}{X}^{(k)}, 
Z\vert_{\os{\circ}{X}^{(k)}})/S*}
({\cal O}\quad \quad} 
\otimes_{\mab Z}\vp^{(2j+k-1)}_{\rm conv}
(\os{\circ}{X}_{T_0}/\os{\circ}{T}))(-j-k+1,u). 
\end{split}  
\end{equation*}   
\end{prop} 
\begin{proof} 
This proposition follows from \cite[(2.2.12), (2.2.13)]{nh3} 
and the proof of \cite[(10.1)]{ndw}. 
\end{proof}

\section{Monodromy operators}\label{sec:mn} 
In this section we define the $p$-adic monodromy operator on the log convergent cohomologcal complex 
of a log smooth scheme. We also define the $p$-adic quasi-monodromy operator 
on the log convergent Steenbrink complex $(A_{\rm conv},P)$ of an SNCL scheme. 
We prove that they are equal. 
\par 
Let the notations be as in \S\ref{sec:mplf} and \S\ref{sec:wfcipp}. 
Let us recall the following triangle (\ref{ali:oqay}): 
\begin{align*} 
R\eps^{\rm conv}_{Y/S*}(\eps^{{\rm conv}*}_{Y/S/\os{\circ}{S}}(F))[-1]\os{\theta \wedge}{\lo}  
\wt{R}\eps_{Y/\os{\circ}{S}*}(F) 
\lo R\eps^{\rm conv}_{Y/S*}(\eps^{{\rm conv}*}_{Y/S/\os{\circ}{S}}(F)) 
\tag{13.0.1}\label{ali:omqay}
\os{+1}{\lo}.
\end{align*} 
Hence we obtain the following boundary morphism 
\begin{align*} 
d\col R\eps^{\rm conv}_{Y/S*}(\eps^{{\rm conv}*}_{Y/S/\os{\circ}{S}}(F))\lo 
R\eps^{\rm conv}_{Y/S*}(\eps^{{\rm conv}*}_{Y/S/\os{\circ}{S}}(F)). 
\tag{13.0.2}\label{ali:omcavy}
\end{align*}

\begin{defi}\label{defi:srdr}
We call the boundary morphism above 
the {\it convergent monodromy operator} of $F$
and we denote it by 
\begin{align*} 
N_{\rm conv}\col R\eps^{\rm conv}_{Y/S*}(\eps^{{\rm conv}*}_{Y/S/\os{\circ}{S}}(F))
\lo R\eps^{\rm conv}_{Y/S*}(\eps^{{\rm conv}*}_{Y/S/\os{\circ}{S}}(F)). 
\tag{13.1.1}\label{ali:omcvy}
\end{align*} 
In the case where $F$ is trivial, we call $N_{\rm conv}$
the {\it convergent monodromy operator} of $Y/S$. 
\end{defi} 

\begin{rema}\label{lemm:eit}
We can also construct the following morphism
\begin{align*} 
\nabla \col Rf^{\rm conv}_{Y/S*}(\eps^{{\rm conv}*}_{Y/S/\os{\circ}{S}}(F))
\lo Rf^{\rm conv}_{Y/S*}(\eps^{{\rm conv}*}_{Y/S/\os{\circ}{S}}(F))
\otimes_{{\cal O}_S}\Om^1_{S/\os{\circ}{S}}
\tag{13.2.1}\label{ali:stt}
\end{align*} 
by using the ``convergent Gauss-Manin connection'' for 
the composite morphism $Y\lo S\lo \os{\circ}{S}$ if 
we prove the base change formula for the log convergent cohomologies. 
(The base change formula will be proved in later sections.)
By using the following residue morphism 
\begin{align*} 
{\rm Res} \col \Om^1_{S/\os{\circ}{S}}\os{\sim}{\lo} {\cal O}_S,
\tag{13.2.2}\label{ali:stet}
\end{align*} 
we obtain the following composite morphism 
\begin{align*} 
N \col Rf^{\rm conv}_{Y/S*}(\eps^{{\rm conv}*}_{Y/S/\os{\circ}{S}}(F))
\os{\nabla}{\lo} Rf^{\rm conv}_{Y/S*}(\eps^{{\rm conv}*}_{Y/S/\os{\circ}{S}}(F))
\otimes_{{\cal O}_S}\Om^1_{S/\os{\circ}{S}}\os{\rm Res}{\lo} 
R\eps^{\rm conv}_{Y/S*}(\eps^{{\rm conv}*}_{Y/S/\os{\circ}{S}}(F)). 
\tag{13.2.3}\label{ali:abtt}
\end{align*} 
Though we can also prove that this $N$ is equal to the induced morphism by 
$N_{\rm conv}$, we do not discuss this in this paper. 
\end{rema}

Next we define the $p$-adic convergent quasi-monodromy operator and 
we formulate the variational $p$-adic monodromy-weight conjecture ((\ref{conj:remc})) 
and the variational filtered log $p$-adic hard Lefschetz conjecture ((\ref{conj:lhilc})). 
The former is a generalization of the $p$-adic monodromy-weight conjecture in \cite{msemi} as in \cite{nb}. 

\par 
Set 
\begin{align*}
A^{\bul \bul}:=H/\tau_0[1]\os{\theta}{\lo} H/\tau_1[1]\{1\}\os{\theta}{\lo} 
\cdots \os{\theta}{\lo} H/\tau_j[1]\{j\} \os{\theta}{\lo} \cdots
\end{align*}  
and 
\begin{align*}
P_kA^{\bul \bul}:=(\cdots \os{\theta}{\lo} \tau_{\max \{2j+k+1,j\}}H
/\tau_j[1]\{j\} \os{\theta}{\lo} \cdots) \quad
(k\in {\mab Z}, j\in {\mab N}).
\end{align*}  
Consider the following morphism 
\begin{align*} 
\nu ^{ij}:={\rm proj}  \col {\cal E}^{\bul}\otimes_{{\cal O}_{{\cal P}^{\rm ex}_{\bul}}}
{\Om}^{i+j+1}_{{\cal P}^{\rm ex}_{\bul}/\os{\circ}{S}}/P_{j+1}  
\lo {\cal E}^{\bul}\otimes_{{\cal O}_{{\cal P}^{\rm ex}_{\bul}}}
{\Om}^{i+j+1}_{{\cal P}^{\rm ex}_{\bul}/\os{\circ}{S}}/P_{j+2}.   
\tag{13.2.4}\label{eqn:reslbb}
\end{align*}  
This morphism induces the morphism  
\begin{equation*} 
\nu^{ij}:={\rm proj}. \col A^{ij}\lo A^{i-1,j+1}. 
\end{equation*}   
Set 
\begin{align*} 
\nu:=
s(\oplus_{i,j\in {\mab N}}\nu^{ij}).
\tag{13.2.5}\label{ali:ntzd}
\end{align*}  
Let 
\begin{equation*} 
\nu \col (A_{\rm conv}(X/S,E),P)
\lo 
(A_{\rm conv}(X/S,E),P\langle -2\rangle) 
\tag{13.2.6}\label{eqn:naxgdxdn}
\end{equation*} 
be a morphism of complexes induced by 
$\{\nu^{ij}\}_{i,j \in {\mab N}}$. 
Here $\langle -2\rangle$ means the (-2)-shift of the filtration: 
$P\langle -2\rangle_kA_{\rm conv}(X/S,E):=
P_{k-2}A_{\rm conv}(X/S,E)$. 
We also have the morphism 
$$\theta \col A^{ij}\lo A^{i,j+1}.$$ 
\par 
Let $g\col X\lo X$ be a morphism  fitting into the following commutative diagram
\begin{equation*} 
\begin{CD} 
X' @>{g'}>> X'' \\
@VVV @VVV \\ 
X @>{g}>> X \\
@VVV @VVV \\ 
S_1 @>>> S_1\\ 
@V{\bigcap}VV @VV{\bigcap}V \\ 
S @>{u}>> S, 
\end{CD}
\tag{13.2.7}\label{cd:xgspxy}
\end{equation*}
where $X''$ 
is another disjoint union 
of the member of an affine open covering of $X$. 
Since the following diagram  
\begin{equation*}
\begin{CD}
A^{ij}
@>{{\rm proj}.}>> A^{i-1,j+1}\\ 
@V{g^{*ij}}VV 
@VV{\deg(u)g^{*i-1,j+1}}V  \\
g_{*}(A^{ij}) @>{{\rm proj}.}>>
g_{*}(A^{i-1,j+1}) 
\end{CD}
\tag{13.2.8}\label{cd:phipnu}
\end{equation*}
is commutative, the morphism (\ref{eqn:naxgdxdn}) is 
the following morphism 
\begin{equation*} 
\nu \col (A_{\rm conv}(X/S,E),P)\lo 
(A_{\rm conv}(X/S,E),P\langle -2\rangle)(-1,u). 
\tag{13.2.9}\label{eqn:axdxdn}
\end{equation*}

\par  
Set 
\begin{equation*} 
B^{ij}
:=A^{i-1,j}(-1,u)\oplus A^{ij} 
\quad ( i,j\in {\mab N})
\tag{13.2.10}\label{eqn:axmdn}
\end{equation*} 
and 
\begin{equation*} 
B^{ij}
:=A^{i-1,j}(-1,u)\oplus A^{ij} 
\quad ( i,j\in {\mab N})
\tag{13.2.11}\label{eqn:anxdn}
\end{equation*} 
(cf.~\cite[p.~246]{st1}).
The horizontal differential morphism
$d' \col B^{ij} \lo B^{i+1,j}$ 
is, by definition, the induced morphism by a morphism 
$d' \col B^{ij} \lo B^{i+1,j}$ 
defined by the following formula: 
\begin{align*} 
d'(\om_1,\om_2)=(\nabla \om_1,-\nabla \om_2) 
\tag{13.2.12}\label{ali:sddclxn}
\end{align*} 
and the vertical one 
$d'' \col B^{ij} \lo B^{i,j+1}$ is the induced morphism 
by a morphism  $d'' \col B^{ij} \lo B^{i,j+1}$ 
defined by the following formula: 
\begin{align*} 
d''(\om_1,\om_2)=
(-\theta \wedge \om_1+\nu(\om_2),\theta \wedge \om_2).
\tag{13.2.13}\label{ali:sddpxn}
\end{align*}  
It is easy to check that 
$B^{\bul \bul}$ is actually a double complex. 
Let $B$ be the single complex of $B^{\bul \bul}$. 
\par
Let 
\begin{align*} 
\mu_{{X/\os{\circ}{S}}} \col 
\wt{R}u_{X/\os{\circ}{S}*}
(\eps^{{\rm conv}*}_{X/\os{\circ}{S}}(E))\lo B
\tag{13.2.14}\label{ali:sgsclxn}\\
\end{align*} 
be a morphism
of complexes induced by the following morphisms
\begin{align*} 
\mu^i_{\bul} \col {\cal E}^{\bul}
\otimes_{{\cal O}_{{\cal P}^{\rm ex}_{\bul}}}
{\Om}^i_{{\cal P}^{\rm ex}_{\bul}/\os{\circ}{S}} 
\lo  & ~{\cal E}^{\bul}
\otimes_{{\cal O}_{{\cal P}^{\rm ex}_{\bul}}}
{\Om}^i_{{\cal P}^{\rm ex}_{\bul}/\os{\circ}{S}}/P_0 \oplus   {\cal E}^{\bul}
\otimes_{{\cal O}_{{\cal P}^{\rm ex}_{\bul}}}
{\Om}^{i+1}_{{\cal P}^{\rm ex}_{\bul}/\os{\circ}{S}}/P_0 
\quad (i\in {\mab N})
\end{align*} 
defined by the following formula 
\begin{align*} 
\mu^i_{\bul}(\om):=(\om~{\rm mod}~P_0,
\theta \wedge \om~{\rm mod}~P_0) \quad 
(\om \in 
{\cal E}^{\bul}
\otimes_{{\cal O}_{{\cal P}^{\rm ex}_{\bul}}}
{\Om}^i_{{\cal P}^{\rm ex}_{\bul}/\os{\circ}{S}}).
\end{align*} 
Then we have the following 
morphism of triangles:
\begin{equation*}
\begin{CD}
@>>> 
A_{\rm conv}(X/S,E)[-1] @>{}>>B  \\
@. @A{(\theta_{X/S}\wedge *)[-1]}AA  
@AA{\mu_{{X/\os{\circ}{S}}}}A
\\
@>>> 
R\eps^{\rm conv}_{X/S*}(\eps^{{\rm conv}*}_{X/S}(E))[-1]   
@>{\theta \wedge}>> \wt{R}\eps^{\rm conv}_{X/S*}(\eps^{{\rm conv}*}_{X/S}(E))
\end{CD}
\tag{13.2.15}\label{cd:sgsctexn}
\end{equation*} 
\begin{equation*}
\begin{CD}  
@>>> A_{\rm conv}(X/S,E)
@>{+1}>>  \\  
@. @A{\theta_{X/S} \wedge}AA \\ 
@>>> R\eps^{\rm conv}_{X/S*}(\eps^{{\rm conv}*}_{X/S}(E))
@>{+1}>>.\\
\end{CD}
\end{equation*}

\parno 
Hence we obtain the following as in \cite[p.~246]{st1} 
and \cite[(11.10)]{ndw}:

\begin{prop}\label{prop:cmzoqm}
The zariskian monodromy operator
\begin{align*} 
N \col &
R\eps^{\rm conv}_{X/S*}
(\eps^*_{X/S}(E))\lo  
R\eps^{\rm conv}_{X/S*}
(\eps^*_{X/S}(E))(-1,u)
\tag{13.3.1}\label{eqn:minbv}\\
\end{align*}  
is equal to  
$$\nu \col 
A_{\rm conv}(X/S,E) 
\lo A_{\rm conv}(X/S,E)(-1,u)$$ 
via the isomorphism {\rm (\ref{ali:pi})}. 
\end{prop}

\begin{coro}\label{coro:nncilp} 
Assume that $\os{\circ}{X}$ is quasi-compact. 
Then  the zariskian monodromy operator 
\begin{align*} 
N \col &  R\eps^{\rm conv}_{X/S*}(\eps^*_{X/S}(E))\lo  R\eps^{\rm conv}_{X/S*}(\eps^*_{X/S}(E))(-1,u)
\tag{13.4.1}\label{ali:mdmycoh} \\
\end{align*} 
is nilpotent.  
\end{coro}
\begin{proof}
(\ref{coro:nncilp}) immediately follows from 
(\ref{prop:cmzoqm}) since $\nu$ is nilpotent.
\end{proof}

Next we give two conjectures as in \cite{nb}.

\begin{conj}[{\bf Variational convergent monodromy-weight  conjecture}]
\label{conj:remc}  
Assume that $\os{\circ}{X} \lo \os{\circ}{S}_1$ is projective. 
Let $k$ and $q$ be nonnegative integers. 
Then we conjecture that the induced morphism 
\begin{equation*} 
N^k \col {\rm gr}^P_{q+k}R^qf_{X/S*}({\cal K}_{X/S}) 
\lo {\rm gr}^P_{q-k}R^qf_{X/S*}({\cal K}_{X/S})(-k,u)  
\tag{13.5.1}\label{eqn:pgme} 
\end{equation*}
by the monodromy operator is an isomorphism.  
\end{conj}

\par 
Let the assumption be as in (\ref{conj:remc}). 
Assume that the relative dimension of $\os{\circ}{X} \lo \os{\circ}{S}_1$ is of pure dimension $d$. 
Let $L$ be a relatively ample line bundle on 
$\os{\circ}{X}/\os{\circ}{S}_1$.  
Let ${\cal J}_{X/S}$ be the ``defining ideal sheaf'' in $(X/S)_{\rm conv}$: 
$\Gam((U,T),{\cal J}_{X/S}):={\rm Ker}({\cal K}_T\lo {\cal K}_U)$ for 
an object $(U,T)\in {\rm Conv}(X/S)$. 
Let $i\col X_{\rm zar} \lo (X/S)_{\rm conv}$ 
be the morphism of topoi defined by 
$i_*(F)((U,T)):=F(U)$ for an object $(U,T)$ in ${\rm Conv}(X/S)$. 
Then we have the following exact sequence 
\begin{align*} 
0\lo 1+{\cal J}_{X/S}\lo {\cal O}^*_{X/S}\lo i_*({\cal O}_X^*)\lo 0. 
\end{align*} 
Hence we have the following boundary morphism 
\begin{align*} 
R^1f_{*}({\cal O}^*_X) \lo R^2f_{X/S*}(1+{\cal J}_{X/S})(1). 
\end{align*} 
Composing the induced morphism 
\begin{align*}
R^2f_{X/S*}(1+{\cal J}_{X/S})(1) \lo R^2f_{X/S*}({\cal K}_{X/S})(1)
\end{align*} 
by the following logarithm morphism 
\begin{align*}
\log \col  1+{\cal J}_{X/S}\lo {\cal K}_{X/S}, 
\end{align*} 
we obtain the following composite morphism (cf.~\cite[\S3]{boi}): 
\begin{align*}
c_{1,{\rm conv}}\col R^1f_{*}({\cal O}^*_X) \lo R^2f_{X/S*}(1+{\cal J}_{X/S})(1) 
\lo  R^2f_{X/S*}({\cal K}_{X/S})(1). 
\end{align*} 
Here note that 
$\log(1+a):=\sum_{n=0}^{\infty}(-1)^na^n/n$ is 
defined for a local section $a$ of ${\cal J}_{X/S}$ since the topology of 
$T$ is $p$-adic for an enlargement $(U,T)$ of $X/S$.

Set $\eta=c_{1,{\rm conv}}(L)\in R^2f_{X/S*}({\cal K}_{X/S})(1)$. 
The variational filtered log $p$-adic hard Lefschetz conjecture
is the following: 

\begin{conj}[{\bf Variational filtered 
convergent log hard Lefschetz conjecture}]\label{conj:lhilc}
Let $L$ be a relatively ample line bundle on $X/S_1$. Set $\eta:=c_{1,{\rm conv}}(L)$. 
Then we conjecture the following$:$ 
\par 
$(1)$ The following cup product 
\begin{equation*} 
\eta^i \col 
R^{d-i}f_{X/S*}({\cal K}_{X/S})\lo R^{d+i}f_{X/S*}({\cal K}_{X/S})(i)
\tag{13.6.1}\label{eqn:fcvilpl} 
\end{equation*}
is an isomorphism. 
\par 
$(2)$ 
In fact, $\eta^i$ is the following isomorphism of filtered sheaves: 
\begin{equation*} 
\eta^i \col 
(R^{d-i}f_{X/S*}({\cal K}_{X/S}),P) 
\os{\sim}{\lo} (R^{d+i}f_{X/S*}({\cal K}_{X/S})(i),P). 
\tag{13.6.2}\label{eqn:filfilpl} 
\end{equation*}
Here $(i)$ means that 
$$P_k(R^{d+i}f_{X/S*}({\cal K}_{X/S})(i))
=P_{k+2i}R^{d+i}f_{X/S*}({\cal K}_{X/S}).$$ 
\end{conj}

In the next section we give results for these conjectures after giving a comparison theorem 
on the relative log convergent cohomology and the relative log crystalline cohomology of $X/S$. 

\section{Comparison theorem}\label{sec:ct}
In this section we compare $(A_{\rm iso{\textrm -}zar},P)$ 
with $(A_{\rm zar},P)\otimes_{\mab Z}^L{\mab Q}$ 
in the case where $\pi{\cal V}$ has a PD-strucutre $\gam$. 
(In the case of the trivial coefficient, it is not necessary to assume that 
$\pi{\cal V}$ has a PD-strucutre $\gam$.) 
\par 
Let ${\cal V}$, $\pi$, $S$ and $S_1$ 
be as in \S\ref{sec:logcd}. 
Assume that $\pi{\cal V}$ has a PD-strucutre $\gam$. 
Because we have assumed that $S$ is flat over ${\cal V}$, 
$\gam$ extends to $S$.  
By abuse of notation we denote this extended PD-structure by $\gam$ again. 
Let $f \col Y\lo S_1$ be a fine log scheme over $S_1$.  
Then we have the log crystalline topos 
$({Y/S})_{\rm crys}$ of $Y/(S,\pi{\cal O}_S,\gam)$.  
Let 
$$u^{\rm crys}_{Y/S} \col (({Y/S})_{\rm crys},{\cal O}_{Y/S}) 
\lo  (Y_{\rm zar},f^{-1}({\cal O}_S))$$ 
be the canonical projection.  
To prove the comparison theorem, we recall the following, 
which is a special case of \cite[(11.4)]{nhw} and a relative version of  
\cite[Theorem 7.7]{oc} and \cite[Theorem 3.1.1]{s2}: 

\begin{theo}[{\bf Comparison theorem}]\label{theo:hct} 
Let $Y$ be a log smooth log scheme over $S_1$. 
Denote by 
\begin{equation*}
\Xi \col {\rm Isoc}_{\rm conv}(Y/S)
(:={\rm Isoc}_{\rm conv, zar}(Y/S)) 
\lo {\rm Isoc}_{\rm crys}(Y/S)
\tag{14.1.1}\label{eqn:icym}
\end{equation*} 
the relative version of 
the functor defined in {\rm \cite[Theorem 5.3.1]{s1}}
and denoted by $\Phi$ in $[${\rm loc.~cit.}$]$. 
Let $E$ be an object of 
${\rm Isoc}_{\rm conv}(Y/S)$. 
Then there exists a functorial isomorphism 
\begin{equation*}
Ru^{\rm conv}_{Y/S*}(E)
\os{\sim}{\lo}
Ru^{\rm crys}_{Y/S*}(\Xi(E)).
\tag{14.1.2}\label{eqn:kymc}  
\end{equation*} 
\end{theo}

\par  
Let $X$ be an SNCL scheme over $S_1$. 
Let $f\col X\lo S$ be the structural morphism. 
To give the comparison theorem, we have to give a generalization of 
$(A_{\rm zar}(X/S,F),P)$ 
for a more generalized sheaf than a quasi-coherent crystal $F$ on the crystalline site 
${\rm Crys}(\os{\circ}{X}/\os{\circ}{S})$ defined in \cite{nb}. 
Let $X_{\bul}$ be a simplicial log scheme obtained by an affine open covering of 
$X$ and  a simplicial immersion 
$X_{\bul} \os{\sus}{\lo} \ol{\cal P}_{\bul}$ into a formally log smooth log formal scheme 
over $\ol{S}$. Let $\ol{\mathfrak D}_{\bul}$ be the log PD-envelope of 
this immersion over $(\os{\circ}{S},p{\cal O}_S,[~])$. 
Set ${\mathfrak D}_{\bul}:=\ol{\mathfrak D}_{\bul}\times_{\ol{S}}S$.  
Assume that, we are given a 
a flat quasi-coherent ${\cal K}_{{\mathfrak D}_{\bul}}$-module 
(${\cal K}_{{\mathfrak D}_{\bul}}:=
{\cal O}_{{\mathfrak D}_{\bul}}\otimes_{\mab Z}{\mab Q}$)
with integrable connection  
$({\cal E}({\mathfrak D}_{\bul}),\nabla)$: 
\begin{equation*} 
\nabla \col {\cal E}({\mathfrak D}_{\bul})\lo 
{\cal E}({\mathfrak D}_{\bul})
\otimes_{{\cal O}_{{\cal P}^{\rm ex}_{\bul}}}
\Om^1_{{\cal P}^{\rm ex}_{\bul}/\os{\circ}{S}}. 
\end{equation*} 
We endow 
${\cal E}({\mathfrak D}_{\bul})
\otimes_{{\cal O}_{{\cal P}^{\rm ex}_{\bul}}}
\Om^{\bul}_{{\cal P}^{\rm ex}_{\bul}/\os{\circ}{S}}$ 
with the tensor product $P$ of the trivial filtration on 
${\cal E}({\mathfrak D}_{\bul})$ and 
the (pre)weight filtration on $\Om^{\bul}_{{\cal P}^{\rm ex}_{\bul}/\os{\circ}{S}}$. 
Then, as in \cite{nb}, we have the filtered double complex 
$(A_{\rm zar}({\cal P}^{\rm ex}_{\bul}/S,
{\cal E}({\mathfrak D}_{\bul}))^{\bul \bul},P)$ as follows: 
\begin{align*} 
A_{\rm zar}({\cal P}^{\rm ex}_{\bul}/S,{\cal E}({\mathfrak D}_{\bul}))^{ij}
& :={\cal E}({\mathfrak D}_{\bul})\otimes_{{\cal O}_{{\cal P}^{\rm ex}_{\bul}}}
{\Om}^{i+j+1}_{{\cal P}^{\rm ex}_{\bul}/\os{\circ}{S}}/P_j 
\tag{14.1.3}\label{cd:accef} \\
& :={\cal E}({\mathfrak D}_{\bul})\otimes_{{\cal O}_{{\cal P}^{\rm ex}_{\bul}}}
{\Om}^{i+j+1}_{{\cal P}^{\rm ex}_{\bul}/\os{\circ}{S}}/
P_j({\cal E}({\mathfrak D}_{\bul})\otimes_{{\cal O}_{{\cal P}^{\rm ex}_{\bul}}}
{\Om}^{i+j+1}_{{\cal P}^{\rm ex}_{\bul}/\os{\circ}{S}})  
\quad (i,j \in {\mab N}). 
\end{align*}   
We consider the following boundary morphisms of 
double complexes: 
\begin{equation*}
\begin{CD}
A_{\rm zar}({\cal P}^{\rm ex}_{\bul}/S,{\cal E}({\mathfrak D}_{\bul}))^{i,j+1}  @.  \\ 
@A{\theta \wedge}AA  @. \\
A_{\rm zar}({\cal P}^{\rm ex}_{\bul}/S,{\cal E}({\mathfrak D}_{\bul}))^{ij}
@>{-\nabla}>> 
A_{\rm zar}({\cal P}^{\rm ex}_{\bul}/S,{\cal E}({\mathfrak D}_{\bul}))^{i+1,j}.\\
\end{CD}
\tag{14.1.4}\label{cd:lccbd} 
\end{equation*}  
The sheaf 
$A_{\rm zar}({\cal P}^{\rm ex}_{\bul}/S,{\cal E}({\mathfrak D}_{\bul}))^{ij}$ 
has a quotient filtration $P$ obtained by the filtration $P$ on 
${\cal E}({\mathfrak D}_{\bul})\otimes_{{\cal O}_{{\cal P}^{\rm ex}_{\bul}}}
{\Om}^{i+j+1}_{{\cal P}^{\rm ex}_{\bul}/\os{\circ}{S}}$: 
\begin{equation*}
P_kA_{\rm zar}({\cal P}^{\rm ex}_{\bul}/S,{\cal E}({\mathfrak D}_{\bul}))^{ij}
:=(P_{2j+k+1}+P_j)
({\cal E}({\mathfrak D}_{\bul})\otimes_{{\cal O}_{{\cal P}^{\rm ex}_{\bul}}}
{\Om}^{i+j+1}_{{\cal P}^{\rm ex}_{\bul}/\os{\circ}{S}})/P_j.
\tag{14.1.5}\label{eq:lpbd} 
\end{equation*} 
Let $(A_{\rm zar}({\cal P}^{\rm ex}_{\bul}/S,{\cal E}
({\mathfrak D}_{\bul})),P)$ be the single filtered complex 
associated to this double complex.  
We assume that the following three conditions (A), (B), (C) hold: 
\par 
(A) We assume that, for another simplicial immersion  
$X'_{\bul} \os{\sus}{\lo} \ol{\cal P}{}'_{\bul}$ and 
for the following commutative diagram 
\begin{equation*} 
\begin{CD} 
X_{\bul} @>{\sus}>> {\cal P}_{\bul} \\ 
@VVV @VV{g_{\bul}}V \\ 
X'_{\bul} @>{\sus}>> 
{\cal P}{}'_{\bul},   
\end{CD} 
\tag{14.1.6}\label{cd:cop}
\end{equation*} 
there exists a morphism 
\begin{equation*} 
\rho_{g_{\bul}} 
\col {\cal E}({\mathfrak D}'_{\bul})
\lo g_{\bul*}({\cal E}({\mathfrak D}_{\bul}))
\tag{14.1.7}\label{cd:cngp}
\end{equation*} 
of ${\cal O}_{{\mathfrak D}'_{\bul}}$-modules 
or ${\cal K}_{{\mathfrak D}'_{\bul}}$-modules 
satisfying the usual cocycle condition 
for the following commutative diagram: 
\begin{equation*} 
\begin{CD} 
X_{\bul} @>{\sus}>> {\cal P}_{\bul} \\ 
@VVV @VV{}V \\ 
X'_{\bul} @>{\sus}>> {\cal P}{}'_{\bul}\\
@VVV @VV{}V \\ 
X''_{\bul} @>{\sus}>> 
{\cal P}{}''_{\bul}.   
\end{CD} 
\tag{14.1.8}\label{cd:cctpt}
\end{equation*} 
Here ${\cal P}'_{\bul}$ is the analogue of 
${\cal P}_{\bul}$ for the immersion 
$X'_{\bul} \os{\sus}{\lo} \ol{\cal P}{}'_{\bul}$
and 
${\mathfrak D}'_{\bul}$ is the analogue of 
${\mathfrak D}_{\bul}$ for the immersion 
$X'_{\bul} \os{\sus}{\lo} {\cal P}'_{\bul}$. 
\par 
(B) We assume that the following morphism 
\begin{align*} 
R\pi_{{\rm zar}*}(A_{\rm zar}({\cal P}'{}^{\rm ex}_{\bul}/S,
{\cal E}({\mathfrak D}'_{\bul}))\lo 
R\pi_{{\rm zar}*}(A_{\rm zar}({\cal P}^{\rm ex}_{\bul}/S,
{\cal E}({\mathfrak D}_{\bul})) 
\end{align*}
induced by the morphisms 
$g_{\bul}\col {\cal P}_{\bul}\lo {\cal P}'_{\bul}$ 
and $\rho_{g_{\bul}}$ is an isomorphism 
in $D^+(f^{-1}({\cal O}_T))$. 
\par 
(C) We assume that, for a positive integer $k$, 
for a nonnegative integer $n$, 
the following morphism 
\begin{equation*} 
{\rm gr}^P_k({\cal E}({\mathfrak D}'_{n})
\otimes_{{\cal O}_{{\cal P}'{}^{\rm ex}_{\! \!n}}}
{\Om}^{\bul}_{{\cal P}'{}^{\rm ex}_{\! \!n}/\os{\circ}{T}})
\os{\sim}{\lo}  
Rg_{n*}({\rm gr}^P_k({\cal E}({\mathfrak D}_{n})
\otimes_{{\cal O}_{{\cal P}^{\rm ex}_{n}}}
{\Om}^{\bul}_{{\cal P}^{\rm ex}_{n}/\os{\circ}{S}}))
\tag{14.1.9}\label{eqn:ptep}
\end{equation*} 
induced by the morphisms 
$g_{\bul}\col {\cal P}_{\bul}\lo {\cal P}'_{\bul}$ 
and $\rho_{g_{\bul}}$ is an isomorphism in $D^+(f^{-1}_{n}({\cal O}_T))$. 
\bigskip
\parno 
Then, by the same standard proof as that of the well-definedness of 
$(A_{\rm zar}(X/S,F),P)$ in \cite{nb},  
we can prove that the filtered complex 
$R\pi_{{\rm zar}*}
((A_{\rm zar}({\cal P}^{\rm ex}_{\bul}/S,
{\cal E}({\mathfrak D}_{\bul})),P))$ 
is a well-defined filtered complex.

\par  
Now let 
$E$ be an object of ${\rm Isoc}_{\rm conv}(\os{\circ}{X}/\os{\circ}{S})$.  
Let $(\ol{\cal E}{}^{\bul},\ol{\nabla})=\{(\ol{\cal E}{}^{\bul}_n,\ol{\nabla})\}_{n=1}^{\infty}$ 
$(\ol{\nabla}\col \ol{\cal E}{}^{\bul}_{n+1}\lo \ol{\cal E}{}^{\bul}_n
\otimes_{{\cal O}_{\ol{\cal P}_{\bul}}}\Om^1_{\ol{\cal P}_{\bul}/\os{\circ}{S}}$)
be the integrable connection obtained by 
$\eps^{{\rm conv}*}_{X/\os{\circ}{S}}(E)$ and the simplicial immersion 
$X_{\bul} \os{\sus}{\lo} \ol{\cal P}_{\bul}$. 
By the universality of the log enlargement $\{{\mathfrak T}_{X_{\bul},n}(\ol{\cal P}_{\bul})\}_{n=1}^{\infty}$,
there exists an $n(\bul)$ (depending on $\bul$ 
such that the morphism $\ol{\mathfrak D}_{\bul}\lo \ol{\cal P}_{\bul}$ 
factors through  a morphism 
$\ol{\lam}{}_{n(\bul)} \col \ol{\mathfrak D}_{\bul}\lo {\mathfrak T}_{X_{\bul},n}(\ol{\cal P}_{\bul})$. 
Hence we obtain the following integrable connection: 
\begin{align*} 
\ol{\lam}{}^*_{n(\bul)+1}(\ol{\cal E}{}^{\bul}_{n(\bul)+1})\lo  
\ol{\lam}{}^*_n(\ol{\cal E}{}^{\bul}_{n(\bul)}) \otimes_{{\cal O}_{\ol{\cal P}{}^{\rm ex}_{\bul}}}
\Om^1_{\ol{\cal P}{}^{\rm ex}_{\bul}/\os{\circ}{S}}.
\end{align*} 
Since $\ol{\lam}{}^*_n(\ol{\cal E}{}^{\bul}_{n(\bul)})$ is independent of for large $n(\bul)$'s, 
we set  $\ol{\lam}{}^*(\ol{\cal E}):=\ol{\lam}{}^*_{n(\bul)}(\ol{\cal E}{}^{\bul}_{n(\bul)})$ 
for a large $n(\bul)$. 
Then we have the following integrable connection: 
\begin{align*} 
\ol{\lam}{}^*(\ol{\cal E}{}^{\bul})\lo  
\ol{\lam}{}^*(\ol{\cal E}{}^{\bul}) \otimes_{{\cal O}_{\ol{\cal P}{}^{\rm ex}_{\bul}}}
\Om^1_{\ol{\cal P}{}^{\rm ex}_{\bul}/\os{\circ}{S}}.
\end{align*} 
The morphism $\ol{\lam}_{n(\bul)}
\col \ol{\mathfrak D}_{\bul}\lo \ol{\cal P}_{\bul}$ 
induces a morphism 
$\lam_{n(\bul)} \col {\mathfrak D}_{\bul}\lo {\mathfrak T}_{X_{\bul},n(\bul)}({\cal P}_{\bul})$. 
Set ${\cal E}{}^{\bul}_{n(\bul)}
:=\ol{\cal E}{}^{\bul}_{n(\bul)}\otimes_{{\cal K}_{{\mathfrak T}_{X_{\bul},n(\bul)}(\ol{\cal P}_{\bul})}}
{\cal K}_{{\mathfrak T}_{X_{\bul},n(\bul)}({\cal P}_{\bul})}$ and 
${\lam}{}^*({\cal E}):={\lam}{}^*_{n(\bul)}(\ol{\cal E}{}^{\bul}_{n(\bul)})$ for a large $n(\bul)$.  
Then we have the following integrable connection: 
\begin{align*} 
\nabla \col \lam^*({\cal E}{}^{\bul})\lo  
\lam^*({\cal E}{}^{\bul}) \otimes_{{\cal O}_{{\cal P}{}^{\rm ex}_{\bul}}}
\Om^1_{{\cal P}{}^{\rm ex}_{\bul}/\os{\circ}{S}}.
\tag{14.1.10}\label{ali:pl}
\end{align*} 

\begin{lemm}\label{lemm:alp}
The connection {\rm (\ref{ali:pl})} satisfies the three conditions {\rm (A), (B), (C)}.  
\end{lemm}
\begin{proof}
The condition (A) hold since it holds for the connection 
\begin{align*} 
\nabla \col {\cal E}{}^{\bul}\lo  {\cal E}{}^{\bul} 
\otimes_{{\cal O}_{{\cal P}{}^{\rm ex}_{\bul}}}\Om^1_{{\cal P}{}^{\rm ex}_{\bul}/\os{\circ}{S}}.
\end{align*} 
\par 
By (\ref{eqn:kymc}) and (\ref{eqn:uybo}), 
we obtain the following commutative diagram 
\begin{equation*}
\begin{CD}
R\pi^{\rm zar}_{X/S*}(\vpl_n\bet_{n*}({\cal E}^{\bul}_n{\otimes}_{{\cal O}_{{\cal P}^{\rm ex}_{\bul}}} 
\Om^{\bul}_{{\cal P}^{\rm ex}_{\bul}/S}))
@>>>
R\pi^{\rm zar}_{X/S*}(\lam^*({\cal E}^{\bul})
{\otimes}_{{\cal O}_{{\cal P}^{\rm ex}_{\bul}}} 
\Om^{\bul}_{{\cal P}^{\rm ex}_{\bul}/S})\\
@| @|\\
Ru^{\rm conv}_{X/S*}(\eps^{{\rm conv}*}_{X/S}(E))
@>{\sim}>>
Ru^{\rm crys}_{X/S*}(\Xi(\eps^{{\rm conv}*}_{X/S}(E))).
\end{CD}
\tag{14.2.1}\label{eqn:kyemc}  
\end{equation*} 
By the same proof as that of (\ref{ali:pi}), we also have the following commutative diagram: 
\begin{equation*}
\begin{CD}
A_{{\rm iso}{\textrm -}{\rm zar}}(X/S,E)
@>>>R\pi^{\rm zar}_{X/S*}
(A_{\rm zar}({\cal P}^{\rm ex}_{\bul}/S,\lam^*({\cal E}^{\bul})))
\\
@A{\theta \wedge}AA @AA{\theta \wedge}A\\
R\pi^{\rm zar}_{X/S*}(\vpl_n\bet_{n*}({\cal E}^{\bul}_n{\otimes}_{{\cal O}_{{\cal P}^{\rm ex}_{\bul}}} 
\Om^{\bul}_{{\cal P}^{\rm ex}_{\bul}/S}))
@>{\sim}>>
R\pi^{\rm zar}_{X/S*}(\lam^*({\cal E}^{\bul}){\otimes}_{{\cal O}_{{\cal P}^{\rm ex}_{\bul}}} 
\Om^{\bul}_{{\cal P}^{\rm ex}_{\bul}/S}).
\end{CD}
\tag{14.2.2}\label{eqn:kyapc}  
\end{equation*} 
Hence the condition (B) holds. 
\par 
We also have the following commutative diagram by using the Poincar\'{e} residue isomorphism:  
\begin{equation*}
\begin{CD}
{\rm gr}_k^PA_{{\rm iso}{\textrm -}{\rm zar}}(X/S,E)@=\\
@VVV \\
R\pi_{{\rm zar}*}({\rm gr}_k^PA_{\rm zar}({\cal P}^{\rm ex}_{\bul}/S,\lam^*({\cal E}^{\bul})))
@>{\sim}>>
\end{CD}
\tag{14.2.3}\label{cd:yapc}  
\end{equation*} 
\begin{equation*}
\begin{CD}
\us{j\geq \max \{-k,0\}}{\bigoplus}
R\pi_{{\rm zar}*}(\vpl_n\bet_{n*}({\cal E}^{\bul}_n{\otimes}_{{\cal O}_{{\cal P}^{\rm ex}_{\bul}}} 
\Om^{\bul}_{\os{\circ}{\cal P}{}^{{\rm ex},(2j+k)}_{\bul}/\os{\circ}{S}}\otimes_{\mab Z}
\vp^{(2j+k)}(\os{\circ}{\cal P}{}^{{\rm ex},(2j+k)}_{\bul}/\os{\circ}{S})))[-2j-k])\\
@VVV\\
\us{j\geq \max \{-k,0\}}{\bigoplus}
R\pi_{{\rm zar}*}(\lam^*({\cal E}^{\bul}){\otimes}_{{\cal O}_{{\cal P}^{\rm ex}_{\bul}}} 
\Om^{\bul}_{\os{\circ}{\cal P}{}^{{\rm ex},(2j+k)}_{\bul}/\os{\circ}{S}}
\otimes_{\mab Z}
\vp^{(2j+k)}(\os{\circ}{\cal P}{}^{{\rm ex},(2j+k)}_{\bul}/\os{\circ}{S}))[-2j-k]).
\end{CD}
\end{equation*} 
Because $\os{\circ}{X}{}^{(k)}/\os{\circ}{S}$ is smooth, 
we can define the following functor by (\ref{eqn:icym}): 
\begin{equation*}
\Xi^{(k)} \col {\rm Isoc}_{\rm conv}(\os{\circ}{X}{}^{(k)}/\os{\circ}{S})
\lo {\rm Isoc}_{\rm crys}(\os{\circ}{X}{}^{(k)}/\os{\circ}{S}). 
\end{equation*} 
Because the right hand sides of (\ref{cd:yapc}) 
are isomorphic to 
$$\bigoplus_{j\geq \max \{-k,0\}} 
a^{(2j+k)}_* 
(Ru^{\rm conv}_{\os{\circ}{X}{}^{(2j+k)}/\os{\circ}{S}*}
(a^{(2j+k)*}_{{\rm conv}}(E)) 
\otimes_{\mab Z}\vp_{\rm conv}^{(2j+k)}(\os{\circ}{X}/\os{\circ}{S})))$$
and 
$$\bigoplus_{j\geq \max \{-k,0\}} 
a^{(2j+k)}_* 
(Ru^{\rm crys}_{\os{\circ}{X}{}^{(2j+k)}/\os{\circ}{S}*}
(\Xi^{(2j+k)}(a^{(2j+k)*}_{{\rm conv}}(E))
\otimes_{\mab Z}\vp_{\rm crys}^{(2j+k)}(\os{\circ}{X}/\os{\circ}{S}))),$$ 
respectively, 
the right vertical morphism is an isomorphism 
by the trivial log version of (\ref{theo:hct}). 
Hence the condition (C) holds. 
\end{proof}

\begin{defi} 
Let $L$ be a crystal on ${\rm Crys}(X/S)$ 
such that $K\otimes_{\cal V}L=\Xi(\eps^{*{\rm conv}}_{X/S}(E))$. 
We set 
\begin{align*} 
(A_{{\rm iso{\textrm -}zar}}(X/S,\os{\circ}{L}),P)
:=
R\pi_{{\rm zar}*}((A_{\rm zar}({\cal P}^{\rm ex}_{\bul}/S,
\lam^*({\cal E})),P)).
\end{align*} 
Here note that we do not define $\os{\circ}{L}$. 
\end{defi}

\begin{coro}\label{coro:fgr}
There exists the following spectral sequence 
\begin{align*} 
E_1^{-k,q+k}=&\bigoplus_{j\geq \max \{-k,0\}} 
R^{q-2j-k}f^{\rm crys}_{\os{\circ}{X}{}^{(2j+k)}/\os{\circ}{S}*}(\Xi^{(2j+k)}(a^{(2j+k)*}_{{\rm conv}}(E))
\otimes_{\mab Z}\vp_{\rm crys}^{(2j+k)}(\os{\circ}{X}/\os{\circ}{S}))(-j-k)
\tag{14.4.1}\label{ali:sxa}\\
&\Lo R^qf^{\rm crys}_{X/S*}(\eps_{X/S}^{{\rm crys}*}(\Xi(E)))
\end{align*} 
\end{coro}
\begin{proof}
We have already proved this in the proof of (\ref{lemm:alp}). 
\end{proof}

\begin{coro}\label{coro:cgrss} 
The spectral sequences $(\ref{ali:wtfapwt})$ and 
$(\ref{ali:sxa})\otimes_{\mab Z}{\mab Q}$ 
are canonically isomorphic. 
\end{coro}



\begin{theo}[{\bf Comparison theorem of $p$-adic isozariskian 
and zariskian Steenbrink complexes}]\label{theo:ccco} 
{\rm (\ref{theo:izz})} holds. 
\end{theo}
\begin{proof} 
We have already proved (\ref{theo:ccco}) 
in the proof of (\ref{lemm:alp}). 
\end{proof}

\begin{prop}[{\bf Base change theorem}]\label{prop:bchange} 
Let ${\cal V}'$, $K'$, $\pi'$, $S'$ and $S_1'$ 
be as in {\rm \S\ref{sec:fpw}}.  
Let  
\begin{equation*}
\begin{CD} 
S' @>{u}>> S \\
@VVV @VVV \\
{\rm Spf}({\cal V}')  @>>> {\rm Spf}({\cal V})
\end{CD}
\tag{14.7.1}\label{cd:bbsmlgs}
\end{equation*}
be  a commutative diagram of $p$-adic formal families of log points. 
Let
\begin{equation*}
\begin{CD}
X'@>{g}>> X\\
 @V{f'}VV  @VV{f}V \\
S'_1@>{u}>> S_1
\end{CD}
\end{equation*}
be a cartesian diagram of SNCL schemes.  
Assume that $\os{\circ}{f}$ is quasi-compact and quasi-separated.  
Then the base change morphism 
\begin{equation*}
{\cal K}_{S'}\otimes^L_{{\cal K}_S}(Rf^{\rm conv}_{X/S*}(\eps^{{\rm conv}*}_{X/S}(g^*_{\rm conv}(E))),P)
\os{\sim}{\lo} (Rf^{\rm conv}_{X'/S*}(\eps^{{\rm conv}*}_{X'/S'}(E)),P)
\tag{14.7.2}\label{eqn:simbct}
\end{equation*} 
is an isomorphism in the filtered derived category  
${\rm DF}({\cal K}_{Y'/S'})$. 
\end{prop}
\begin{proof}
This follows from (\ref{eqn:cad}) and the log crystalline base change theorem 
and the classical crystalline base change theorem and 
the proof of (\ref{lemm:alp}) and the descending induction on the number of the filtration $P$. 
\end{proof}

\begin{rema}
By (\ref{prop:iekp}) we can replace $p$ by $\pi$ for the case where $E$ is trivial. 
\end{rema}

Henceforth we consider the case of the trivial coefficient. 
Using (\ref{eqn:cad}) for the trivial coefficient, we obtain the following:

\begin{coro}
$(1)$ If the crystalline analogue of {\rm (\ref{conj:remc})} holds, 
then {\rm (\ref{conj:remc})} holds.   
\par 
$(2)$ If the crystalline analogue of {\rm (\ref{conj:lhilc})} holds, 
then {\rm (\ref{conj:lhilc})} holds.   
\end{coro}

\begin{rema}
As is well-known, the crystalline analogues of (\ref{conj:remc}) and (\ref{conj:lhilc})
and in some cases. See  \cite{nb} and the references in [loc.~cit.].  
\end{rema}

\begin{theo}[{\bf $E_2$-degeneration}]\label{theo:ssd}
The spectral sequence $(\ref{ali:wtfapwt})$ degenerates at $E_2$ if $\os{\circ}{X}$ is proper over 
$\os{\circ}{S}_1$ and if $E$ is the trivial coefficient. 
\end{theo}
\begin{proof}  
In \cite{nh3}, by the same proof as that of 
\cite[(2.17.2)]{nh2}, 
we have proved the $E_2$-degeneration of 
$(\ref{ali:ixsxa})\otimes_{\mab Z}{\mab Q}$ 
when $\os{\circ}{X}$ is proper over $\os{\circ}{S}_1$. 
Hence (\ref{theo:ssd}) follows from (\ref{eqn:cad}) for the trivial coefficient. 
\end{proof}

Let us recall the notation (\ref{ali:ccwez}). 
Set 
\begin{align*} 
(Rf_{X/S*}(\eps^{{\rm conv}*}_{X/S}(E)),P):=
Rf_*(Ru^{\rm conv}_{X/S*}(\eps^{{\rm conv}*}_{X/S}(E)),P). 
\end{align*}

\begin{theo}[{\bf Strict compatibility}]\label{theo:stcom}  
Let $f \col X \lo S_1$ and $f' \col X'\lo S_1$ be proper SNCL schemes over $S_1$.
Let $g\col X' \lo X$ be a morphism of simplicial log 
schemes over $S_1$. 
Let $q$ be an integer. 
Then the induced morphism 
\begin{equation*} 
g^* \col 
R^qf^{\rm conv}_{X/S*}
({\cal K}_{X/S})
\lo R^qf^{\rm conv}_{X'/S*}
({\cal K}_{X'/S}) 
\tag{14.12.1}\label{eqn:hfcds}
\end{equation*}
is strictly compatible with the weight filtration.
\end{theo}
\begin{proof} 
In \cite{nb} we have proved the strict compatibility of $g^*$ 
with respect to the weight filtration in the case of 
log crystalline cohomologies.  
Now (\ref{theo:stcom}) follows from the comparison theorem 
(\ref{theo:hct}). 
\end{proof}

\bigskip
\bigskip
\parno
Yukiyoshi Nakkajima 
\parno
Department of Mathematics,
Tokyo Denki University,
5 Asahi-cho Senju Adachi-ku,
Tokyo 120--8551, Japan. 
\parno
{\it E-mail address\/}: nakayuki@cck.dendai.ac.jp
\end{document}